\documentclass[10pt]{article}
\usepackage{amssymb,amsfonts,amsmath, amsthm, amsbsy}
\usepackage{color,graphicx, subcaption, placeins, array, mathtools, url, color, rotating} 
\usepackage[text={6.75in,9.5in},centering,letterpaper]{geometry}
\usepackage[hypertexnames=false,colorlinks=true,urlcolor=blue,linkcolor=blue,citecolor=blue]{hyperref}
\usepackage{tikz}
\usepackage[shortlabels]{enumitem}
\usepackage[margin = 10pt, font=small, labelfont=bf, labelsep = endash]{caption}

\everymath={\displaystyle}
 \numberwithin{equation}{section}

\newtheorem{Theorem}{Theorem}[section]

\newtheorem{Lemma}[Theorem]{Lemma}
\newtheorem{Proposition}[Theorem]{Proposition}

\newtheorem{Remark}[Theorem]{Remark}

\newcommand{\R}{\mathbb{R}}
\newcommand{\C}{\mathbb{C}}

\def\Re{\mathop{\mathrm{Re}}}

\newcommand{\rmd}{\mathrm{d}}

\renewcommand{\ker}{\mathrm{Ker}\,}

\newcommand{\ad}{\mathrm{ad}}

\newcommand{\eps}{\varepsilon}

\newcommand{\h}{\mathcal{H}}

\title{Stable planar vegetation stripe patterns on sloped terrain in dryland ecosystems}
\author{Robbin Bastiaansen\thanks{Mathematisch Instituut, Universiteit Leiden, Niels Bohrweg 1, 2333CA Leiden, The Netherlands} \and Paul Carter\thanks{Department of Mathematics, University of Arizona, 617 N Santa Rita Ave, Tucson, AZ 85721, USA} \and Arjen Doelman\footnotemark[1]}

\begin{document}
\maketitle

\begin{abstract}
In water-limited regions, competition for water resources results in the formation of vegetation patterns; on sloped terrain, one finds that the vegetation typically aligns in stripes or arcs. We consider a two-component reaction-diffusion-advection model of Klausmeier type describing the interplay of vegetation and water resources and the resulting dynamics of these patterns. We focus on the large advection limit on constantly sloped terrain, in which the diffusion of water is neglected in favor of advection of water downslope. Planar vegetation pattern solutions are shown to satisfy an associated singularly perturbed traveling wave equation, and we construct a variety of  traveling stripe and front solutions using methods of geometric singular perturbation theory. In contrast to prior studies of similar models, we show that the resulting patterns are spectrally stable to perturbations in two spatial dimensions using exponential dichotomies and Lin's method. We also discuss implications for the appearance of curved stripe patterns on slopes in the absence of terrain curvature.
\end{abstract}

\section{Introduction}

Large parts of earth have an arid climate (deserts) with low mean annual precipitation and little to no vegetation; even larger parts of earth have a semi-arid climate with somewhat more precipitation, which allows (some) vegetation to grow. However, human pressure and global climate change have been turning semi-arid climates into arid climates, with severe consequences for life in these areas~\cite{UNCCDreport, gowda2018signatures}. This so-called desertification process has been studied extensively over the years, from both ecological and mathematical perspectives. These studies have shown the importance and omnipresence of spatial patterning of vegetation, which is widely recognized as the first step in the desertification process~\cite{bastiaansendata, rietkerk2008regular, gowda2018signatures, rietkerk2004self, noy1975stability, may1977thresholds, rietkerk1997site, gowda2014transitions}. On flat ground, the reported patterns are gaps, labyrinths and spots, while on sloped terrain, (curved) banded or striped patterns can form~\cite{von2001diversity, Rietkerk2002, deblauwe2011environmental, gandhi2018influence}; this article is focused on the latter, and in particular the stabilizing effect of terrain slope on striped vegetation patterns.

To understand the formation and dynamics of vegetation patterns in semi-arid climates, many conceptual models have been formulated~\cite{klausmeier1999regular, von2001diversity, Rietkerk2002, gilad2004ecosystem}. All of these dryland models describe the interplay between the available water and the concentration of vegetation, in different levels of detail. The simplest models only have two components: $U$, the water in the system and $V$, the vegetation. These two-component models generally have the following (rescaled) form:
\begin{equation}
	\begin{cases}
		U_t & = D \Delta U + S U_x + a - U - G(U,V)V, \\
		V_t & = \Delta V - m V + R(V)G(U,V)V.
	\end{cases}\label{eq:twoComponentDrylandModel}
\end{equation}
In~\eqref{eq:twoComponentDrylandModel}, the movement of water is modeled as a combined effect of diffusion ($D \Delta U$) and advection ($SU_x$), where $D$ is the diffusion constant and $S$ is a measure for the slope of the terrain. We assume the terrain is constantly sloped, so that uphill corresponds to the positive $x$ direction. The dispersal of plants is described by diffusion ($\Delta V$). The reaction terms describe the change in water due to rainfall ($+a$), evaporation of water ($-U$) and uptake by plants ($-G(U,V)V$). Simultaneously, the change of plant biomass is due to mortality ($-mV$) and plant growth ($R(V)G(U,V)$). 

In this formulation, $G$ and $R$ are functions that describe, respectively, the amount of water that is taken up by the plant's roots and the density-dependent growth rate of the vegetation. Because the presence of vegetation increases the soil's permeability, $G$ is typically assumed to increase with both $U$ and $V$. The conversion rate $R$ is decreasing with $V$ and for a specific $V^* > 0$ we have $R(V^*) = 0$. This value, $V^*$, is called the carrying capacity of the system and describes the total concentration of vegetation that can be supported at a certain location. In light of these ecological intuitions, one expects that the function $R(V)G(U,V)$ should take the from as depicted in Figure~\ref{fig:CGform} (for fixed $U$). A simple choice which satisfies these constraints is given by $R(V) = 1 - b V$ and $G(U,V) = UV$, where $1/b$ is the carrying capacity. For clarity of presentation, we fix this choice for the remainder of this paper; however, we emphasize that, with minor modifications, the following analysis can be shown to hold for a different choice of the functions $R$ and/or $G$ which take the same qualitative form. 
\begin{figure}
	\centering
	\includegraphics[width = 0.5\textwidth]{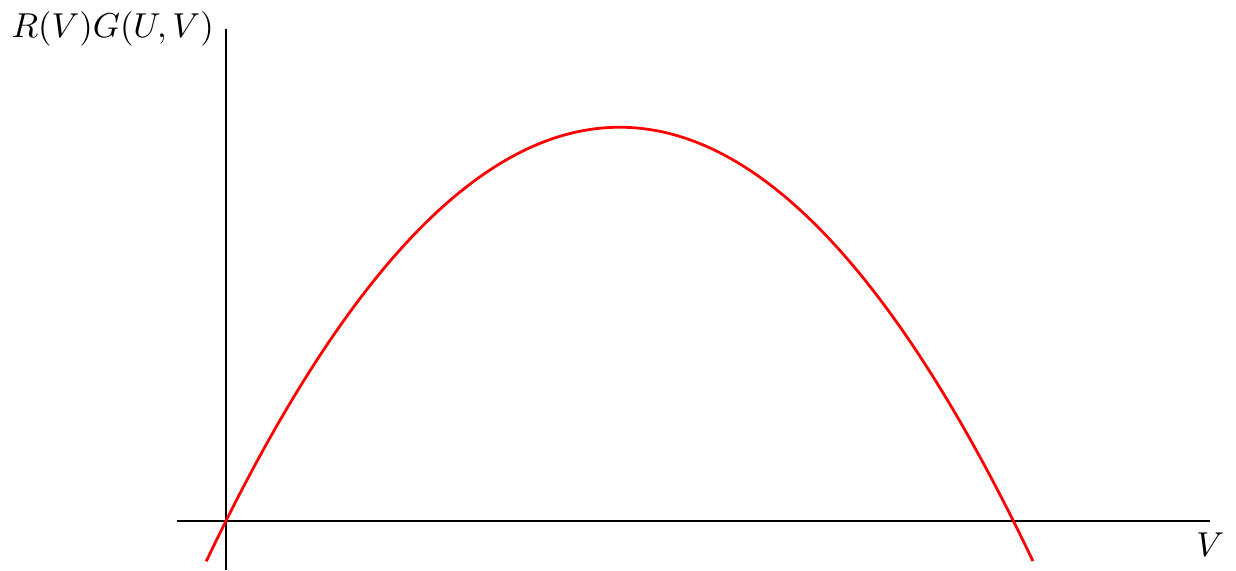}
	\caption{The qualitative form of $R(V)G(U,V)$ for fixed $U$ based on ecological intuition of dryland ecosystems.}
	\label{fig:CGform}
\end{figure}

Finally, in~\eqref{eq:twoComponentDrylandModel}, the displacement of water is modeled as a combined effect of diffusion and advection. However, in reality banded patterns are mainly observed on sloping grounds, where movement of water is dominated by the downhill flow and diffusive motion is of lesser importance~\cite{von2001diversity, Rietkerk2002, deblauwe2011environmental, gandhi2018influence}. Note that this agrees with recent studies on ecosystem models that show banded vegetation is unstable against lateral perturbations of sufficiently small wavenumber when diffusion is large enough (i.e. $D$ large enough compared to $S$)~\cite{siero2015striped, sewaltspatially}. Therefore, as a first step, we ignore the diffusion of water completely (as in~\cite{klausmeier1999regular}) and set $D = 0$. Moreover, due to the separation of scales between movement of water and dispersion of vegetation, we take $S = \frac{1}{\varepsilon}$, where $0 < \varepsilon \ll 1$ is a small parameter.

To summarize, the dryland model we consider in this article is given by
\begin{equation}
	\begin{cases}
		U_t & = \frac{1}{\varepsilon} U_x + a - U - G(U,V)V, \\
		V_t & = \Delta V - m V + R(V)G(U,V)V,
	\end{cases}\label{eq:modKlausmeier}
\end{equation}
where $a, m, b > 0$, $0 < \varepsilon \ll 1$ and the functions $R$ and $G$ are given by
\begin{equation}
	\begin{split}
		G(U,V) & = UV, \qquad R(V)  = 1 - b V
	\end{split}
\end{equation}

\begin{Remark}\label{rem:b0klaus}
Noteworthy, one of the first dryland ecosystem models, by Klausmeier~\cite{klausmeier1999regular}, takes $G(U,V) = UV$ and $R(V) = 1$. This corresponds to the assumption of infinite carrying capacity, and taking $b = 0$ in our formulation. Therefore in the limit $b \downarrow 0$ our model is the original Klausmeier model, and our model can thus be seen as a modified Klausmeier model. We emphasize, however, that the results in this article hold only for $b>0$. The limiting case $b=0$ turns out to be highly degenerate (see Remark~\ref{rem:b0degenerate}) and requires additional technical considerations; this is analyzed in detail in~\cite{CDklausmeier}. 
\end{Remark}

The model~\eqref{eq:modKlausmeier} admits a spatially homogeneous steady state
\begin{equation}\label{eq:desertState}
	(U,V) = (U_0,V_0) = (a,0),
\end{equation}
corresponding to the desert-state of the system. When $\frac{a}{m} > 2 \left(b + \sqrt{1+b^2}\right)$ there are also two additional vegetated steady state solutions, $(U_1,V_1)$ and $(U_2,V_2)$, where
\begin{equation}
\begin{split}
	U_{1,2} & = m \left( \frac{a}{m} - \frac{V_{1,2}}{1 - b V_{1,2}} \right) = m\ \frac{ \frac{a}{m} + 2 \frac{a}{m} b^2 + 2b \pm  \sqrt{(\frac{a}{m})^2 - 4 \left(1+ \frac{a}{m} b\right)}}{2(1+b^2)} ; \\
	V_{1,2} & = \frac{\frac{a}{m} \mp \sqrt{(\frac{a}{m})^2 - 4 \left(1+ \frac{a}{m} b\right)}}{2\left(1+\frac{a}{m}b\right)}.
\end{split}\label{eq:uniformSteadyStates}
\end{equation}
For $\frac{a}{m} = 2\left(b + \sqrt{1+b^2}\right)$ these two steady states coincide. The desert state, $(U_0,V_0)$ is stable against all homogeneous perturbations; the first vegetated state, $(U_1,V_1)$, is unstable against these perturbations and the last steady state, $(U_2,V_2)$, is stable if $V_2 > \frac{1}{2b}$ -- see Appendix~\ref{sec:stabilitySteadyStates}. The condition $V_2 > \frac{1}{2b}$, corresponding to $\frac{a}{m} > 4b + \frac{1}{b}$, is not strict; however the following analysis of banded vegetation patterns nonetheless restrict our results to this region.

\begin{Remark}
	Ecologically, the parameter $a$ is a measure for the rainfall and $m$ for the mortality of plants. Therefore, $\frac{a}{m}$ is a natural measure for the amount of resources needed for vegetation (patterns) to exist: if $m$ is large, vegetation dies faster and more water is needed to maintain vegetation; when $m$ is small, plants die slowly and less water is needed. Hence, $\frac{a}{m}$ is a natural bifurcation parameter. Also note that $\frac{a}{m}$ usually is taken as a small bifurcation parameter in studies of the extended-Klausmeier or generalized Klausmeier-Gray-Scott systems~\cite{van2013rise, BD2018, sewaltspatially, doelman2000slowly}.
\end{Remark}

In this article we aim to study patterned solutions to~\eqref{eq:modKlausmeier}, which arise as traveling wave solutions to~\eqref{eq:modKlausmeier}. We assume these waves travel in the uphill direction, and we define the traveling wave coordinate $\xi := x - ct$, where $c$ is the movement speed. Moreover, we set $(U,V)(x,y,t) = (u,v)(\xi,y,t)$, which results in the equation
\begin{equation}
	\begin{cases}
		u_t & = \frac{1}{\varepsilon} u_\xi + c u_\xi + a - u - G(u,v)v, \\
		v_t & = (\partial_\xi^2+\partial_y^2) v + c v_\xi - m v + R(v)G(u,v)v.
	\end{cases}\label{eq:TWPDE}
\end{equation}
Stationary solutions to~\eqref{eq:TWPDE} which are constant in $y$ correspond to traveling wave solutions of~\eqref{eq:modKlausmeier}; these solutions satisfy the first order traveling wave ODE
\begin{align}
\begin{cases}\label{eq:twode}
u_\xi&=\frac{\eps}{1+\eps c}\left(u-a+G(u,v)v\right),\\
v_\xi&=q,\\
q_\xi &=mv-R(v)G(u,v)v-cq. 
\end{cases}
\end{align}
This equation has an equilibrium at $(u,v,q)=(a,0,0)$ which represents the homogeneous desert state $(U_0,V_0)$ of~\eqref{eq:modKlausmeier}. There are two additional equilibrium points at $(u,v,q) = (u_{1,2},v_{1,2},0)$ corresponding to the other homogeneous steady states $(U_{1,2},V_{1,2})$ of~\eqref{eq:modKlausmeier}.

Based on the parameters of the model, several different patterned solutions to~\eqref{eq:modKlausmeier} can emerge that correspond to homoclinic or heteroclinic orbits of~\eqref{eq:twode}. Single vegetation stripe patterns occur as orbits that are homoclinic to the desert state. Similarly, vegetation gap patterns occur as orbits that are homoclinic to the vegetated state $(u_2,v_2,0)$. Besides these, there are also heteroclinic connections between the vegetated state and the desert state (and vice-versa) that represent transitions, or infiltration waves, between these uniform stationary states. Plots of these patterned solutions are shown in Figure~\ref{fig:numericalPatterns}.

\begin{figure}
	\centering
		\begin{subfigure}[t]{0.23\textwidth}
			\centering
			\includegraphics[width=\textwidth]{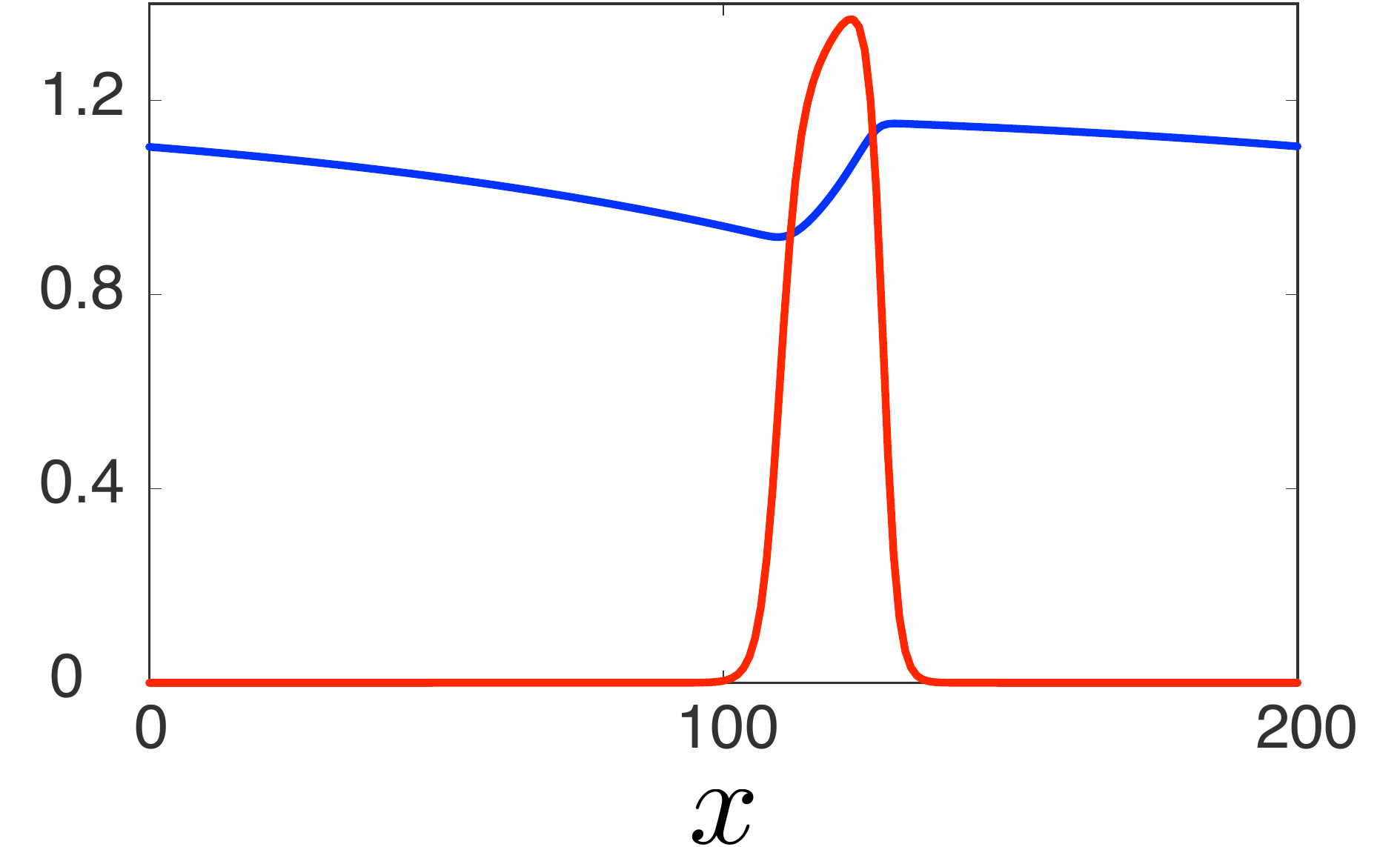}
			\caption{Vegetation stripe}
		\end{subfigure}
~
		\begin{subfigure}[t]{0.23\textwidth}
			\centering
			\includegraphics[width=\textwidth]{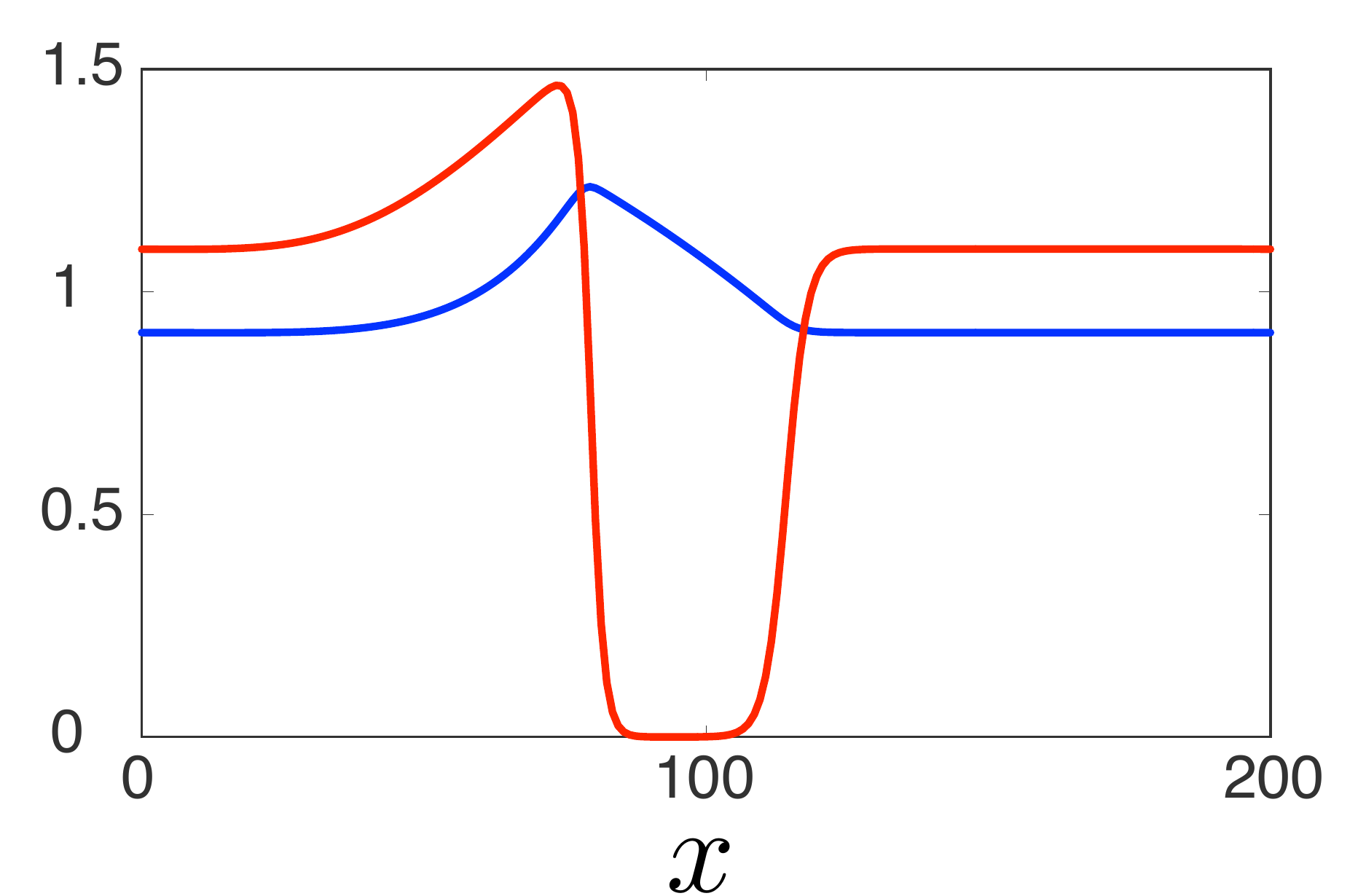}
			\caption{Vegetation gap}
		\end{subfigure}
~		
		\begin{subfigure}[t]{0.23\textwidth}
			\centering
			\includegraphics[width=\textwidth]{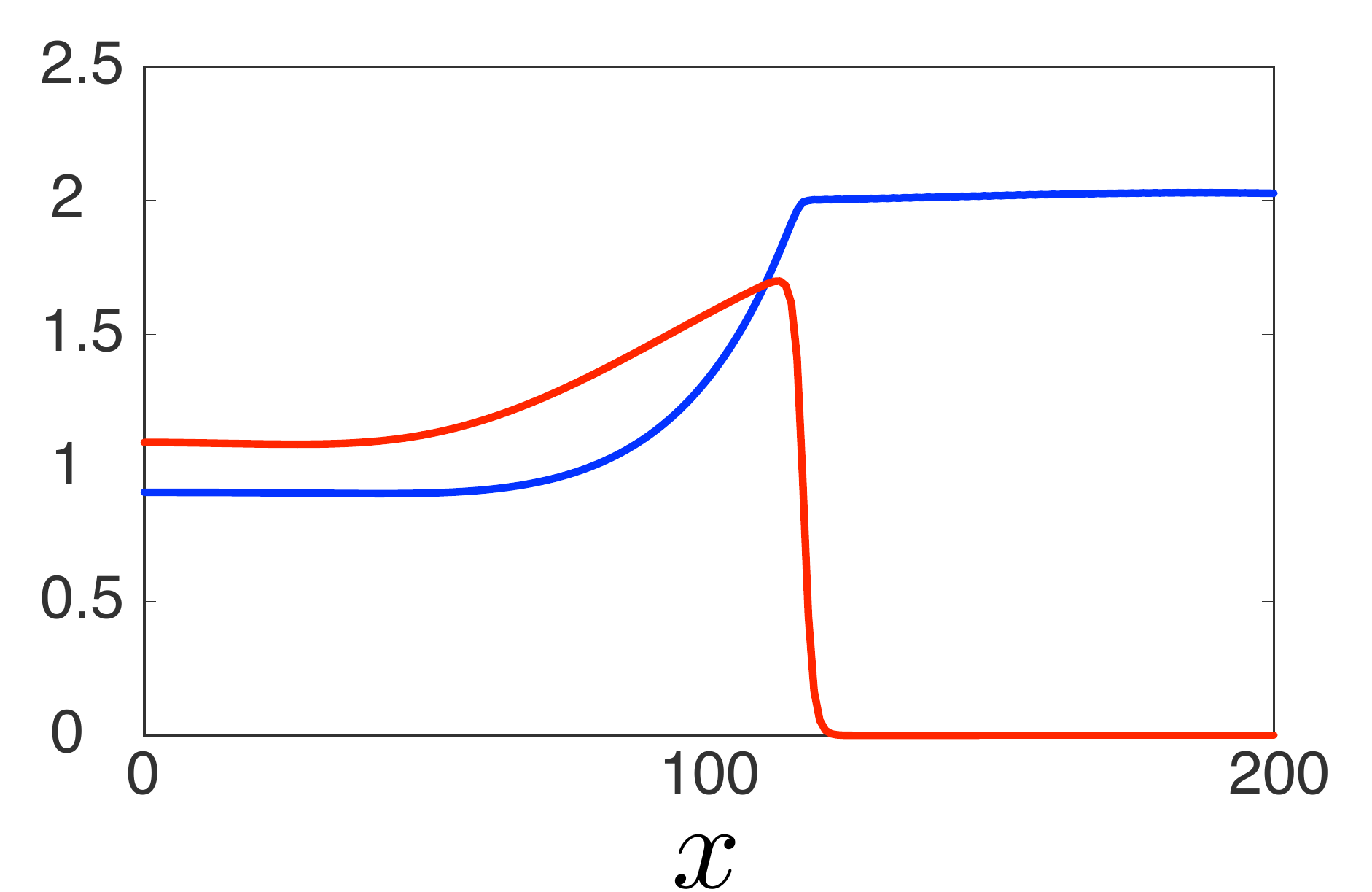}
			\caption{Vegetation-desert front}
		\end{subfigure}
~
		\begin{subfigure}[t]{0.23\textwidth}
			\centering
			\includegraphics[width=\textwidth]{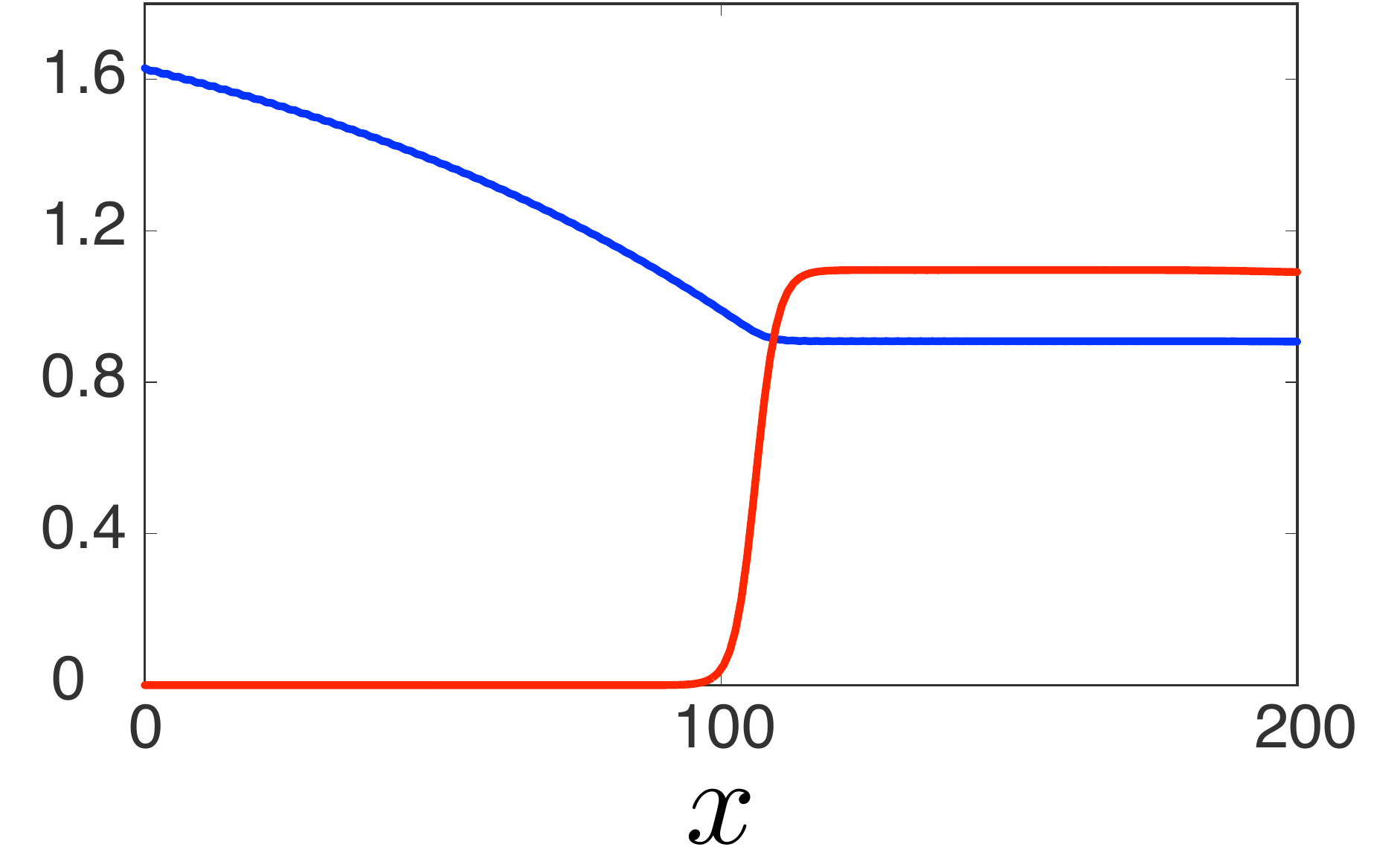}
			\caption{Desert-vegetation front}
		\end{subfigure}
	\caption{Shown are the different patterned solutions of~\eqref{eq:modKlausmeier} that are studied in this paper. Presented figures show cross-sections $u(x)$ (blue) and $v(x)$ (red) of direct numerical simulations with $\varepsilon = 0.01$, $m = 0.45$, $b = 0.5$ and $a = 1.2$ (a) or $a = 2.0$ (b-d). The 2D pattern is a trivial extension of these patterns in the $y$-direction, visualization of which is shown in Figure~\ref{fig:numericsStraightb0p5}.}
	\label{fig:numericalPatterns}
\end{figure}

In this article, we first establish existence of the aforementioned patterns rigorously. To that end, we exploit the scale separation in~\eqref{eq:twode} using the methods of geometric singular perturbation theory~\cite{fenichel1979geometric}. Using a fast-slow decomposition, these patterns are shown to correspond to the union of trajectories on so-called invariant slow manifolds of~\eqref{eq:twode} and fast connections between these slow manifolds. Specifically, \eqref{eq:twode} has three slow manifolds: one manifold, $\mathcal{M}^\ell$ ($\ell$ for left), consists of states without vegetation and the two others, $\mathcal{M}^m$ (middle) and $\mathcal{M}^r$ (right), consist of states with vegetation. Fast front-type solutions $\phi_\dagger$ exist which connect $\mathcal{M}^\ell$ to $\mathcal{M}^r$, and likewise there exist fast front solutions $\phi_\diamond$ which connect $\mathcal{M}^r$ to $\mathcal{M}^\ell$. Using these, stripes, gaps and fronts can be constructed for various parameter values. Pulse solutions to~\eqref{eq:modKlausmeier} consist of trajectories on $\mathcal{M}^\ell$ and $\mathcal{M}^r$ and \emph{two} fast front-type connections; front solutions to~\eqref{eq:modKlausmeier} only posses \emph{one} fast front-type connection. In Figure~\ref{fig:bd_reduced} these patterns are shown in the $\varepsilon = 0$ limit, where they are characterized by their speed in a sample bifurcation diagram.

\begin{figure}
			\centering
			\includegraphics[width=0.8\textwidth]{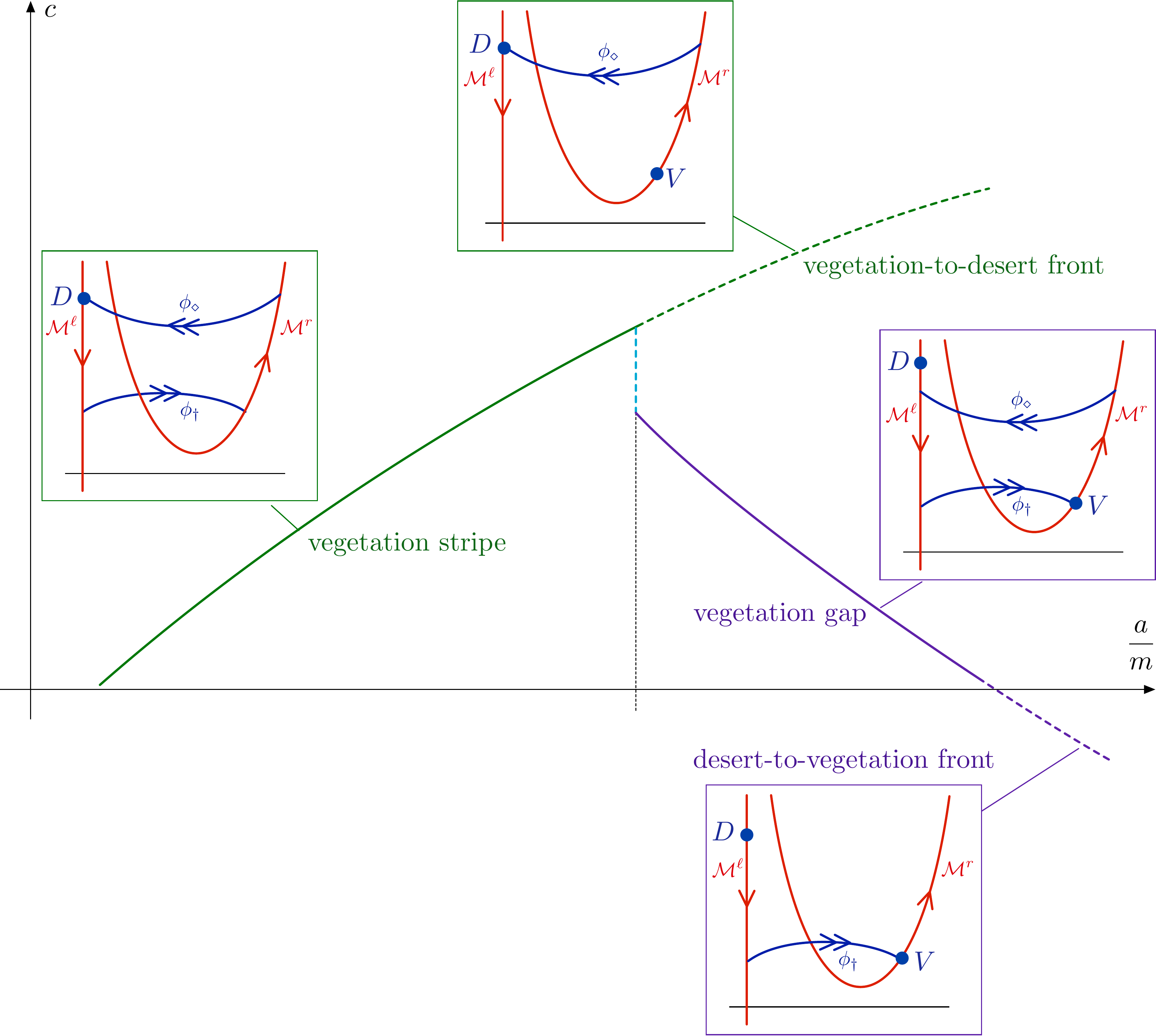}
\caption{A sample singular $\varepsilon = 0$ bifurcation diagram in $(a/m,c)$ parameter space. The solid green line indicates stripe solutions, while the solid purple line denotes the gap solutions. Vegetation-to-desert fronts are indicated by the dashed green line. Finally, desert-to-vegetation front solutions are given by the dashed and solid purple lines. Schematic depictions of the associated singular limit geometries are depicted in the insets, where the labels $D$ and $V$ denote the locations of the desert and vegetated equilibrium states, respectively. The precise bifurcation structure depends on the value of the parameter $b$; see~\S\ref{sec:singularsolns}.}
\label{fig:bd_reduced}
\end{figure}

The main theme of this paper is the spectral stability of the patterns. Because the main building blocks of all of the patterns are normally hyperbolic slow manifolds and fast front-type connections between these, we argue that destabilization can, a priori, only be caused by a `small' eigenvalue, one of which is created by every front-type connection. However, using formal asymptotic computations this possibility is excluded: all described patterns to~\eqref{eq:modKlausmeier} -- stripes, gaps and fronts -- are thus (always) stable against two-dimensional perturbations. These formal arguments are also verified rigorously by carefully constructing eigenfunctions using techniques previously employed to prove stability of traveling pulses in the FitzHugh--Nagumo system in~\cite{cdrs}; similar arguments were also used in~\cite{HOLZ, hupkes2013stability}. However, in those previous works, only stability with respect to perturbations in one spatial dimension was considered. By performing a Fourier decomposition in the transverse ($y$) direction, we show that these methods can also be used to obtain $2$D spectral stability of the full planar traveling waves.

Furthermore, the $2$D stability of the (straight) planar vegetation patterns implies that slightly curved variants of the same patterns, sometimes called corner defect solutions, are also solutions to~\eqref{eq:modKlausmeier} that are -- again -- 2D stable. An example of one of these solutions is given in Figure~\ref{fig:introCorner}. Existing techniques developed in~\cite{HS1, HS2} can be applied to infer that the orientation of these patterns is related to the speed $c$ of their associated straight patterns; in particular we predict that when $c > 0$ the corresponding corners are oriented convex upslope, and when $c < 0$ they are convex downslope. The presence of these curved patterns provides a possible explanation for observed vegetation arcs -- even in the absence of topographic mechanisms~\cite{gandhi2018influence}.

\begin{figure}
	\centering
	\begin{subfigure}[t]{0.45\textwidth}
		\centering	
		\includegraphics[width=0.8\textwidth]{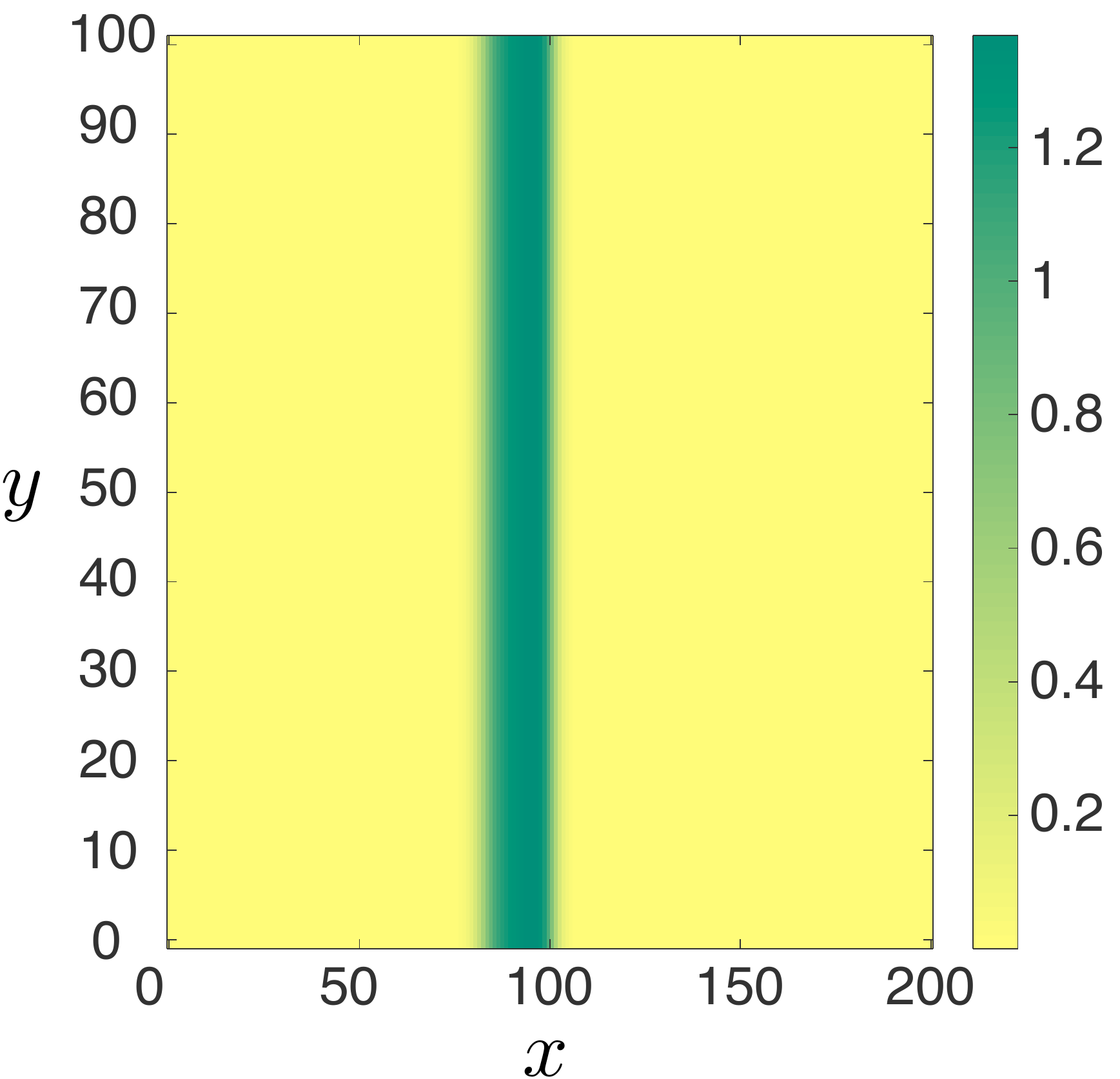}
		\caption{Straight vegetation stripe}
	\end{subfigure}
~
	\begin{subfigure}[t]{0.45\textwidth}
		\centering
		\includegraphics[width=0.8\textwidth]{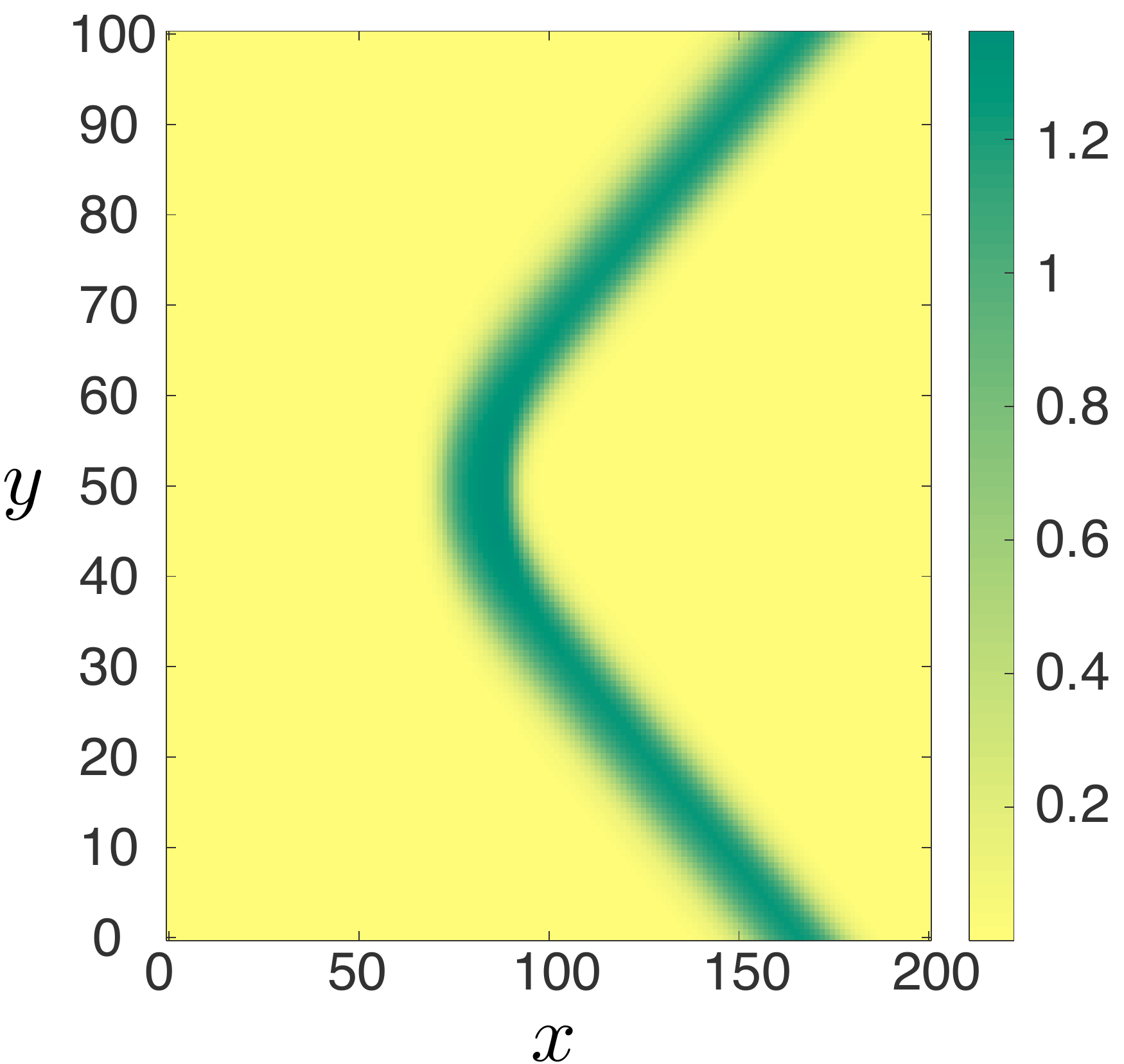}
		\caption{Curved vegetation `corner'}
	\end{subfigure}
	\caption{A snapsthot of a straight (a) and slightly bent (b) traveling vegetation stripe solution (with $c > 0$), obtained via direct numerical simulation of~\eqref{eq:modKlausmeier} with $\varepsilon = 0.01$, $m = 0.45$, $b = 0.5$ and $a = 1.2$.}
	\label{fig:introCorner}
\end{figure}

\begin{Remark}
	In an ecological context, traveling (spatially) periodic orbits are perhaps more relevant than the traveling pulse solutions constructed in this paper. However, once these pulse solutions are found, the periodic ones typically follow naturally~\cite{sewaltspatially} -- as is the case here. Furthermore, properties of these periodic orbits are closely related to those of the pulse solutions. See also \S~\ref{sec:singularperiodicorbits}.
\end{Remark}

The set-up for the rest of this article is as follows. In \S\ref{sec:slowfast}, we study~\eqref{eq:twode} as a slow/fast system in the context of geometric singular perturbation theory. We determine the slow manifolds $\mathcal{M}^\ell$, $\mathcal{M}^m$ and $\mathcal{M}^r$ and the fast connections $\phi_\dagger$ and $\phi_\diamond$ that connect the manifolds $\mathcal{M}^\ell$ and $\mathcal{M}^r$, which are then used to construct singular stripe, gap and front solutions to~\eqref{eq:twode}. In \S\ref{sec:existence}, we prove the persistence of these solutions for sufficiently small $\varepsilon > 0$. Next, in \S\ref{sec:stabilityFormal}, we compute the essential and point spectra of all these patterns using (formal) asymptotic computations, and show that all patterns are stable against all two dimensional perturbations. Subsequently, in \S\ref{sec:stabilityRigorous} these stability statements are made rigorous by carefully constructing eigenfunctions. In \S\ref{sec:corners} we inspect existence and stability of weakly bent (corner) solutions to~\eqref{eq:twode}. Then, in \S\ref{sec:numerics} we present the results of numerical computations on closely related spatially periodic patterns and numerical simulations of both straight and bent patterns. We conclude with a brief discussion of the results in \S\ref{sec:discussion}.

\section{Slow-fast analysis of traveling wave equation}\label{sec:slowfast}
In this section, we study the traveling wave equation~\eqref{eq:twode} as a slow-fast system in the singular limit $\eps=0$.  A discussion of the critical manifolds is given in~\S\ref{sec:criticalmanifold}. In~\S\ref{sec:layer}, we describe the singular layer problem, and we construct families of singular front solutions. We describe the reduced flow on the critical manifolds in~\S\ref{sec:reduced}, and we construct singular traveling front and stripe solutions in~\S\ref{sec:singularsolns}. Finally,~\S\ref{sec:mainexistenceresults} contains statements of our main existence results.

\subsection{Critical manifolds}\label{sec:criticalmanifold}
The traveling wave ODE~\eqref{eq:twode} is a two-fast-one-slow system. We obtain the fast subsystem or layer problem by setting $\eps=0$ in~\eqref{eq:twode}, which results in the system
\begin{align}
\begin{cases}\label{eq:fast0}
u' &=0,\\
v'&=q,\\
q' &=mv-R(v)G(u,v)v-cq, 
\end{cases}
\end{align}
or, equivalently, the collection of planar ODEs
\begin{align}
\begin{cases}\label{eq:layer}
v'&=q,\\
q' &=mv-R(v)G(u,v)v-cq, 
\end{cases}
\end{align}
parameterized by $u$. We note that $(v,q)=(0,0)=:p_0(u)$ is always an equilibrium of~\eqref{eq:layer}; there are additional equilibria $(v,0)$ whenever $v$ satisfies $R(v)G(u,v)=m$. Thus we see that there are additional equilibria $p_\pm(u):=(v_\pm(u),0)$, where
\begin{align}\label{eq:vplusminus}
v_\pm(u) = \frac{1\pm\sqrt{1-4bm/u}}{2b},
\end{align}
provided $u\geq 4bm$. We see that~\eqref{eq:layer} admits three equilibria for $u>4bm$, two equilibria for $u=4bm$, and a single equilibrium for $u<4bm$.

Denoting the right-hand-side of~\eqref{eq:layer} by 
\begin{align}
F(v,q;u):= \begin{pmatrix}q\\mv-R(v)G(u,v)v-cq\end{pmatrix},
\end{align}
we consider the linearization of~\eqref{eq:layer} about each of the three equilibria $p_0,p_\pm$ that is given by
\begin{align}
D_{(v,q)}F(0,0;u)&= \begin{pmatrix}0&1\\m&-c\end{pmatrix},\\
D_{(v,q)}F(v_\pm(u),0;u)&= \begin{pmatrix}0&1\\\frac{u-4mb\pm\sqrt{u^2-4mbu}}{2b}&-c\end{pmatrix}.
\end{align}
For $c>0$, we deduce that the equilibrium $p_0(u)$ is always a saddle. When $u>4bm$, the equilibrium $p_-(u)$ is a stable node or spiral, and the equilibrium $p_+(u)$ is a saddle. When $u=4bm$, the equilibrium $p_+(4bm)=p_-(4bm)$ is not hyperbolic.

In the full system, the equilibria of the layer problem~\eqref{eq:layer} form critical manifolds, given by three normally hyperbolic branches
\begin{align}\begin{split}
\mathcal{M}^\ell_0 &= \{v=q=0\},\\
\mathcal{M}^m_0 &= \{p_-(u):u>4bm\},\\
\mathcal{M}^r_0 &= \{p_+(u):u>4bm\},
\end{split}
\end{align}
with the branches $\mathcal{M}^m_0, \mathcal{M}^r_0$ meeting at a nonhyperbolic fold point $\mathcal{F} = p_+(4bm)=p_-(4bm)$; see Figure~\ref{fig:reduced_single}. For $u_1,u_2\in \mathbb{R}$, we will use the notation
\begin{align}\begin{split}\label{eq:criticalmanifoldsegment}
\mathcal{M}^j_0[u_1,u_2] &:= \mathcal{M}^j_0\cap \{ u_1\leq u\leq u_2 \}
\end{split}
\end{align}
to refer to a compact segment of one of the critical manifolds $\mathcal{M}^j_0$, $j=\ell,m,r$.

\begin{figure}
\centering
\includegraphics[width=0.45\textwidth]{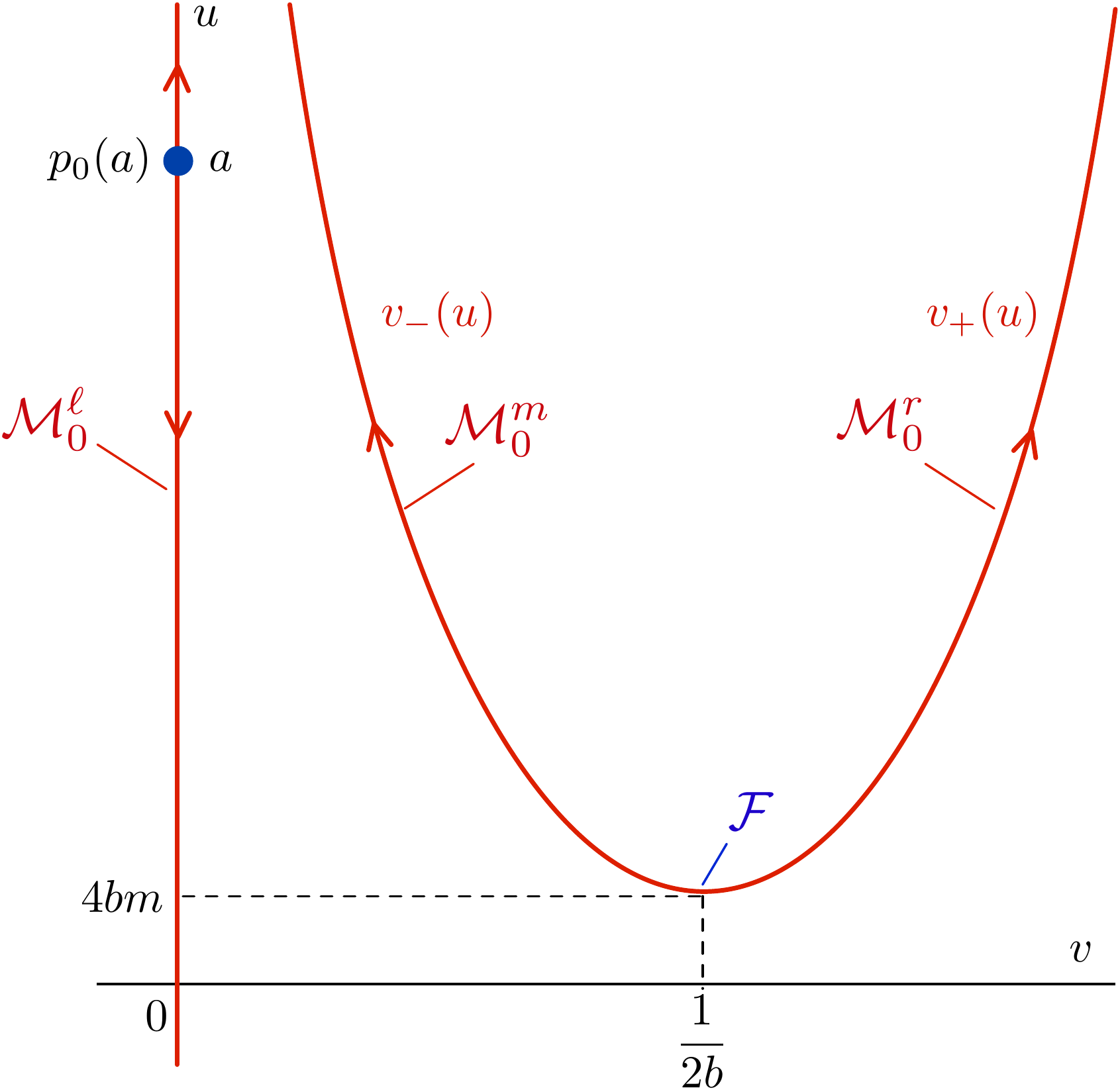}
\caption{Shown are the three branches of the critical manifold $\mathcal{M}_0$ and the associated reduced flow~\eqref{eq:reducedl}-\eqref{eq:reducedr} in the case $\frac{a}{m}<2\left(b+\sqrt{1+b^2}\right)$. There is a single equilibrium at $p_0(a)$ on the left branch $\mathcal{M}^\ell_0$ corresponding to the desert state $(u,v,q)=(a,0,0)$. }
\label{fig:reduced_single}
\end{figure}

We recall that there are (up to) three equilibria of the full system, given by $(u,v,q)=(a,0,0)$ and $(u,v,q) = (u_{1,2},v_{1,2},0)$; see Figures~\ref{fig:reduced_single} and~\ref{fig:reduced_multiple}. The equilibrium at $(u,v,q)=(a,0,0)$ lies on the left branch $\mathcal{M}^\ell_0$ and corresponds to $p_0(a)$, while that at $(u,v,q) = (u_{1},v_{1},0)$ corresponds to $p_-(u_1)$ and lies on the middle branch $\mathcal{M}^m_0$. The location of the equilibrium  $(u,v,q) = (u_{2},v_{2},0)$ depends on the parameter values: if $a / m>4b+1/b$, then it lies on the middle branch $\mathcal{M}^m_0$ at $p_-(u_2)$, while if $a / m>4b+1/b$, then it lies on the right branch $\mathcal{M}^r_0$ at $p_+(u_2)$. When $a / m=4b+1/b$, the equilibrium $(u,v,q) = (u_{2},v_{2},0)$ coincides with the fold $\mathcal{F}$.

\begin{Remark}\label{rem:b0degenerate}
We recall that the case $b=0$ corresponds to the original Klausmeier model~\cite{klausmeier1999regular}; see Remark~\ref{rem:b0klaus}. From the geometry of the critical manifold (see Figure~\ref{fig:reduced_single}), the degeneracy of the limit $b\to 0$ becomes apparent. In particular, the branch $\mathcal{M}^r_0$ of the critical manifold is sent to infinity, and the left branch $\mathcal{M}^\ell_0$ coincides with the hyperbola $v = m/u$ in the plane $q=0$. In the forthcoming analysis, we will consider only the case $b>0$. However, we note that under appropriate rescalings, it is possible to unfold the degenerate case $b=0$ and construct traveling wave solutions. Additional complications arise in the singular perturbation analysis due to loss of normal hyperbolicity along the critical manifold, for which blow up desingularization techniques are needed. We refer to~\cite{CDklausmeier} for the details.
\end{Remark}

\begin{figure}
\hspace{.0\textwidth}
\begin{subfigure}{.45 \textwidth}
\centering
\includegraphics[width=1\linewidth]{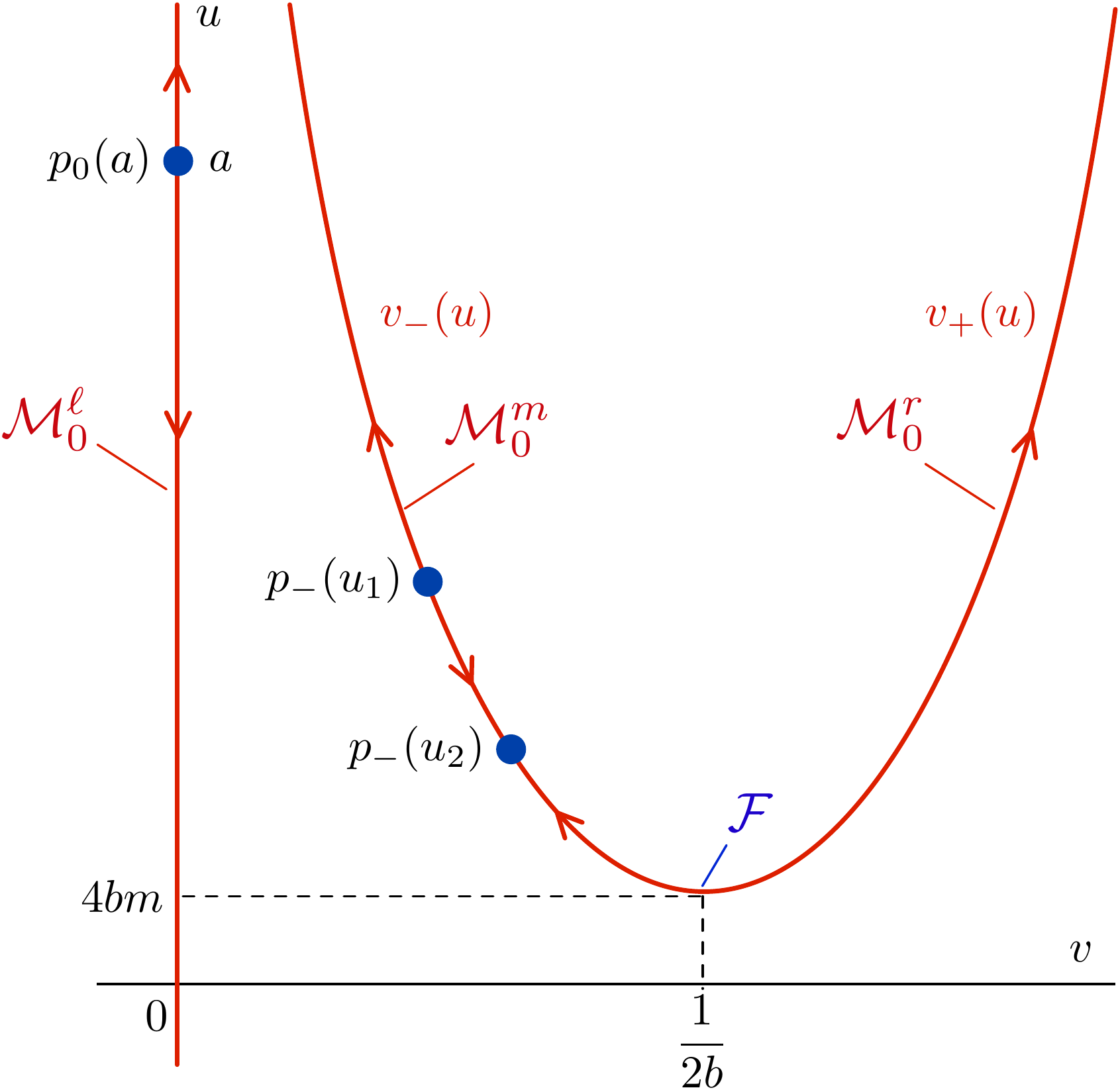}
\caption{$2\left(b+\sqrt{1+b^2}\right)<\frac{a}{m}<4b+\frac{1}{b}$}
\end{subfigure}
\hspace{.05\textwidth}
\begin{subfigure}{.45 \textwidth}
\centering
\includegraphics[width=1\linewidth]{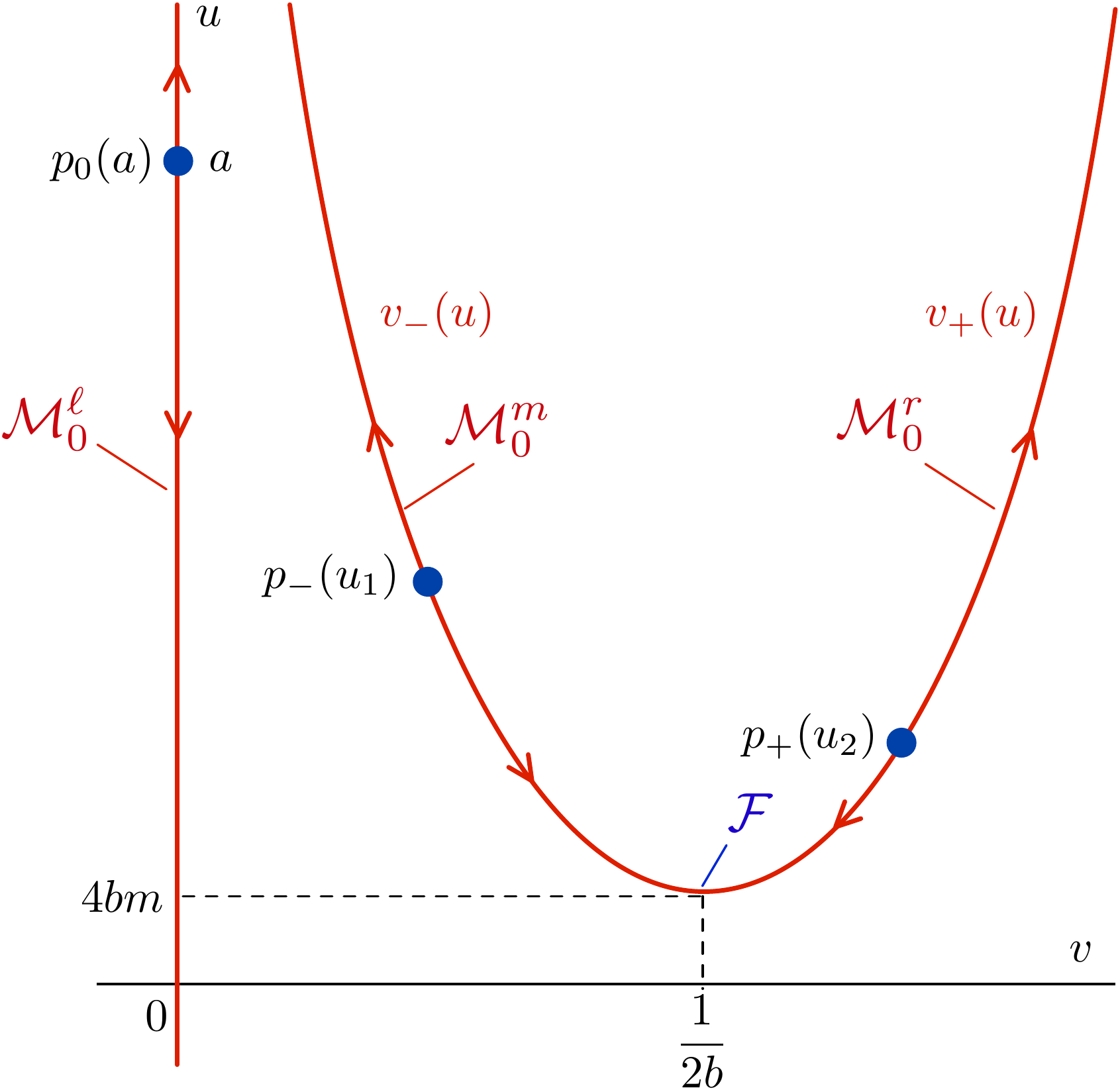}
\caption{$\frac{a}{m}>4b+\frac{1}{b}$}
\end{subfigure}
\hspace{.025\textwidth}
\caption{Shown are the three branches of the critical manifold $\mathcal{M}_0$ and the associated reduced flow~\eqref{eq:reducedl}-\eqref{eq:reducedr} in the case $\frac{a}{m}>2\left(b+\sqrt{1+b^2}\right)$. The reduced problem admits two addtional equilibria corresponding to the vegetated states $(u,v,q) = (u_{j},v_{j},0), j=1,2$. The equilibrium $(u,v,q) = (u_{1},v_{1},0)$ corresponds to $p_-(u_1)$ and lies on the middle branch $\mathcal{M}^m_0$. If $\frac{a}{m}>4b+1/b$, the equilibrium  $(u,v,q) = (u_{2},v_{2},0)$ lies on the middle branch $\mathcal{M}^m_0$ and corresponds to $p_-(u_2)$, while if $\frac{a}{m}>4b+1/b$, it lies on the right branch $\mathcal{M}^r_0$ at $p_+(u_2)$.}
\label{fig:reduced_multiple}
\end{figure}

\subsection{Layer fronts}\label{sec:layer}

We are interested in fronts between the two saddle equilibria $p_0(u)=(0,0)$ and $p_+(u)=(v_+(u),0)$; equivalently, we search for connections between the outer branches $\mathcal{M}^\ell_0 , \mathcal{M}^r_0 $. For each value of $u> 4mb$, there are two such fronts $ \phi_\diamond(\xi;u)$, $\phi_\dagger(\xi;u)$, with explicit $v$ profiles given by
\begin{align}
\begin{split}
v_{\diamond}(\xi;u) &= \frac{v_+(u)}{2}\left(1 - \tanh \left(\frac{v_+\sqrt{ub}}{2\sqrt{2}} \xi \right)\right),\\
v_{\dagger}(\xi;u) &= \frac{v_+(u)}{2}\left(1 + \tanh \left(\frac{v_+\sqrt{ub}}{2\sqrt{2}} \xi \right)\right),
\end{split}
\end{align}
and wave speeds
\begin{align}
\begin{split}
c^*_{\diamond}(u) &=  \frac{\sqrt{2bu}}{2}\left(v_+(u)-2v_-(u)\right)\\
c^*_{\dagger}(u) &= - \frac{\sqrt{2bu}}{2}\left(v_+(u)-2v_-(u)\right).
\end{split}
\end{align}
The $\diamond$-fronts connect $p_+$ to $p_0$, while the $\dagger$-fronts connect $p_0$ to $p_+$; see Figure~\ref{fig:singular_fronts}.

\begin{figure}
\hspace{.05\textwidth}
\begin{subfigure}{.4 \textwidth}
\centering
\includegraphics[width=1\linewidth]{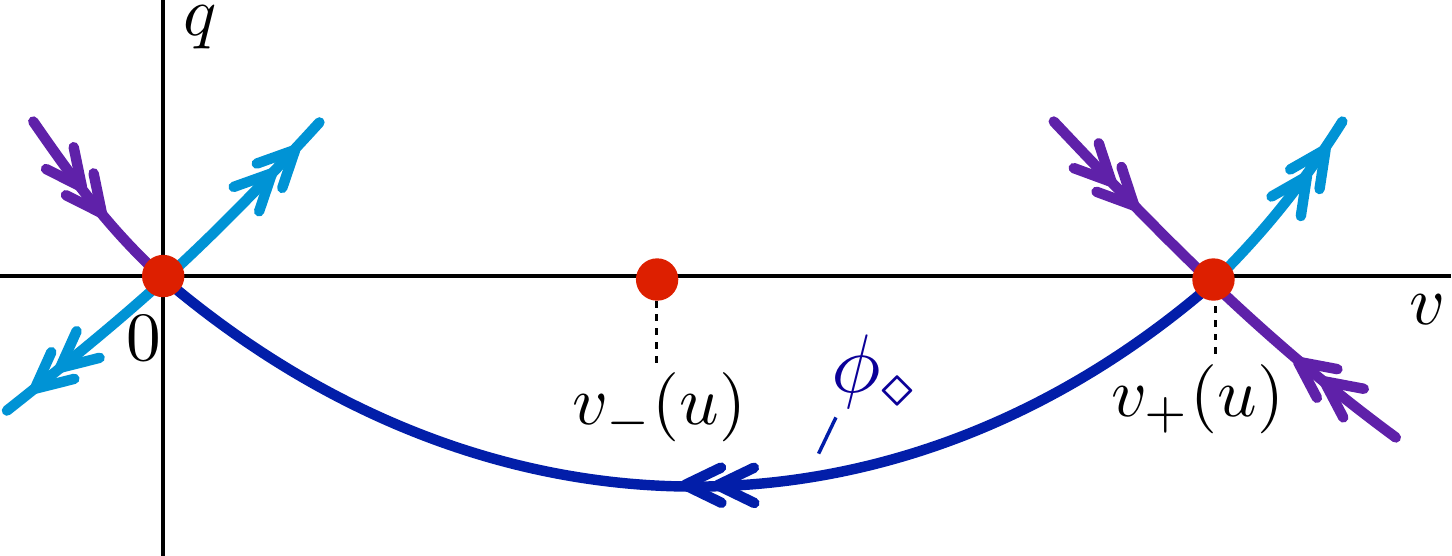}
\caption{$c=c_\diamond(u)$}
\end{subfigure}
\hspace{.05\textwidth}
\begin{subfigure}{.4 \textwidth}
\centering
\includegraphics[width=1\linewidth]{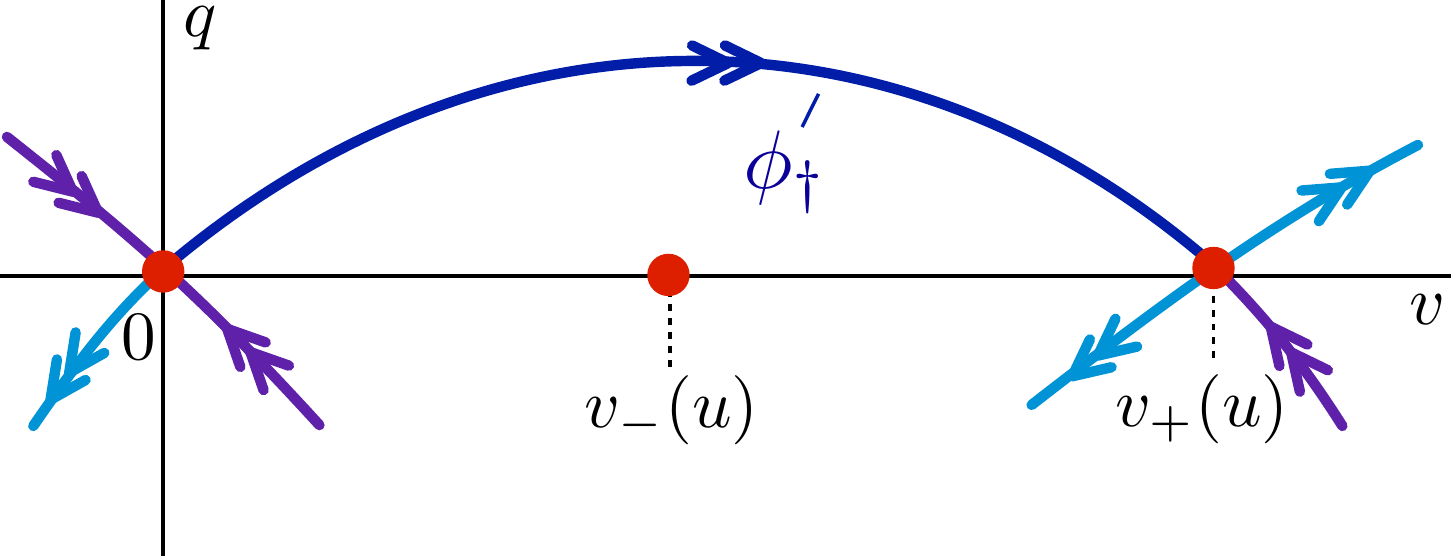}
\caption{$c=c_\dagger(u)$}
\end{subfigure}
\hspace{.01\textwidth}
\caption{Shown are the singular fronts fronts $ \phi_\diamond(\xi;u)$, $\phi_\dagger(\xi;u)$ of the layer problem~\eqref{eq:layer}.}
\label{fig:singular_fronts}
\end{figure}

When $u=4mb$, the situation is slightly different as the equilibria $p_\pm(u)$ collide in a saddle-node bifurcation at the fold point $\mathcal{F}$, and the equilibrium $p_+(u)$ is no longer a saddle. However, it is still possible to find fronts between $p_0$ and $p_+(4bm)=p_-(4bm)$. In particular, there exists a front connecting $p_+(4bm)$ to $p_0(4bm)$ for any
\begin{align}
\begin{split}
c\leq c_{\diamond,\mathrm{crit}} &= b\sqrt{2m}\left(v_+(4bm)-2v_-(4bm)\right)\\
&=-\sqrt{\frac{m}{2}}.
\end{split}
\end{align}
When $c=c_{\diamond,\mathrm{crit}}$ this front decays exponentially in backwards time, while for lesser speeds it decays only algebraically. Similarly, there exists a front connecting $p_0(4bm)$ to $p_+(4bm)$ for any
\begin{align}
\begin{split}
c\geq c_{\dagger,\mathrm{crit}} &= -b\sqrt{2m}\left(v_+(4bm)-2v_-(4bm)\right)\\
&=\sqrt{\frac{m}{2}}.
\end{split}
\end{align}
When $c=c_{\dagger,\mathrm{crit}}$ this front decays exponentially in forwards time, while for greater speeds it decays only algebraically.

In particular, provided $a>4bm$, these fronts exist when $u=a$. Therefore we have a front connecting $p_+(a)$ to $p_0(a)$ -- the equilibrium $(a,0,0)$ of the full system~\eqref{eq:twode} -- when
\begin{align}
\begin{split}
c&=c^*_\diamond(a)\\
&=\frac{1}{2\sqrt{2b}}\left(-\sqrt{a}+3\sqrt{a-4bm}\right).
\end{split}
\end{align}
 We now search for fronts which exist simultaneously for the same speed but different value of $u$, in particular for $u\leq a$. We have the following.
\begin{Lemma}\label{lem:desertlayerfronts}
For each $\frac{a}{m}\geq\frac{9}{2}b$, there exists a pair of fronts $ \phi_\diamond(\xi;a), \phi_\dagger(\xi;u^*(a))$ with speed 
\begin{align}
c&=c^*(a):=\frac{1}{2\sqrt{2b}}\left(-\sqrt{a}+3\sqrt{a-4bm}\right).
\end{align}
The front $ \phi_\diamond(\xi;a)$ connects $p_+(a)$ to $p_0(a)$ in the layer system~\eqref{eq:layer} for $u=a$, while the front $\phi_\dagger(\xi;u^*(a))$ connects $p_0(u^*(a))$ to  $p_+(u^*(a))$ in the layer system~\eqref{eq:layer} for $u=u^*(a)\leq a$, where
\begin{align}\label{eq:ustar}
u^*(a):=\left\{\begin{array}{c c} \frac{1}{8} \left(17 a-18 b m -15 \sqrt{a^2-4 ab m}\right), \quad & \frac{9}{2}b\leq \frac{a}{m}<\frac{25}{4}b;\\ 4bm, \quad&\frac{a}{m}\geq \frac{25}{4}b. \end{array}\right.
\end{align}
\end{Lemma}
\begin{proof}
When $\frac{a}{m}=\frac{9}{2}b$, we have $c^*_\diamond(a)=0=c^*_\dagger(a)$, and therefore both heteroclinic orbits lie simultaneously in the plane $u=a$, forming a heteroclinic loop. For values of $\frac{9}{2}b<\frac{a}{m}<\frac{25}{4}b$, the second heteroclinic orbit exists for a value of $4bm<u^*<a$ given by~\eqref{eq:ustar}, which can be obtained by solving the relation $c^*_\diamond(a)=c^*_\dagger(u)$ for $u=u^*(a)$. 

For $a\geq \frac{25bm}{4}$, the second heteroclinic orbit occurs when $u=u^*(a)=4bm$; the decay is exponential in forward time when $a= \frac{25bm}{4}$, and algebraic for $a> \frac{25bm}{4}$.
\end{proof}
\begin{Remark}
In the case $4b \leq \frac{a}{m} \leq \frac{9}{2}b$, there (also) exists a second front $\phi_\dagger(\xi;u^*(a))$ with speed $c = c^*(a)$ that connects $p_0(u^*(a))$ to $p_+(u^*(a))$ in the layer system~\eqref{eq:layer} for $u=u^*(a)$, where 
\begin{equation*}
u^*(a) =  \frac{1}{8} \left(17 a-18 b m -15 \sqrt{a^2-4 ab m}\right).
\end{equation*}
However, in this case $u^*(a) > a$, which -- because of the flow on $\mathcal{M}_0^r$ (see \S\ref{sec:reduced}) -- prevents the existence of a homoclinic connection in the full system.
\end{Remark}

We recall that for $a / m >4b+1/b$, the equilibrium $p_+(u_2)$ on the right branch $\mathcal{M}^r_0$ corresponds to the equilibrium $(u_2,v_2,0)$ of the full system~\eqref{eq:twode}. For $a/m=4b+1/b$, this equilibrium lies precisely on the fold $\mathcal{F}$. We now search for singular fronts to this equilibrium for values of $a/m\geq4b+1/b$, and the argument is similar as above.  When $a/m>4b+1/b$, there exists a front connecting $p_0(u_2)$ to $p_+(u_2)$ when
\begin{align}
\begin{split}
c&=c^*_\dagger(u_2)\\
&= -\frac{1}{2\sqrt{2b}}\left(-\sqrt{u_2}+3\sqrt{u_2-4bm}\right),
\end{split}
\end{align}
and when $a/m=4b+1/b$ this front exists for each $c\geq c_{\dagger,\mathrm{crit}}$, with exponential decay in forward time for $c=c_{\dagger,\mathrm{crit}}$ and algebraic decay when $c>c_{\dagger,\mathrm{crit}}$. We again search for fronts which exist simultaneously for the same speed but different value of $u$, and we have the following lemma, analogous to Lemma~\ref{lem:desertlayerfronts}.
\begin{Lemma}\label{lem:veggielayerfronts}Concerning the layer problem~\eqref{eq:layer}, the following hold.
\begin{enumerate}[(i)]
\item \label{lem:veggielayerfrontsi}
For each $4b+\frac{1}{b}< \frac{a}{m}\leq \frac{9}{2}b+\frac{2}{b}$, there exists a pair of fronts $ \phi_\diamond(\xi;\hat{u}_2(a)), \phi_\dagger(\xi;u_2)$ with speed $\hat{c}(a) = c_\dagger^*(u_2)$. The front $ \phi_\dagger(\xi;u_2)$ connects $p_0(u_2)$ to $p_+(u_2)$ in the layer system~\eqref{eq:layer} for $u=u_2$, while the front $\phi_\diamond(\xi;\hat{u}_2(a))$ connects  $p_+(\hat{u}_2(a))$ to $p_0(\hat{u}_2(a))$ in the layer system~\eqref{eq:layer} for $u=\hat{u}_2(a)$, where
\begin{align}\label{eq:uhat}
\hat{u}_2(a):= \frac{1}{8} \left(17 u_2-18 b m -15 \sqrt{u_2^2-4 u_2b m}\right).
\end{align}
\item \label{lem:veggielayerfrontsii}
When $a / m=4b+1/b$, for each $c\geq c_{\dagger,\mathrm{crit}}$, there exists a pair of fronts $ \phi_\dagger(\xi;u_2), \phi_\diamond(\xi;\hat{u}(c))$, where $\hat{u}(c)$ is an increasing function of $c$ which satisfies $\hat{u}(c_{\dagger,\mathrm{crit}}) =\hat{u}_2(4mb+m/b) $. 

\end{enumerate}
\end{Lemma}
\begin{proof}
For~\ref{lem:veggielayerfrontsi}, when $\frac{a}{m}=\frac{9}{2}b+\frac{2}{b}$, we have $c^*_\diamond(u_2)=0=c^*_\dagger(u_2)$, and therefore both heteroclinic orbits lie simultaneously in the plane $u=u_2$, forming a heteroclinic loop. For values of $4b+\frac{1}{b}< \frac{a}{m}< \frac{9}{2}b+\frac{2}{b}$, the second heteroclinic orbit exists for a value of $\hat{u}_2>u_2$ given by the solution of~\eqref{eq:uhat}, which can be obtained by solving the relation $c^*_\diamond(u)=c^*_\dagger(u_2)$ for $u=\hat{u}_2$.

For~\ref{lem:veggielayerfrontsii}, when $a/m=4b+1/b$, the equilibrium $p_+(u_2)$ lies precisely on the fold $\mathcal{F}$ and hence we obtain the fronts $\phi_\dagger(\xi;u_2)$ for each $c\geq c_{\dagger,\mathrm{crit}}$. The facts regarding $\hat{u}(c)$ follow by noticing that the relation
\begin{align}
\begin{split}
c^*_\diamond(u) &=  \frac{\sqrt{2bu}}{2}\left(v_+(u)-2v_-(u)\right)\\
&=\frac{1}{2\sqrt{2b}}\left(-\sqrt{u}+3\sqrt{u-4bm}\right)
\end{split}
\end{align}
defines $c^*_\diamond(u)$ as a strictly increasing function of $u$, and that $u_2=4bm$ when $a/m=4b+1/b$, so that $\hat{u}_2(4mb+m/b) = 25bm/4$, and $c^*_\diamond(25bm/4)=c_{\dagger,\mathrm{crit}}$.
\end{proof}

\subsection{Slow flow}\label{sec:reduced}
We now examine the slow flow restricted to the critical manifolds $\mathcal{M}^\ell_0$ and  $\mathcal{M}^r_0 $. We rescale $\tau=\eps \xi$ and obtain the corresponding slow system
\begin{align}
\begin{cases}\label{eq:slow}
u_\tau&=\frac{1}{1+\eps c}\left(u-a+G(u,v)v\right)\\
\eps v_\tau&=q\\
\eps q_\tau &=mv-R(v)G(u,v)v-cq.
\end{cases}
\end{align}
By setting $\eps=0$, we obtain the reduced flow on $\mathcal{M}^\ell_0$ as
\begin{align}
\begin{split}\label{eq:reducedl}
u_\tau&=u-a,
\end{split}
\end{align}
on $\mathcal{M}^m_0$ as
\begin{align}
\begin{split}\label{eq:reducedm}
u_\tau&=u-a+G(u,v_-(u))v_-(u),
\end{split}
\end{align}
and on $\mathcal{M}^r_0$ as
\begin{align}
\begin{split}\label{eq:reducedr}
u_\tau&=u-a+G(u,v_+(u))v_+(u).
\end{split}
\end{align}
See Figures~\ref{fig:reduced_single} and~\ref{fig:reduced_multiple} for depictions of the reduced flow, depending on the value of $a/m$. We see that for $u<a$, under the reduced flow on $\mathcal{M}^\ell_0$, $u$ is always decreasing, while on $\mathcal{M}^r_0$, $u$ is always increasing, provided $a/m<4b+1/b$. When $a/m=4b+1/b$, there exists an equilibrium of the full system $(u_2,v_2,0)$ which coincides with the fold $\mathcal{F}$, which thus takes the form of a canard point~\cite{krupaszmolyan2001}. As $a$ increases through this value, this equilibrium moves up along the right branch $\mathcal{M}^r_0$. In that case, the flow is away from this equilibrium point; that is, $u$ is decreasing when $u < u_2$ and increasing when $u > u_2$.

\subsection{Singular orbits} \label{sec:singularsolns}

In the previous sections we have studied the slow flow on the manifolds $\mathcal{M}_0^\ell$ and $\mathcal{M}_0^r$ and the dynamics of fast transitions between these manifolds. In this section, we use this knowledge to construct families of singular orbits, which will serve as the basis for constructing traveling front and pulse solutions to~\eqref{eq:modKlausmeier}. These singular orbits are constructed for open regions in $(a,b,m)$ parameter space, with the wavespeed $c$ in general determined uniquely by the value of $(a,b,m)$. The bifurcation structure, as well as the singular limit geometry of the associated solution orbits, is depicted in the bifurcation diagrams in Figures~\ref{fig:bd_case1} and~\ref{fig:bd_case2}. These diagrams show the dependence of the wave speed $c$ on the value of the quantity $a/m$, in the regions $b<2/3$ and $b>2/3$, as the bifurcation structure changes qualitatively as $b$ crosses through the critical value $2/3$.

We first consider traveling pulse solutions, which can be thought of as two front-type solutions glued together to create a profile which is bi-asymptotic to one of the equilibrium states with a plateau in between. These come in two varieties: vegetation stripe solutions, considered in~\S\ref{sec:singularstripes}, which manifest as homoclinic orbits to the desert equilibrium state $p_0(a)$, and vegetation gap solutions, considered in~\S\ref{sec:singulargaps}, which arise as homoclinic orbits to the equilibrium $p_+(u_2)$. In both cases, the corresponding homoclinic orbits are composed of two portions of the slow manifolds $\mathcal{M}_0^\ell$ and $\mathcal{M}_0^r$ concatenated with two fast jumps in between, which exist for the same value of $c$. The singular limit geometry for these solutions is shown in the bifrucation diagrams Figures~\ref{fig:bd_case1} and~\ref{fig:bd_case2} (see also Figure~\ref{fig:singular_pulse_desert} for more details), in which the stripe solutions are defined along the upper solid green, and the gap solutions are defined along the upper solid purple curve. The distinction between the cases $b<2/3$ and $b>2/3$ is related to the manner in which these two curves interact; this is discussed in more detail in~\S\ref{sec:singularstripes}.

Next we consider with singular front solutions in~\S\ref{sec:singularfronts}, characterized by a sharp transition from the uniform desert state to the uniformly vegetated state or vica versa. In the slow/fast framework of the traveling wave equation~\eqref{eq:twode}, these solutions manifest as heteroclinic orbits between the equilibria $p_0(a)$ and $p_+(u_2)$, and are composed of a single slow segment along one of the manifolds $\mathcal{M}_0^\ell$ and $\mathcal{M}_0^r$ concatenated with a fast jump to the opposite slow manifold. In the diagrams Figures~\ref{fig:bd_case1} and~\ref{fig:bd_case2}, these singular front solutions are defined along the upper solid and dashed green and purple curves in the region $a/m>4b+1/b$. The green curves correspond to front solutions in which the vegetated state is downslope of the desert state, while the desert state is downslope of the vegetated state along the purple curves.

We briefly discuss periodic orbits in~\S\ref{sec:singularperiodicorbits}, and in the following section~\S\ref{sec:mainexistenceresults}, we state our main existence results regarding traveling front, stripe, and gap solutions to~\eqref{eq:modKlausmeier}.

\begin{figure}
\centering
	\begin{subfigure}[t]{\textwidth}
	\centering
		\includegraphics[width=0.8\textwidth]{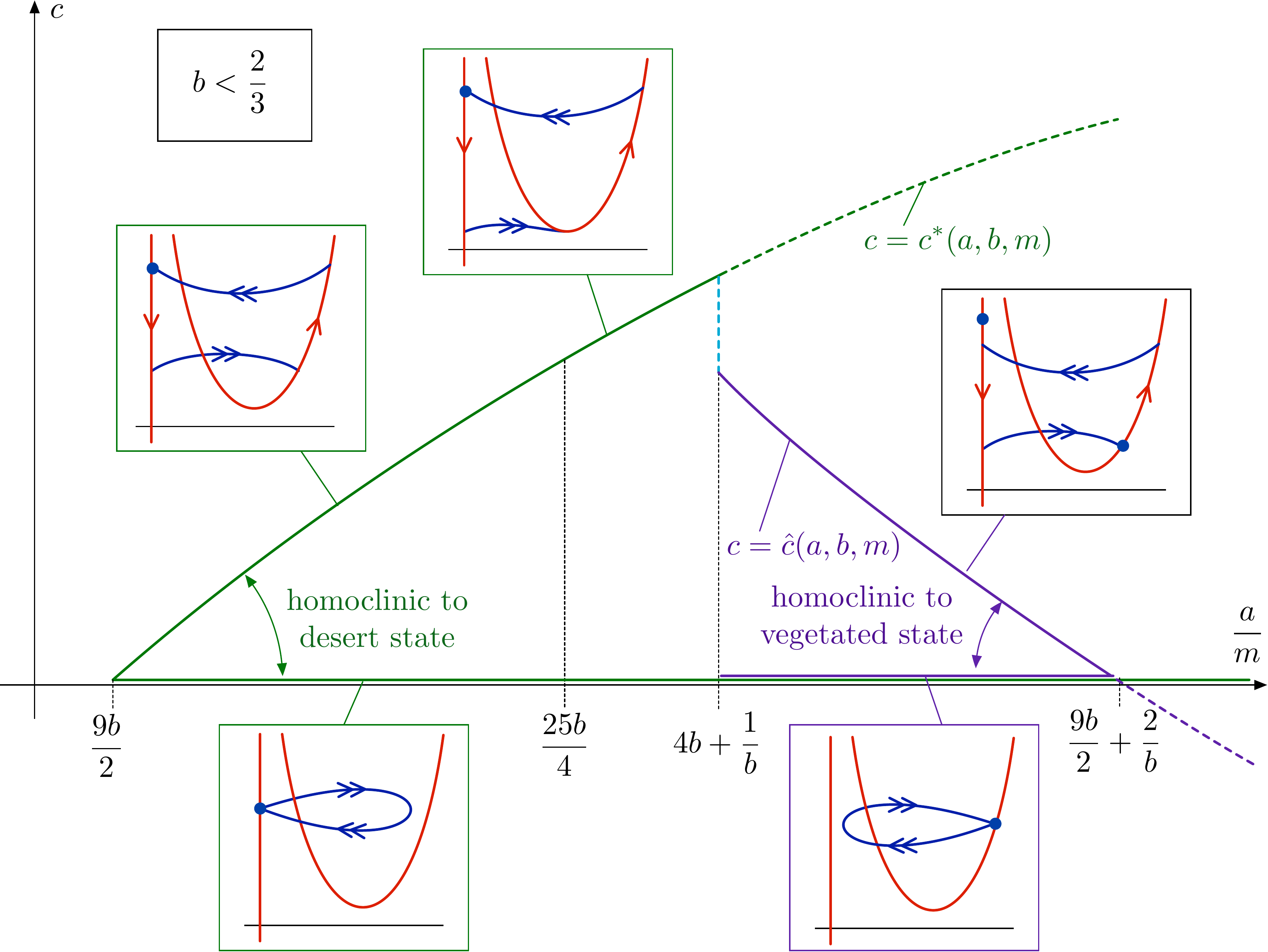}
	\caption{}
	\label{fig:bd_case1}
	\end{subfigure}
	\begin{subfigure}[t]{\textwidth}
	\centering
	\includegraphics[width=0.8\textwidth]{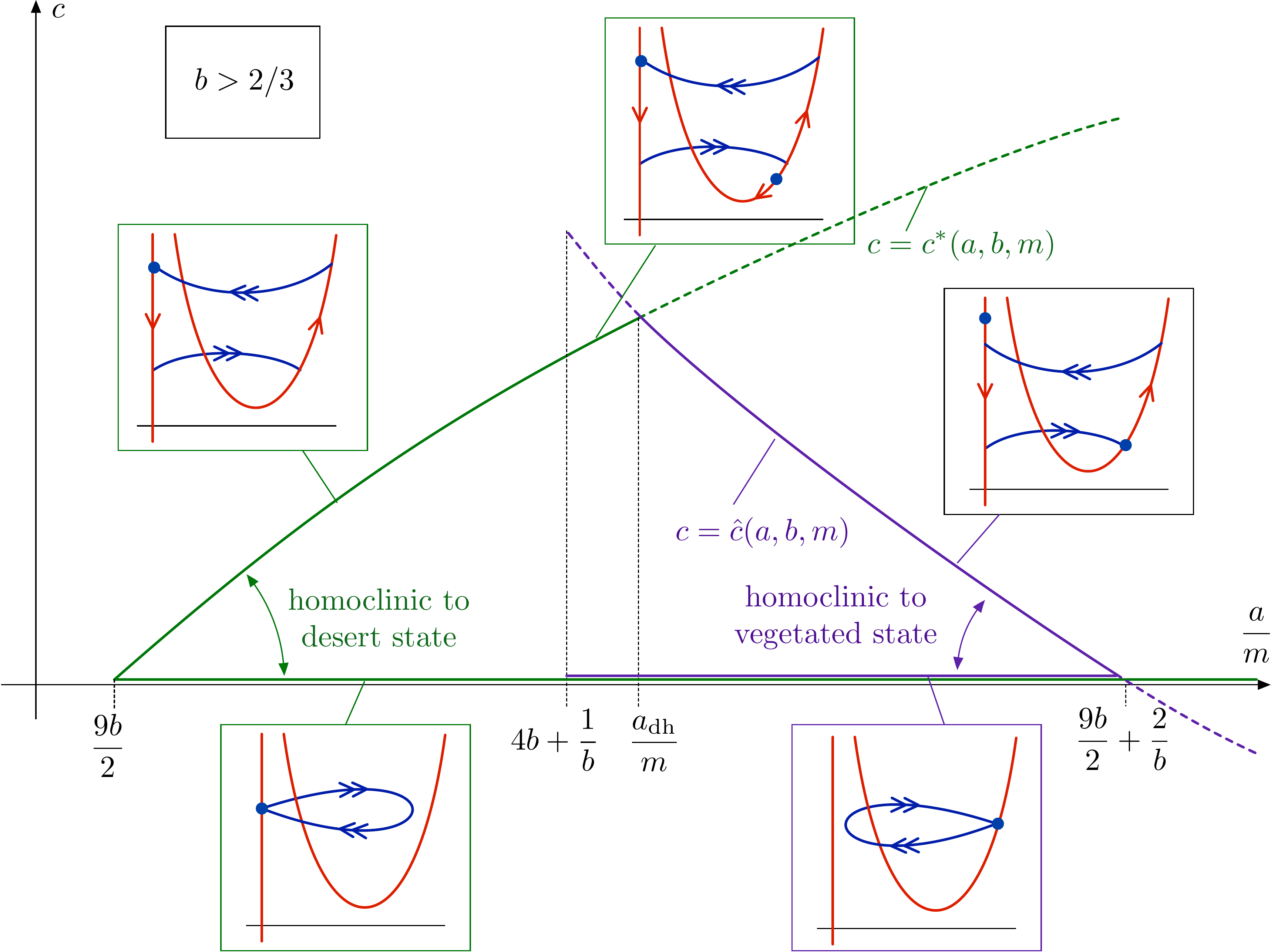}
	\caption{}
	\label{fig:bd_case2}
	\end{subfigure}
\caption{Shown are the singular $\eps = 0$ bifurcation diagrams in $(a,c)$ parameter space in the cases $b < 2/3$ (a) and $b > 2/3$ (b).}
\end{figure}

\subsubsection{Homoclinic orbits to the desert state $p_0(a)$}\label{sec:singularstripes}
By Lemma~\ref{lem:desertlayerfronts}, for each $\frac{a}{m}\geq\frac{9b}{2}$, there exists a pair of fronts $ \phi_\diamond(\xi;a), \phi_\dagger(\xi;u^*(a))$ with the same speed
\begin{align}
c&=c^*(a):=\frac{1}{2\sqrt{2b}}\left(-\sqrt{a}+3\sqrt{a-4bm}\right).
\end{align}
We can concatenate these fronts with portions of the critical manifolds $\mathcal{M}^{\ell,r}_0$ in order to construct singular homoclinic solutions to the equilibrium $p_0(a)$. However, when $a/m>4b+1/b$, the equilibrium $p_+(u_2)$ lies on $\mathcal{M}^r_0$ and can block these orbits. For each $\frac{a}{m} \geq \frac{9b}{2}$, we have a candidate singular homoclinic orbit to the desert state $p_0(a)$ given by
\begin{align}\label{eq:singhomd}
\mathcal{H}_\mathrm{d}(a):=\mathcal{M}^\ell_0[u^*(a),a]\cup \phi_\dagger(u^*(a)) \cup \mathcal{M}^r_0[u^*(a),a] \cup \phi_\diamond(a),
\end{align}
corresponding to a vegetation stripe solution (see Figure~\ref{fig:singular_pulse_desert}), where the notation $\mathcal{M}^j_0[u_1,u_2]$ was defined in~\eqref{eq:criticalmanifoldsegment}. This orbit will be blocked if the equilibrium $p_+(u_2)$ lies on $\mathcal{M}^r_0$ with $u^*(a)\leq u_2$. There are two cases based on the expression for $u^*(a)$ in~\eqref{eq:ustar}. If $a/m\geq 25b/4$, then this orbit is blocked whenever $p_+(u_2)$ lies on $\mathcal{M}^r_0$, that is, for any value of $a/m\geq 4b+1/b$. If $a/m< 25b/4$, then this orbit is blocked if $u_2\geq u^*(a)$, which occurs when 
\begin{align}
\frac{a}{m} \geq \bar{a}_\mathrm{dh}:= 2b+\frac{5\sqrt{3}b^2}{2\sqrt{4+3b^2}}+\frac{8}{\sqrt{12+9b^2}}.
\end{align}
We therefore expect a different singular bifurcation diagram for the cases $4b+1/b>25b/4$ or $4b+1/b<25b/4$ (i.e. $b < 2/3$ respectively $b > 2/3$). In the former case the singular front $\phi_\dagger(\xi;u^*(a))$ can jump precisely onto the fold point $\mathcal{F}$; in the latter case this is not possible. Equivalently, the structure changes depending on whether $b<2/3$ or $b>2/3$ (see Figures~\ref{fig:bd_case1} and~\ref{fig:bd_case2}). We define the quantity
\begin{align}
\bar{a}(b):=\begin{cases} 4b+1/b & b\leq 2/3\\ \bar{a}_\mathrm{dh} & b>2/3\end{cases}.
\end{align}
Then for each $b,m>0$, we can construct the singular homoclinic orbits $\mathcal{H}_\mathrm{d}(a)$ for $\frac{9}{2}b \leq \frac{a}{m} \leq \bar{a}(b)$. We note that when $b\leq 2/3$ and $\frac{a}{m}\in[4b+1/b, 25b/4]$, the front $\phi_\dagger(u^*(a))$ jumps precisely onto the nonhyperbolic fold point $\mathcal{F}$. While it is possible to construct homoclinic orbits in this regime as well as determine the stability of the underlying traveling wave solution~\cite{beck,cdrs,cas} using geometric blow-up methods, we do not consider this case here. Rather we restrict our attention to orbits which jump on/off normally hyperbolic portions of the critical manifold. To that end, we define the quantity
\begin{align}
\bar{a}_\mathrm{hyp}(b):=\begin{cases} 25b/4 & b\leq 2/3\\ \bar{a}_\mathrm{dh} & b>2/3\end{cases},
\end{align}
and consider only the singular homoclinic orbits $\mathcal{H}_\mathrm{d}(a)$ for $\frac{9}{2}b \leq \frac{a}{m} < \bar{a}_\mathrm{hyp}(b)$.

\begin{Remark}\label{rem:layerhomoclinics}
In addition to the class of homoclinic orbits described above, there also exist singular homoclinic orbits to the equilibrium $p_0(a)$ lying entirely in the plane $u=a$. These orbits in fact correspond to solutions of the layer problem~\eqref{eq:layer} for $u=a$ and $c=0$, and they are depicted along the lower green curves in the bifurcation diagrams in Figures~\ref{fig:bd_case1} and~\ref{fig:bd_case2}. As with the singular homoclinic orbits $\mathcal{H}_\mathrm{d}(a)$ constructed in this section, it is possible to show that these layer homoclinic orbits also persist for sufficiently small $\epsilon>0$ using geometric singular perturbation arguments, and in fact they lie on the same continuation branch; see Figure~\ref{fig:bd_numerics}. Furthermore, the bifurcation structure near these orbits is surprisingly rich; a detailed analysis is carried out in~\cite{CD_3Dhomoclinics}. However, unlike the orbits $\mathcal{H}_\mathrm{d}(a)$, the resulting traveling wave solutions are typically unstable as solutions to~\eqref{eq:modKlausmeier}, and we therefore refrain from analyzing these solutions in this work.
\end{Remark}

\begin{figure}
\centering
\includegraphics[width=0.5\textwidth]{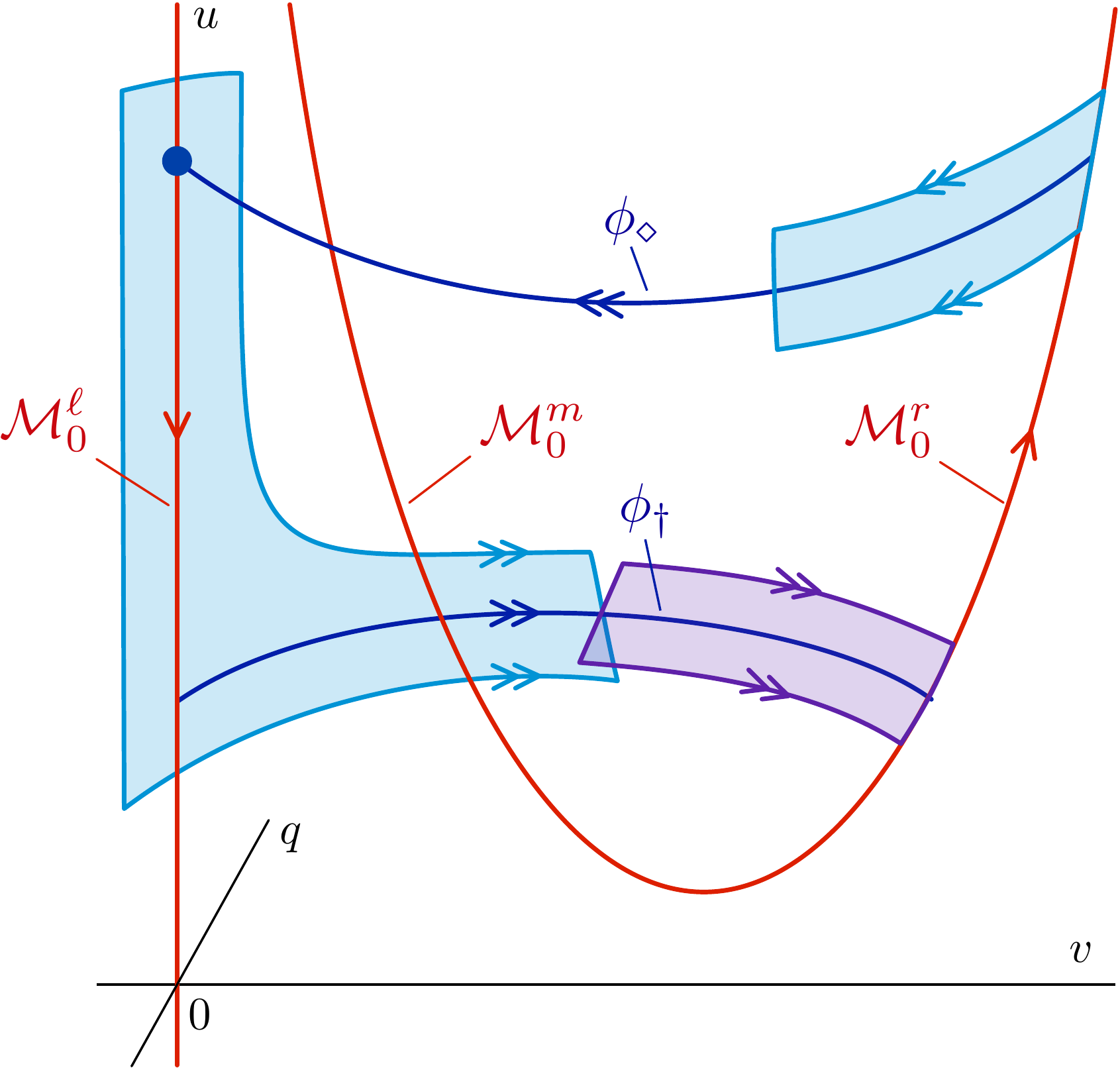}
\caption{Shown is the singular orbit $\mathcal{H}_\mathrm{d}(a)$ homoclinic to the desert state $p_0(a)$. The orbit first traverses a portion of the manifold $\mathcal{M}^\ell_0$, then the front $\phi_\dagger(u^*(a))$, followed by a portion of the critical manifold $\mathcal{M}^r_0$, and finally the front $\phi_\diamond(a)$. }
\label{fig:singular_pulse_desert}
\end{figure}

\subsubsection{Homoclinic orbits to the vegetated state $p_+(u_2)$}\label{sec:singulargaps}
Similarly, we can construct singular homoclinic orbits to the vegetated state $p_+(u_2)$, using the fronts from Lemma~\eqref{lem:veggielayerfronts}. By similar arguments as above, we obtain singular homoclinic orbits 
\begin{align}
\mathcal{H}_\mathrm{v}(a):= \mathcal{M}^r_0[u_2,\hat{u}_2(a)]  \cup \phi_\diamond(\hat{u}_2(a))\cup \mathcal{M}^\ell_0[u_2,\hat{u}_2(a)] \cup \phi_\dagger(u_2),
\end{align}
corresponding to vegetation gap solutions. For each $b,m>0$, these orbits can be constructed for parameters $\bar{a}(b) \leq \frac{a}{m} \leq \frac{9}{2} b + 2/b$. 

\begin{Remark}
Additionally, in the case $b<2/3$, using Lemma~\eqref{lem:veggielayerfrontsii}, when $a=4bm+m/b$, we also obtain homoclinic orbits 
\begin{align}
\widehat{\mathcal{H}}_\mathrm{v}(c):= \mathcal{M}^r_0[u_2,\hat{u}(c)] \cup \phi_\diamond(\hat{u}(c))\cup \mathcal{M}^\ell_0 [u_2,\hat{u}(c)] \cup \phi_\dagger(u_2)
\end{align}
for each $c_{\dagger,\mathrm{crit}}\leq c\leq c^*(4bm+m/b)$.
\end{Remark}

\begin{Remark}
Similarly as in~\S\ref{sec:singularstripes}, there exist singular homoclinic orbits $p_+(u_2)$ lying entirely in the plane $u=u_2$ for $c=0$; see Remark~\ref{rem:layerhomoclinics}. These orbits are depicted in Figures~\ref{fig:bd_case1} and~\ref{fig:bd_case2} along the lower purple curves. We remark on their presence here, but we refer to~\cite{CD_3Dhomoclinics} for a detailed singular bifurcation analysis.
\end{Remark}


\subsubsection{Heteroclinic orbits connecting desert state $p_0(a)$ and vegetated state $p_+(u_2)$}\label{sec:singularfronts}

To construct singular heteroclinic solution that connect the steady state $p_0(a)$ to the steady state $p_+(u_2)$, we can concatenate $\mathcal{M}_0^\ell$ with a front $\phi_\dagger$ that limits onto the fixed point $p_+(u_2)$. The latter fronts only exist when $p_+(u_2)$ lies on $\mathcal{M}_0^r$, i.e. when $\frac{a}{m} > 4 b + \frac{1}{b}$. Hence, a singular heteroclinic orbit connecting $p_0(a)$ to $p_+(u_2)$ is given by
\begin{equation}
	\mathcal{H}_{\mathrm{dv}}(a) := \mathcal{M}_0^\ell[u_2,a]  \cup \phi_{\dagger}(u_2),
\end{equation}
the speed of which is $c = \hat{c}(a)$.

Similarly, a heteroclinic orbit connecting $p_+(u_2)$ to $p_0(a)$ can be found by concatenating $\mathcal{M}_0^r$ with a front $\phi_\diamond$ that limits onto the fixed point $p_0(a)$. Again, this can only happen when $\frac{a}{m} > 4b + \frac{1}{b}$; a candidate orbit is given by
\begin{equation}
	\mathcal{H}_{\mathrm{vd}}(a) := \mathcal{M}_0^r[u_2,a] \cup \phi_\diamond(a),
\end{equation}
the speed of which is $c = c^*(a)$.

\begin{Remark}
We note that there exist additional heteroclinic orbits for values of $2(b+\sqrt{1+b^2})<\frac{a}{m} < 4 b + \frac{1}{b}$. However, in this parameter regime, the steady state $(U_2,V_2)$ corresponding to the equilibrium $p_+(u_2)$ is unstable (against some non-uniform perturbations) in the original PDE~\eqref{eq:modKlausmeier}. Hence a heteroclinic orbit in this regime corresponds to a front which invades the unstable vegetated state. We do not analyze such invasion fronts in this work; rather, we focus on the bistable regime, corresponding to the singular heteroclinic orbits $\mathcal{H}_{\mathrm{vd}}(a)$ described above.
\end{Remark}

\subsubsection{Periodic orbits}\label{sec:singularperiodicorbits}
In this section, we comment briefly on periodic orbits. Following the construction as for singular homoclinic orbits in~\S\ref{sec:singularstripes}--\ref{sec:singulargaps}, it is also possible to construct singular periodic orbits by concatenating portions of the critical manifolds $\mathcal{M}^\ell_0,\mathcal{M}^r_0$ with fast layer transitions in between, provided the relevant segments of $\mathcal{M}^\ell_0,\mathcal{M}^r_0$ do not contain either of the equilibria $p_0(a)$ or $p_+(u_2)$. Hence, one expects to find singular periodic orbits for any value of $\frac{9b}{2}<\frac{a}{m}<\frac{9b}{2}+\frac{2}{b}$, and any value of the wavespeed $0<c<\min\{c^*(a,b,m),\hat{c}(a,b,m)\}$. Further, general theory predicts that such periodic orbits persist for small $\eps>0$~\cite{SOT}; these solutions correspond to wavetrain solutions of~\eqref{eq:modKlausmeier}, or periodic vegetation stripes. While such solutions are perhaps more ecologically relevant, in the following we focus on traveling pulse solutions as the question of stability, particularly in two spatial dimensions, is more analytically tractable.

We remark that periodic wavetrain solutions have been found in a similar slow-fast context in the FitzHugh--Nagumo equation~\cite{CAR,HAS}, and furthermore, their spectral stability (in one spatial dimension) has been studied in~\cite{ESZ}.

\subsection{Main existence results}\label{sec:mainexistenceresults}

In this section, we have studied~\eqref{eq:modKlausmeier} in the singular limit $\varepsilon \downarrow 0$. Here, we have found several homoclinic and heteroclinic orbits. These orbits persist for $\varepsilon > 0$, as we will prove in \S\ref{sec:existence}. To summarize our findings, we end this section with our main existence results.

\begin{Theorem}[Vegetation stripe solution] \label{thm:stripeexistence} Fix $b,m>0$ and $a$ such that $\frac{a}{m} \in \left( \frac{9}{2}b, \bar{a}_\mathrm{hyp}(b)\right)$. There exists $\eps_0>0$ such that for $\eps\in(0,\eps_0)$,~\eqref{eq:modKlausmeier} admits a traveling pulse solution $\phi_\mathrm{d}(\xi;a,\eps) =(U_\mathrm{d},V_\mathrm{d})(\xi;a,\eps)  $ with speed 
\begin{align}
c_\mathrm{d}(a,\eps)=c^*(a)+\mathcal{O}(\eps)
\end{align}
and satisfying $\lim_{|\xi|\to\infty}(U_\mathrm{d},V_\mathrm{d})(\xi;a,\eps)=(U_0,V_0)$. The length of the vegetation stripe is given to leading order by
\begin{equation}
	\eps\, L_\mathrm{d} := \int_{u^*(a)}^a \frac{du}{u - a + u v_+(u)^2}.
\end{equation}

\end{Theorem}

\begin{Theorem}[Vegetation gap solution]\label{thm:gapexistence}
Fix $b,m>0$ and $a$ such that $\frac{a}{m} \in \left( \bar{a}(b), \frac{9}{2}b + \frac{2}{b} \right)$. There exists $\eps_0>0$ such that for $\eps\in(0,\eps_0)$,~\eqref{eq:modKlausmeier} admits a traveling pulse solution $\phi_\mathrm{v}(\xi;a,\eps) =(U_\mathrm{v},V_\mathrm{v})(\xi;a,\eps)  $ with speed 
\begin{align}
c_\mathrm{v}(a,\eps)=\hat{c}(a)+\mathcal{O}(\eps)
\end{align}
and satisfying $\lim_{|\xi|\to\infty}(U_\mathrm{v},V_\mathrm{v})(\xi;a,\eps)=(U_2,V_2)$. The length of the vegetation gap is given to leading order by
\begin{equation}
	\eps\, L_\mathrm{v} := \int_{\hat{u}_2}^{u_2} \frac{du}{u - a} = \log\left( \frac{u_2(a)-a}{\hat{u}_2-a}\right).
\end{equation}

\end{Theorem}

\begin{Theorem}[Desert front solution] \label{thm:desertFrontExistence} 
Fix $b,m > 0$ and $a$ such that $\frac{a}{m} > 4b + \frac{1}{b}$. There exists $\varepsilon_0 > 0$ such that for $\varepsilon \in (0,\varepsilon_0)$, \eqref{eq:modKlausmeier} admits a traveling front solution $\phi_{\mathrm{dv}}(\xi;a,\varepsilon) = (U_{\mathrm{dv}},V_{\mathrm{dv}})(\xi;a,\varepsilon)$ with speed
\begin{equation}
	c_{\mathrm{dv}}(a,\varepsilon) = c^*(a) + \mathcal{O}(\eps)
\end{equation}
and satisfying $\lim_{\xi \to -\infty} (U_{\mathrm{dv}},V_{\mathrm{dv}})(\xi;a,\eps) = (U_0,V_0)$ and $\lim_{\xi \to \infty} (U_{\mathrm{dv}},V_{\mathrm{dv}})(\xi;a,\eps) = (U_2,V_2)$.
\end{Theorem}

\begin{Theorem}[Vegetation front solution] \label{thm:vegetationFrontExistence}
Fix $b,m > 0$ and $a$ such that $\frac{a}{m} > 4b + \frac{1}{b}$. There exists $\varepsilon_0 > 0$ such that for $\varepsilon \in (0,\varepsilon_0)$, \eqref{eq:modKlausmeier} admits a traveling front solution $\phi_{\mathrm{vd}}(\xi;a,\varepsilon) = (U_{\mathrm{vd}},V_{\mathrm{vd}})(\xi;a,\varepsilon)$ with speed
\begin{equation}
	c_{\mathrm{vd}}(a,\varepsilon) = \hat{c}(a) + \mathcal{O}(\eps)
\end{equation}
and satisfying $\lim_{\xi \to -\infty} (U_{\mathrm{vd}},V_{\mathrm{vd}})(\xi;a,\eps) = (U_2,V_2)$ and $\lim_{\xi \to \infty} (U_{\mathrm{vd}},V_{\mathrm{vd}})(\xi;a,\eps) = (U_0,V_0)$.
\end{Theorem}

\section{Persistence of solutions for $0<\eps\ll1$}\label{sec:existence}
In this section, we prove that the singular orbits constructed in~\ref{sec:singularsolns} perturb to solutions of~\eqref{eq:twode} for sufficiently small $\eps>0$ using methods of geometric singular perturbation theory. In~\S\ref{sec:transversality}, we prove technical lemmata regarding the transversality of the fast connections $\phi_{\dagger,\diamond}$, and we discuss the proofs of Theorems~\ref{thm:stripeexistence}--\ref{thm:vegetationFrontExistence} in~\S\ref{sec:existenceproof}.

\subsection{Transversality along singular orbits}\label{sec:transversality}

We consider the layer system~\eqref{eq:fast0}  
\begin{align}\label{eq:layer0}
\begin{cases}
u'&=0\\
v'&=q\\
q'&= mv-(1-bv)uv^2-cq.
\end{cases}
\end{align}
As outlined in~\S\ref{sec:layer}, this system possesses heteroclinic connections $\phi_\mathrm{\diamond,\dagger} = (v_{\diamond,\dagger},q_{\diamond,\dagger})$ between the left and right critical manifolds $\mathcal{M}^{\ell,r}_0$, where the speed $c$ for a given heteroclinic orbit depends on the value of $u$ (as well as the other parameters). We define the stable and unstable manifolds, $\mathcal{W}^\mathrm{s}(\mathcal{M}^j_0)$ and $\mathcal{W}^\mathrm{u}(\mathcal{M}^j_0)$, of a critical manifold $\mathcal{M}^j_0,j=\ell,r$, as the union of the stable and unstable manifolds, respectively, of the corresponding equilibria of the layer problem~\eqref{eq:layer0}.

Then an orbit $\phi_\dagger$ lies in the intersection of $\mathcal{W}^\mathrm{u}(\mathcal{M}^\ell_0)$ and $\mathcal{W}^\mathrm{s}(\mathcal{M}^r_0)$, while an orbit $\phi_\diamond$ lies in the intersection of $\mathcal{W}^\mathrm{u}(\mathcal{M}^r_0)$ and $\mathcal{W}^\mathrm{s}(\mathcal{M}^\ell_0)$. For a given orbit $\phi_\dagger$, which we suppose exists for some values of $(c,u) = (c_0,u_0)$, we aim to determine how this connection breaks as $(c,u)$ varies near $(c_0,u_0)$; that is, we determine the transversality of the intersection of $\mathcal{W}^\mathrm{u}(\mathcal{M}^\ell_0)$ and $\mathcal{W}^\mathrm{s}(\mathcal{M}^r_0)$ with respect to $(c,u)$. We find the following.
\begin{Lemma}\label{lem:transf}
Consider a heteroclinic orbit $\phi_\dagger$ which lies in the intersection of $\mathcal{W}^\mathrm{u}(\mathcal{M}^\ell_0)$ and $\mathcal{W}^\mathrm{s}(\mathcal{M}^r_0)$ for some $(c,u) = (c_0,u_0)$. Then this intersection is transverse in $(c,u)$, and we compute the splitting of $\mathcal{W}^\mathrm{u}(\mathcal{M}^\ell_0)$ and $\mathcal{W}^\mathrm{s}(\mathcal{M}^r_0)$ along $\phi_\dagger$ via the distance function
\begin{align}\label{eq:splittingf}
D_\dagger(\tilde{c},\tilde{u}) = M_\dagger^c\tilde{c}+M_\dagger^u\tilde{u}+\mathcal{O}(\tilde{c}^2+\tilde{u}^2)
\end{align}
where $\tilde{c} := c-c_0, \tilde{u} := u-u_0$, and
\begin{align}\label{eq:melnikovcomputef}
\begin{split}
M_\dagger^c&= \int^{\infty}_{-\infty} e^{c_0\xi}q_\dagger(\xi)^2 \rmd\xi > 0,\\
M_\dagger^u&= \int^{\infty}_{-\infty} e^{c_0\xi}(1-bv_\dagger(\xi))v_\dagger(\xi)^2q_\dagger(\xi) \rmd\xi>0.
\end{split}
\end{align}
\end{Lemma}
\begin{proof}
We use Melnikov theory to compute the distance between $\mathcal{W}^\mathrm{u}(\mathcal{M}^\ell_0)$ and $\mathcal{W}^\mathrm{s}(\mathcal{M}^r_0)$ to first order in $|c-c_0|$ and $|u-u_0|$. We consider the adjoint equation of the linearization of~\eqref{eq:layer0} about the front $\phi_\dagger$ given by
\begin{align}\label{eq:layer0adj}
\psi' = \begin{pmatrix}
0&-m+uv_\dagger(\xi)(2-3bv_\dagger(\xi))\\[1em]
-1&c
\end{pmatrix}\psi. 
\end{align}
The space of bounded solutions is one-dimensional and spanned by 
\begin{align}
\begin{split}
\psi_\dagger(\xi):= e^{c_0\xi}\begin{pmatrix}q_\dagger'(\xi)\\ -v_\dagger'(\xi)\end{pmatrix}\\
=e^{c_0\xi}\begin{pmatrix}q_\dagger'(\xi)\\ -q_\dagger(\xi)\end{pmatrix}
\end{split}
\end{align}
Let $F_0$ denote the right hand side of~\eqref{eq:layer0}, and define the Melnikov integrals
\begin{align}\label{eq:melnikov}
M_\dagger^\nu&:= \int^{\infty}_{-\infty} D_\nu F_0(\phi_\dagger(\xi))\cdot \psi_\dagger(\xi)\rmd\xi,
\end{align}
for $\nu = c,u$. The quantities $M_\dagger^c, M_\dagger^u$ measure the distance between $\mathcal{W}^\mathrm{u}(\mathcal{M}^\ell_0)$ and $\mathcal{W}^\mathrm{s}(\mathcal{M}^r_0)$ to first order in $|c-c_0|$ and $|u-u_0|$, respectively. We compute
\begin{align*}
\begin{split}
M_\dagger^c&= \int^{\infty}_{-\infty} e^{c_0\xi}q_\dagger(\xi)^2 \rmd\xi > 0,\\
M_\dagger^u&= \int^{\infty}_{-\infty} e^{c_0\xi}(1-bv_\dagger(\xi))v_\dagger(\xi)^2q_\dagger(\xi) \rmd\xi>0.
\end{split}
\end{align*}
As these are nonzero, we deduce that the intersection of $\mathcal{W}^\mathrm{u}(\mathcal{M}^\ell_0)$ and $\mathcal{W}^\mathrm{s}(\mathcal{M}^r_0)$ along $\phi_\dagger$ is transverse in both $c$ and $u$, and we arrive at the distance function~\eqref{eq:splittingf}.
\end{proof}

Analogously, we can determine the transversality of the intersection of $\mathcal{W}^\mathrm{u}(\mathcal{M}^r_0)$ and $\mathcal{W}^\mathrm{s}(\mathcal{M}^\ell_0)$ along an orbit $\phi_\diamond$.
\begin{Lemma}\label{lem:transb}
Consider a heteroclinic orbit $\phi_\diamond$ which lies in the intersection of $\mathcal{W}^\mathrm{u}(\mathcal{M}^r_0)$ and $\mathcal{W}^\mathrm{s}(\mathcal{M}^\ell_0)$ for some $(c,u) = (c_0,u_0)$. Then this intersection is transverse in $(c,u)$, and we compute the splitting of $\mathcal{W}^\mathrm{u}(\mathcal{M}^r_0)$ and $\mathcal{W}^\mathrm{s}(\mathcal{M}^\ell_0)$ along $\phi_\diamond$ via the distance function
\begin{align}\label{eq:splittingb}
D_\diamond(\tilde{c},\tilde{u}) = M_\diamond^c\tilde{c}+M_\diamond^u\tilde{u}+\mathcal{O}(\tilde{c}^2+\tilde{u}^2),
\end{align}
where $\tilde{c} := c-c_0, \tilde{u} := u-u_0$, and
\begin{align}\label{eq:melnikovcomputeb}
\begin{split}
M_\diamond^c&= \int^{\infty}_{-\infty} e^{c_0\xi}q_\diamond(\xi)^2 \rmd\xi > 0,\\
M_\diamond^u&= \int^{\infty}_{-\infty} e^{c_0\xi}(1-bv_\diamond(\xi))v_\diamond(\xi)^2q_\diamond(\xi) \rmd\xi<0.
\end{split}
\end{align}
\end{Lemma}

\subsection{Proof of existence results}\label{sec:existenceproof}
In this section, we conclude the proof of Theorem~\ref{thm:stripeexistence}. The proof of Theorem~\ref{thm:gapexistence} is similar. The proofs of Theorems~\ref{thm:desertFrontExistence} and~\ref{thm:vegetationFrontExistence} also follow a similar argument -- albeit less involved -- and we omit the details.

\begin{proof}[Proof of Theorem~\ref{thm:stripeexistence}]
Based on the analysis in~\S\ref{sec:slowfast}, we obtain a traveling pulse solution of~\eqref{eq:modKlausmeier} as a perturbation from the singular homoclinic orbit $\mathcal{H}_\mathrm{d}(a)$ (see~\eqref{eq:singhomd} and Figure~\ref{fig:pulse_construct_desert}) within the traveling wave ODE~\eqref{eq:twode} for a speed $c\approx c^*(a)$. We will construct a homoclinic orbit for $0<\eps\ll1$ as an intersection of the stable and unstable manifolds $\mathcal{W}^\mathrm{s}(p_0(a))$ and $\mathcal{W}^\mathrm{u}(p_0(a))$ of the equilibrium $p_0(a)$ corresponding to the desert state.

\begin{figure}
\centering
\includegraphics[width=0.5\textwidth]{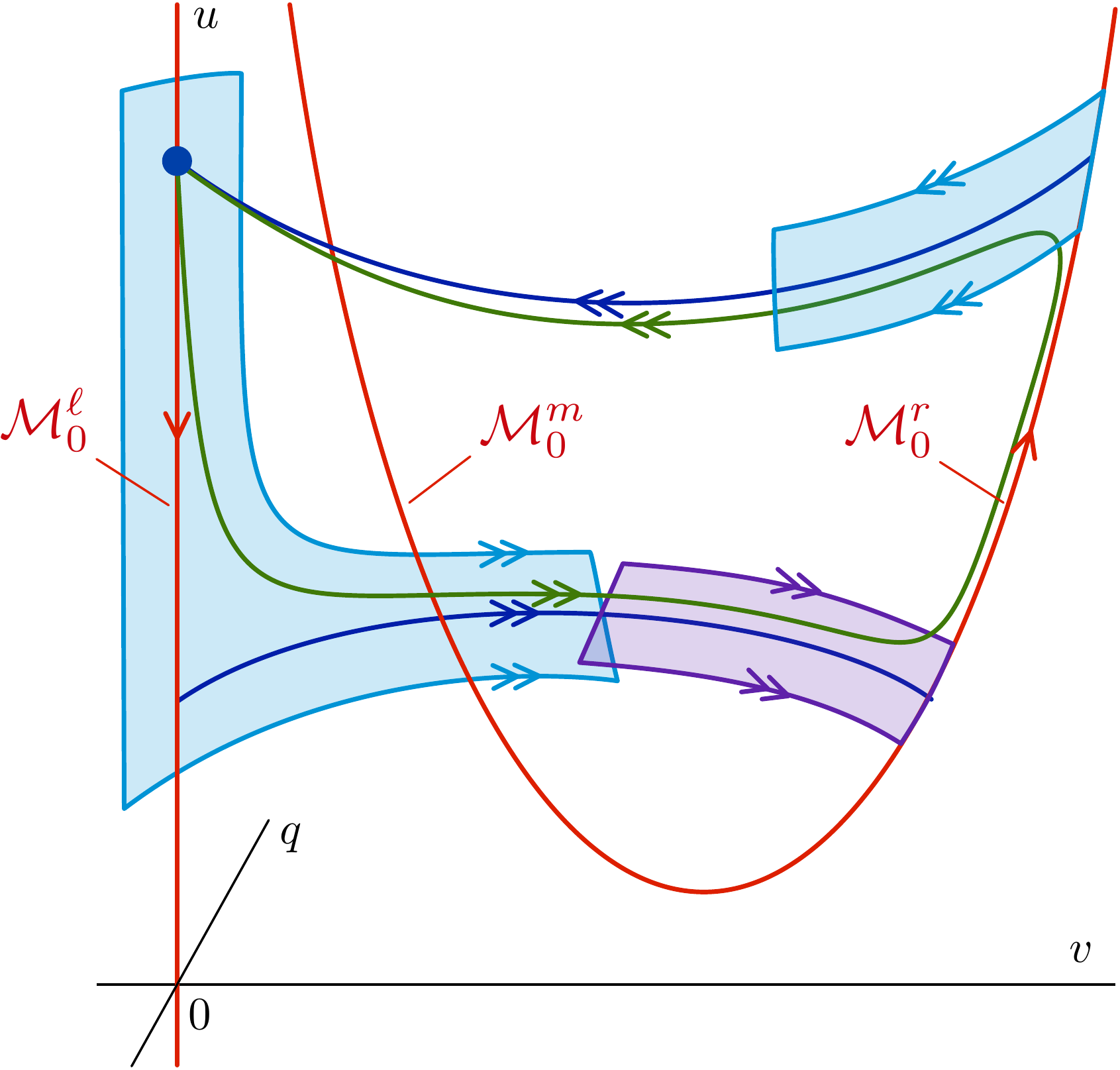}
\caption{The traveling pulse solution of Theorem~\ref{thm:stripeexistence} is obtained for $0<\eps\ll1$ as a perturbation of the singular homoclinic orbit $\mathcal{H}_\mathrm{d}(a)$. }
\label{fig:pulse_construct_desert}
\end{figure}

For $\eps_0>0$ sufficiently small, from standard methods of geometric singular perturbation theory, as the left branch $\mathcal{M}^\ell_0$ of the critical manifold is normally hyperbolic, it persists for $\eps\in(0,\eps_0)$ as a one-dimensional locally invariant slow manifold $\mathcal{M}^\ell_\eps$. Similarly, away from the fold $\mathcal{F}$, the right branch $\mathcal{M}^r_0$ of the critical manifold is normally hyperbolic and persists for $\eps\in(0,\eps_0)$ as a one-dimensional locally invariant slow manifold $\mathcal{M}^r_\eps$. The two-dimensional (un)stable manifolds $\mathcal{W}^\mathrm{u}(\mathcal{M}^j_0)$ and $\mathcal{W}^\mathrm{s}(\mathcal{M}^j_0)$, $j=\ell,r$, persist for $\eps\in(0,\eps_0)$ as two-dimensional locally invariant manifolds $\mathcal{W}^\mathrm{u}(\mathcal{M}^j_\eps)$ and $\mathcal{W}^\mathrm{s}(\mathcal{M}^j_\eps)$, $j=\ell,r$. 

As the equilibrium $p_0(a)$ is repelling with respect to the reduced flow on $\mathcal{M}^\ell_0$ (see~\S\ref{sec:reduced}), for sufficiently small $\eps>0$, the two-dimensional unstable manifold $\mathcal{W}^\mathrm{u}(p_0(a))$ of $p_0(a)$ coincides with $\mathcal{W}^\mathrm{u}(\mathcal{M}^\ell_\eps)$. The equilibrium $p_0(a)$ also admits a one-dimensional stable manifold $\mathcal{W}^\mathrm{s}(p_0(a))$ which precisely corresponds the strong stable fiber of $\mathcal{W}^\mathrm{s}(\mathcal{M}^\ell_\eps)$ with basepoint $p_0(a)$. We note that for $\eps=0$ and $c=c^*(a)$, the manifold $\mathcal{W}^\mathrm{s}(p_0(a))$ is precisely the singular front $\phi_\diamond(a)$.

Using the results of Lemma~\ref{lem:transf} for $c_0=c^*(a), u_0=u^*(a)$, for each fixed $c\approx c^*(a)$ the two-dimensional manifolds $\mathcal{W}^\mathrm{u}(\mathcal{M}^\ell_0)$ and $\mathcal{W}^\mathrm{s}(\mathcal{M}^r_0)$ intersect transversely along the front $\phi_\dagger(u^*(a))$. This transversality persists for sufficiently small $\eps>0$, and using the fact that $\mathcal{W}^\mathrm{u}(p_0(a))=\mathcal{W}^\mathrm{u}(\mathcal{M}^\ell_\eps)$, we deduce the transverse intersection of $\mathcal{W}^\mathrm{u}(p_0(a))$ and $\mathcal{W}^\mathrm{s}(\mathcal{M}^r_\eps)$ for each $c\approx c^*(a)$ and each sufficiently small $\eps>0$. We now track $\mathcal{W}^\mathrm{u}(p_0(a))$ as it passes near $\mathcal{M}^r_\eps$; by the exchange lemma~\cite{jkk,ssch2}, there is a constant $\eta > 0$ such that $\mathcal{W}^\mathrm{u}(p_0(a))$ aligns $C^1$-$\mathcal{O}(e^{-\eta/\eps})$-close to $\mathcal{W}^\mathrm{u}(\mathcal{M}^r_\eps)$ upon exiting a neighborhood of $\mathcal{M}^r_\eps$ near the front $\phi_\diamond(a)$.

Using Lemma~\ref{lem:transb} for $c_0=c^*(a), u_0=a$, we can compute the distance between $\mathcal{W}^\mathrm{u}(\mathcal{M}^r_\eps)$ and $\mathcal{W}^\mathrm{s}(p_0(a))$ along the singular front $\phi_\diamond(a)$ using the distance function~\eqref{eq:splittingb}. In order to find a homoclinic orbit, we are interested in intersections of $\mathcal{W}^\mathrm{u}(p_0(a))$ and $\mathcal{W}^\mathrm{s}(p_0(a))$. By the $C^1$-$\mathcal{O}(e^{-\eta/\eps})$-closeness of $\mathcal{W}^\mathrm{u}(p_0(a))$ and $\mathcal{W}^\mathrm{u}(\mathcal{M}^r_\eps)$, the resulting distance function differs only by $\mathcal{O}(e^{-\eta/\eps})$ terms. Hence we compute the distance between $\mathcal{W}^\mathrm{u}(p_0(a))$ and $\mathcal{W}^\mathrm{s}(p_0(a))$ along $\phi_\diamond(a)$ as
\begin{align}\label{eq:homoclinicsplitting}
D(\tilde{c},\tilde{u},\eps) = M_\diamond^c\tilde{c}+\mathcal{O}(\eps+\tilde{c}^2),
\end{align}
where $M_\diamond^c\neq0$ and $\tilde{c} = c-c^*(a)$. We solve for $D(\tilde{c},\tilde{u},\eps)=0$ when
\begin{align}
c=c_\mathrm{d}(a,\eps) = c^*(a)+\mathcal{O}(\eps),
\end{align}
which corresponds to an intersection of $\mathcal{W}^\mathrm{u}(p_0(a))$ and $\mathcal{W}^\mathrm{s}(p_0(a))$ along a homoclinic orbit of~\eqref{eq:twode}.
\end{proof}

\section{Stability}\label{sec:stabilityFormal}

In the previous sections we have constructed several different localized solutions to~\eqref{eq:TWPDE}: homoclinics to the desert state $(u,v)=(U_0,V_0) = (a,0)$, homoclinics to the vegetated state $(u,v) = (U_2,V_2)$ -- see~\eqref{eq:uniformSteadyStates} -- and heteroclinics connecting these states. In this section we study the linear stability of these solutions using formal arguments; rigorous proofs follows in \S\ref{sec:stabilityRigorous}. We denote a steady state solution to~\eqref{eq:TWPDE} by $(u_s,v_s)$ -- without specifying yet which steady state solution -- and we linearize around this state by setting $(u,v)(\xi,t) = (u_s,v_s)(\xi) + e^{\lambda t+i\ell y} (\bar{u},\bar{v})(\xi)$. The linear stability problem then reads
\begin{equation}
	\begin{cases}
		\lambda \bar{u} & = \frac{1+\varepsilon c_s}{\varepsilon} \bar{u}_\xi - \left(1 + v_s^2\right) \bar{u} - 2 u_s v_s \bar{v}, \\
		\lambda \bar{v} & = \bar{v}_{\xi\xi} + c_s \bar{v}_\xi + \left( -m-\ell^2 + (2 - 3 b v_s) u_s v_s\right)\bar{v} + (1 - b v_s)v_s^2 \bar{u}. \label{eq:eigenvalueProblem}
	\end{cases}
\end{equation}
Here, $c_s$ denotes the speed of the steady state under consideration. With the introduction of $\bar{q} := \bar{v}_\xi$ we can write this stability problem in matrix form as
\begin{equation}
	\left( \begin{array}{c} \bar{u}_\xi \\ \bar{v}_\xi \\ \bar{q}_\xi \end{array} \right) =	A
 \left( \begin{array}{c} \bar{u} \\ \bar{v} \\ \bar{q} \end{array}\right),
 \mbox{ where } A = \left( \begin{array}{ccc}
\frac{\varepsilon}{1+\varepsilon c_s} \left[1 + \lambda + v_s^2 \right] & \frac{\varepsilon}{1+\varepsilon c_s} 2 u_s v_s & 0 \\
0 & 0 & 1 \\
- (1 - b v_s) v_s^2 & m +\ell^2+ \lambda - (2 - 3 b v_s) u_s v_s & -c_s
\end{array}\right).\label{eq:eigenvalueProblemMatrixForm}
\end{equation}

The rest of this section is devoted to finding the spectrum $\Sigma$ of this eigenvalue problem for the different stationary solutions to~\eqref{eq:TWPDE}, using formal computations. This spectrum consists of an essential spectrum $\Sigma_\mathrm{ess}$ and a point spectrum $\Sigma_\mathrm{pt}$. The essential spectrum is studied in \S\ref{sec:essentialSpectrum} and the point spectrum in \S\ref{sec:pointSpectrumFormal}. In~\S\ref{sec:stabilityTheorems} we formulate theorems based on our findings, the proofs of which are given in~\S\ref{sec:stabilityRigorous}.

\subsection{Essential spectrum}\label{sec:essentialSpectrum}

The essential spectrum consists of all eigenvalues $\lambda$ such that an asymptotic matrix of~\eqref{eq:eigenvalueProblemMatrixForm} has a spatial eigenvalue with real part zero. Depending on the type of steady state solution we are inspecting, the asymptotic matrix or matrices might be different. However, since we are only considering steady state solutions that limit to either the desert state $(u,v) = (a,0)$ or the vegetated state $(u,v) = (U_2,V_2)$, there are only two possible asymptotic matrices; when $(u_s,v_s)$ limits to $(a,0)$ (for either $\xi \rightarrow \infty$ or $\xi \rightarrow -\infty$) we have $A_\mathrm{d}$ as asymptotic matrix and when $(u_s,v_s)$ limits to  $(U_2,V_2)$ we have $A_\mathrm{v}$, where these matrices are given by
\begin{align}\label{eq:asympAd}
	A_\mathrm{d}(\lambda;\ell) &= 
\left( \begin{array}{ccc}
\frac{\varepsilon}{1+\varepsilon c_s} [1 + \lambda] & 0 & 0 
\\ 0 & 0 & 1 
\\ 0 & m +\ell^2+ \lambda & -c_s
\end{array} \right)\\ \label{eq:asympAv}
A_\mathrm{v}(\lambda;\ell) &= 
\left( \begin{array}{ccc}
\frac{\varepsilon}{1+\varepsilon c_s} \left[ 1 + \lambda + V_2^2 \right] & \frac{\varepsilon}{1+\varepsilon c_s} 2 U_2 V_2 & 0 \\
0 & 0 & 1 \\
- (1 - b V_2) V_2^2 & m+\ell^2 + \lambda - (2 - 3 b V_2) U_2 V_2 & - c_s
\end{array}\right),
\end{align}
where the values for $U_2$ and $V_2$ are given in~\eqref{eq:uniformSteadyStates}. 

\begin{Lemma}
Concerning the asymptotic matrices $A_\mathrm{d}, A_\mathrm{v}$ defined in~\eqref{eq:asympAd}--\eqref{eq:asympAv}, we have the following.
\begin{enumerate}[(i)]
\item \label{lem:essd}The matrix $A_\mathrm{d}$ is hyperbolic for all $\lambda \in \mathbb{C}$ satisfying \begin{align}\Re \lambda>-\min\{m+\ell^2,1\}.\end{align}

\item \label{lem:essv} For values of $a,m,b>0$ satisfying $\frac{a}{m}>4b+\frac{1}{b}$, the matrix $A_\mathrm{v}$ is hyperbolic for all $\lambda \in \mathbb{C}$ satisfying \begin{align}\Re \lambda>-\min\left\{1+\frac{1}{4b^2},\quad \frac{2m\left(b\sqrt{a^2-4m(m+ab)}-m\right)}{2m+ab-b\sqrt{a^2-4m(m+ab)}}+\ell^2\right\}<0. \end{align}
\end{enumerate}
\end{Lemma}
\begin{proof}
For~\ref{lem:essd}, a straightforward computation reveals that $A_\mathrm{d}$ is non-hyperbolic when $\lambda \in \{ \lambda \in \C: \Re \lambda = -1\} \cup \{ \lambda = -m-\ell^2 - k^2 + i c_s k; k \in \R\}$; see Figure~\ref{fig:essential_spec}.

For~\ref{lem:essv}, we compute that $A_\mathrm{v}$ is non-hyperbolic when
\begin{align}\label{eq:essv_solve}
\left(\frac{\varepsilon}{1+\varepsilon c_s} \left( 1 + \lambda + V_2^2 \right)-i\nu\right)\left(i\nu c_s-\nu^2- m-\ell^2 - \lambda + (2 - 3 b V_2) U_2 V_2\right)-\frac{\varepsilon}{1+\varepsilon c_s} 2 U_2 V_2(1 - b V_2) V_2^2=0,
\end{align}
for some $\nu\in \R$. We note that 
\begin{align}
\Re \left(i\nu c_s-\nu^2- m -\ell^2- \lambda + (2 - 3 b V_2) U_2 V_2\right)<0
\end{align}
whenever
\begin{align}
\Re \lambda > -m-\ell^2 + (2 - 3 b V_2) U_2 V_2. 
\end{align}
Furthermore, using the expressions~\eqref{eq:uniformSteadyStates}, when $\frac{a}{m}>4b+\frac{1}{b}$, we have that $V_2>\frac{1}{2b}$ and
\begin{align*}
-m-\ell^2 + (2 - 3 b V_2) U_2 V_2&=-\frac{2m\left(b\sqrt{a^2-4m(m+ab)}-m\right)}{2m+ab-b\sqrt{a^2-4m(m+ab)}}-\ell^2\\
&<0.
\end{align*}
for all $\ell \in\R$. By rearranging~\eqref{eq:essv_solve}, we deduce that $A_\mathrm{v}$ is non-hyperbolic when
\begin{align}\label{eq:essv_solve_re}
\lambda=-1-V_2^2+\frac{2 U_2 V_2(1 - b V_2) V_2^2}{\left(i\nu c-\nu^2- m-\ell^2 - \lambda + (2 - 3 b V_2) U_2 V_2\right)}+i\nu\frac{1+\varepsilon c_s}{\varepsilon}.
\end{align}
Taking real parts of~\eqref{eq:essv_solve_re} in the region 
\begin{align}
\Re \lambda > -\frac{2m\left(b\sqrt{a^2-4m(m+ab)}-m\right)}{2m+ab-b\sqrt{a^2-4m(m+ab)}} -\ell^2,
\end{align} we have that $\Re \lambda<-1-V_2^2$, and noting $V_2>\frac{1}{2b}$, the result follows.
\end{proof}

\begin{figure}
\centering
\includegraphics[width=0.3\textwidth]{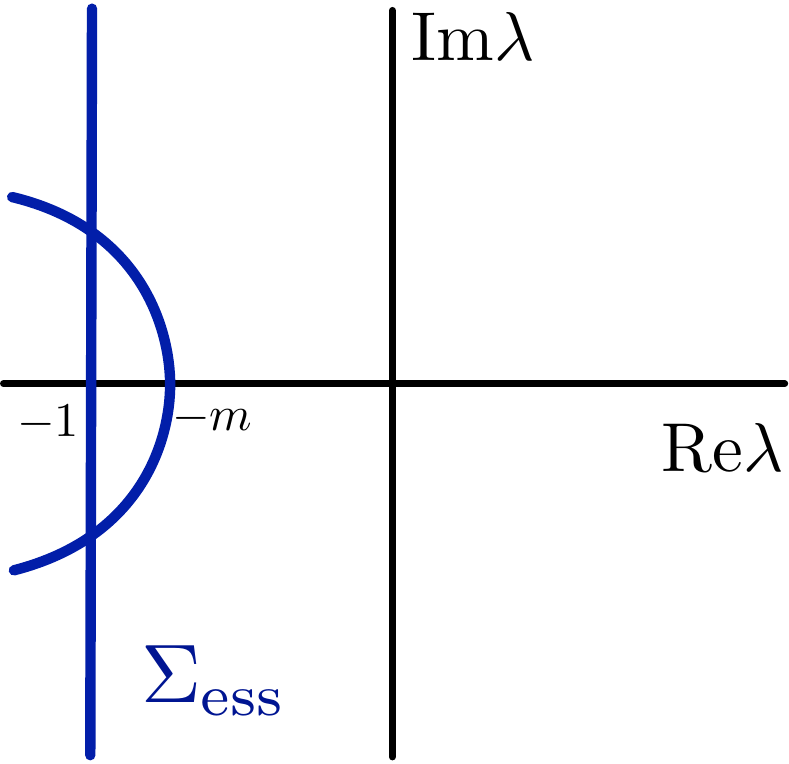}
\caption{Shown is the essential spectrum $\Sigma_{\mathrm{ess}}$ associated with the desert state $(u,v) = (a,0)$ in the case $\ell=0$. }
\label{fig:essential_spec}
\end{figure}

Thus, since both $A_\mathrm{d}$ and $A_\mathrm{v}$ stay hyperbolic for all $\lambda$ with $\Re \lambda \geq 0$ for the relevant parameter values, the essential spectrum of all of the types of steady state solutions found in Section~\ref{sec:slowfast} is located in the left half-plane.

\subsection{Point spectrum}\label{sec:pointSpectrumFormal}

In this section we study the point spectrum $\Sigma_\mathrm{pt}$ using formal perturbation theory. Here we focus on $1$D stability, that is $\ell=0$. Rigorous proofs of the statements in this section, and the extension to all $\ell\in\R$, follow in~\S\ref{sec:stabilityRigorous}. 

We observe that the slow manifolds $\mathcal{M}^{\ell,r}_0$ are hyperbolic (away from the fold point $\mathcal{F}$) and consist entirely of saddle equilibria of the fast layer problem~\eqref{eq:fast0}. Hence, these slow manifolds should not contribute any eigenvalues; the only eigenvalues come from the contribution of the fast fronts $\phi_\dagger$ and $\phi_\diamond$. That is, eigenvalues in the point spectrum lie close to the eigenvalues of the fast-reduced subsystem~\eqref{eq:fast0}. Since $\phi_\dagger$ and $\phi_\diamond$ are fronts and~\eqref{eq:fast0} is translational invariant, standard Sturm-Liouville theory indicates that they carry an eigenvalue $\lambda = 0$ and possibly several other eigenvalues that are all real and negative. Therefore, if there are potentially unstable eigenvalues in the point spectrum $\Sigma_\mathrm{pt}$ they need to lie close to $\lambda = 0$. Specifically, there are as many eigenvalues close to $0$ as there are fronts in the steady state solution $(u_s,v_s)$ under consideration.

Because the full system~\eqref{eq:TWPDE} is translational invariant, $\lambda = 0$ is an eigenvalue of the full system. When we study the stability of a heteroclinic connection (connecting the desert state $p_0(a)$ to the vegetated state $p_+(u_2)$ or vice-versa) this is the only eigenvalue close to $0$; in particular $\Sigma_\mathrm{pt} \backslash \{0\} \subset \{\lambda \in \C: \Re \lambda < 0\}$. On the other hand, when we study the stability of a homoclinic connection (connecting either the desert state $p_0(a)$ or the vegetated state $p_+(u_2)$ to itself), there is an additional eigenvalue close to $0$. This eigenvalue -- of the homoclinic steady state solutions -- can, in principle, move either to the left or to the right (making the steady state unstable). In this section, we use perturbation theory to track this movement and pinpoint the location of the second eigenvalue formally. 

\subsection{Formal computation of small eigenvalues}
Let $(u_s,v_s)$ be an exact solution to~\eqref{eq:TWPDE}. The linearized stability problem~\eqref{eq:eigenvalueProblem} can be recast to the following form

\begin{equation}
\mathcal{L}(\ell) \left( \begin{array}{c} \bar{u}\\\bar{v}\end{array}\right) = \lambda  \left( \begin{array}{c} \bar{u}\\\bar{v}\end{array}\right),
\qquad
\mathcal{L}(\ell) := \left( \begin{array}{cc} 
\eps^{-1}(1+\eps c_s)\partial_\xi - (1 + v_s^2) & - 2 u_s v_s \\
(1-bv_s)v_s^2 & \partial_\xi^2 + c_s \partial_\xi - m - \ell^2 + (2 - 3 b v_s) u_s v_s
\end{array}\right).
\label{eq:stabilityProblemOperator}
\end{equation}
For simplicity, we focus on the operator $\mathcal{L}(0)$ corresponding to the case $\ell=0$; the case of $\ell \in\R$ is similar and is carried out in detail in~\S\ref{sec:stabilityRigorous}. 

Since we are looking for a small (order $\mathcal{O}(\varepsilon)$) eigenvalue closely related to the derivatives of the fast fronts $\phi_\dagger = (u_\dagger,v_\dagger)^T$ and $\phi_\diamond = (u_\diamond,v_\diamond)^T$, in particular at leading order~\eqref{eq:stabilityProblemOperator} is satisfied in the fast $\xi$-fields by any linear combination of $\phi_\dagger'$ and $\phi_\diamond'$. We denote the fast region with the front $\phi_\dagger$ by $I_\dagger$ and the fast region with the front $\phi_\diamond$ by $I_\diamond$. Then, to find the small eigenvalues we therefore use regular expansion and determine the eigenvalues with a Fredholm solvability condition. In particular, we first focus on the fast fields and we expand the eigenvalue and $(\bar{u},\bar{v})^T$ in these fast regions as
\begin{alignat}{4}
	 \left(\begin{array}{c} \bar{u} \\ \bar{v} \end{array}\right) & = \alpha_j \phi_j' &&+ \varepsilon \left(\begin{array}{c} \bar{u}_{j,1} \\ \bar{v}_{j,1} \end{array}\right) && + \mathcal{O}(\varepsilon^2), & \qquad (\xi \in I_j, j = \dagger,\diamond)\\
	 \lambda & =  0 && + \varepsilon \tilde{\lambda} && + \mathcal{O}(\varepsilon^2), &
\intertext{where $\alpha_{\dagger,\diamond}$ are constants to be determined. Moreover, we also need to expand the exact solution $(u_s,v_s)^T$ as well as the speed $c_s$:}
	 \left(\begin{array}{c} u_s \\ v_s \end{array}\right) & = \left(\begin{array}{c} u_{j} \\ v_{j} \end{array}\right) && + \varepsilon \left(\begin{array}{c} u_{j,1} \\ v_{j,1} \end{array}\right) && + \mathcal{O}(\varepsilon^2), &\qquad (\xi \in I_j, j = \dagger,\diamond)\\
	 c_s & = c_{0} && + \varepsilon c_1 && + \mathcal{O}(\varepsilon^2), &
\end{alignat}
where $(u_j,v_j)^T\, (j  = \dagger,\diamond)$ and $c_0$ are the leading order approximations of the exact solutions as constructed in~\S\ref{sec:mainexistenceresults}, Theorems~\ref{thm:stripeexistence} and~\ref{thm:gapexistence}. 
Substitution in~\eqref{eq:stabilityProblemOperator} leads at order $\mathcal{O}(\varepsilon)$ to the following equation (the $\mathcal{O}(1)$ equations are automatically satisfied):
\begin{equation}
	\begin{cases}
		\bar{u}_1' & = 2 \alpha_j u_j v_j v_j', \\
		\mathcal{L}^r_j \bar{v}_1 & = \left( \tilde{\lambda} - c_1 \partial_\xi - \left[2-6 b v_{j}\right]u_{j} v_{j,1} - \left[2-3bv_{j}\right]v_{j}u_{j,1}\right) \alpha_j v_j' - \left[1-bv_{j}\right]v_{j}^2\bar{u}_1,
	\end{cases}\quad (\xi \in I_j,\,\, j = \dagger,\diamond)
\label{eq:eigenvalueProblemOrderEpsFastFields}
\end{equation}
where 
\begin{equation}
\mathcal{L}^r_j := \partial_\xi^2 + c_0 \partial_\xi - m + (2-3b v_{j})u_{j}v_{j}.
\end{equation}
In~\eqref{eq:eigenvalueProblemOrderEpsFastFields} terms with $c_1$, $v_{j,1}$ and $u_{j,1}$ appear, and to determine these, we expand the existence problem~\eqref{eq:twode} in $\varepsilon$ as well. In the fast fields the order $\mathcal{O}(\varepsilon)$ terms read
\begin{equation}
	\begin{cases}
		u_{s,1}' & = u_j - a + u_j v_j^2, \\
		\mathcal{L}^r_j v_{j,1} & = -(1-b v_j)v_j u_{j,1} - c_1 v_j'.
	\end{cases}\quad (\xi \in I_j,\,\, j = \dagger,\diamond)
\label{eq:existenceProblemOrderEps}
\end{equation}
Taking the derivative with respect to $\xi$ of the second equation then yields
\begin{equation}
\mathcal{L}^r_j v_{j,1}'  = \left( -c_1 \partial_\xi - \left[2-6 b v_j\right]u_j v_{j,1} - \left[2-3bv_j\right]v_ju_{j,1}\right)v_j' - \left[1-bv_j\right]v_j^2 u_{j,1}'
\end{equation}
Substitution in~\eqref{eq:eigenvalueProblemOrderEpsFastFields} then reduces the core stability problem to
\begin{equation}
	\begin{cases}
		\bar{u}_1' & = 2 \alpha_j u_j v_j v_j', \\
		\mathcal{L}^r_j \bar{v}_1 & = \alpha_j \mathcal{L}^r_j v_{j,1}' + \tilde{\lambda}\alpha_j v_j' + \left[1-bv_j\right]v_j^2 \left(\alpha_j u_{j,1}' - \bar{u}_1\right).
	\end{cases}\quad (\xi \in I_j,\,\, j  = \dagger,\diamond)
\label{eq:eigenvalueProblemOrderEpsFastFields2}
\end{equation}
From this equation it is clear that $\bar{u}_{j,1}$ can be found by integration (regardless of the value of $\tilde{\lambda}$, $\alpha_\dagger$ and $\alpha_\diamond$). However, since $\mathcal{L}^r_j$ has a non-trivial kernel, we have to impose a solvability condition on $\bar{v}_{j,1}$. We define $v_j^*$ as a solution to the adjoint equation $(\mathcal{L}^r_j)^* v_j^* = 0$ and note that
\begin{equation}\label{eq:adjointSolutionLr}
	v_j^*(\xi) = e^{c_0\xi} v_j'(\xi), \qquad (\xi \in I_j, j = \dagger,\diamond).
\end{equation}
Thus we obtain the following Fredholm solvability condition
\begin{equation}
	0 = \alpha_j \tilde{\lambda} \int_{-\infty}^\infty (v_j')^2 e^{c_0 \xi} d\xi + \int_{-\infty}^\infty \left[1-bv_j\right]v_j^2 e^{c_0\xi} v_j' \left(\alpha_j u_{j,1}' - \bar{u}_{j,1}\right)\ d\xi \qquad (j  = \dagger,\diamond) \label{eq:stabilityFredholmCondition1}
\end{equation}
We observe from~\eqref{eq:eigenvalueProblemOrderEpsFastFields} and~\eqref{eq:existenceProblemOrderEps} that $\alpha_j u_{j,1}' - \bar{u}_{j,1}$ is constant in the fast fields $I_j\, (j = \dagger,\diamond)$. Thus the Fredholm condition reduces to
\begin{equation}
	0 = \alpha_j \tilde{\lambda} \int_{-\infty}^\infty (v_j')^2 e^{c_0 \xi} d\xi + \left(\alpha_j u_{j,1}' - \bar{u}_{j,1}\right) \int_{-\infty}^\infty \left[1-bv_j\right]v_j^2 e^{c_0\xi} v_j'  \ d\xi \qquad (j = \dagger,\diamond)\label{eq:stabilityFredholmCondition1}
\end{equation}
Note that we thus have two solvability conditions. Only when both are satisfied simultaneously, it is possible to find $(\bar{u},\bar{v})^T$ that solve~\eqref{eq:stabilityProblemOperator}. The terms in~\eqref{eq:stabilityFredholmCondition1} change depending on the type of steady state solution we are considering, and in particular, to which equilibrium state these solutions are homoclinic, as this determines the value of $\alpha_j u_{j,1}' - \bar{u}_{j,1}$.

\paragraph{Homoclinics to desert state}
In this situation, $u_{\diamond,1}(\xi) \rightarrow 0$ for $\xi \rightarrow \infty$ in $I_\diamond$, since the jump here is onto the fixed point. Moreover, $\bar{u}_{\diamond,1}(\xi) \rightarrow 0$ for $\xi \rightarrow \infty$ in $I_\diamond$ to ensure integrability of the eigenfunction. Thus, the condition in $I_\diamond$ is
\begin{equation}
	\alpha_\diamond \tilde{\lambda} M_{\diamond,\lambda}^\mathrm{d} = 0, \label{eq:FredholmConditionDesertHomoclinic1}
\end{equation}
where
\begin{equation}
 M_{\diamond,\lambda}^\mathrm{d} := \int_{-\infty}^\infty v_\diamond'(\xi)^2 e^{c^*(a)\xi}\ d\xi > 0. \label{eq:MdDiamondLambda}
\end{equation}
Therefore, either $\tilde{\lambda} = 0$ or $\alpha_\diamond = 0$. The former gives us back the translational invariant eigenvalue (with eigenfunction $(\bar{u},\bar{v})^T = (u_s',v_s')^T$, so we focus on the latter possibility. Note that $\alpha_\diamond = 0$ implies that $\bar{u}_{\diamond,1} = 0$ in the fast field $I_\diamond$. Thus, this provides a matching condition for the equations in the slow field between the fast fields $I_\dagger$ and $I_\diamond$. By expanding the slow field equation in the slow variable, it immediately follows, from this fact, that the eigenfunction must be $0$ in the slow field between $I_\dagger$ and $I_\diamond$ as well. Hence we conclude that $\bar{u}_{\dagger,1}(\xi) \rightarrow 0$ for $\xi \rightarrow \infty$ in $I_\dagger$ as well. Moreover, $u_{\dagger,1}(\xi) \rightarrow u_\dagger - a - u_\dagger v_+(u_\dagger)^2 = u^*(a) - a + u^*(a)v_+(u^*(a))^2$ for $\xi \rightarrow \infty$ in $I_\dagger$ -- see equation~\eqref{eq:existenceProblemOrderEps} and Theorem~\ref{thm:stripeexistence}. Thus the second solvability condition becomes
\begin{equation}
	\alpha_\dagger \left[ \tilde{\lambda} M_{\dagger,\lambda}^\mathrm{d} + M_{\dagger,\varepsilon}^\mathrm{d}\right] = 0, \label{eq:FredholmConditionDesertHomoclinic2}
\end{equation}
where
 \begin{align}
	 M_{\dagger,\lambda}^\mathrm{d} & := \int_{-\infty}^\infty v_\dagger'(\xi)^2 e^{c^*(a)\xi}\ d\xi > 0, \label{eq:MdDaggerLambda}\\
	 M_{\dagger,\eps}^\mathrm{d} & := \left[ u^*(a) - a + u^*(a) v_+(u^*(a))^2 \right] \int_{-\infty}^\infty (1-bv_\dagger(\xi)) v_\dagger(\xi)^2 e^{c^*(a) \xi} v_\dagger'(\xi)\ d\xi > 0.\label{eq:MdDaggerDelta}
 \end{align}
The signs of these are positive, since $v_\dagger$ is increasing with $\xi$, and the quantity $\left(u^*(a) - a + u^*(a) v_+(u^*(a))^2\right)$ is positive per construction. Because taking $\alpha_\dagger = 0$ leads to the trivial solution (on $\mathbb{R}$), we therefore obtain the additional eigenvalue $\lambda = \varepsilon \tilde{\lambda} = - \varepsilon \frac{M_{\dagger,\varepsilon}^\mathrm{d}}{M_{\dagger,\lambda}^\mathrm{d}} < 0$, which indicates that the eigenvalue $\lambda$ close to zero has moved into the stable half-plane $\{\lambda \in \C: \Re \lambda < 0\}$. A plot of the corresponding eigenfunction, computed numerically, is given in Figure~\ref{fig:spec_numerics_ef}.

\paragraph{Homoclinics to the vegetated state} This case is very similar. However, now the solution in $I_\dagger$ limits to the fixed point of~\eqref{eq:twode}. Using similar arguments, we then find the following condition in $I_\dagger$:
\begin{equation}
	\alpha_\dagger \tilde{\lambda} M_{\dagger,\lambda}^\mathrm{v} = 0,
\end{equation}
where
\begin{equation}
M_{\dagger,\lambda}^\mathrm{v}  := \int_{-\infty}^\infty v_\dagger'(\xi)^2 e^{\hat{c}(a)\xi}\ d\xi > 0\label{eq:MvDaggerLambda}.
\end{equation}
This time we need to take $\alpha_\dagger = 0$. Similar to before, matching through the slow field yields $\bar{u}_{\dagger,1}(\xi) \rightarrow 0$  and $u_{\diamond,1} \rightarrow u_\diamond - a = \hat{u}_2(a) - a$ for $\xi \rightarrow \infty$ in $I_\diamond$. Therefore the second condition for this steady state reads
\begin{equation}
	\alpha_\diamond \left[ \tilde{\lambda} M_{\diamond,\lambda}^\mathrm{v} + M_{\diamond, \varepsilon}^d\right] = 0,
\end{equation}
where
\begin{align}
	 M_{\diamond,\lambda}^\mathrm{v} &:= \int_{-\infty}^\infty v_\diamond'(\xi)^2 e^{\hat{c}(a)\xi}\ d\xi > 0,\label{eq:MvDiamondLambda}\\
	 M_{\diamond,\eps}^\mathrm{v} & := \left[ \hat{u}_2(a) - a \right] \int_{-\infty}^\infty (1-bv_\diamond(\xi)) v_\diamond(\xi)^2 e^{\hat{c}(a) \xi} v_\diamond'(\xi)\ d\xi > 0\label{eq:MvDiamondDelta}.
 \end{align}
Because $\hat{u}_2(a) - a < 0$ and $v_\diamond$ is decreasing with $\xi$, the sign of all these terms are positive again. Therefore we obtain the additional eigenvalue $\lambda = \varepsilon \tilde{\lambda} = - \varepsilon \frac{M_{\diamond,\varepsilon}^\mathrm{v}}{M_{\diamond,\lambda}^\mathrm{v}} < 0$, and again the eigenvalue has moved into the stable half-plane.

\subsection{Main stability results}\label{sec:stabilityTheorems}

In the previous sections we have formally determined the spectrum of all kind of steady state solutions to~\eqref{eq:TWPDE}. The computations in these sections hold for 1D perturbations of the steady state in question. We do, however, also want to understand the stability of these steady states under 2D perturbations. For that, we linearize around this state by setting $(u,v)(\xi,y,t) = (u_s,v_s)(\xi) + e^{\lambda t + i \ell y} (\bar{u},\bar{v})(\xi)$, where $\ell\in \R$ is the transverse wavenumber, which results in the family of linearized PDE operators
\begin{align}
\mathcal{L}(\ell) := \left( \begin{array}{cc} 
\eps^{-1}(1+\eps c_s)\partial_\xi - 1 - v_s^2 & - 2 u_s v_s \\
(1-bv_s)v_s^2 & \partial_\xi^2 +\ell^2+ c_s \partial_\xi - m + (2 - 3 b v_s) u_s v_s
\end{array}\right).
\label{eq:linPDEoperator}
\end{align}
Linear stability is then determined by the corresponding family of eigenvalue problems
\begin{align}\label{eq:2Devalproblem}
\mathcal{L}(\ell) \begin{pmatrix} U\\V\end{pmatrix}= \lambda  \begin{pmatrix} U\\V\end{pmatrix}, \qquad \ell \in \mathbb{R}. \end{align}

Introducing $\Psi := (\bar{u},\bar{v},\bar{v}')^T$ we write the eigenvalue problem~\eqref{eq:2Devalproblem} as the first order nonautonomous ODE
	\begin{align}
	\Psi' &=	A(\xi;\lambda, \ell,\eps)\Psi,\qquad 
	 A(\xi;\lambda, \ell,\eps)=\begin{pmatrix}
\frac{\varepsilon}{1+\varepsilon c_s} \left[1 + \lambda + v_s^2 \right] & \frac{\varepsilon}{1+\varepsilon c_s} 2 u_s v_s & 0 \\
0 & 0 & 1 \\- (1 - b v_s) v_s^2 & m + \lambda +\ell^2- (2 - 3 b v_s) u_s v_s & -c_s\end{pmatrix}.\label{eq:evalprobleml}
	\end{align}
By introducing $\hat{\lambda} = \lambda - \ell^2$ it is suggested that previous computations for the point spectrum, in~\S\ref{sec:pointSpectrumFormal}, still hold up to leading order by replacing $\lambda$ with $\hat{\lambda}$. The change in the essential spectrum is a bit more subtle, but computations are similar to those in~\S\ref{sec:essentialSpectrum}. To summarize our findings, we formulate several stability theorems for the various types of steady state solutions; these are proved rigorously in~\S\ref{sec:stabilityRigorous}.

\begin{Theorem}[Spectrum of traveling front solutions] \label{thm:heteroclinicStability} Let $a,b,m,\varepsilon$ as in Theorem~\ref{thm:desertFrontExistence} or~\ref{thm:vegetationFrontExistence} and let $\phi_h$ denote a traveling front solution as in the same theorem. Then, the following hold.
\begin{enumerate}[(i)]
\item  The spectrum of the operator $\mathcal{L}(0)$ is contained in the set $\{\lambda \in \C:\Re \lambda <0\}\cup \{0\}$, and the spectrum of the operator $\mathcal{L}(\ell),\ell\neq0$ is contained in the set $\{\lambda \in \C:\Re \lambda <0\}$. 
\item The eigenvalue $\lambda_0(0)=0$ of $\mathcal{L}(0)$ is simple and continues to an eigenvalue of $\mathcal{L}(\ell), |\ell|\leq L_M$ for some $L_M\gg1$, satisfying $\lambda_0'(0)=0$ and
\begin{align}
	\lambda_0(\ell) &= - \ell^2 + \mathcal{O}(|\varepsilon \log \varepsilon|^2), \qquad \lambda''_0(\ell) = -2 + \mathcal{O}(|\varepsilon \log \varepsilon|^2), \qquad |\ell|\leq L_M.
\end{align}
\item The remaining spectrum of $\mathcal{L}(\ell)$ is bounded away from the imaginary axis uniformly in $\eps>0$ sufficiently small and $\ell\in\R$.
\end{enumerate}
\end{Theorem}

\begin{Theorem}[Spectrum of vegetation stripe solutions] \label{thm:stripeStability} Let $a,b,m,\varepsilon$ as in Theorem~\ref{thm:stripeexistence} and let $\phi_\mathrm{d}$ be a traveling pulse `stripe' solution as in Theorem~\ref{thm:stripeexistence}. Then, the following hold.
\begin{enumerate}[(i)]
\item The spectrum of the operator $\mathcal{L}(0)$ is contained in the set $\{\lambda \in \C:\Re \lambda <0\}\cup \{0\}$, and the spectrum of the operator $\mathcal{L}(\ell),\ell\neq0$ is contained in the set $\{\lambda \in \C:\Re \lambda <0\}$. 

\item The eigenvalue $\lambda_0(0)=0$ of $\mathcal{L}(0)$ is simple and continues to an eigenvalue of $\mathcal{L}(\ell), |\ell|\leq L_M$ for some $L_M\gg1$, satisfying $\lambda_0'(0)=0$ and
\begin{align}
	\lambda_0(\ell) &= - \ell^2 + \mathcal{O}(|\varepsilon \log \varepsilon|^2), \qquad \lambda''_0(\ell) = -2 + \mathcal{O}(|\varepsilon \log \varepsilon|^2), \qquad |\ell|\leq L_M.
\end{align}
\item  The operator $\mathcal{L}(\ell),|\ell|\leq L_M$ admits an additional critical eigenvalue 
\begin{equation}
\lambda_c(\ell) = -\ell^2 - \frac{M^\mathrm{d}_{\dagger,\eps}}{M^\mathrm{d}_{\dagger,\lambda}} \varepsilon + \mathcal{O}(|\varepsilon \log \varepsilon|^2),\qquad |\ell|\leq L_M,
\end{equation}
where $M^\mathrm{d}_{\dagger,\lambda}$ and $M^\mathrm{d}_{\dagger,\eps}$ are as defined in~\eqref{eq:MdDaggerLambda} and~\eqref{eq:MdDaggerDelta}.
\item The remaining spectrum of $\mathcal{L}(\ell)$ is bounded away from the imaginary axis uniformly in $\eps>0$ sufficiently small and $\ell\in\R$.
\end{enumerate}
\end{Theorem}

\begin{Theorem}[Spectrum of vegetation gap solutions] \label{thm:gapStability} Let $a,b,m,\varepsilon$ as in Theorem~\ref{thm:gapexistence} and let $\phi_\mathrm{v}$ be a travelling pulse `gap' solution as in Theorem~\ref{thm:gapexistence}. Then, the following hold.
\begin{enumerate}[(i)]
\item The spectrum of the operator $\mathcal{L}(0)$ is contained in the set $\{\lambda \in \C:\Re \lambda <0\}\cup \{0\}$, and the spectrum of the operator $\mathcal{L}(\ell),\ell\neq0$ is contained in the set $\{\lambda \in \C:\Re \lambda <0\}$. 

\item The eigenvalue $\lambda_0(0)=0$ of $\mathcal{L}(0)$ is simple and continues to an eigenvalue of $\mathcal{L}(\ell), |\ell|\leq L_M$ for some $L_M\gg1$, satisfying $\lambda_0'(0)=0$ and
\begin{align}
	\lambda_0(\ell) &= - \ell^2 + \mathcal{O}(|\varepsilon \log \varepsilon|^2), \qquad \lambda''_0(\ell) = -2 + \mathcal{O}(|\varepsilon \log \varepsilon|^2), \qquad |\ell|\leq L_M.
\end{align}

\item The operator $\mathcal{L}(\ell),|\ell|\leq L_M$ admits an additional critical eigenvalue 
\begin{equation}
\lambda_c(\ell) = -\ell^2 - \frac{M^\mathrm{v}_{\diamond,\eps}}{M^\mathrm{v}_{\diamond,\lambda}} \varepsilon + \mathcal{O}(|\varepsilon \log \varepsilon|^2),\qquad |\ell|\leq L_M,
\end{equation}
where $M^\mathrm{v}_{\diamond,\lambda}$ and $M^\mathrm{v}_{\diamond,\eps}$ are as defined in~\eqref{eq:MvDiamondLambda} and~\eqref{eq:MvDiamondDelta}.
\item  The remaining spectrum of $\mathcal{L}(\ell)$ is bounded away from the imaginary axis uniformly in $\eps>0$ sufficiently small and $\ell\in\R$.
\end{enumerate}
\end{Theorem}

\begin{figure}
\hspace{.1\textwidth}
\begin{subfigure}{.3 \textwidth}
\centering
\includegraphics[width=1\linewidth]{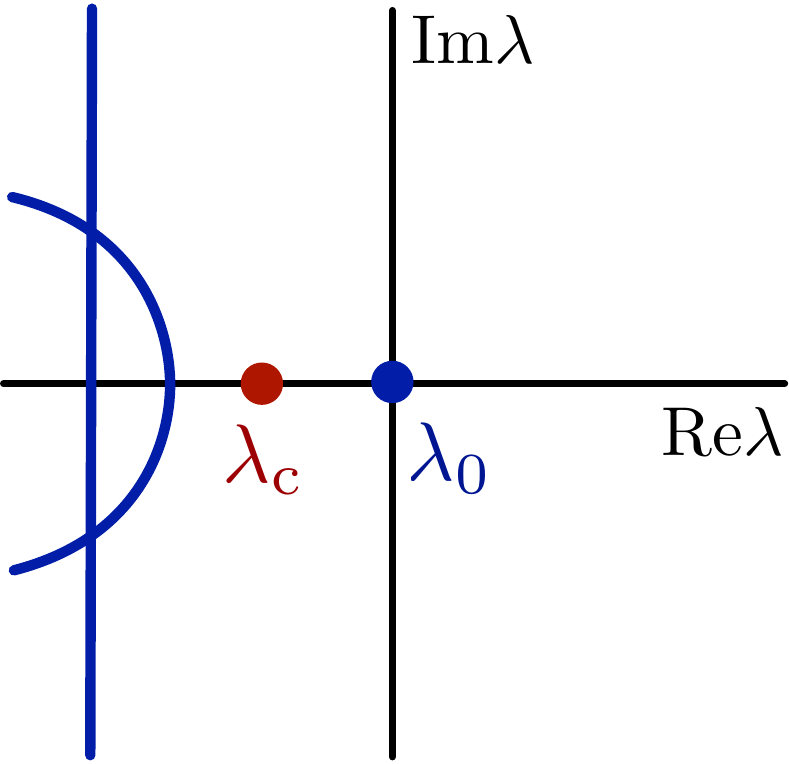}
\end{subfigure}
\hspace{.15\textwidth}
\begin{subfigure}{.3 \textwidth}
\centering
\includegraphics[width=1\linewidth]{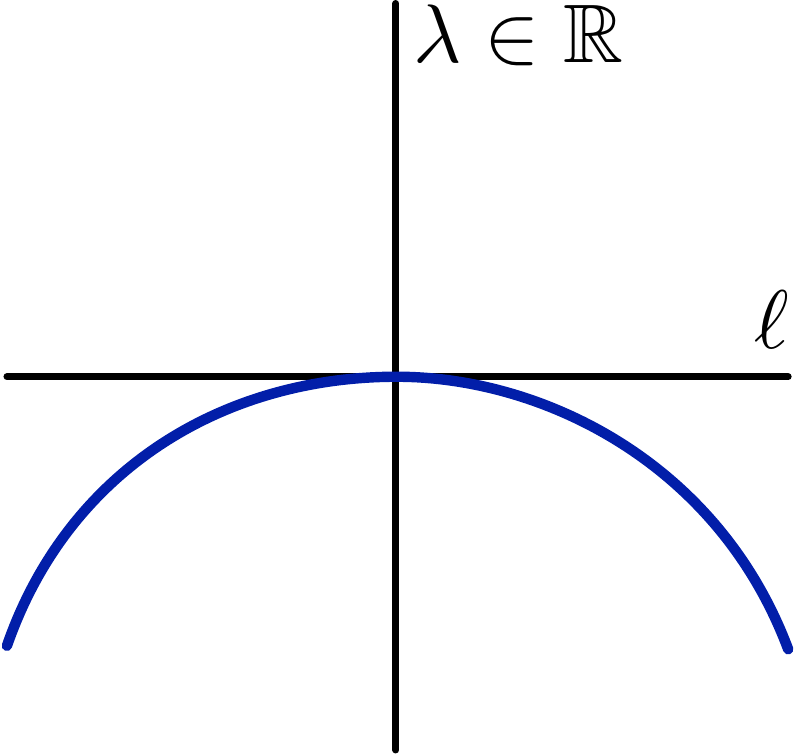}
\end{subfigure}
\hspace{.1\textwidth}
\caption{Shown are the results of Theorem~\ref{thm:stripeStability}. The left panel depicts the spectrum of the $\ell=0$ operator $\mathcal{L}(0)$, corresponding to $1D$ stability. The point spectrum contains two critical eigenvalues $\lambda_0, \lambda_c$ close to the origin, while the remainder of the spectrum is bounded away from the imaginary axis in the left half plane. The right panel depicts a schematic of the continuation of the critical eigenvalue $\lambda_0$ for $|\ell|>0$.  }
\label{fig:specL}
\end{figure}

\begin{figure}
\hspace{.05\textwidth}
\begin{subfigure}{.4 \textwidth}
\centering
\includegraphics[width=1\linewidth]{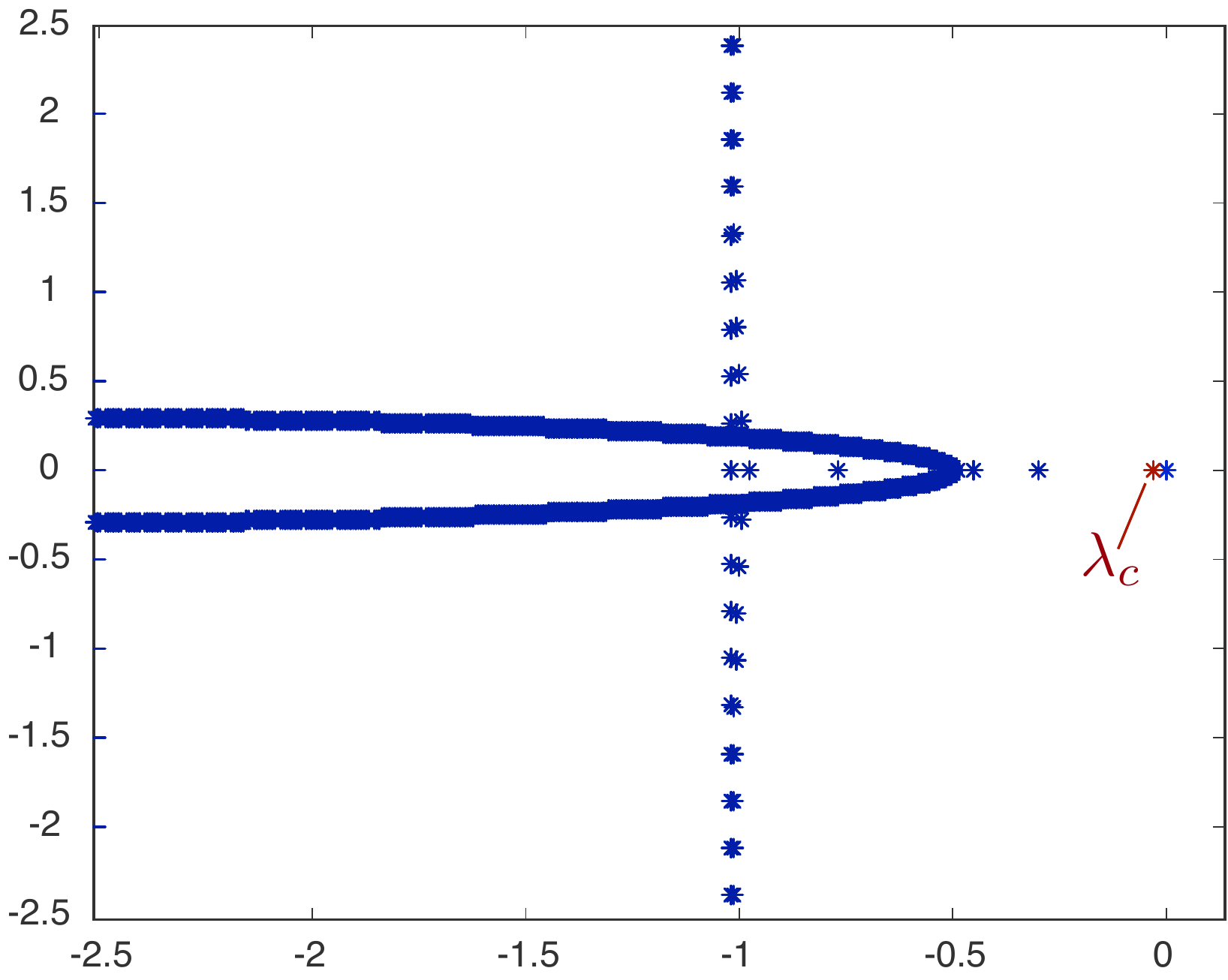}
\caption{}
\end{subfigure}
\hspace{.05\textwidth}
\begin{subfigure}{.4 \textwidth}
\centering
\includegraphics[width=1\linewidth]{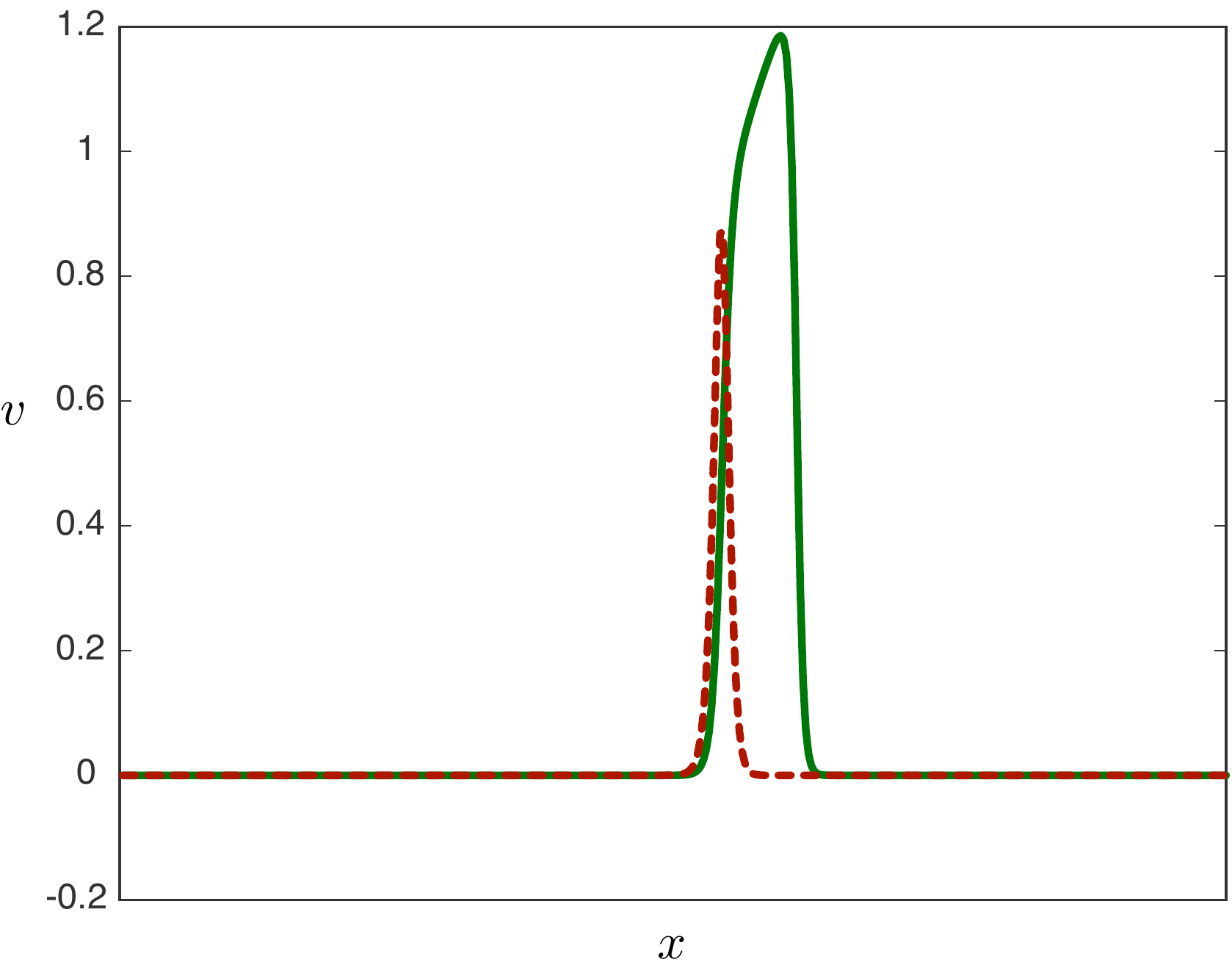}
\caption{}
\label{fig:spec_numerics_ef}
\end{subfigure}
\hspace{.01\textwidth}
\caption{Shown is the numerically computed $1$D spectrum (left panel) associated with a traveling pulse solution of~\eqref{eq:modKlausmeier} found for $a=1.61,b=0.6,m=0.5, \eps=0.003$. The $v$ profile of the solution is shown in the right panel, along with the eigenfunction corresponding to the critical eigenvalue $\lambda_c$.}
\label{fig:spec_numerics}
\end{figure}

\section{Rigorous proof for stability theorems}\label{sec:stabilityRigorous}
The theorems in \S\ref{sec:stabilityTheorems} are based on computations of the essential spectrum in \S\ref{sec:essentialSpectrum} and a formal computation of the point spectrum in \S\ref{sec:pointSpectrumFormal}. The former directly provides proof for the theorem statements concerning the essential spectrum. The latter, however, does not provide a rigorous proof for the theorem statements concerning the point spectrum; to that end, in this section we provide the rigorous justification for the formal point spectrum computations in \S\ref{sec:pointSpectrumFormal}. We restrict ourselves to the study of the traveling pulse `stripe' solution $\phi_\mathrm{d}$ as in Theorem~\ref{thm:stripeexistence} and Theorem~\ref{thm:stripeStability}. The setup and proof for the traveling `gap' solution $\phi_{\mathrm{v}}$ as in Theorem~\ref{thm:gapexistence} and Theorem~\ref{thm:gapStability} is similar; the setup and proofs for the traveling heteroclinic orbits $\phi_{\mathrm{vd}}$ and $\phi_{\mathrm{dv}}$ as in Theorem~\ref{thm:desertFrontExistence}, Theorem~\ref{thm:vegetationFrontExistence} and Theorem~\ref{thm:heteroclinicStability} are also very similar, though less involved. Therefore, the details of these are omitted.

To determine the point spectrum, it is useful to split the complex plane in several regions. For $M \gg 1$ and $\delta\ll1$ fixed independent of $\eps$, we define the following regions (see Figure~\ref{fig:stabilityRegions})
	\begin{align}\begin{split}\label{eq:stabilityregions}
	R_1(\delta)&:=\{\zeta\in \C: |\zeta|\leq \delta\}\\
	R_2(\delta,M)&:=\{\zeta\in \C: \delta<|\zeta|<M, \Re\zeta>-\delta\}\\
	R_3(M)&:= \{\zeta\in \C: |\zeta|>M, |\arg(\zeta)|<2\pi/3\}.
	\end{split}
	\end{align}
	
	In~\S\ref{sec:stabilityLargeL}, we first show that large wavenumbers $\ell$ do not contribute eigenvalues, and hence it suffices to restrict to a region of bounded $\ell$. We then set $\tilde{\lambda}(\ell) := \lambda - \ell^2$ and study the behavior of solutions to\eqref{eq:evalprobleml} for $\tilde{\lambda} $ in the various regions~\eqref{eq:stabilityregions}. The region $R_3$ is considered in~\S\ref{sec:stabilityRegionR3}. In~\S\ref{sec:r1r2setup}, we collect preliminary results in order to set up the analysis for $\tilde{\lambda}(\ell)$ in the regions $R_1$ and $R_2$, which are analyzed in~\S\ref{sec:regionr1} and~\ref{sec:regionr2}, respectively. We briefly conclude the proof of Theorem~\ref{thm:stripeStability} in~\S\ref{sec:finishstabproof}.
	
\begin{figure}
	\centering
	\includegraphics[width = 0.4 \textwidth]{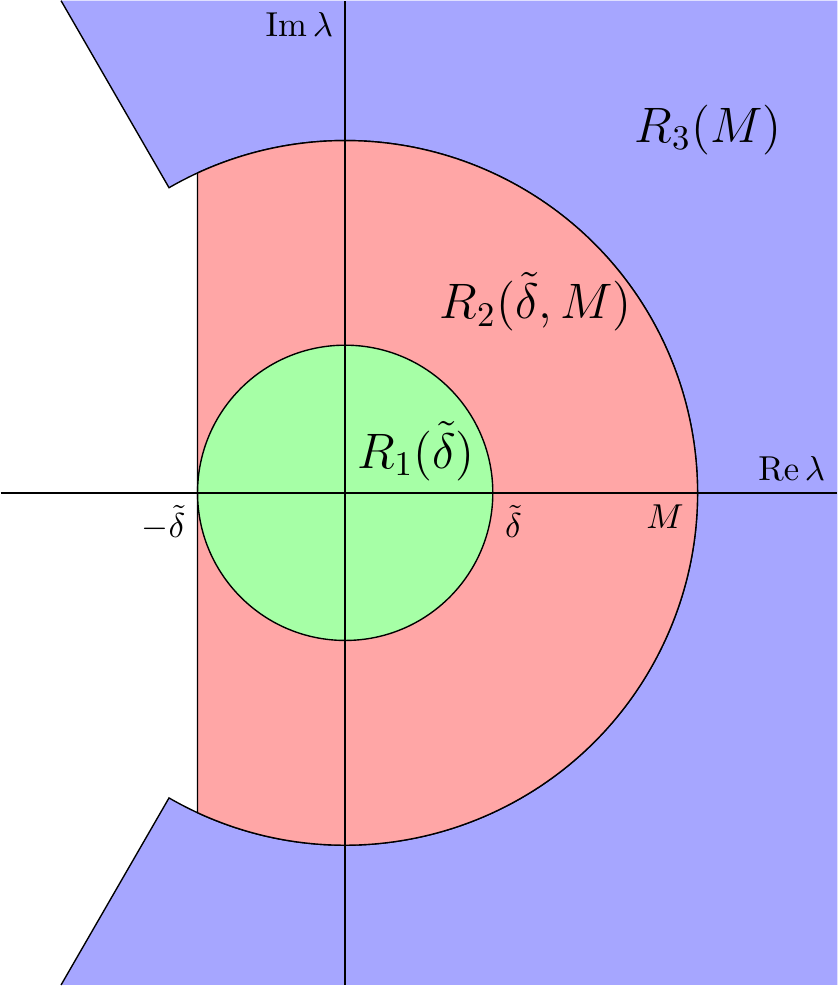}
\caption{Sketch of the regions $R_1(\delta)$, $R_2(\delta,M)$ and $R_3(M)$ as considered in the analysis of the point spectrum.}
\label{fig:stabilityRegions}
\end{figure}
	
	\subsection{Reduction to region of bounded $|\ell|$}\label{sec:stabilityLargeL}
	
	In this section, we show that it suffices to consider bounded wavenumbers $|\ell|\leq L_M$ for some $L_M\gg 1$.
	
	\subsubsection{The region $|\ell| \gg 1$}\label{sec:largel}
	We first consider the region of large transverse wavenumber, that is we consider $(\lambda, \ell)$ such that $\lambda \in R_1(\delta)\cup R_2(\delta,M)\cup R_3(M)$ and $|\ell|\geq L_M$ for a fixed constant $1\ll L_M\ll M$ independent of $\eps$. In this region, we perform a rescaling of the stability problem~\eqref{eq:evalprobleml} and show that the rescaled problem is a small perturbation of a constant coefficient problem which admits exponential di/trichotomies and no exponentially localized solutions.
	
	We rescale $\bar{\xi} = \sqrt{\lambda + \ell^2} \xi, \bar{q} = q/\sqrt{\lambda + \ell^2}$, which results in the system
		\begin{align}
	\frac{d\Psi}{d\bar{\xi}} &=	\bar{A}(\bar{\xi};\lambda, \ell,\eps)\Psi,\qquad \bar{A}(\bar{\xi};\lambda, \ell,\eps)=\bar{A}_1(\lambda, \ell,\eps)+\bar{A}_2(\bar{\xi};\ell,\eps)
\label{eq:evalproblemlres}
	\end{align}
	where $\bar{A}_1(\lambda, \ell,\eps)$ is the constant coefficient matrix
	\begin{align*}
		 \bar{A}_1(\lambda, \ell,\eps)=\begin{pmatrix}
\frac{\varepsilon}{1+\varepsilon c_s} \frac{\lambda}{\sqrt{\lambda+\ell^2}} & 0 & 0 \\
0 & 0 & 1 \\0 &\frac{ \lambda +\ell^2}{| \lambda + \ell^2 |}& 0\end{pmatrix}
	\end{align*}
	and 
		\begin{align*}
		\bar{A}_2(\bar{\xi};\ell,\eps)=\mathcal{O}\left(\frac{1}{\sqrt{\lambda+\ell^2}}\right)
	\end{align*}
	uniformly in $\bar{\xi}, \eps$. We consider $|\ell|\geq L_M$ for some sufficiently large, fixed constant $L_M$. We can compute the eigenvalues of $\bar{A}_1(\lambda, \ell,\eps)$ explicitly as
	\begin{align*}
	\nu_\pm &= \pm \sqrt{ \frac{\lambda + \ell^2}{|\lambda + \ell^2|}}, \qquad \nu_\eps=\frac{\varepsilon}{1+\varepsilon c_\mathrm{d}} \frac{\lambda}{\sqrt{\lambda + \ell^2}}.
	\end{align*}
	For $\lambda \in R_1(\delta)\cup R_2(\delta,M)\cup R_3(M)$ for any $\delta\ll 1$ and $M\gg L_M$, we note that the pair of eigenvalues $\nu_\pm$ have absolute real part greater than $1/2$, because $|\arg \sqrt{ (\lambda + \ell^2) / |\lambda + \ell^2|}|<\pi/3$. One of these eigenvalues has negative real part and the other positive real part.

For the third eigenvalue $\nu_\eps$, there are three cases: $\Re \nu_\eps>1/4, |\Re \nu_\eps|\leq1/4,$ or $\Re \nu_\eps<-1/4$. If $\Re \nu_\eps>1/4$, then, by roughness, \eqref{eq:evalproblemlres} admits exponential dichotomies and hence no exponentially localized solutions. If $|\Re \nu_\eps|\leq1/4$, by roughness~\eqref{eq:evalproblemlres} admits exponential trichotomies with one-dimensional center subspace. Any bounded solution must lie entirely in the center subspace. By continuity, the eigenvalues of the asymptotic matrix $\bar{A}^{\pm\infty}(\lambda, \ell,\eps)=\lim_{\bar{\xi}\to \pm \infty} \bar{A}(\bar{\xi};\lambda, \ell,\eps)$ are separated so that only the eigenvalue $\nu_\eps$ has absolute real part less than $1/4+\kappa$ for some small $\kappa>0$. For $\lambda$ to the right of the essential spectrum, we have that $\Re\nu_\eps>0$. Let $\Psi_\mathrm{c}$ be the corresponding eigenvector. Any solution $\Psi(\xi)$ in the center subspace satisfies $\lim_{\xi \to \pm \infty} \Psi(\xi)e^{-\nu_\eps\xi} = \zeta_\pm \Psi_\mathrm{c}$ for some $\zeta_\pm \in \C\setminus\{0\}$, which contradicts the fact that $\Psi(\xi)$ is bounded. Finally we note that the case $\Re \nu_\eps<-1/4$ cannot occur for $\lambda$ to the right of the essential spectrum since in this region the asymptotic matrix $\bar{A}^{\pm\infty}(\lambda, \ell,\eps)$ has two eigenvalues of positive real part and one of negative real part.
	
	Thus we conclude that for $|\ell|\geq L_M$ and any $\lambda \in R_1(\delta)\cup R_2(\delta,M)\cup R_3(M)$ to the right of the essential spectrum,~\eqref{eq:evalprobleml} admits no exponentially localized solutions.

	\subsubsection{Setup for $|\ell|\leq L_M$}

 In the following sections, we will consider the region where $|\ell|$ is bounded. We begin by setting $\tilde{\lambda}=\tilde{\lambda}(\ell):=\lambda+\ell^2$. Under this transformation,~\eqref{eq:evalprobleml} becomes
		\begin{align}
	\Psi' &=	\tilde{A}(\xi;\tilde{\lambda}, \ell,\eps)\Psi, \label{eq:evalproblemtilde}
	\end{align}
	where
	\begin{align} 
	 \tilde{A}(\xi;\tilde{\lambda}, \ell,\eps):=A(\xi;\tilde{\lambda}-\ell^2, \ell,\eps)=\begin{pmatrix}
\frac{\varepsilon}{1+\varepsilon c_s} \left[1 + \tilde{\lambda}-\ell^2 + v_s^2 \right] & \frac{\varepsilon}{1+\varepsilon c_s} 2 u_s v_s & 0 \\
0 & 0 & 1 \\
- (1 - b v_s) v_s^2 & m + \tilde{\lambda}- (2 - 3 b v_s) u_s v_s & -c_s
\end{pmatrix}.
	\end{align}

In the following we characterize all eigenvalues $\lambda\in \C$ such that
	\begin{align}
	(\tilde{\lambda}, \ell)\in R_1(\delta)\cup R_2(\delta,M)\cup R_3(M)\times [-L_M,L_M].
	\end{align}
This characterizes all eigenvalues $\lambda \in \C$ with $\Re \lambda > - \ell^2 - \delta$ and thus all eigenvalues $\lambda \in \C$ with $\Re \lambda > - \delta$. In particular, all unstable eigenvalues with $\Re \lambda \geq 0$ are found this way.
	
	\subsection{The region $(\tilde{\lambda}(\ell),\ell)\in R_3(M)\times[-L_M,L_M]$}\label{sec:stabilityRegionR3}
	In this region, we follow a similar strategy to that in~\S\ref{sec:largel} and perform the rescaling $\hat{\xi} = \sqrt{|\tilde{\lambda}|} \xi, \hat{q} = q/ \sqrt{|\tilde{\lambda}|}$, which results in the system
		\begin{align}
	\frac{d\Psi}{d\hat{\xi}} &=	\hat{A}(\hat{\xi};\tilde{\lambda}, \ell,\eps)\Psi,\qquad \hat{A}(\hat{\xi};\tilde{\lambda}, \ell,\eps)=\hat{A}_1(\tilde{\lambda}, \ell,\eps)+\hat{A}_2(\hat{\xi};\tilde{\lambda},\ell,\eps)
\label{eq:evalproblemlreslambda}
	\end{align}
	where $\hat{A}_1(\tilde{\lambda}, \ell,\eps)$ is the constant coefficient matrix
	\begin{align*}
		 \hat{A}_1(\tilde{\lambda}, \ell,\eps)=\begin{pmatrix}
\frac{\varepsilon}{1+\varepsilon c_s} \frac{\tilde{\lambda}}{\sqrt{|\tilde{\lambda}|}} & 0 & 0 \\
0 & 0 & 1 \\0 &\frac{ \tilde{\lambda}}{|\tilde{\lambda}|}& 0\end{pmatrix}
	\end{align*}
	and 
		\begin{align*}
		\hat{A}_2(\hat{\xi};\tilde{\lambda},\ell,\eps)=\mathcal{O}\left(\frac{1}{\sqrt{|\tilde{\lambda}|}}\right),
	\end{align*}
	uniformly in $\hat{\xi},\eps,$ and $|\ell|\leq L_M$, where we recall that $1\ll L_M \ll M$. The remainder of the argument follows analogously as in~\S\ref{sec:largel}, and we conclude that for any fixed $L_M$, any sufficiently large $M$ and any $(\tilde{\lambda}(\ell),\ell)\in R_3(M)\times[-L_M,L_M]$ with $\lambda=\tilde{\lambda}-\ell^2$ to the right of the essential spectrum,~\eqref{eq:evalprobleml} admits no exponentially localized solutions.

\subsection{Setup for the region $(\tilde{\lambda}(\ell),\ell)\in R_1(\delta)\cup R_2(\delta,M)\times[-L_M,L_M]$}\label{sec:r1r2setup}

In the previous section we have deduced that all eigenvalues need to be located in the region $(\tilde{\lambda}(\ell),\ell) \in R_1(\delta) \cup R_2(\delta,M) \times [-L_m,L_m]$. The analysis in this region is more involved and we need a specific set-up for this region, the details of which are explained in the next subsections.

\subsubsection{Estimates from the existence analysis}

To study the stability of the traveling pulse solution $\phi_d$, we need to be able to approximate it pointwise by its singular limit, and bound the resulting error terms. The following theorem establishes these estimates.
	
	\begin{Theorem}\label{thm:existencebounds}
For each $\nu>0$ sufficiently large, there exists $\eps_0 > 0$ such that the following holds. Let $\phi_\mathrm{d}(\xi)$ be a traveling-pulse solution as in Theorem~\ref{thm:stripeexistence} for $0<\eps<\eps_0$, and define $L_\eps:=-\nu \log \eps$. There exists $0<Z_\eps =\mathcal{O}(1/\eps)$ such that:
\begin{enumerate}[(i)]
\item \label{thm:existenceboundl} For $\xi \in I_\ell := (-\infty, -L_\eps]$, $\phi_\mathrm{d}(\xi)$ is approximated by the left slow manifold $\mathcal{M}_0^\ell$ with
\begin{align*} d(\phi_\mathrm{d}(\xi),\mathcal{M}_0^\ell) =\mathcal{O}(\eps). \end{align*}
\item \label{thm:existenceboundf} For $\xi \in I_\dagger := [-L_\eps,L_\eps]$, $\phi_\mathrm{d}(\xi)$ is approximated by the front $\phi_{\dagger}$ with
\begin{align*}
\left|\phi_\mathrm{d}(\xi)-\left(\begin{array}{c} u^*(a)\\ \phi_{\dagger}(\xi) \end{array}\right)\right|&=\mathcal{O}( \eps\!\log\eps),\qquad \left|\phi_\mathrm{d}'(\xi)-\left(\begin{array}{c}0\\ \phi_{\dagger}'(\xi)\end{array}\right)\right| =\mathcal{O}( \eps\!\log\eps).
\end{align*}
\item  \label{thm:existenceboundr} For $\xi \in I_r := [L_\eps,Z_{a,\eps} - L_\eps]$, $\phi_\mathrm{d}(\xi)$ is approximated by the right slow manifold $\mathcal{M}_0^r$ with
\begin{align*} d(\phi_\mathrm{d}(\xi),\mathcal{M}_0^r) =\mathcal{O}(\eps).\end{align*}
\item  \label{thm:existenceboundb} For $\xi \in I_\diamond := [Z_{a,\eps} -L_\eps, \infty)$, $\phi_\mathrm{d}(\xi)$ is approximated by the front $\phi_{\diamond}$ with
\begin{align*}\left|\phi_\mathrm{d}(\xi)- \left(\begin{array}{c}a\\ \phi_{\diamond}(\xi - Z_{a,\eps}) \end{array}\right)\right| &= \mathcal{O}( \eps\!\log\eps), \qquad \left|\phi_\mathrm{d}'(\xi)- \left(\begin{array}{c}0\\  \phi_{\diamond}'(\xi - Z_{a,\eps})\end{array}\right)\right| = \mathcal{O}( \eps\!\log\eps)\end{align*}
\end{enumerate}
\end{Theorem}
\begin{proof}
The proof is similar to Theorem 4.3 in~\cite{cdrs}. The estimates are based on the proximity of the solution to the singular limit; along each of the slow manifolds, and along each of the fast jumps outside small neighborhoods of the slow manifolds, these estimates follow directly from the existence analysis, and $\phi_\mathrm{d}(\xi)$ is within $\mathcal{O}(\eps)$ of the corresponding singular profile. The regions in between, i.e. where $\phi_\mathrm{d}(\xi)$ transitions from a fast jump to a slow manifold or vice versa, are more delicate and require corner-type estimates, which result in the $\mathcal{O}(\eps \log \eps)$ errors; see, e.g.~\cite[Theorem 4.5]{cdrs} or~\cite{ESZ,HOLZ}.
\end{proof}
	
	\subsubsection{Weighted eigenvalue problem}
	In this section we introduce a small exponential weight to the stability problem~\eqref{eq:evalproblemtilde}. This weight is introduced to deal with the  inconvenience that arises due to the fact that when $\eps=0$, along the critical manifolds $\mathcal{M}^\ell_0, \mathcal{M}^r_0$ the matrix $A$ admits three spatial eigenvalues: one negative, one positive, and a zero eigenvalue which corresponds to the slow direction. On the other hand, for $\eps>0$ the asymptotic matrix $\mathcal{A}_\mathrm{d}$ is hyperbolic with two positive spatial eigenvalues and one negative eigenvalue. In the following, we will construct exponential dichotomies for~\eqref{eq:evalproblemtilde} along each of the slow manifolds $\mathcal{M}^\ell_\eps, \mathcal{M}^r_\eps$ and each of the fast jumps, and for the following computations it will be convenient to preserve this dichotomy splitting at $\eps=0$ and preserve the exponential decay in forward (resp. backward) time within the corresponding stable (resp. unstable) dichotomy subspaces. To this end, for each $\eta\geq0$ we consider the weighted eigenvalue problem
	\begin{align}\label{eq:Wevalproblem}
	\Psi' &=	A_\eta(\xi;\tilde{\lambda}, \ell,\eps)\Psi,
	\end{align}
	where
	\begin{align}
	 A_\eta(\xi;\tilde{\lambda}, \ell,\eps) &:= \tilde{A}(\xi;\tilde{\lambda}, \ell,\eps)+\eta I=\begin{pmatrix}
\frac{\varepsilon}{1+\varepsilon c} \left[1 + \tilde{\lambda}-\ell^2 + v_\mathrm{d}^2 \right] +\eta& \frac{\varepsilon}{1+\varepsilon c} 2 u_\mathrm{d} v_\mathrm{d} & 0 \\
0 & \eta & 1 \\
- (1 - b v_\mathrm{d}) v_\mathrm{d}^2 & m + \tilde{\lambda} - (2 - 3 b v_\mathrm{d}) u_\mathrm{d} v_\mathrm{d} & \eta-c
\end{pmatrix}.
	\end{align}
	The effect of introducing the weight $\eta$ is to shift the spectrum (i.e. the spatial eigenvalues) of the matrix $\tilde{A}(\xi;\tilde{\lambda},\ell,\eps)$ to the right. For any $\tilde{\lambda}$ chosen so that $\lambda=\tilde{\lambda}-\ell^2$ lies to the right of the essential spectrum of $\mathcal{L}$, the asymptotic matrix $\tilde{A}^{\pm \infty}(\tilde{\lambda},\ell,\eps)=\lim_{\xi \to \infty}\tilde{A}(\xi;\tilde{\lambda},\ell,\eps)$ admits two eigenvalues of positive real part and one of negative real part. Provided $\eta$ is chosen so that $A_\eta^{\pm \infty}(\tilde{\lambda},\ell,\eps)=\lim_{\xi \to \infty}A_\eta(\xi;\tilde{\lambda},\ell,\eps)$ retains this spectral splitting, the original eigenvalue problem~\eqref{eq:evalprobleml} admits a nontrivial exponentially localized solution $\Psi(\xi)$ if and only if the weighted problem~\eqref{eq:Wevalproblem} admits a solution given by $e^{\eta \xi} \Psi(\xi)$.

We proceed by determining $\eta, \nu > 0$ such that the spectrum of the coefficient matrix $A_\eta(\xi;\tilde{\lambda}, \ell,\eps)$ of~\eqref{eq:Wevalproblem} has a consistent splitting into one unstable and two stable eigenvalues for any $\tilde{\lambda}\in R_1(\delta)\cup R_2(\delta,M)$ such that $\lambda=\tilde{\lambda}-\ell^2$ lies to the right of the essential spectrum of  $\mathcal{L}$ and any $\xi\in I_\ell \cup I_r$, where $I_\ell, I_r$ are as in Theorem~\ref{thm:existencebounds}. This consistent splitting will be used to construct exponential dichotomies for~\eqref{eq:Wevalproblem} on the intervals $I_\ell, I_r$. This is the content of the following proposition.

\begin{Proposition} \label{prop:slowexpdich}
There exists $C,\mu,\eta,\eps_0 > 0$ such that for $\eps\in(0, \eps_0)$,~\eqref{eq:Wevalproblem} admits exponential dichotomies on the intervals $I_\ell = (-\infty,-L_\eps]$ and $I_r = [L_\eps,Z_\eps - L_\eps)$ with constants $C, \mu > 0$, and the associated projections $\mathcal{Q}_{\ell, r}^{\mathrm{u},\mathrm{s}}(\xi;\tilde{\lambda},\eps)$ are analytic in $\tilde{\lambda}\in  R_1(\delta)\cup R_2(\delta,M)$ and satisfy
\begin{align*}
\left\|[\mathcal{Q}_{\ell}^\mathrm{s} - \mathcal{P}](-L_\eps;\tilde{\lambda},\eps)\right\|, \left\|[\mathcal{Q}_{r}^\mathrm{s} - \mathcal{P}]( L_\eps;\tilde{\lambda},\eps)\right\|, \left\|[\mathcal{Q}_{r}^\mathrm{s} - \mathcal{P}](Z_\eps- L_\eps;\tilde{\lambda},\eps)\right\| \leq C|\eps\!\log \eps|,
\end{align*}
where $ \mathcal{P}(\xi;\tilde{\lambda},\eps)$ denotes the spectral projection onto the stable eigenspace of the coefficient matrix $A_\eta(\xi;\tilde{\lambda}, \ell,\eps)$ in~\eqref{eq:Wevalproblem}. 
\end{Proposition}
\begin{proof}
By Theorem~\ref{thm:existencebounds}, for $\xi \in I_\ell \cup I_r$, the pulse solution is $\mathcal{O}(\eps)$-close to the slow manifolds $\mathcal{M}^\ell_\eps$ and $\mathcal{M}^r_\eps$, respectively. For $|\ell|\leq L_M$ bounded and any $\tilde{\lambda}\in  R_1(\delta)\cup R_2(\delta,M)$, on $I_\ell$ the matrix $A_\eta(\xi;\tilde{\lambda}, \ell,\eps)$ has slowly varying coefficients and is an $\mathcal{O}(\eps)$ perturbation of the constant-coefficient matrix 
	\begin{align}
A^\ell_\eta(\xi;\tilde{\lambda}, \ell,\eps)=\begin{pmatrix}
\eta& 0 & 0 \\
0 & \eta & 1 \\
0 & m + \tilde{\lambda}&\eta -c^*(a)
\end{pmatrix}.
	\end{align}
	For any sufficiently small $\eta>0$ fixed independently of $\eps$ and $\tilde{\lambda}\in  R_1(\delta)\cup R_2(\delta,M)$, this matrix is hyperbolic with two eigenvalues with positive real part and one with negative real part and a spectral gap with lower bound independent of $\tilde{\lambda}\in  R_1(\delta)\cup R_2(\delta,M)$. By continuity this also holds for $A_\eta(\xi;\tilde{\lambda}, \ell,\eps)$ for $\xi \in I_\ell$, and since $A_\eta(\xi;\tilde{\lambda}, \ell,\eps)$ has slowly varying coefficients on this interval (see~\cite[Proposition 6.1]{COP}), as in the proof of~\cite[Proposition 6.5]{cdrs}, we can construct exponential dichotomies for~\eqref{eq:Wevalproblem} on $I_\ell$ with constants $C,\mu$ independent of $\tilde{\lambda}\in  R_1(\delta)\cup R_2(\delta,M)$ and all sufficiently small $\eps$. 
	
	We proceed similarly along $I_r$, noting that here the matrix $A_\eta(\xi;\tilde{\lambda}, \ell,\eps)$ again has slowly varying coefficients but is now an $\mathcal{O}(\eps)$ perturbation of the matrix 
	\begin{align}
A^r_\eta(\xi;\tilde{\lambda}, \ell,\eps)=\begin{pmatrix}
\eta& 0 & 0 \\
0 & \eta & 1 \\
-(1-bv_\mathrm{d})v_\mathrm{d} & m + \tilde{\lambda} - (2 - 3 b v_\mathrm{d}) u_\mathrm{d} v_\mathrm{d} &\eta -c^*(a)
\end{pmatrix}.
	\end{align}
	where $(u_\mathrm{d},v_\mathrm{d})$ lies within $\mathcal{O}(\eps)$ of the set $\{(u,v)=(u,v_+(u)):u\in[u^*(a),a]\}$ where $v_+$ is as in~\eqref{eq:vplusminus}. On this set, we note that since $m = (1 - b v_+(u)) u v_+(u)$, $u>0$ and $v_+(u) \geq \frac{1}{2b}$, we have that 
	\begin{align}
	m  - (2 - 3 b v_+(u)) u v_+(u) =(-1 + 2 b v_+(u)) u v_+(u) \geq 0.
	\end{align}
	Hence for $\delta>0$ sufficiently small $A^r_\eta(\xi;\tilde{\lambda}, \ell,\eps)$ is hyperbolic with two eigenvalues with positive real part and one with negative real part and a spectral gap with lower bound independent of $\tilde{\lambda}\in  R_1(\delta)\cup R_2(\delta,M)$. The existence of exponential dichotomies for $A_\eta(\xi;\tilde{\lambda}, \ell,\eps)$ on $I_r$ then proceeds similarly to the case of $I_\ell$ above.
\end{proof}

	\subsection{The region $(\tilde{\lambda}(\ell),\ell)\in R_1(\delta)\times[-L_M,L_M]$}\label{sec:regionr1}
	The argument below is based on the analysis in~\cite{cdrs} regarding the stability of traveling pulse solutions in the FitzHugh--Nagumo equation.  The fundamental idea is to construct potential eigenfunctions as solutions to~\eqref{eq:evalprobleml} using Lin's method: the solutions are constructed along three separate intervals which form a partition of the real line and are matched at two locations corresponding to the two fast jumps in the layer problem; see Figure~\ref{fig:eigenfunction_construct}. The resulting matching conditions give bifurcation equations which can be solved using the eigenvalue $\lambda$ as a free parameter, and to leading order these conditions correspond to the Fredholm conditions~\eqref{eq:FredholmConditionDesertHomoclinic1} and~\eqref{eq:FredholmConditionDesertHomoclinic2}.
	
	\begin{figure}
\centering
\includegraphics[width=0.5\textwidth]{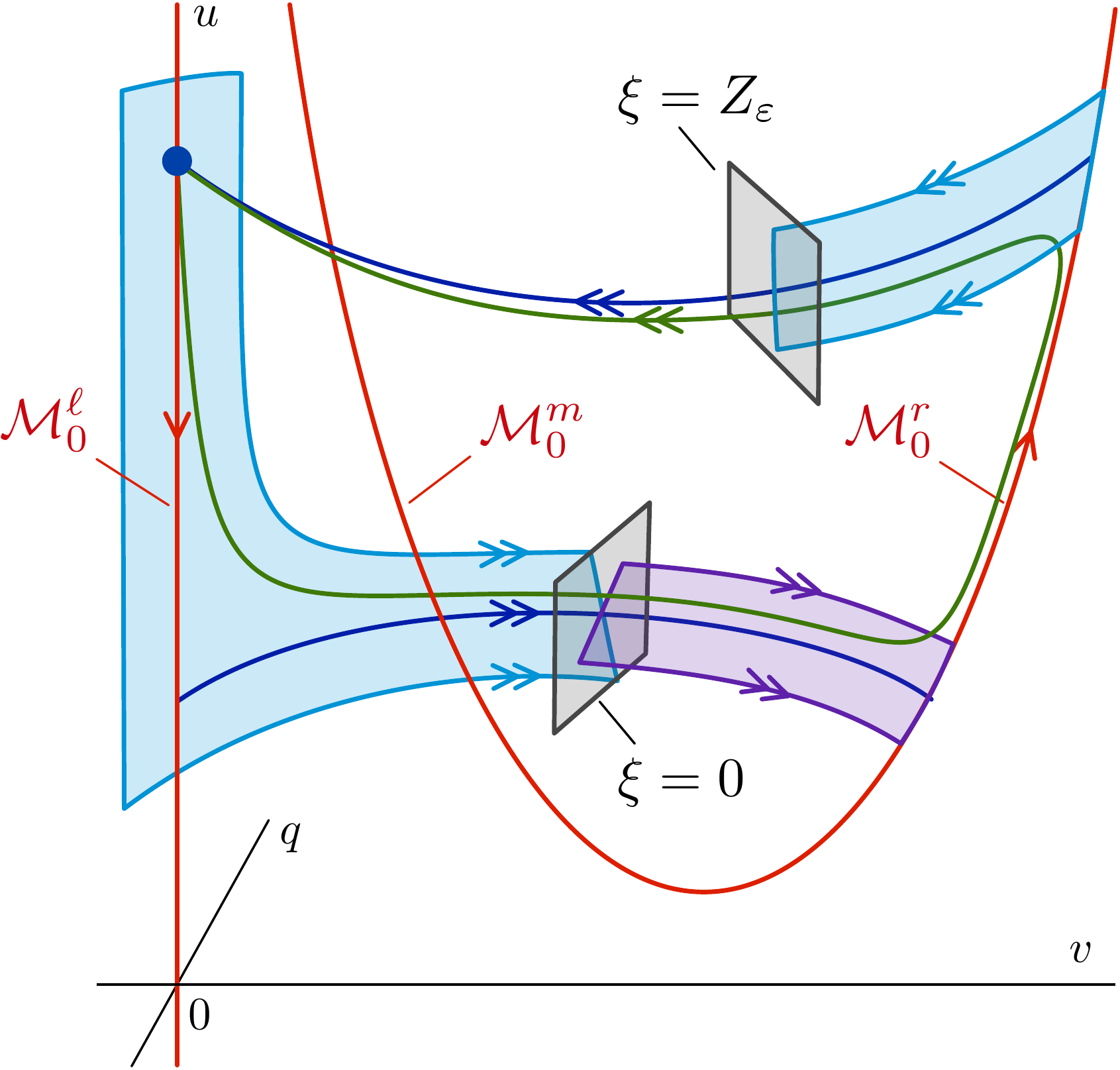}
\caption{Shown is the geometric setup for the construction of potential eigenfunctions using Lin's method. The solutions are constructed along the three intervals $(-\infty,0], [0,Z_\eps], [Z_\eps, \infty)$ and are then matched at $\xi=0$ and $\xi=Z_\eps$ corresponding to the two fast jumps in the layer problem~\eqref{eq:layer}. }
\label{fig:eigenfunction_construct}
\end{figure}

	\subsubsection{Reduced eigenvalue problems along fast jumps} 

	We consider the reduced eigenvalue problems
\begin{align}
\Psi' = A_{j,\eta}(\xi)\Psi, \quad A_{j,\eta}(\xi) := \left(\begin{array}{ccc} \eta & 0 & 0 \\ 0 & \eta & 1 \\ -(1 - b v_j(\xi)) v_j(\xi)^2 & m-(2 - 3 b v_j(\xi)) u_j v_j(\xi) & \eta-c^*(a) \end{array}\right), \quad j = \dagger,\diamond, \label{eq:Wevalproblemreduced}
\end{align}
obtained by considering~\eqref{eq:Wevalproblem} with $\eps=\tilde{\lambda}=0$ and approximating $\phi_\mathrm{d}$ by the fast front solutions $\phi_j, j = \dagger,\diamond$. In~\eqref{eq:Wevalproblemreduced}, $v_j(\xi)$ denotes the $v$-component of $\phi_j(\xi)$, and $u_\dagger=u^*(a), u_\diamond=a$. Hence, for $\xi\in I_\dagger = [-L_\eps,L_\eps]$,~\eqref{eq:Wevalproblem} can be written as the perturbation
\begin{align}\label{eq:Wreducedfront}
 \Psi' = \left(A_{\dagger,\eta}(\xi)+B_\dagger(\xi;\tilde{\lambda}, \ell,\eps) \right)\Psi,\qquad B_\dagger(\xi;\tilde{\lambda}, \ell,\eps) &:=A_\eta(\xi;\tilde{\lambda}, \ell, \eps)-A_{\dagger,\eta}(\xi)
\end{align}
and for $\xi\in [-L_\eps,\infty)$,~\eqref{eq:Wevalproblem} can be written as the perturbation
\begin{align}\label{eq:Wreducedback}
 \Psi' = \left(A_{\diamond,\eta}(\xi) + B_\diamond(\xi;\tilde{\lambda}, \ell,\eps) \right)\Psi,\qquad B_\diamond(\xi;\tilde{\lambda}, \ell,\eps):=A_\eta(\xi+ Z_{\eps};\tilde{\lambda}, \ell, \eps)-A_{\diamond,\eta}(\xi).
\end{align}
We note by Theorem~\ref{thm:existencebounds}~\ref{thm:existenceboundf} and~\ref{thm:existenceboundb} that the perturbation matrices $B_\dagger, B_\diamond$ satisfy
\begin{align}
\begin{split}
\|B_\dagger(\xi;\tilde{\lambda}, \ell,\eps)\| &\leq C(\eps|\!\log\eps| + |\tilde{\lambda}|), \quad \xi \in [-L_\eps,L_\eps],\\
\|B_\diamond(\xi;\tilde{\lambda}, \ell,\eps)\| &\leq C(\eps|\!\log\eps| + |\tilde{\lambda}|), \quad \xi \in [-L_\eps,\infty).
\end{split} \label{Bbound}\end{align}

Next, we note that~\eqref{eq:Wevalproblemreduced} has a lower triangular block structure and leaves the two-dimensional subspace $\{0\}\times \C^2 \subset \C^3$ invariant, the dynamics on which are given by
\begin{align}
\psi' = C_{j,\eta}(\xi)\psi, \quad C_{j,\eta}(\xi) := \left(\begin{array}{cc} \eta & 1 \\  m-(2 - 3 b v_j(\xi)) u_j v_j(\xi) & \eta-c^*(a)  \end{array}\right), \quad j = \dagger,\diamond. \label{eq:Wreducedinvariant}
\end{align}
The space of bounded solutions of~\eqref{eq:Wreducedinvariant} is one-dimensional and spanned by
\begin{align}\label{eq:defpsi}
\psi_j(\xi):=e^{\eta \xi} \phi_j'(\xi), \quad j = \dagger,\diamond.\end{align}
Likewise, the associated adjoint system
\begin{align}
\psi' = -C_{j,\eta}(\xi)^*\psi, \quad j = \dagger,\diamond,  \label{eq:Wreducedinvariantadjoint}
\end{align}
has a one-dimensional space of bounded solutions spanned by
\begin{align} \psi_{j,\ad}(\xi) := \left(\begin{array}{c} q_j'(\xi) \\ -v_j'(\xi)\end{array}\right)e^{(c^*(a)-\eta)\xi}, \quad \ j = \dagger,\diamond. \label{eq:defpsiad}\end{align}
Note the similarities with~\eqref{eq:adjointSolutionLr} in the formal computation. The system~\eqref{eq:Wreducedinvariant} admits exponential dichotomies on both half-lines, which can be extended to the full system~\eqref{eq:Wevalproblemreduced} by exploiting the lower triangular block structure and using variation of constants formulae. This is the content of the following proposition.

\begin{Proposition} \label{prop:expdich} There exist $C, \mu > 0$ such that the following hold.
\begin{enumerate}[(i)]
\item \label{prop:expdichi} The system~\eqref{eq:Wreducedinvariant} admits exponential dichotomies on $\R_{\pm}$ with constants $C, \mu > 0$, projections $\Pi_{j,\pm}^{\mathrm{u},\mathrm{s}}(\xi) = \Pi_{j,\pm}^{\mathrm{u},\mathrm{s}}(\xi;a)$, and corresponding (un)stable evolutions $\Phi_{j,\pm}^{\mathrm{u},\mathrm{s}}(\xi,\hat{\xi}) = \Phi_{j,\pm}^{\mathrm{u},\mathrm{s}}(\xi,\hat{\xi};a)$, $j = \dagger,\diamond$. The projections can be chosen so that
\begin{align} \begin{split}R(\Pi_{j,+}^\mathrm{s}(0)) = \mathrm{Span}(\psi_j(0)) = R(\Pi_{j,-}^\mathrm{u}(0)), \quad R(\Pi_{j,+}^\mathrm{u}(0)) = \mathrm{Span}(\psi_{j,\ad}(0)) = R(\Pi_{j,-}^\mathrm{s}(0)), \quad j = \dagger,\diamond.\end{split} \label{projdich}\end{align}
\item  \label{prop:expdichii} The system~\eqref{eq:Wevalproblemreduced} admits exponential dichotomies on $\R_\pm$ with constants $C, \mu > 0$, projections $Q_{j,\pm}^{\mathrm{u},\mathrm{s}}(\xi) = Q_{j,\pm}^{\mathrm{u},\mathrm{s}}(\xi;a), j = \dagger,\diamond,$ and (un)stable evolutions $T_{j,\pm}^{\mathrm{u},\mathrm{s}}(\xi,\hat{\xi}) = T_{j,\pm}^{\mathrm{u},\mathrm{s}}(\xi,\hat{\xi};a)$. We have that
\begin{align}
\begin{split}
Q_{j,+}^\mathrm{s}(\xi) &= \left(\begin{array}{cc}0 & 0\\ -\int_0^\xi e^{-\eta(\xi - \hat{\xi})}\Phi_{j,+}^\mathrm{s}(\xi,\hat{\xi})F_j(\xi) d\hat{\xi}& \Pi_{j,+}^\mathrm{s}(\xi)  \end{array}\right) = 1 - Q_{j,+}^\mathrm{u}(\xi),
\quad \xi \geq 0,\\
Q_{j,-}^\mathrm{s}(\xi) &= \left(\begin{array}{cc} 0 & 0\\ -\int_{-\infty}^\xi e^{-\eta(\xi - \hat{\xi})}\Phi_{j,-}^\mathrm{s}(\xi,\hat{\xi})F_j(\xi) d\hat{\xi}&\Pi_{j,-}^\mathrm{s}(\xi)   \end{array}\right) = 1 - Q_{j,-}^\mathrm{u}(\xi), \quad \xi \leq 0,
\end{split} \label{eq:projdichfull}
\end{align}
where $F_j(\xi): = \left( 0,-(1 - b v_j(\xi)) v_j(\xi)^2\right)^t$. Furthermore, the projections satisfy
\begin{align} \begin{split}R(Q_{j,+}^\mathrm{u}(0)) &= \mathrm{Span}(\omega_{j,\ad}(0),\Psi_0),\quad R(Q_{j,+}^\mathrm{s}(0)) = \mathrm{Span}(\omega_j(0)),
\\ R(Q_{j,-}^\mathrm{u}(0)) &=  \mathrm{Span}(\omega_j(0),\Psi_0), \quad R(Q_{j,-}^{\mathrm{s}}(0)) = \mathrm{Span}(\omega_{j,\ad}(0)), \end{split}  \label{eq:projdichfull0}
\end{align}
where 
\begin{align}
\omega_j(\xi) :=  \left(\begin{array}{c}0\\  \psi_j(\xi) \end{array}\right), \quad \omega_{j,\ad}(\xi) := \left(\begin{array}{c}0\\  \psi_{j,\ad}(\xi) \end{array}\right), \quad  \Psi_0 := \left(\begin{array}{c}1\\  0 \\ 0\end{array}\right), \quad j = \dagger,\diamond, \label{eq:omegadef}\end{align}
with $\psi_j(\xi)$ and $\psi_{j,\ad}(\xi)$ defined in~\eqref{eq:defpsi} and~\eqref{eq:defpsiad}, respectively.
\end{enumerate}
\end{Proposition}
\begin{proof}
For~\ref{prop:expdichi}, we refer to~\cite[Proposition 6.6]{cdrs}. The exponential dichotomies in~\ref{prop:expdichii} can be constructed from those in~\ref{prop:expdichi} using variation of constants formulae, by exploiting the block triangular structure in~\eqref{eq:Wevalproblemreduced}; see~\cite[Corollary 6.7]{cdrs}.
\end{proof}

	\subsubsection{Construction of eigenfunctions}
	In this section, we use the exponential dichotomies from Proposition~\ref{prop:expdich}, variation of constants formulae, and the estimates from Theorem~\ref{thm:existencebounds} to construct potential eigenfunctions. These eigenfunctions are constructed in three pieces along the intervals $(-\infty,0],[0,Z_\eps],[Z_\eps,\infty)$ (see Figure~\ref{fig:eigenfunction_construct}), and then matched together at $\xi=0,Z_\eps$; the associated matching conditions can then be solved to find eigenvalues $\tilde{\lambda}$. We begin with the following proposition, which describes potential eigenfunctions along each of the three intervals.
	\begin{Proposition}
\label{prop:eigenconstruct}
Let $B_j$ be as in~\eqref{eq:Wreducedfront} and~\eqref{eq:Wreducedback}, and $\omega_j,\Psi_0$ as in~\eqref{eq:omegadef} for $j = \dagger,\diamond$. There exists $\delta,\eps_0, C, q > 0$ such that for $\tilde{\lambda} \in R_1(\delta)$ and all sufficiently small $\eps>0$, the following hold.
\begin{enumerate}[(i)]
\item \label{prop:eigenconstructi}
Any solution $\Psi_{\dagger,-}(\xi,\tilde{\lambda})$ to~\eqref{eq:Wevalproblem}, which decays exponentially in backward time, satisfies
\begin{align}
\Psi_{\dagger,-}(0,\tilde{\lambda}) &= \beta_{\dagger,-}\omega_{\dagger}(0) + \zeta_{\dagger,-} Q_{\dagger,-}^\mathrm{u}(0) \Psi_0 + \beta_{\dagger,-}\int_{-L_\eps}^0 T_{\dagger,-}^{\mathrm{s}}(0,\hat{\xi})B_{\dagger}(\hat{\xi};\tilde{\lambda}, \ell,\eps)\omega_{\dagger}(\hat{\xi})d\hat{\xi} + \h_{\dagger,-}(\beta_{\dagger,-},\zeta_{\dagger,-}), \label{eq:decay-}\end{align}
for some $\beta_{\dagger,-}, \zeta_{\dagger,-} \in \C$, where $\h_{\dagger,-}$ is a linear map satisfying
\begin{align*} \|\h_{\dagger,-}(\beta_{\dagger,-},\zeta_{\dagger,-})\| \leq  C\left((\eps |\!\log \eps | + |\tilde{\lambda}|)|\zeta_{\dagger,-}|+(\eps |\!\log \eps | + |\tilde{\lambda}|)^2|\beta_{\dagger,-}|\right).\end{align*}

\item \label{prop:eigenconstructii} Any solution $\Psi_{\mathrm{sl}}(\xi,\tilde{\lambda})$ to~\eqref{eq:Wevalproblem} which is bounded along the slow manifold $\mathcal{M}^r_\eps$ satisfies
\begin{align}
  \label{eq:slowentry} \Psi_{\mathrm{sl}}(0,\tilde{\lambda})&= \beta_{\dagger}\omega_{\dagger}(0)+ \beta_{\dagger}\int_{L_\eps}^0 T_{\dagger,+}^{\mathrm{u}}(0,\hat{\xi})B_{\dagger}(\hat{\xi};\tilde{\lambda}, \ell,\eps)\omega_{\dagger}(\hat{\xi})d\hat{\xi} + \h_\dagger(\beta_\dagger,\beta_\diamond, \zeta_\diamond),\\
\label{eq:slowexit}
\Psi_{\mathrm{sl}}(Z_{\eps},\tilde{\lambda}) &= \beta_{\diamond}\omega_{\diamond}(0)+\zeta_{\diamond}Q_{\diamond,-}^\mathrm{u}(0)\Psi_0+ \beta_{\diamond}\int_{-L_\eps}^0 T_{\diamond,-}^{\mathrm{s}}(0,\hat{\xi})B_{\diamond}(\hat{\xi};\tilde{\lambda}, \ell,\eps)\omega_{\diamond}(\hat{\xi})d\hat{\xi} + \h_\diamond(\beta_\dagger,\beta_\diamond,\zeta_\diamond),\end{align}
for some $\beta_{\dagger},\beta_{\diamond},\zeta_{\diamond} \in \C$, where $\h_\dagger$ and $\h_\diamond$ are linear maps satisfying
\begin{align*}\|\h_\dagger(\beta_\dagger,\beta_\diamond,\zeta_\diamond)\| & \leq C\left((\eps |\!\log \eps | + |\tilde{\lambda}|)^2|\beta_{\dagger}|+e^{-q/\eps}(|\beta_{\diamond}|+|\zeta_{\diamond}|)\right),\\
\|\h_\diamond(\beta_\dagger,\beta_\diamond,\zeta_\diamond)\|& \leq C\left((\eps |\!\log \eps | + |\tilde{\lambda}|)|\zeta_{\diamond}|+(\eps |\!\log \eps | + |\tilde{\lambda}|)^2|\beta_{\diamond}| + e^{-q/\eps}|\beta_{\dagger}|\right).\end{align*}
\item \label{prop:eigenconstructiii}
Any solution $\Psi_{\diamond,+}(\xi,\tilde{\lambda})$ to~\eqref{eq:Wevalproblem} which decays exponentially in forward time satisfies
\begin{align}
\Psi_{\diamond,+}(Z_\eps,\tilde{\lambda}) &= \beta_{\diamond,+} \omega_{\diamond}(0) + \beta_{\diamond,+}\int_{\infty}^0 T_{\diamond,+}^{\mathrm{u}}(0,\hat{\xi})B_{\diamond}(\hat{\xi};\tilde{\lambda}, \ell,\eps)\omega_{\diamond}(\hat{\xi})d\hat{\xi} + \h_{\diamond,+}(\beta_{\diamond,+}),\label{eq:decay+}
\end{align}
for some $\beta_{\diamond,+} \in \C$, where $\h_{\diamond,+}$ is a linear map satisfying
\begin{align*} \|\h_{\diamond,+}(\beta_{\diamond,+})\| \leq C(\eps |\!\log\eps| + |\tilde{\lambda}|)^2|\beta_{\diamond,+}|, \end{align*}
\end{enumerate}
Moreover, the functions $\psi_{\dagger,-}(\xi,\tilde{\lambda})$, $\psi^{\mathrm{sl}}(\xi,\tilde{\lambda})$, and $\psi_{\diamond,+}(\xi,\tilde{\lambda})$ are analytic in $\tilde{\lambda}$.
\end{Proposition}
\begin{proof}
Using the exponential dichotomies from Propositions~\ref{prop:slowexpdich} and~\ref{prop:expdich}\ref{prop:expdichii}, the proof is nearly identical to the proofs of Propositions~6.8--6.10 in~\cite{cdrs}.
\end{proof}

It remains to solve the matching conditions which arise when attempting to glue together the three solutions from Proposition~\ref{prop:eigenconstruct}~\ref{prop:eigenconstructi}--\ref{prop:eigenconstructiii} at $\xi=0$ and $\xi=Z_\eps$, in order to construct an exponentially localized eigenfunction.
\begin{Theorem} \label{thm:matching}
There exists $\delta, \eps_0 > 0$ such that for $\eps \in (0,\eps_0)$ and $|\ell| \leq L_M$, the eigenvalue problem~\eqref{eq:Wevalproblem} has precisely two  eigenvalues $\tilde{\lambda}_0(\ell),\tilde{\lambda}_c(\ell) \in R_1(\delta)$ given by
\begin{align*} \tilde{\lambda}_0(\ell)=\mathcal{O}(|\eps\!\log \eps |^2), \qquad \tilde{\lambda}_c(\ell) &= -\frac{M_{\dagger,\eps}^\mathrm{d}}{M_{\dagger,\tilde{\lambda}}^\mathrm{d}}\eps + \mathcal{O}\left(\left|\eps \! \log\eps\right|^2\right), \end{align*}
where
\begin{align}
	M_{\dagger,\tilde{\lambda}}^\mathrm{d} &:= \int_{-\infty}^\infty v_\dagger'(\xi;a)^2 e^{c^*(a)\xi}\ d\xi > 0,\\
	  M_{\dagger,\eps}^\mathrm{d} & := \left[ u^*(a) - a + u^*(a) v_+(u^*(a))^2 \right] \int_{\infty}^\infty (1-bv_\dagger(\xi)) v_\dagger(\xi)^2 e^{c^*(a) \xi} v_\dagger'(\xi)\ d\xi > 0.
\end{align} 
The derivatives of $\tilde{\lambda}_0(\ell)$ with respect to $\ell$ satisfy the same estimates, and $\tilde{\lambda}_0(0)=\tilde{\lambda}_0'(0)=0$.
\end{Theorem}

\begin{proof}
We recall from Proposition~\ref{prop:eigenconstruct} that any exponentially localized solution must satisfy the conditions~\eqref{eq:decay-}--\eqref{eq:decay+} at $\xi=0,Z_\eps$ for some $\beta_{\dagger,-},\zeta_{\dagger,-},\beta_\dagger,\beta_\diamond,\zeta_\diamond, \beta_{\diamond,+} \in\C$. Therefore, to obtain an exponentially localized solution to~\eqref{eq:Wevalproblem} we match the solutions $\Psi_{\dagger,-},\Psi_{\mathrm{sl}}$ at $\xi = 0$ and the solutions $\Psi_{\mathrm{sl}},\Psi_{\diamond,+}$ at $\xi = Z_{a,\eps}$, which results in matching conditions which must be satisfied by $\tilde{\lambda}$ and $\eps$ which can be solved to find eigenfunctions. Since the projections $Q_{j,+}^{\mathrm{u},\mathrm{s}}(0)$ associated with the exponential dichotomy of~\eqref{eq:Wevalproblemreduced} established in Proposition~\ref{prop:expdich}\ref{prop:expdichii} satisfy
\begin{align*}
Q_{j,+}^\mathrm{u}(0)+Q_{j,+}^\mathrm{s}(0)=I, \quad j=\dagger,\diamond,
\end{align*}
this is equivalent to ensuring that the differences $\Psi_{\dagger,-}(0,\tilde{\lambda})-\Psi_{\mathrm{sl}}(0,\tilde{\lambda})$ and $\Psi_{\mathrm{sl}}(Z_\eps,\tilde{\lambda})-\Psi_{\diamond,+}(Z_\eps,\tilde{\lambda})$ vanish under the projections $Q_{\dagger,+}^{\mathrm{u},\mathrm{s}}(0)$ and $Q_{\diamond,+}^{\mathrm{u},\mathrm{s}}(0)$, respectively.

We first note that we must have $\beta_\dagger = \beta_{\dagger,-}$ and $\beta_\diamond = \beta_{\diamond,+}$. This can be seen by applying $Q_{j,+}^\mathrm{s}(0)$, $j = \dagger,\diamond,$ to the differences $\Psi_{\dagger,-}(0,\tilde{\lambda})-\Psi_{\mathrm{sl}}(0,\tilde{\lambda})$ and $\Psi_{\mathrm{sl}}(Z_\eps,\tilde{\lambda})-\Psi_{\diamond,+}(Z_\eps,\tilde{\lambda})$, respectively, using the expressions~\eqref{eq:decay-}--\eqref{eq:decay+}. 

We next recall the vectors $\omega_{j,\mathrm{ad}}(0)$ and $\Psi_0$ defined in~\eqref{eq:omegadef}. By~\eqref{eq:projdichfull0} the vectors $\Psi_0$ and
\begin{align*} \Psi_{j,\perp} := \omega_{j,\mathrm{ad}}(0) - \left(\int_{-\infty}^0  e^{\eta \xi}\left\langle \psi_{j,\ad}(\xi),F_j(\xi)\right\rangle d\xi \right) \Psi_0, \quad F_j(\xi) = \left(\begin{array}{c} 0 \\ -(1 - b v_j(\xi)) v_j(\xi)^2 \end{array}\right), \quad j = \dagger,\diamond, \end{align*}
span $R(Q_{j,+}^\mathrm{u}(0))$. Hence we aim to show that the inner products of the differences $\Psi_{\dagger,-}(0,\tilde{\lambda})-\Psi_{\mathrm{sl}}(0,\tilde{\lambda})$ and $\Psi_{\mathrm{sl}}(Z_\eps,\tilde{\lambda})-\Psi_{\diamond,+}(Z_\eps,\tilde{\lambda})$ with $\Psi_0$ and $\Psi_{j,\perp}$ vanish for $j = \dagger,\diamond$, respectively. Using~\eqref{eq:decay-}--\eqref{eq:decay+} we first project along along $\Psi_0$, whereby
\begin{align}
\begin{split}
0 &= \left\langle \Psi_0, \Psi_{\dagger,-}(0,\tilde{\lambda})-\Psi_{\mathrm{sl}}(0,\tilde{\lambda}) \right\rangle = \zeta_{\dagger,-} + \mathcal{O}\left(\left(\eps |\!\log \eps | + |\tilde{\lambda}|\right) \left(|\beta_\dagger|+ |\zeta_{\dagger,-}|\right)+e^{-q/\eps}(|\beta_\diamond| + |\zeta_\diamond|)\right),\\
0 &= \left\langle \Psi_0, \Psi_{\mathrm{sl}}(Z_\eps,\tilde{\lambda})-\Psi_{\diamond,+}(Z_\eps,\tilde{\lambda}) \right\rangle = \zeta_\diamond+ \mathcal{O}\left(\left(\eps |\!\log \eps |+|\tilde{\lambda}|\right)\left(|\beta_\diamond| +|\zeta_\diamond|\right)+e^{-q/\eps}|\beta_\dagger|\right),
\end{split} \label{match1}
\end{align}
where we used Theorem~\ref{thm:existencebounds}~\ref{thm:existenceboundf} and~\ref{thm:existenceboundb}, and~\eqref{Bbound}. Provided $|\tilde{\lambda}|,\eps > 0$ are sufficiently small, we can solve~\eqref{match1} for $\zeta_{\dagger,-}$ and $\zeta_\diamond $ to obtain
\begin{align}
\begin{split}\label{eq:matchzeta}
\zeta_{\dagger,-} &= \mathcal{O}\left((\eps |\!\log \eps |+|\tilde{\lambda}|)|\beta_\dagger| + e^{-q/\eps}|\beta_\diamond|\right)\\
\zeta_\diamond &= \mathcal{O}\left((\eps |\!\log \eps |+|\tilde{\lambda}|)|\beta_\diamond| + e^{-q/\eps}|\beta_\dagger|\right). 
\end{split}
\end{align}
We substitute~\eqref{eq:matchzeta} into~\eqref{eq:decay-}--\eqref{eq:decay+} and noting $\Psi_{j,\perp}\in\ker(Q_{j,-}^\mathrm{u}(0)^*) = R(Q_{j,-}^\mathrm{s}(0)^*) \subset R(Q_{j,+}^\mathrm{u}(0)^*)$ for $j = \dagger,\diamond$, we obtain the final conditions by projecting with $\Psi_{j,\perp}, j=\dagger, \diamond,$ whereby
\begin{align}
\begin{split}0 &= \left\langle \Psi_{\dagger, \perp}, \Psi_{\dagger,-}(0,\tilde{\lambda})-\Psi_{\mathrm{sl}}(0,\tilde{\lambda}) \right\rangle\\ 
&= \beta_\dagger \underbrace{\int_{-L_\eps}^{L_\eps} \left\langle T_\dagger(0,\xi)^* \Psi_{\dagger,\perp}, B_\dagger(\xi;\tilde{\lambda}, \ell,\eps) \omega_\dagger(\xi)\right\rangle d\xi }_{=:\mathcal{I}_\dagger}+ \mathcal{O}\left(\left(\eps|\!\log \eps | + |\tilde{\lambda}|\right)^2 |\beta_\dagger| +e^{-q/\eps}|\beta_\diamond|\right),\end{split} \label{eq:matchperp1}\\
\begin{split}0 &= \left\langle \Psi_{\diamond, \perp}, \Psi_{\mathrm{sl}}(Z_\eps,\tilde{\lambda})-\Psi_{\diamond,+}(Z_\eps,\tilde{\lambda}) \right\rangle\\
 &= \beta_\diamond \underbrace{\int_{-L_\eps}^{\infty} \left\langle T_\diamond(0,\xi)^* \Psi_{\diamond,\perp}, B_\diamond(\xi;\tilde{\lambda}, \ell,\eps) \omega_\diamond(\xi)\right\rangle d\xi}_{=:\mathcal{I}_\diamond} + \mathcal{O}\left(\left(\eps|\!\log \eps | +|\tilde{\lambda}|\right)^2 |\beta_\diamond| + e^{-q/\eps}|\beta_\dagger|\right).\end{split} \label{eq:matchperp2}
\end{align}

To estimate the integrals $\mathcal{I}_j$ for $j=\dagger,\diamond$ appearing in~\eqref{eq:matchperp1}--\eqref{eq:matchperp2}, we note that $T_{j}(0,\xi)^* \Psi_{j,\perp} $ is the solution to the adjoint equation
\begin{align}
\Psi'=-A_{j,\eta}^*\Psi
\end{align}
of~\eqref{eq:Wevalproblemreduced} satisfying $\Psi(0)=\Psi_{j,\perp}$; hence we calculate
\begin{align}
T_{j}(0,\xi)^* \Psi_{j,\perp} = \left(\begin{array}{c} -\int_{-\infty}^\xi \left\langle \psi_{j,\ad}(\hat{\xi}), F_j(\hat{\xi}) \right\rangle d\hat{\xi}\\  \psi_{j,\ad}(\xi)\end{array}\right) = \left(\begin{array}{c}  -\int_{-\infty}^\xi e^{(c^*(a)-\eta) \hat{\xi}} (1 - b v_j(\hat{\xi})) v_j(\hat{\xi})^2v_j'(\hat{\xi}) d\hat{\xi}\\ e^{(c^*(a)-\eta) \xi} q_j'(\xi) \\ -e^{(c^*(a)-\eta) \xi} v_j'(\xi)\end{array}\right), \label{adjointintegral}
\end{align}
for $\xi\in\R$ and $j=\dagger,\diamond$. We now approximate $\mathcal{I}_j$ by first extracting the leading order $\tilde{\lambda}$ contribution, whereby we obtain
\begin{align}
\begin{split} \label{eq:matchapproxfront}
\mathcal{I}_\dagger &= \underbrace{\int_{-L_\eps}^{L_\eps} \left\langle e^{\eta \xi } T_\dagger(0,\xi)^* \Psi_{\dagger,\perp}, B_\dagger(\xi;0, \ell,\eps) \phi_\mathrm{d}'(\xi)\right\rangle d\xi }_{=:\mathcal{J}_\dagger} - M^\mathrm{d}_{\dagger,\tilde{\lambda}} \tilde{\lambda} + \mathcal{O}\left(|\eps\!\log\eps|(|\tilde{\lambda}|+|\eps\!\log\eps|)\right)\end{split}\\
 \begin{split}
\mathcal{I}_\diamond
&= \underbrace{\int_{-L_\eps}^{\infty} \left\langle e^{\eta \xi } T_\diamond(0,\xi)^* \Psi_{\diamond,\perp}, B_\diamond(\xi;0, \ell,\eps) \phi_\mathrm{d}'(\xi+Z_\eps)\right\rangle d\xi}_{=:\mathcal{J}_\diamond} - M^\mathrm{d}_{\diamond,\tilde{\lambda}} \tilde{\lambda}  + \mathcal{O}\left(|\eps\!\log\eps|(|\tilde{\lambda}|+|\eps\!\log\eps|)\right),
\end{split}
 \label{eq:matchapproxback}
\end{align}
where 
\begin{align}
M^\mathrm{d}_{\dagger,\tilde{\lambda}}&:=\int_{-\infty}^{\infty} e^{c^*(a) \xi} \left(v_\dagger'(\xi)\right)^2 d\xi = \int_{-L_\eps}^{L_\eps} e^{c^*(a) \xi} \left(v_\dagger'(\xi)\right)^2 d\xi+\mathcal{O}(\eps)\\
M^\mathrm{d}_{\diamond,\tilde{\lambda}}&:=\int_{-\infty}^{\infty} e^{c^*(a) \xi} \left(v_\diamond'(\xi)\right)^2 d\xi = \int_{-L_\eps}^{\infty} e^{c^*(a) \xi} \left(v_\diamond'(\xi)\right)^2 d\xi+\mathcal{O}(\eps),
\end{align}
where we used the fact that the integrands decay exponentially to estimate the tails of the integrals. Finally, in order to obtain the leading order $\eps$ contribution, it remains to estimate the integrals $\mathcal{J}_j$ for $j=\dagger,\diamond$ which appear in the expressions~\eqref{eq:matchapproxfront}--\eqref{eq:matchapproxback}. To do this, we note that the derivative $\Phi_\mathrm{d}'(\xi)$ of the pulse solution solves the linearized equations when $\ell=0$, and therefore satisfies 
\begin{align}\Phi_\mathrm{d}''(\xi) = \left(A_{\dagger,0}(\xi)+B_\dagger(\xi;0,0,\eps) \right)\Phi_\mathrm{d}'(\xi),\qquad \xi\in[-L_\eps, L_\eps]
\end{align}
and
\begin{align} \Phi_\mathrm{d}''(\xi+Z_\eps)= \left(A_{\diamond,0}(\xi) + B_\diamond(\xi;0,0,\eps) \right)\Phi_\mathrm{d}'(\xi+Z_\eps), \qquad \xi\in[-L_\eps, \infty).
\end{align}
In particular, for $\xi\in[-L_\eps, L_\eps]$, we obtain
\begin{align*}B_\dagger(\xi;0,\ell,\eps)\Phi_\mathrm{d}'(\xi)&= \left[ \partial_\xi-A_{\dagger,0}(\xi)+B_\dagger(\xi;0,\ell,\eps)-B_\dagger(\xi;0,0,\eps)\right] \Phi_\mathrm{d}'(\xi)\\
&=\begin{pmatrix} 0\\ \left[ \partial_\xi-C_{\dagger,0}(\xi)\right]\begin{pmatrix} v_\mathrm{d}'(\xi)\\ q_\mathrm{d}'(\xi)\end{pmatrix} \end{pmatrix}+\begin{pmatrix}u_\mathrm{d}''(\xi)-\frac{\eps\ell^2}{1+\eps c_\mathrm{d}}u_\mathrm{d}'(\xi)\\ 0\\ (1-bv_\dagger(\xi))v_\dagger(\xi)^2u_\mathrm{d}'(\xi)\end{pmatrix}
\end{align*}
and similarly 
\begin{align*}B_\diamond(\xi;0,\ell,\eps)\Phi_\mathrm{d}'(\xi+Z_\eps)&= \left[ \partial_\xi-A_{\diamond,0}(\xi)+B_\diamond(\xi;0,\ell,\eps)-B_\diamond(\xi;0,0,\eps)\right] \Phi_\mathrm{d}'(\xi+Z_\eps)\\
&=\begin{pmatrix} 0\\ \left[ \partial_\xi-C_{\diamond,0}(\xi)\right]\begin{pmatrix} v_\mathrm{d}'(\xi+Z_\eps)\\ q_\mathrm{d}'(\xi+Z_\eps)\end{pmatrix} \end{pmatrix}+\begin{pmatrix} u_\mathrm{d}''(\xi+Z_\eps)-\frac{\eps\ell^2}{1+\eps c_\mathrm{d}}u_\mathrm{d}'(\xi+Z_\eps)\\0\\ (1-bv_\diamond(\xi))v_\diamond(\xi)^2u_\mathrm{d}'(\xi+Z_\eps)\end{pmatrix}
\end{align*}
for $\xi\in[-L_\eps, \infty)$. Using the fact that $\psi_{j,\ad}(\xi)$ solves~\eqref{eq:Wreducedinvariantadjoint}, we have
\begin{align}
\left[ \partial_\xi-C_{j,0}(\xi)\right]^*\left(e^{\eta \xi}\psi_{j,\ad}(\xi)\right)=0, \qquad j=\dagger, \diamond,
\end{align}
and we therefore obtain
\begin{align*}
\mathcal{J}_\dagger &=\int_{-L_\eps}^{L_\eps} \left\langle e^{\eta \xi } T_\dagger(0,\xi)^* \Psi_{\dagger,\perp}, \begin{pmatrix}u_\mathrm{d}''(\xi)-\frac{\eps\ell^2}{1+\eps c_\mathrm{d}}u_\mathrm{d}'(\xi)\\ 0\\ (1-bv_\dagger(\xi))v_\dagger(\xi)^2u_\mathrm{d}'(\xi)\end{pmatrix}\right\rangle d\xi\\
&=-\int_{-L_\eps}^{L_\eps}  \left(e^{c^*(a) \xi} v_\dagger'(\xi)(1-bv_\dagger(\xi))v_\dagger(\xi)^2u_\mathrm{d}'(\xi)+u_\mathrm{d}''(\xi)\int_{-\infty}^\xi e^{c^*(a)\hat{\xi}} (1 - b v_\dagger(\hat{\xi})) v_\dagger(\hat{\xi})^2v_\dagger'(\hat{\xi}) d\hat{\xi} \right) d\xi\\
&\qquad +\frac{\eps\ell^2}{1+\eps c_\mathrm{d}}\int_{-L_\eps}^{L_\eps}\left(u_\mathrm{d}'(\xi) \int_{-\infty}^\xi e^{c^*(a)\hat{\xi}} (1 - b v_\dagger(\hat{\xi})) v_\dagger(\hat{\xi})^2v_\dagger'(\hat{\xi}) d\hat{\xi} \right)d\xi+\mathcal{O}(\eps^2),
\end{align*}
where we used the fact that the integrands decay exponentially. Integrating by parts, we have that
\begin{align*}
\mathcal{J}_\dagger&=-\int_{-L_\eps}^{L_\eps} \frac{d}{d\xi} \left(u_\mathrm{d}'(\xi)\int_{-\infty}^\xi e^{c^*(a)\hat{\xi}} (1 - b v_\dagger(\hat{\xi})) v_\dagger(\hat{\xi})^2v_\dagger'(\hat{\xi}) d\hat{\xi} \right)d\xi\\
&\qquad +\frac{\eps\ell^2}{1+\eps c_\mathrm{d}}\left[u_\mathrm{d}(\xi) \int_{-\infty}^\xi e^{c^*(a)\hat{\xi}} (1 - b v_\dagger(\hat{\xi})) v_\dagger(\hat{\xi})^2v_\dagger'(\hat{\xi}) d\hat{\xi} \right]_{-L_\eps}^{L_\eps}\\
&\qquad  \qquad -\frac{\eps\ell^2}{1+\eps c_\mathrm{d}}\int_{-L_\eps}^{L_\eps}u_\mathrm{d}(\xi)  e^{c^*(a)\xi} (1 - b v_\dagger(\xi)) v_\dagger(\xi)^2v_\dagger'(\xi) d\xi +\mathcal{O}(\eps^2)\\
&=- u_\mathrm{d}'(L_\eps)\int_{-\infty}^{L_\eps} e^{c^*(a)\xi} (1 - b v_\dagger(\xi)) v_\dagger(\xi)^2v_\dagger'(\xi) d\xi +\mathcal{O}(\eps^2|\!\log\eps|)\\
&=-\eps\left[ u^*(a) - a + u^*(a) v_+(u^*(a))^2 \right] \int_{-\infty}^{\infty} e^{c^*(a)\xi} (1 - b v_\dagger(\xi)) v_\dagger(\xi)^2v_\dagger'(\xi) d\xi +\mathcal{O}(\eps^2|\!\log\eps|),
\end{align*}
where we again used the fact that the integrands decay exponentially, and we estimated $u_\mathrm{d}(\xi)=u^*(a)+\mathcal{O}(\eps\log \eps)$ for $ \xi\in[-L_\eps,L_\eps]$ and
\begin{align*}
u_\mathrm{d}'(L_\eps)&=\eps\left[ u_\mathrm{d}(L_\eps) - a + u_\mathrm{d}(L_\eps) v_\mathrm{d}'(L_\eps)^2 \right]\\
&=\eps\left[ u^*(a) - a + u^*(a) v_+(u^*(a))^2 +\mathcal{O}(|\eps\!\log\eps|)\right],
\end{align*}
using Theorem~\ref{thm:existencebounds}. Hence we have that 
\begin{align}
\mathcal{J}_\dagger &= -M^\mathrm{d}_{\dagger,\eps}\eps +\mathcal{O}(\eps^2|\!\log\eps|),
\end{align}
where
\begin{align}
M^\mathrm{d}_{\dagger,\eps}:=\left[ u^*(a) - a + u^*(a) v_+(u^*(a))^2 \right] \int_{-\infty}^{\infty} e^{c^*(a)\xi} (1 - b v_\dagger(\xi)) v_\dagger(\xi)^2v_\dagger'(\xi) d\xi>0.
\end{align}
Performing a similar computation for $\mathcal{J}_\diamond$, we arrive at 
\begin{align}
\mathcal{J}_\diamond &= - \lim_{\xi\to\infty}u_\mathrm{d}'(Z_\eps+\xi)\int_{-\infty}^{\infty} e^{c^*(a)\xi} (1 - b v_\dagger(\xi)) v_\dagger(\xi)^2v_\dagger'(\xi) d\xi +\mathcal{O}(\eps^2)=\mathcal{O}(\eps^2),
\end{align}
due to the fact that $u_\mathrm{d}'(Z_\eps+\xi)\to0$ as $\xi \to\infty$.

Substituting the expressions for $\mathcal{I}_j, \mathcal{J}_j, j=\dagger, \diamond$, into the remaining conditions~\eqref{eq:matchperp1}--\eqref{eq:matchperp2}, we find the following linear system of equations for $(\beta_\dagger,\beta_\diamond)$, solutions of which correspond to eigenfunctions of~\eqref{eq:Wevalproblem}:
\begin{align}
\begin{split}
\mathcal{M}(\tilde{\lambda},\eps)\left(\begin{array}{c} \beta_{\dagger} \\ \beta_{\diamond}\end{array}\right) &= 0,
\end{split}\label{eq:finalmatch}
\end{align}
where
\begin{align}
\begin{split}
\mathcal{M}(\tilde{\lambda},\eps)&:=\left(\begin{array}{cc} -\tilde{\lambda} M^\mathrm{d}_{\dagger,\tilde{\lambda}} -M^\mathrm{d}_{\dagger,\eps}\eps+ \mathcal{O} \left((\eps |\!\log \eps |+|\tilde{\lambda}|)^2\right) & \mathcal{O}(e^{-q/\eps})\\
\mathcal{O}(e^{-q/\eps}) & -\tilde{\lambda} M^\mathrm{d}_{\diamond,\tilde{\lambda}}+ \mathcal{O} \left((\eps |\!\log \eps |+|\tilde{\lambda}|)^2\right)\end{array}\right). \end{split} \label{eq:finalmatchmatrix}
\end{align}
Since the solutions $\Psi_{\dagger,-},\Psi_\mathrm{sl},\Psi_{\diamond,+}$ from Proposition~\ref{prop:eigenconstruct} and the matrices $B_j$ are analytic in $\tilde{\lambda}$, all entries in the matrix $\mathcal{M}(\tilde{\lambda},\eps)$~\eqref{eq:finalmatchmatrix}, and futhermore its determinant $D(\tilde{\lambda},\eps)$, are analytic in $\tilde{\lambda}$. Note that the quantities $M^\mathrm{d}_{\dagger,\eps}$ and $M^\mathrm{d}_{j,\tilde{\lambda}}, j=\dagger,\diamond$ are nonzero and independent of $\tilde{\lambda}, \eps$. Hence, provided $\delta,\eps > 0$ are sufficiently small, we have 
\begin{align*} |D(\tilde{\lambda}, \eps) - \tilde{\lambda} M^\mathrm{d}_{\diamond,\tilde{\lambda}}(\tilde{\lambda} M^\mathrm{d}_{\dagger,\tilde{\lambda}} + \eps M^\mathrm{d}_{\dagger,\eps})| < |\tilde{\lambda} M^\mathrm{d}_{\diamond,\tilde{\lambda}}(\tilde{\lambda} M^\mathrm{d}_{\dagger,\tilde{\lambda}} + \eps M^\mathrm{d}_{\dagger,\eps})|. \end{align*}
for $\tilde{\lambda} \in \partial R_1(\delta) = \{\tilde{\lambda} \in \C : |\tilde{\lambda}| = \delta\}$, and by Rouch\'e's Theorem $D(\tilde{\lambda}, \eps)$ has precisely two roots $\tilde{\lambda}_0,\tilde{\lambda}_1$ in $R_1(\delta)$ which are $\mathcal{O}(|\eps\!\log \eps |^2)$-close to the roots 
\begin{align*}\tilde{\lambda}=0,\qquad \tilde{\lambda}=-\frac{M^\mathrm{d}_{\dagger,\eps}}{M^\mathrm{d}_{\dagger,\tilde{\lambda}}}\eps\end{align*}
 of $\tilde{\lambda} M^\mathrm{d}_{\diamond,\tilde{\lambda}}(\tilde{\lambda} M^\mathrm{d}_{\dagger,\tilde{\lambda}} + \eps M^\mathrm{d}_{\dagger,\eps})$.  We deduce that~\eqref{eq:Wevalproblem} has two real eigenvalues in the region $R_1(\delta)$ given by
\begin{align*}
\tilde{\lambda}_0(\ell) &= \mathcal{O}(|\eps\!\log \eps |^2), \qquad \tilde{\lambda}_c(\ell) = -\frac{M^\mathrm{d}_{\dagger,\eps}}{M^\mathrm{d}_{\dagger,\tilde{\lambda}}}\eps+\mathcal{O}(|\eps\!\log \eps |^2),
\end{align*}
and by implicitly differentiating the characteristic equation of~\eqref{eq:finalmatchmatrix}, we furthermore obtain that the derivatives of $\tilde{\lambda}_0(\ell)$ with respect to $\ell$ satisfy the same estimates. We note that the derivative $\Phi_\mathrm{d}'$ of the pulse solution is an eigenfunction with eigenvalue $0$ when $\ell=0$ due to translation invariance, hence $\lambda_0(0)=0$. Furthermore, since~\eqref{eq:finalmatchmatrix} depends on $\ell$ only via the quantity $\ell^2$, we obtain that $\tilde{\lambda}_0'(0) = 0$.

\end{proof}

\subsection{The region $(\tilde{\lambda}(\ell),\ell)\in R_2(\delta,M)\times [-L_M,L_M]$}\label{sec:regionr2}

We now consider the final remaining region, $\tilde{\lambda}(\ell)\in R_2(\delta,M)$ for $|\ell|$ bounded. The fundamental idea is the same as for the region $R_1(\delta)$; using exponential dichotomies along the fast jumps and the slow manifolds, we attempt to construct potential eigenfunctions. However, in this region it is possible to construct exponential dichotomies along each of the intervals $I_\ell, I_\dagger,I_r,I_\diamond$, and by comparing their projections at the endpoints of these intervals we obtain estimates which preclude the existence of a nontrivial exponentially localized eigenfunction. We note that the exponential dichotomies along $I_r$ and $I_\ell$ are guaranteed by Proposition~\ref{prop:slowexpdich}. The existence of exponential dichotomies along $I_\dagger$ and $I_\diamond$ is due to the fact that the associated reduced problems along each of the fast jumps admit no eigenvalues for $\tilde{\lambda}(\ell)\in R_2(\delta,M)$. 

To see this, proceeding in a similar fashion as in~\S\ref{sec:regionr1}, we consider the following reduced problems along $I_\dagger$ and $I_\diamond$ obtained for $\eps=0$ and $\tilde{\lambda}\in R_2(\delta,M)$.
\begin{align}
\psi_\xi = A_{j,\eta}(\xi;\tilde{\lambda})\psi, \quad A_{j,\eta}(\xi;\tilde{\lambda}):= \left(\begin{array}{ccc} \eta & 0 & 0 \\ 0 & \eta & 1 \\ -(1 - b v_j(\xi)) v_j(\xi)^2 & m+\tilde{\lambda}-(2 - 3 b v_j(\xi)) u_j v_j(\xi) & \eta-c^*(a) \end{array}\right), \label{eq:Wevalproblemreduced2}
\end{align}
Here $j = \dagger,\diamond,$ where again $v_j(\xi)$ denotes the $v$-component of $\phi_j(\xi)$, and $u_\dagger=u^*(a), u_\diamond=a$. As in~\S\ref{sec:regionr1}, the lower triangular structure allows us to restrict to a two-dimensional invariant subspace with dynamics
\begin{align}
\psi' = C_{j,\eta}(\xi;\tilde{\lambda})\psi, \quad C_{j,\eta}(\xi;\tilde{\lambda}) := \left(\begin{array}{cc} \eta & 1 \\  m+\tilde{\lambda}-(2 - 3 b v_j(\xi)) u_j v_j(\xi) & \eta-c^*(a)  \end{array}\right), \quad j = \dagger,\diamond. \label{eq:Wreducedinvariant2}
\end{align}
We note that the front profiles $v_\dagger(\xi)$ and $v_\diamond(\xi)$ are solutions to the scalar equations
\begin{align*}
v_t = v_{\xi\xi}+c^*(a)v_\xi -mv+(1-bv)u_jv^2,\qquad j = \dagger,\diamond,
 \end{align*}
 and critically, the linear system~\eqref{eq:Wreducedinvariant2} is precisely the (weighted) eigenvalue problem one obtains by considering their stability with eigenvalue parameter $\tilde{\lambda}$. Since the derivatives $v_j'(\xi), j = \dagger,\diamond$ define exponentially localized eigenfunctions with no zeros when $\tilde{\lambda}=0$, Sturm-Liouville theory precludes the existence of eigenvalues in $R_2(\delta,M)$, provided $\delta$ is sufficiently small. Thus~\eqref{eq:Wreducedinvariant2} admits exponential dichotomies, which can be extended to the full system~\eqref{eq:Wevalproblemreduced2} using variation of constants formulae. Finally, these exponential dichotomies can be extended to the stability problem~\eqref{eq:Wevalproblem} on the intervals $I_\dagger$ and $I_\diamond$ using roughness results.
 
 Once exponential dichotomies are established along each of the intervals $I_\ell, I_\dagger,I_r,I_\diamond$, it remains to compare their projections at the endpoints of each interval. Using the estimates in Theorem~\ref{thm:existencebounds} combined with repeated use of a technical lemma~\cite[Lemma 6.10]{HOLZ}, it is possible to show that each pair of projections are sufficiently close at each endpoint, and further that any exponentially localized solution to~\eqref{eq:Wevalproblem} must be trivial. This is summarized in the following proposition.
 \begin{Proposition} \label{propR2}
Fix $M > 0$. There exists $\delta> 0$ such that for each sufficiently small $\eps>0$ and each $\tilde{\lambda} \in R_2(\delta,M)$, the eigenvalue problem~\eqref{eq:Wevalproblem} admits no nontrivial exponentially localized solution.
\end{Proposition}
The proof of Proposition~\ref{propR2} follows the argument as outlined above, and is similar to the proof of~\cite[Proposition 6.20]{cdrs}. For completeness, we include this in Appendix~\ref{app:r2proof}.

\subsection{Proof of Theorem~\ref{thm:stripeStability}}\label{sec:finishstabproof}

\begin{proof}[Proof of Theorem~\ref{thm:stripeStability}]
This is a direct consequence of \S\ref{sec:essentialSpectrum}, \S\ref{sec:stabilityLargeL}, \S\ref{sec:stabilityRegionR3}, Theorem~\ref{thm:matching} and Proposition~\ref{propR2}. The fact that the translational eigenvalue $\tilde{\lambda}_0(0)=0$ is simple follows from a similar argument as in~\cite[Proposition 6.14]{cdrs}.
\end{proof}

\section{Defects and curved vegetation pattern solutions}\label{sec:corners}
In this section we consider~\eqref{eq:modKlausmeier} with a small diffusion term added to the water component.
\begin{equation}
	\begin{cases}
		u_t & =D \Delta u+ \frac{1}{\varepsilon} u_x + a - u - G(u,v)v, \\
		v_t & = \Delta v - m v + R(v)G(u,v)v,
	\end{cases}\label{eq:modKlausmeier+D}
\end{equation}
where $D\ll 1$. The reason for this is mainly technical, in order to draw on results concerning planar interface propagation in parabolic equations. However, to accurately describe water movement on flat terrains a diffusion term is necessary~\cite{van2013rise} -- see also the upcoming discussion section,~\S\ref{sec:discussion}.

The results of Theorems~\ref{thm:stripeexistence}--\ref{thm:vegetationFrontExistence} and Theorems~\ref{thm:heteroclinicStability}--\ref{thm:gapStability} concern the existence and stability of straight stripe, gap, and front solutions; that is, the traveling patterns are constant in the direction transverse to the slope and are essentially one-dimensional patterns. We reiterate that these patterns are, however, stable to perturbations in two spatial dimensions.

We now consider the system~\eqref{eq:modKlausmeier+D} for which, by a perturbation argument, the results of Theorems~\ref{thm:stripeexistence}--\ref{thm:vegetationFrontExistence}, and furthermore the results of Theorems~\ref{thm:heteroclinicStability}--\ref{thm:gapStability}, are expected to hold for sufficiently small $D>0$. Within this system, we are able to call on general results on the existence and stability of corner defects in planar wave propagation~\cite{HS1, HS2}. In essence, considering a straight vegetation stripe, gap, or front solution satisfying certain hypotheses (see below), for nearby wave speeds there exist stripe solutions at slightly offset angles. Two oppositely angled such stripes can meet at a corner defect, forming a ``curved" stripe solution, which can be oriented convex upslope (exterior corner) or downslope (interior corner). Further, some of these solutions can be shown to be stable.  In particular, we will argue using the results of~\cite{HS1, HS2} that nearby vegetation stripe, gap, or front solutions of~\eqref{eq:modKlausmeier+D}, there exist stable interior corner defects, and in the case of certain front solutions, there exist stable exterior corner defects.

Consider a traveling wave solution $(u,v)(x,y,t) = (u_s, v_s)(\xi)$ of~\eqref{eq:modKlausmeier+D} with speed $c=c_s$, and $\xi=x-ct$. An almost planar interface $\sigma$-close to $(u_s, v_s)(\xi)$ with speed $c$ is a solution of the form
\begin{align}
(u,v)(x,y,t) = (u_s, v_s)(\xi+h(y))+(\tilde{u},\tilde{v})(\xi,y),
\end{align}
where $h\in C^2(\mathbb{R})$ and
\begin{align}
\sup_{y\in\mathbb{R}}|h'(y)| <\sigma, \quad \sup_{y\in\mathbb{R}}\|(\tilde{u},\tilde{v})(\cdot,y)\|_{H^1(\mathbb{R}, \mathbb{R}^2)}<\sigma, \quad |c-c_s|<\sigma
\end{align}
 This solution is a planar interface if $h''=0$ and a corner defect if $h''\not\equiv0$, and $h'(y)\to \eta_\pm$ as $y\to \infty$. A corner defect can be classified depending on the asymptotic orientations $\eta_\pm$ as an (i) interior corner ($\eta_+<\eta_-$), (ii) exterior corner ($\eta_-<\eta_+$), (iii) step ($\eta_+=\eta_-\neq0$), or (iv) hole ($\eta_+=\eta_-=0$).
 
 Depending on the original traveling wave solution $ (u_s, v_s)(\xi)$, it may be possible to determine which type(s) of defects can arise. As stated above, a corner defect is essentially composed of slightly angled stripe solutions meeting along an interface.  An angled stripe solution can be written as a traveling wave
 \begin{align}
 (u,v)(x,y,t)= (u,v)(\xi),\qquad  \xi=x\cos \varphi +y\sin \varphi - ct
 \end{align} 
 where the case $\varphi=0$ corresponds to a solution which is constant in the direction transverse to the slope as before. Substituting this ansatz into~\eqref{eq:modKlausmeier+D} results in the traveling wave ODE
 \begin{equation}
	\begin{cases}
		-c u_\xi & =D  u_{\xi\xi}+ \frac{\cos\varphi}{\varepsilon} u_\xi + a - u - G(u,v)v, \\
		-c v_\xi & =  v_{\xi\xi} - m v + R(v)G(u,v)v.
	\end{cases}\label{eq:TWODE+D}
\end{equation}
By setting $\tilde{\eps} = \eps/\cos\varphi$, we see that~\eqref{eq:TWODE+D} is the same traveling wave equation one obtains in the case of $\varphi=0$, except with $\eps$ replaced by $\tilde{\eps}$. For small values of $\varphi$, we have that
\begin{align}
\tilde{\eps} = \eps(1+\mathcal{O}(\varphi^2))
\end{align}
and~\eqref{eq:TWODE+D} can therefore be solved to find an angled traveling wave solution when
\begin{align}
c = c(\varphi)= c_s+\mathcal{O}(\eps \varphi^2).
\end{align}
 The quantity $c(\varphi)$ is called the nonlinear dispersion relation and relates the speed of propagation and angle of the traveling wave solution. A related quantity 
 \begin{align}
d(\varphi):=\frac{c(\varphi)}{\cos(\varphi)}
\end{align}
called the directional dispersion, or flux, relates the angle to the speed of propagation in the direction of the original traveling wave $(u_s,v_s)$, i.e. the $x$-direction. The flux near $\varphi=0$ is said to be convex if $d''>0$, concave if $d''<0$, and flat if $d''\equiv0$ for small $|\varphi|$. In~\cite{HS2}, the authors related the convexity of the flux to the type of corner defect which is selected: in particular when $d$ is convex, there exist interior corner defects for nearby speeds $c>c_s$, while for $d$ concave there exist exterior corner defects for speeds $c<c_s$.

In the case of~\eqref{eq:TWODE+D}, the directional dispersion is computed as
 \begin{align}
d(\varphi):=c_s\left( 1+\frac{\varphi^2}{2} \right)+\mathcal{O}\left( \eps \varphi^2, \varphi^4 \right),
\end{align}
from which we find that
 \begin{align}
d''(\varphi):=c_s+\mathcal{O}\left( \eps, \varphi^2 \right),
\end{align}
that is, to leading order the convexity is determined by the speed of propagation of the original traveling wave $(u_s,v_s)$. In particular, for sufficiently small $\eps$, the directional dispersion is convex for $c_s>0$ and concave for $c_s<0$. Hence in the setting of Theorems~\ref{thm:stripeexistence},~\ref{thm:gapexistence}, or~\ref{thm:desertFrontExistence}, one expects to see nearby interior corner solutions, but \emph{not} exterior corner solutions. That is, the resulting curved vegetation stripe, gap, or front is oriented convex downslope. However, in the setting of Theorem~\ref{thm:vegetationFrontExistence}, the convexity depends on the value of $a/m$ as the speed $c_s$ can be negative if $a$ is large enough. In particular, one expects interior corner solutions if $\frac{a}{m}<\frac{9b}{2}+\frac{2}{b}$, but exterior corners (oriented convex upslope) if $\frac{a}{m}>\frac{9b}{2}+\frac{2}{b}$.

\section{Numerics}\label{sec:numerics}

In this section we present numerical results related to Theorems~\ref{thm:stripeexistence}--\ref{thm:vegetationFrontExistence} and Theorems~\ref{thm:heteroclinicStability}--\ref{thm:gapStability} regarding the existence and stability of front, stripe, and gap pattern solutions of~\eqref{eq:modKlausmeier} . In particular, we discuss the results of numerical continuation of stripe and gap traveling wave solutions, and direct numerical simulation of planar stripe, gap, and front solutions, as well as corner defect solutions as discussed in~\S\ref{sec:corners}.

\subsection{Continuation of traveling stripes and gaps}
Theorems~\ref{thm:stripeexistence}--\ref{thm:gapexistence} predict the existence of traveling stripe and gap solutions to~\eqref{eq:modKlausmeier} which solve the traveling wave ODE~\eqref{eq:twode}. These solutions were constructed as perturbations of singular homoclinic orbits, organized by the singular bifurcations diagrams in Figures~\ref{fig:bd_case1} and~\ref{fig:bd_case2}, corresponding to the cases of $b<2/3$ and $b>2/3$, respectively. Figure~\ref{fig:bd_numerics} depicts the results of numerical continuation of speed $c$ versus $a$ for traveling stripes and gaps, conducted in AUTO for the parameter values $\eps=3\cdot10^{-4}$, $m=0.5$, and values of $b=0.6,0.7$ on either side of the critical value $b=2/3$. The continuation curves corresponding to vegetation stripe solutions are depicted in green, while those corresponding to gap solutions are in purple, with the relevant singular bifurcation curves depicted as dashed lines.

We note that the upper branches of the bifurcation curves for both stripes and gaps continue towards $c=0$ and eventually turn back onto lower branches which persist for a range of $a$ values and small speeds $c\ll1$. These waves arise as perturbations of a family of fast planar homoclinic orbits, as discussed in Remark~\ref{rem:layerhomoclinics}, and we expect they are unstable (even to $1$D perturbations) as traveling wave solutions of~\eqref{eq:modKlausmeier}. Interestingly, the lower branch of stripe solutions continues for increasing $a$, while the lower branch of gap solutions eventually turns back near the canard value $\frac{a}{m}=4b+1/b$ due to interaction of the equilibrium $p_+(u_2)$ with the fold point $\mathcal{F}$. 

\begin{Remark} We also remark that in the case of $b<2/3$, depicted in the left panel of Figure~\ref{fig:bd_numerics}, that the upper branch of gap solutions also approaches the canard point. Here this branch transitions into a ``double-gap" solution, resembling two copies of the primary homoclinic orbit. This transition is similar to canard transitions observed in systems such as the FitzHugh--Nagumo equation~\cite{CSbanana,champneys2007shil,guckenheimer2010homoclinic}, albeit with a somewhat different mechanism due to the presence of the additional  equilibrium $p_0(a)$.
\end{Remark}

We also depict the results of continuation of both stripe and gap solutions for fixed values of rainfall $a=1.2$ (stripes) and $a=2$ (gaps), with $m=0.45, b=0.5$, and $\eps=0.01$. As discussed in~\S\ref{sec:singularperiodicorbits}, it is expected that nearby the single traveling stripe or gap solutions are periodic wave train solutions corresponding to repeating vegetation patterns which exist for a range of wave speeds, and that these patterns can similarly be constructed by perturbing from singular periodic orbits in the traveling wave equation~\eqref{eq:twode}. We verify this prediction by numerically continuing the stripe (and gap) solutions as periodic orbits for decreasing period, the results of which are depicted in Figure~\ref{fig:wavelength_continuation}. We observe that in general the wave speed $c$ decreases as the period $T$ decreases, as do the total biomass $B:=\int_0^T v\ \mathrm{d} x$ and the maximum value of $v$ over one period, denoted by $v_\mathrm{max}$. Lastly the results of continuation of periodic orbits in $(a,k)$-space for fixed wave speeds $c=\{0.15,0.2,0.25,0.3,0.35\}$ are depicted in Figure~\ref{fig:wavenumber_vs_rainfall}; here $k$ denotes the wavenumber of the corresponding pattern. 

These numerical results align with previous work on (similar) ecosystem models; similar trends are found in, for instance, studies on the Klausmeier vegetation model~\cite{sherratt2013pattern}, on extended Klausmeier models~\cite{siteur2014beyond, BD2018, bastiaansendata}, on the Klausmeier-Gray-Scott model~\cite{sewaltspatially} and the Rietkerk model~\cite{dagbovie2014pattern}. Moreover, measurements on the speed of migrating vegetation patterns, indeed, show vegetation patterns with higher wavelength move faster~\cite{deblauwe2012determinants, bastiaansendata}. Finally, recent in-situ measurement on the above ground biomass in the Horn of Africa corroborate displayed trends in biomass~\cite{bastiaansendata}.

\begin{figure}
\hspace{.025\textwidth}
\begin{subfigure}{.45 \textwidth}
\centering
\includegraphics[width=1\linewidth]{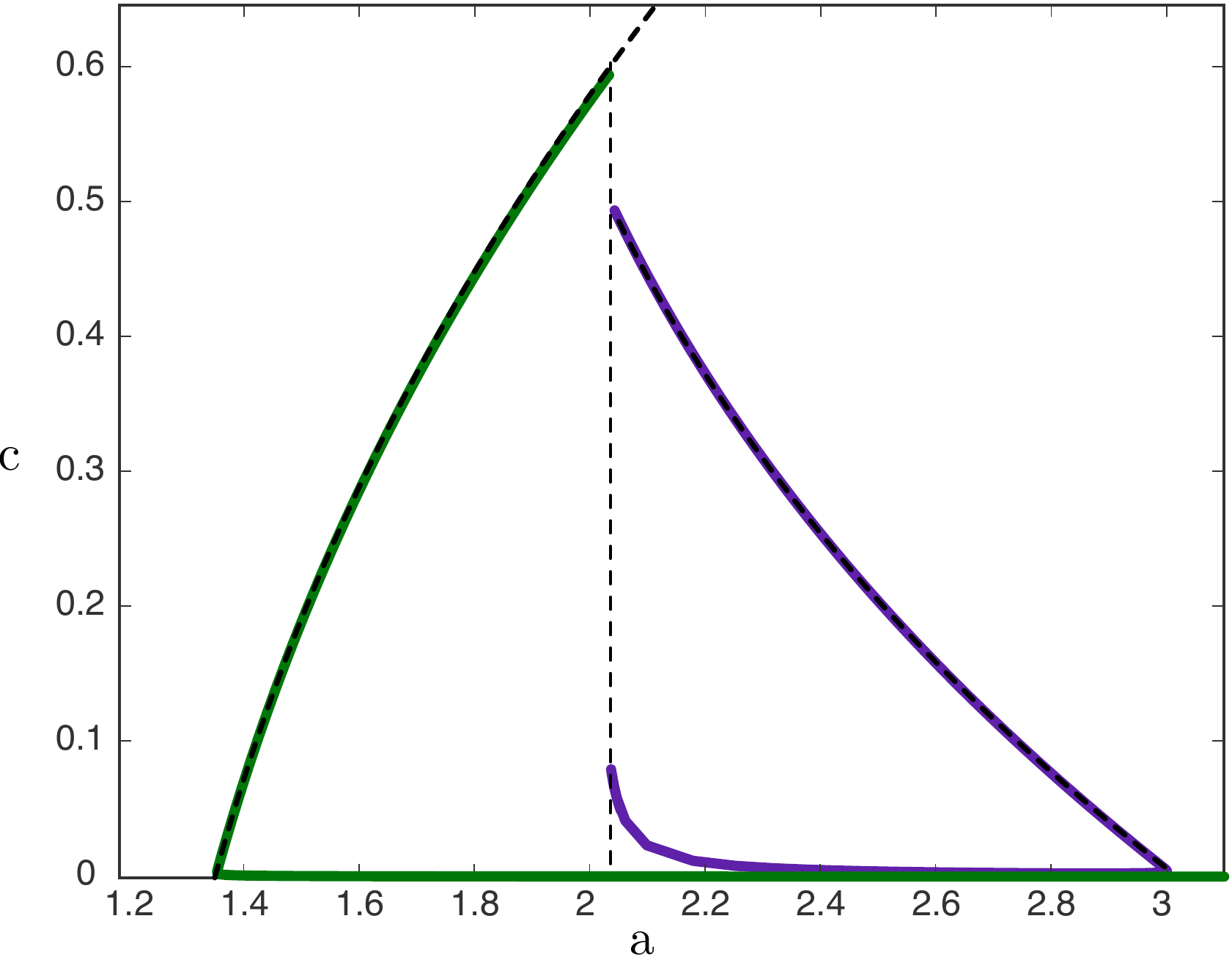}
\end{subfigure}
\hspace{.025\textwidth}
\begin{subfigure}{.45 \textwidth}
\centering
\includegraphics[width=1\linewidth]{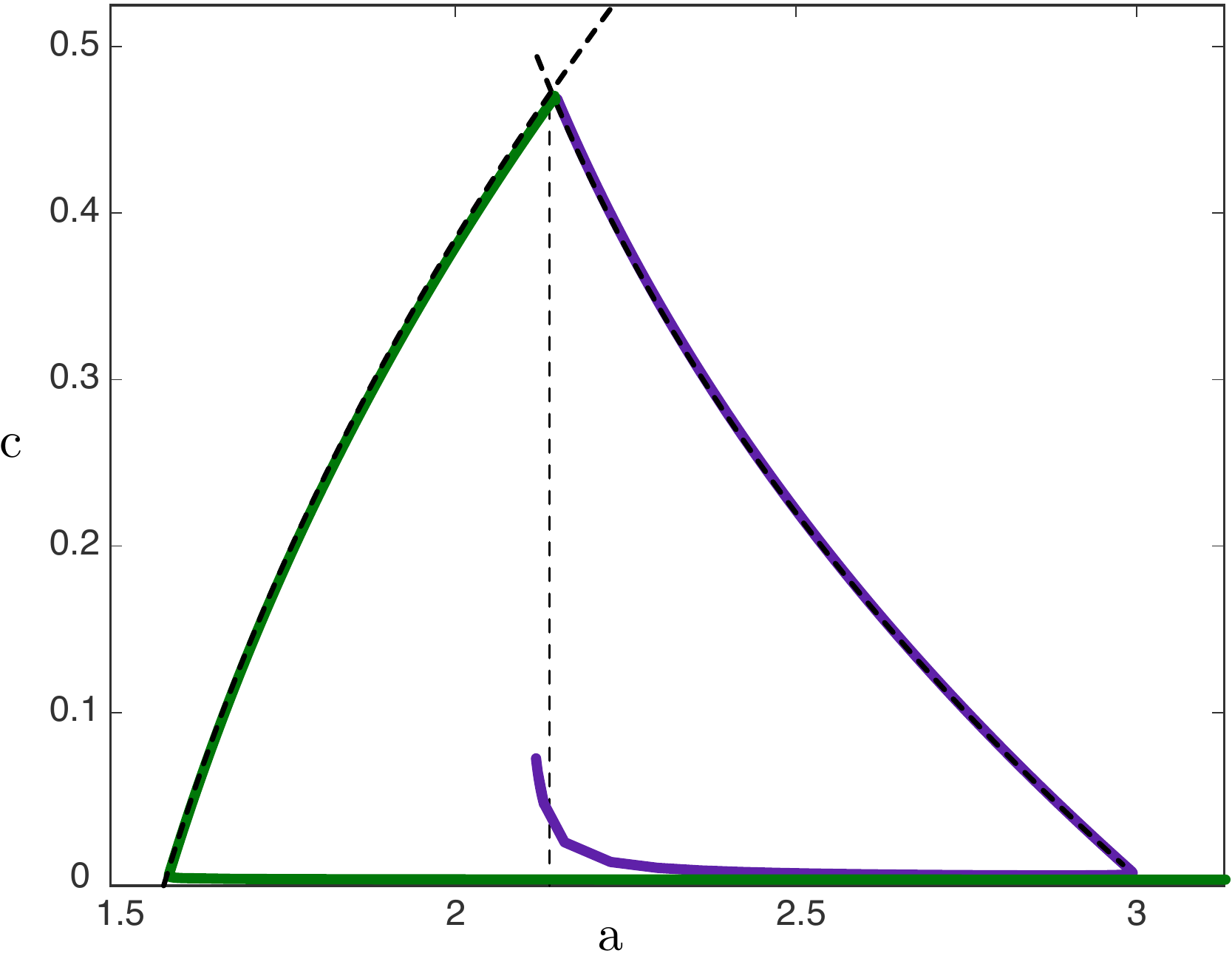}
\end{subfigure}
\hspace{.02\textwidth}
\caption{Shown are numerically computed bifurcation diagrams of vegetation stripes (green curves) and gaps (purple curves) for the parameter values $m=0.5, \eps=0.0003$, and  $b=0.6$ (left panel), $b=0.74$ (right panel). The solutions were obtained via parameter continuation in AUTO for the traveling wave equation~\eqref{eq:twode}. Also plotted in dashed black are the curves $c=c^*(a)$ and $c=\hat{c}(a)$. The vertical dashed curve denotes the location of $\bar{a}$ in each panel. }
\label{fig:bd_numerics}
\end{figure}

\begin{figure}
	\centering
	\begin{subfigure}[t]{0.3\textwidth}
		\centering
			\includegraphics[width=\textwidth]{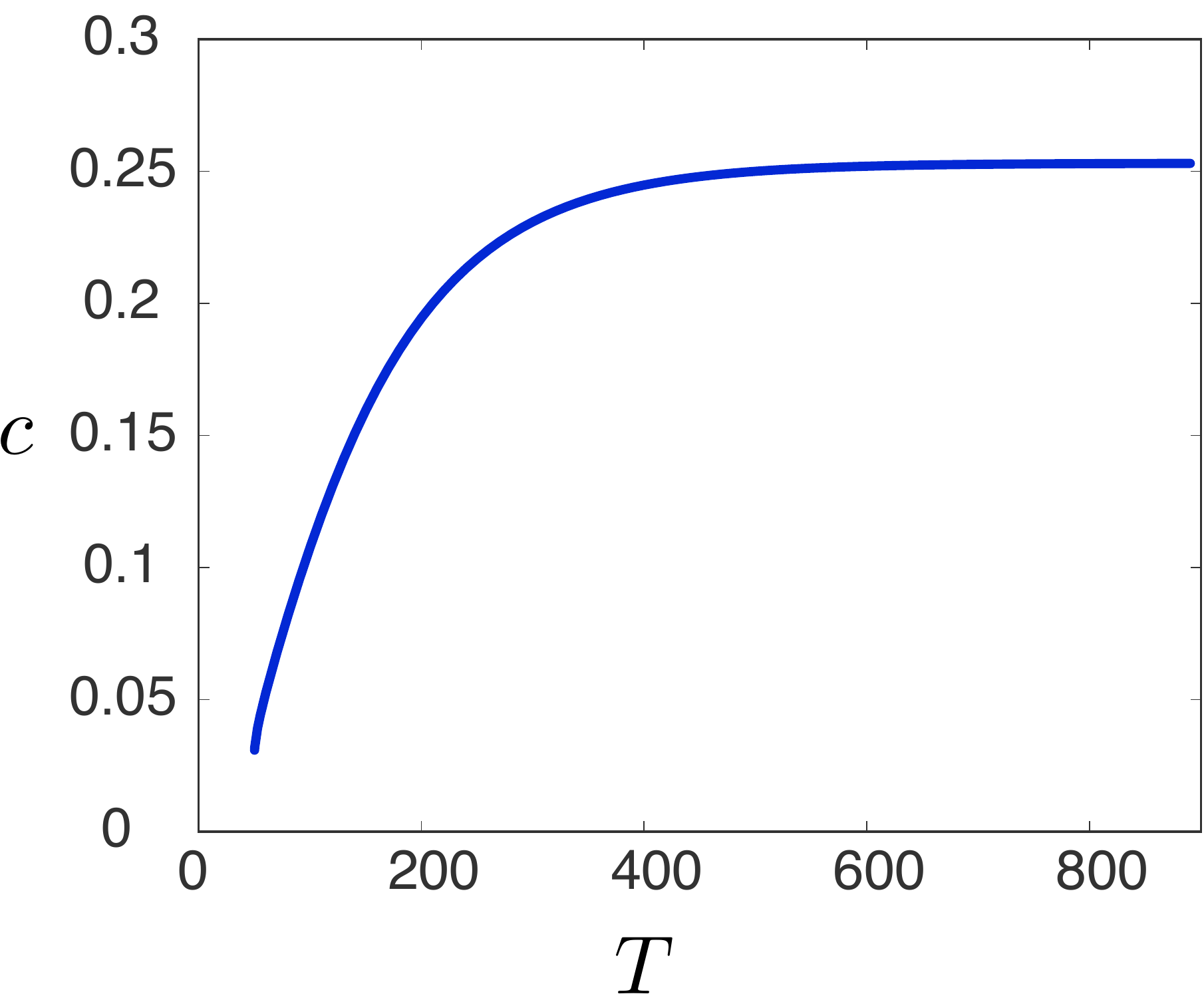}
		\caption{Speed of periodic stripes}
	\end{subfigure}
	\hspace{0.025\textwidth}
	\begin{subfigure}[t]{0.3\textwidth}
		\centering
			\includegraphics[width=\textwidth]{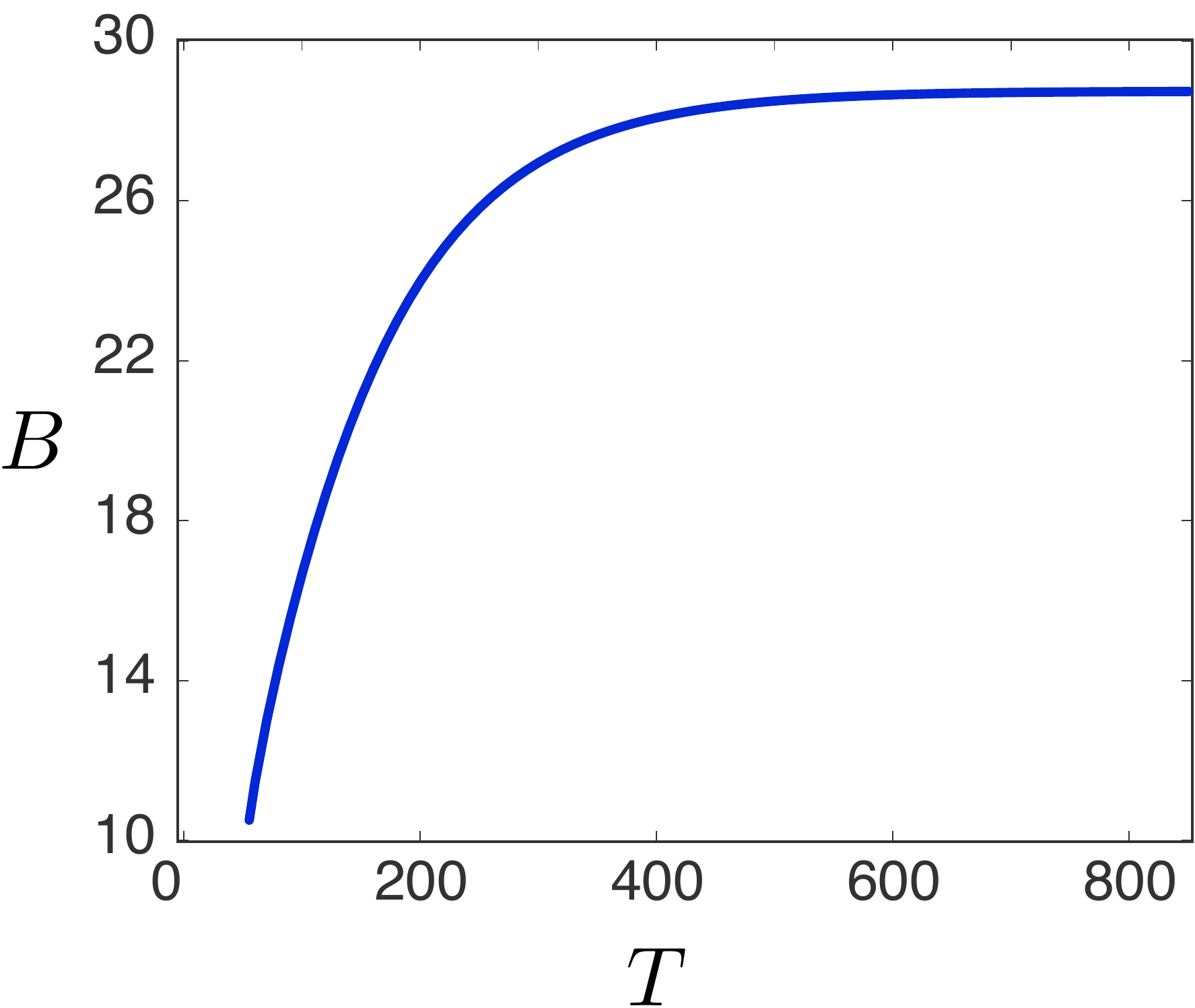}
		\caption{Biomass of periodic stripes}
	\end{subfigure}
	\hspace{0.025\textwidth}
	\begin{subfigure}[t]{0.3\textwidth}
		\centering
			\includegraphics[width=\textwidth]{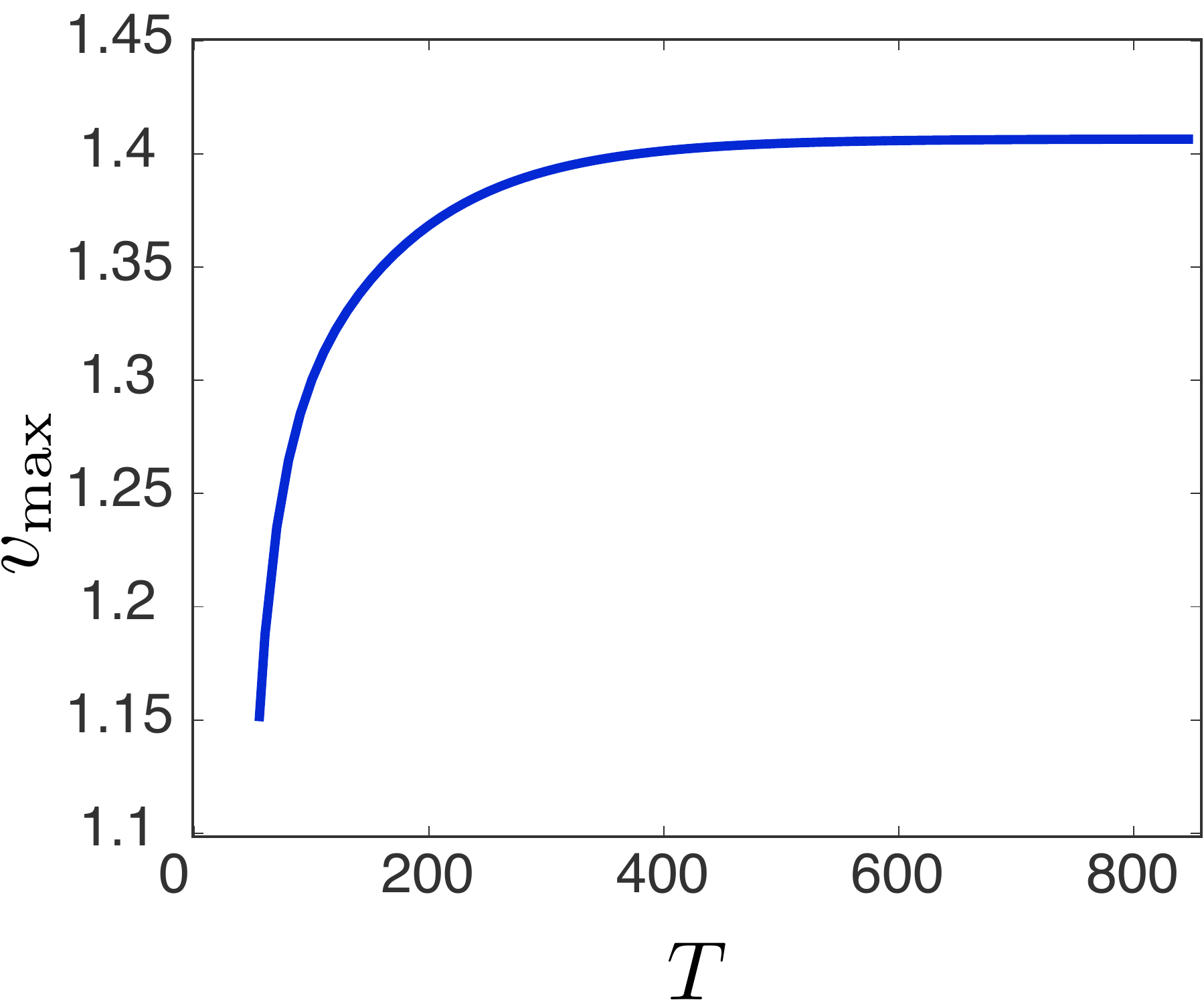}
		\caption{$v_\mathrm{max}$ of periodic stripes}
	\end{subfigure}
	\begin{subfigure}[t]{0.3\textwidth}
		\centering
			\includegraphics[width=\textwidth]{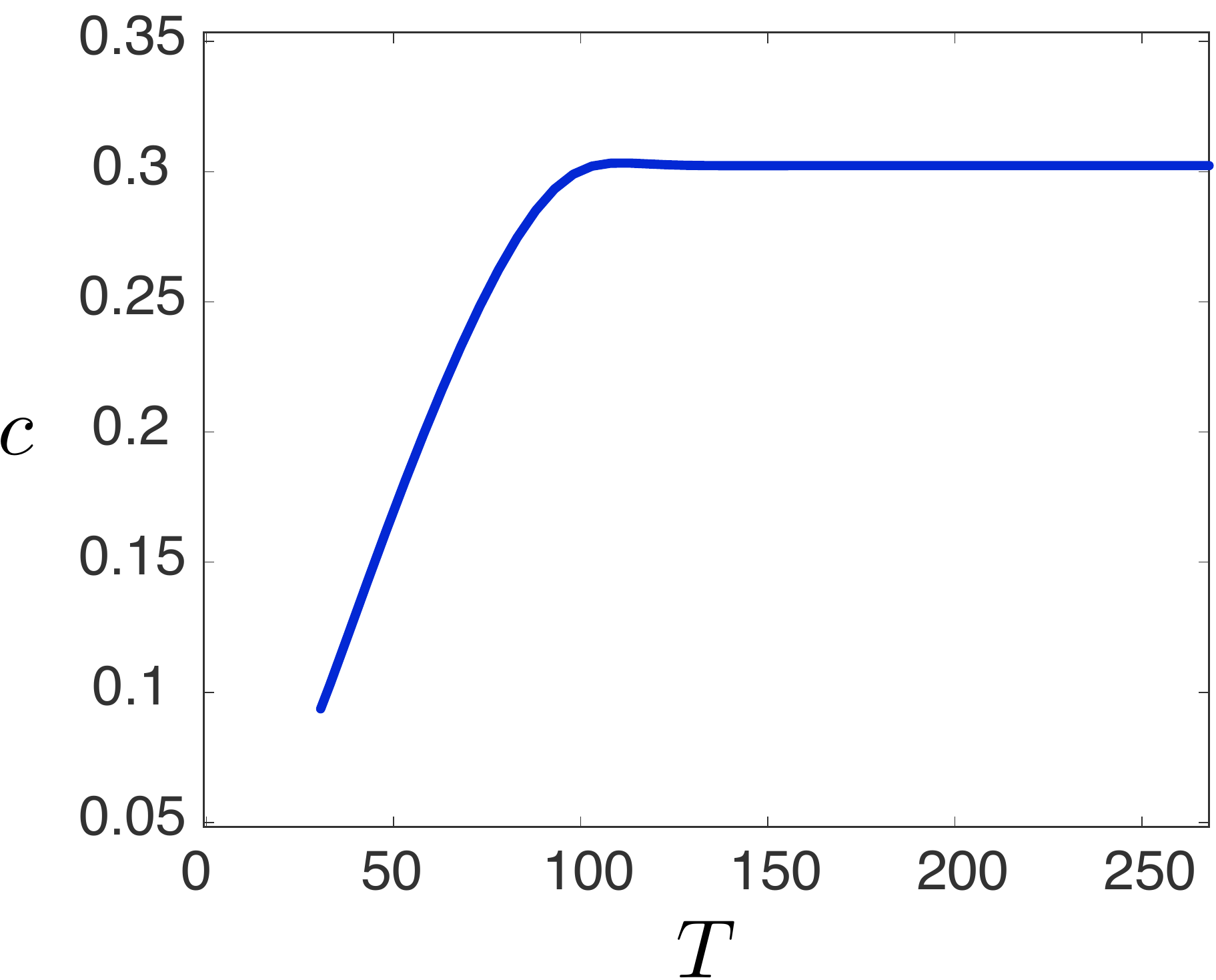}
		\caption{speed of periodic gaps}
	\end{subfigure}
	\hspace{0.025\textwidth}
	\begin{subfigure}[t]{0.3\textwidth}
		\centering
			\includegraphics[width=\textwidth]{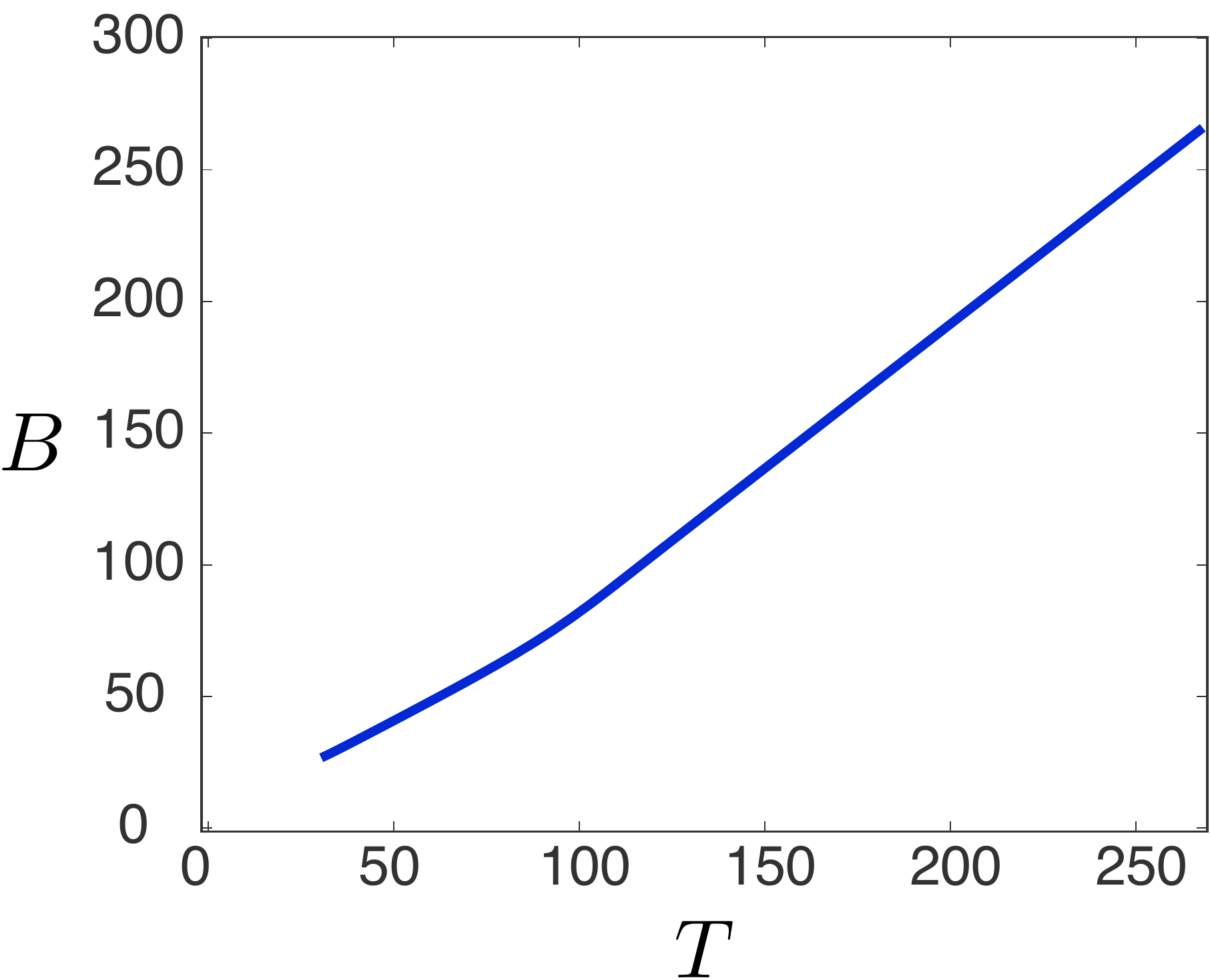}
		\caption{Biomass of periodic gaps}
	\end{subfigure}
	\hspace{0.025\textwidth}
	\begin{subfigure}[t]{0.3\textwidth}
		\centering
			\includegraphics[width=\textwidth]{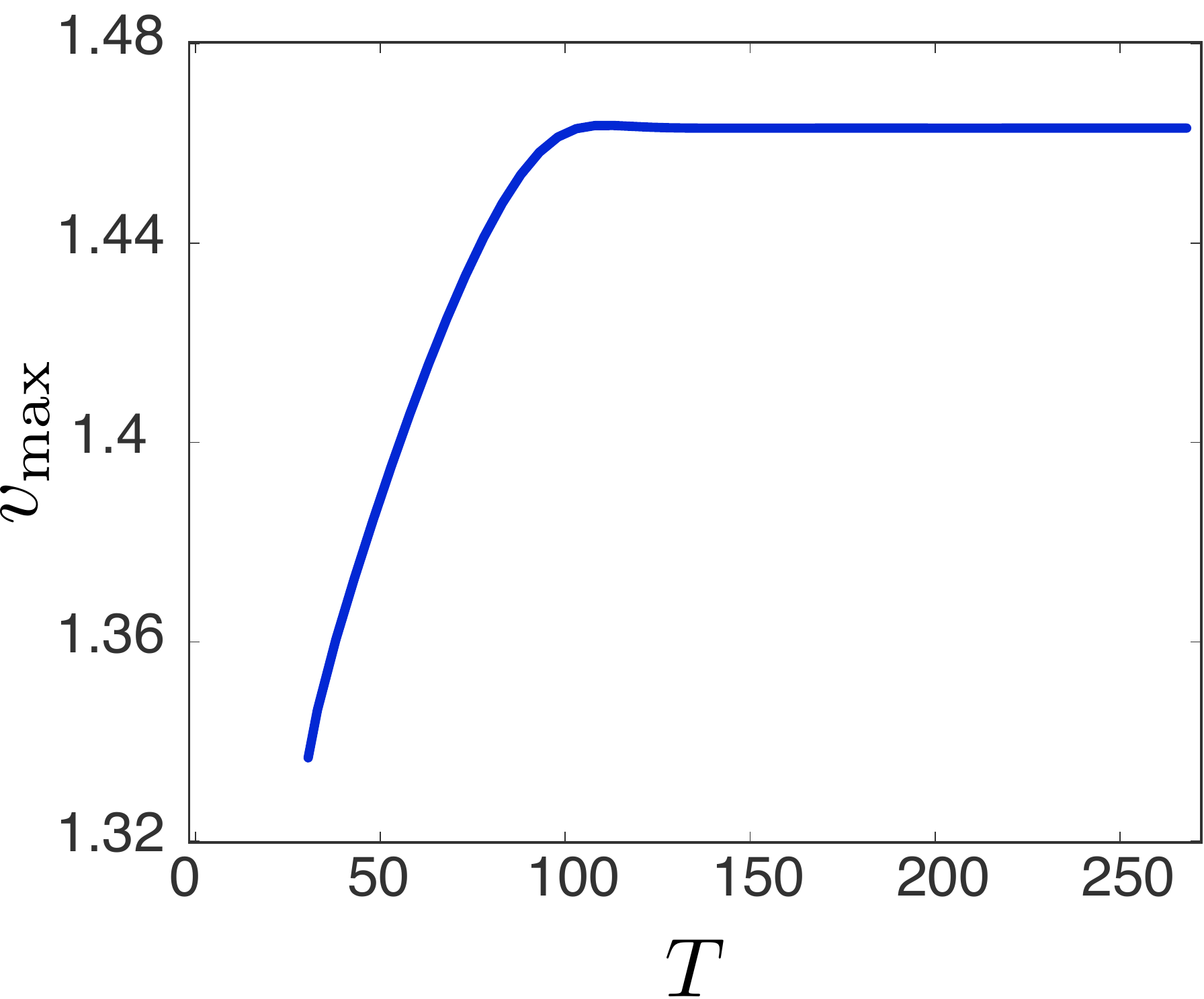}
		\caption{$v_\mathrm{max}$ of periodic gaps}
	\end{subfigure}\\
\caption{Results of numerical continuation of periodic stripe (a-c) and gap (d-f) pattern solutions for decreasing wavelength for the parameter values $m=0.45, b=0.5, \eps=0.01$ and $a=1.2$ (stripes), $a=2$ (gaps). Shown are plots of speed $c$ of the pattern vs. period $T$ (left panels),  biomass $B:=\int_0^Tv~\mathrm{d} x$ vs period $T$ (middle panels), and $v_\mathrm{max}$ versus the period $T$, where $v_\mathrm{max}$ denotes the maximum of $v$ over one period (right panels).}
\label{fig:wavelength_continuation}
\end{figure}

\begin{figure}
		\centering
			\includegraphics[width=0.5\textwidth]{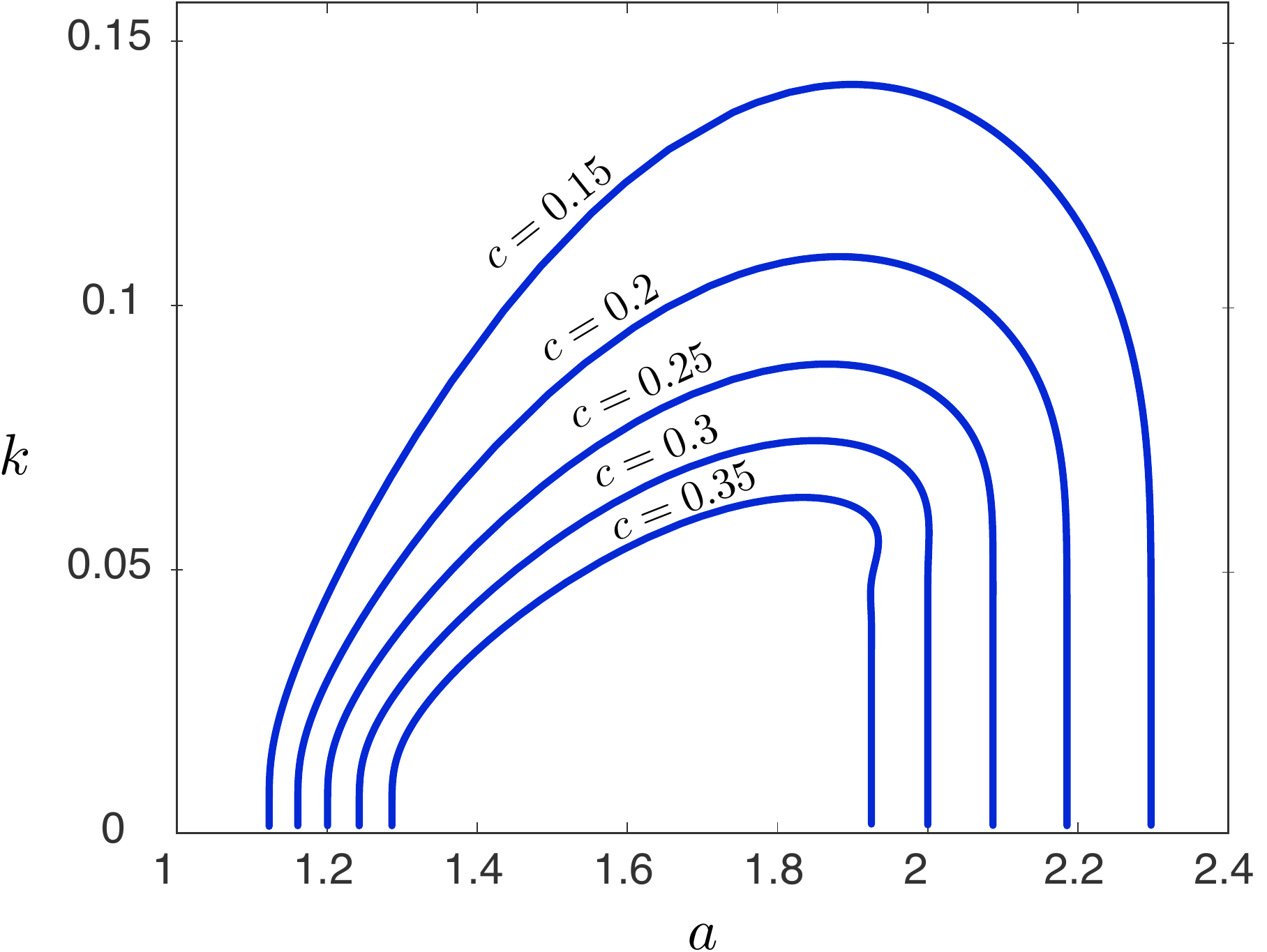}
		\caption{Results of numerical continuation of periodic stripe/gap patterns for spatial wavenumber $k$ versus $a$ for fixed $b=0.5, m=0.45, \eps=0.01$ and wave speeds $c=\{0.15,0.2,0.25,0.3,0.35\}$. }
\label{fig:wavenumber_vs_rainfall}
\end{figure}

\subsection{Direct simulations}
In this section we present direct numerical simulations of the various traveling wave solutions predicted by Theorems~\ref{thm:stripeexistence}--\ref{thm:vegetationFrontExistence}. To that end, we have spatially discretized the PDE~\eqref{eq:modKlausmeier} with a uniformly spaced grid in both $x$ and $y$ directions, which was integrated using a Runge--Kutta solver. In all simulations, the initial conditions were constructed using the approximate expressions derived in the previous sections of this article.

First, we have tested the existence and 2D stability of straight (i.e. non-curved) patterns. The results for $b = 0.5 < 2/3$ are given in Figure~\ref{fig:numericsStraightb0p5} and for $b = 0.75 > 2/3$ in Figure~\ref{fig:numericsStraightb0p75}. In both cases, all solutions from Theorems~\ref{thm:stripeexistence}--\ref{thm:vegetationFrontExistence} could be obtained easily and were (2D) stable in our simulations (and in fact all seem to have a quite large domain of attraction).

Moreover, we numerically inspected corner solutions as described in~\S\ref{sec:corners}. Again, numerical simulations corroborate theoretical predictions -- see Figure~\ref{fig:numericsCorners}. In fact, we were able to find corner-type solutions for each front or pulse in Theorems~\ref{thm:stripeexistence}--\ref{thm:vegetationFrontExistence}. When the speed of the straight pattern is positive, i.e. $c_s > 0$, it is possible to find curved patterns which are oriented convex downslope (interior defect) and when $c_s < 0$ the curved pattern is oriented convex upslope (exterior defect); recall that upslope corresponds to the direction of increasing $x$. This matches the prediction given by the directional dispersion, as outlined in~\S\ref{sec:corners}.

\begin{figure}
	\centering
	\begin{subfigure}[t]{0.18\textwidth}
		\centering
			\includegraphics[width=\textwidth]{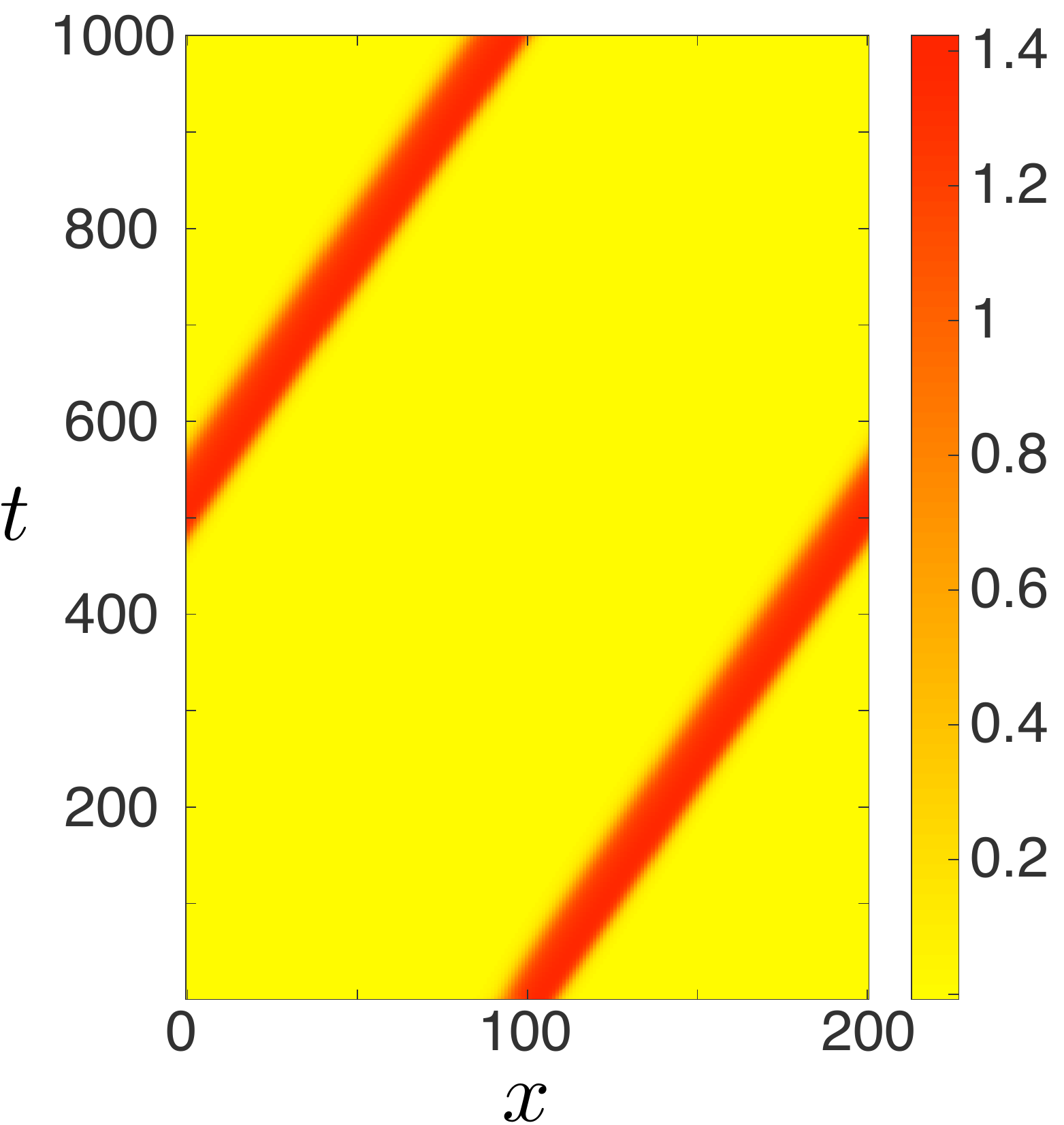}
		\caption{Stripe}
	\end{subfigure}
	\begin{subfigure}[t]{0.18\textwidth}
		\centering
			\includegraphics[width=\textwidth]{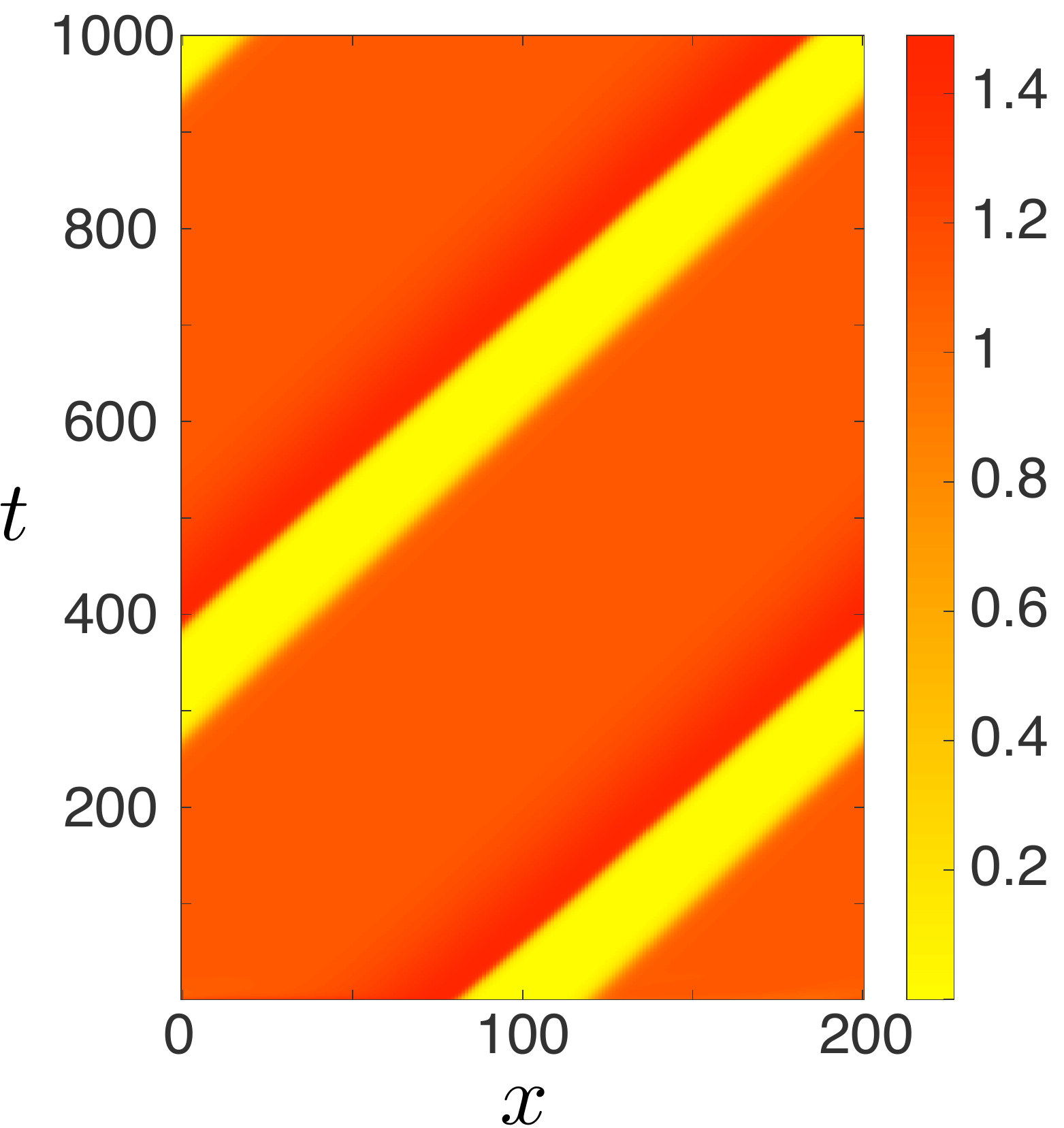}
		\caption{Gap}
	\end{subfigure}
	\begin{subfigure}[t]{0.18\textwidth}
		\centering
			\includegraphics[width=\textwidth]{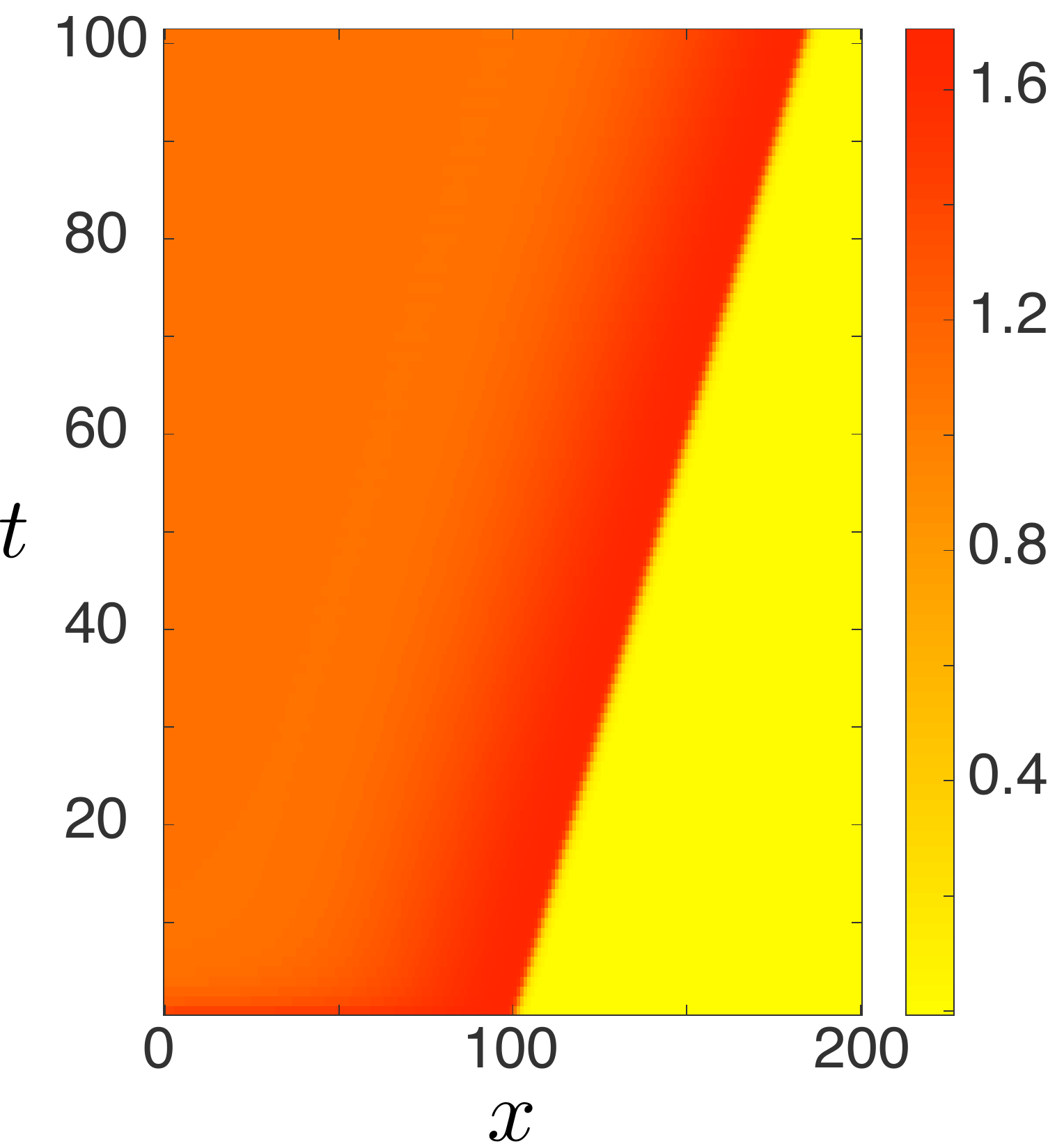}
		\caption{Vegetation front}
	\end{subfigure}
	\begin{subfigure}[t]{0.18\textwidth}
		\centering
			\includegraphics[width=\textwidth]{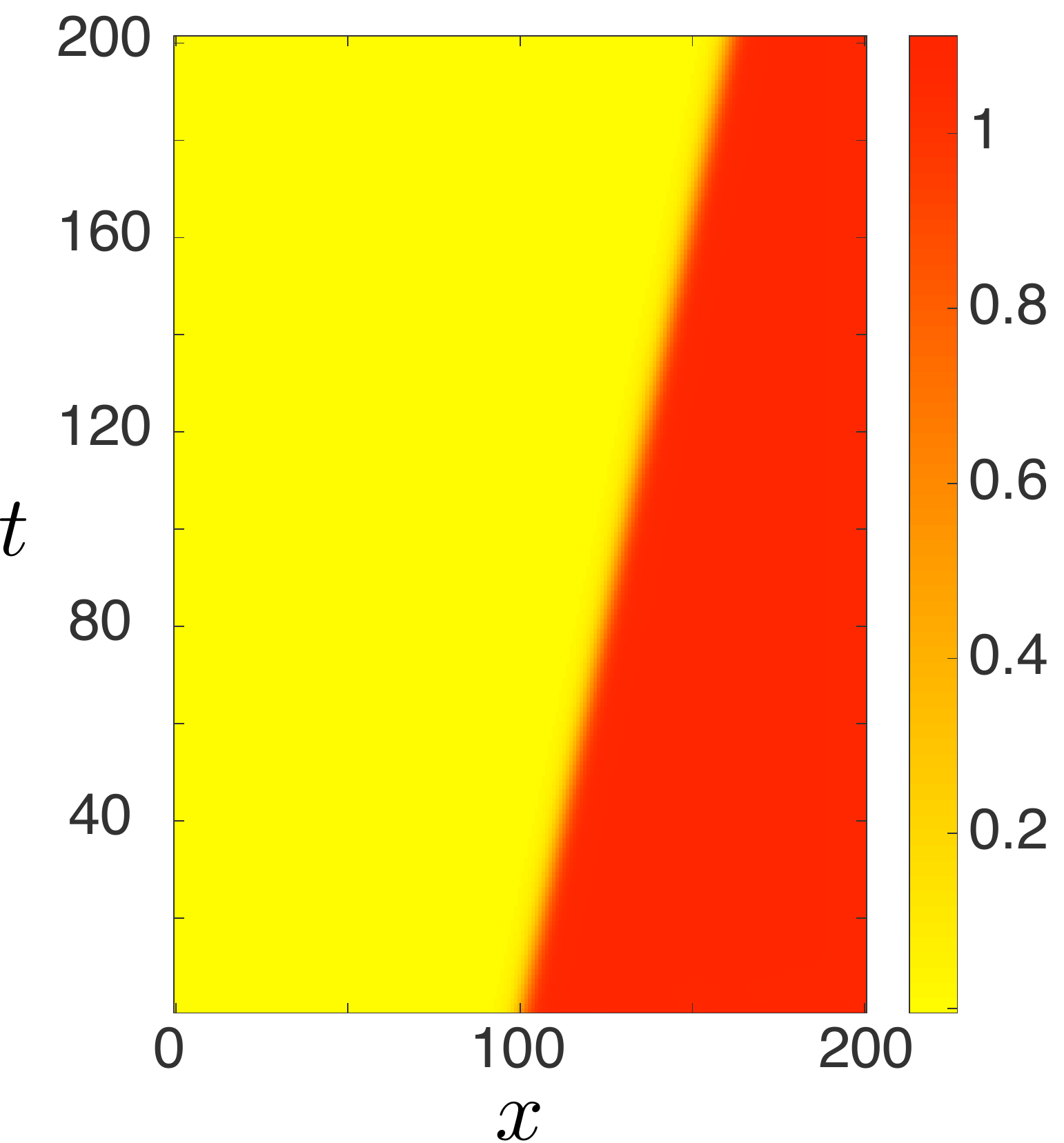}
		\caption{Desert front}
	\end{subfigure}
	\begin{subfigure}[t]{0.18\textwidth}
		\centering
			\includegraphics[width=\textwidth]{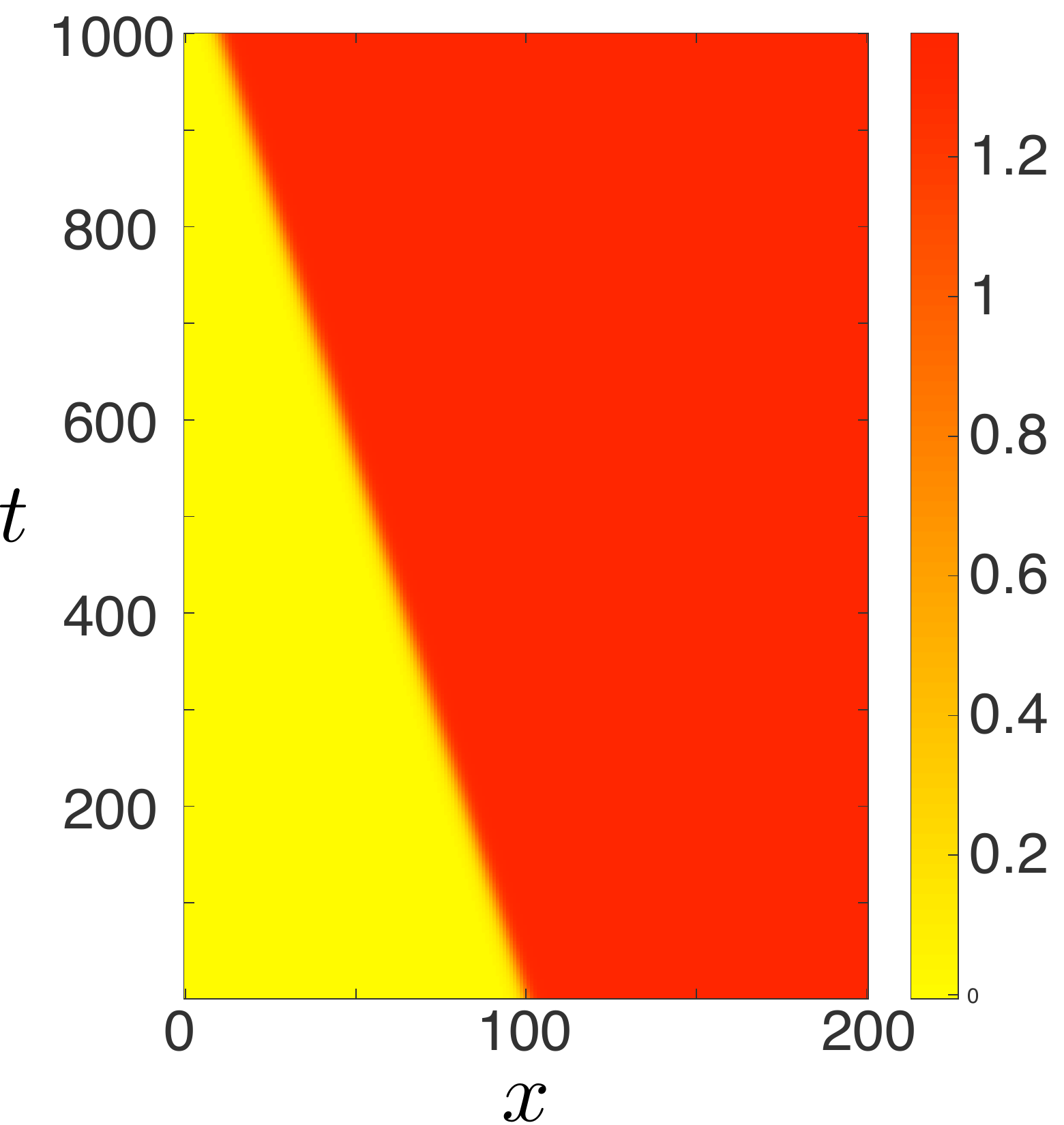}
		\caption{Desert front}
	\end{subfigure}

	\begin{subfigure}[t]{0.18\textwidth}
		\centering
			\includegraphics[width=\textwidth]{Figures/numericsStraightB0p5Stripe-PLANE.pdf}
		\caption{Stripe}
	\end{subfigure}
	\begin{subfigure}[t]{0.18\textwidth}
		\centering
			\includegraphics[width=\textwidth]{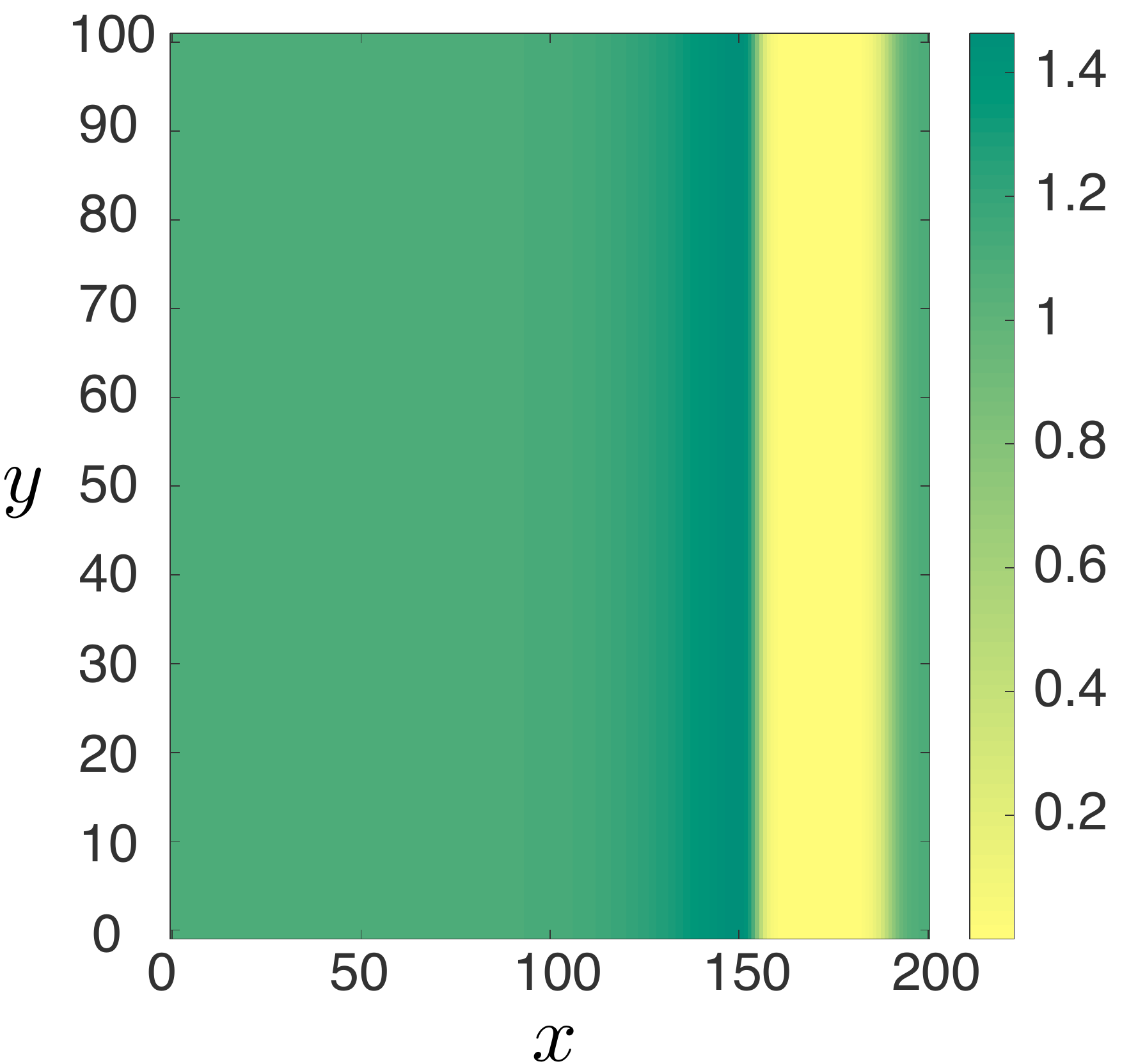}
		\caption{Gap}
	\end{subfigure}
	\begin{subfigure}[t]{0.18\textwidth}
		\centering
			\includegraphics[width=\textwidth]{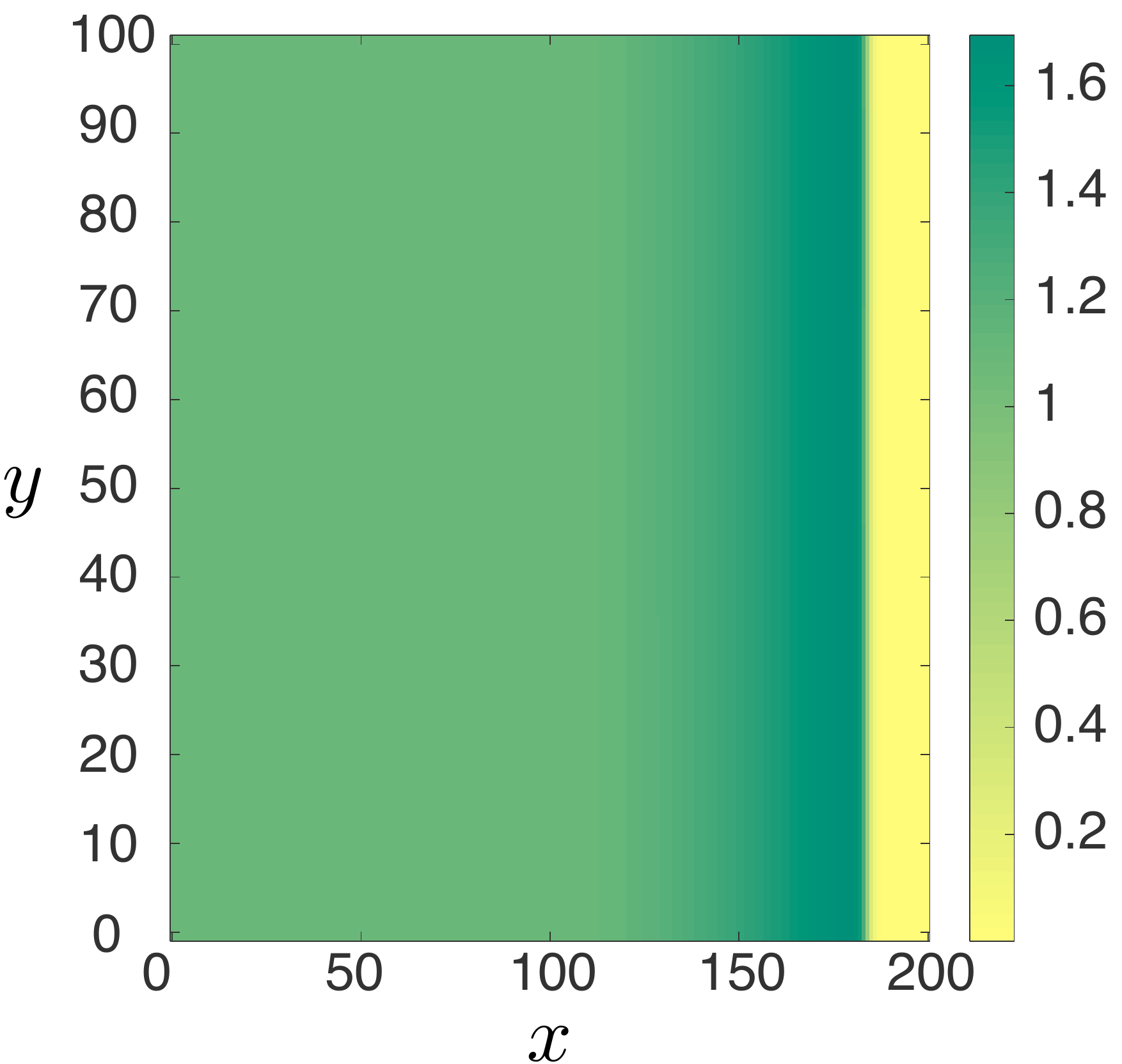}
		\caption{Vegetation front}
	\end{subfigure}
	\begin{subfigure}[t]{0.18\textwidth}
		\centering
			\includegraphics[width=\textwidth]{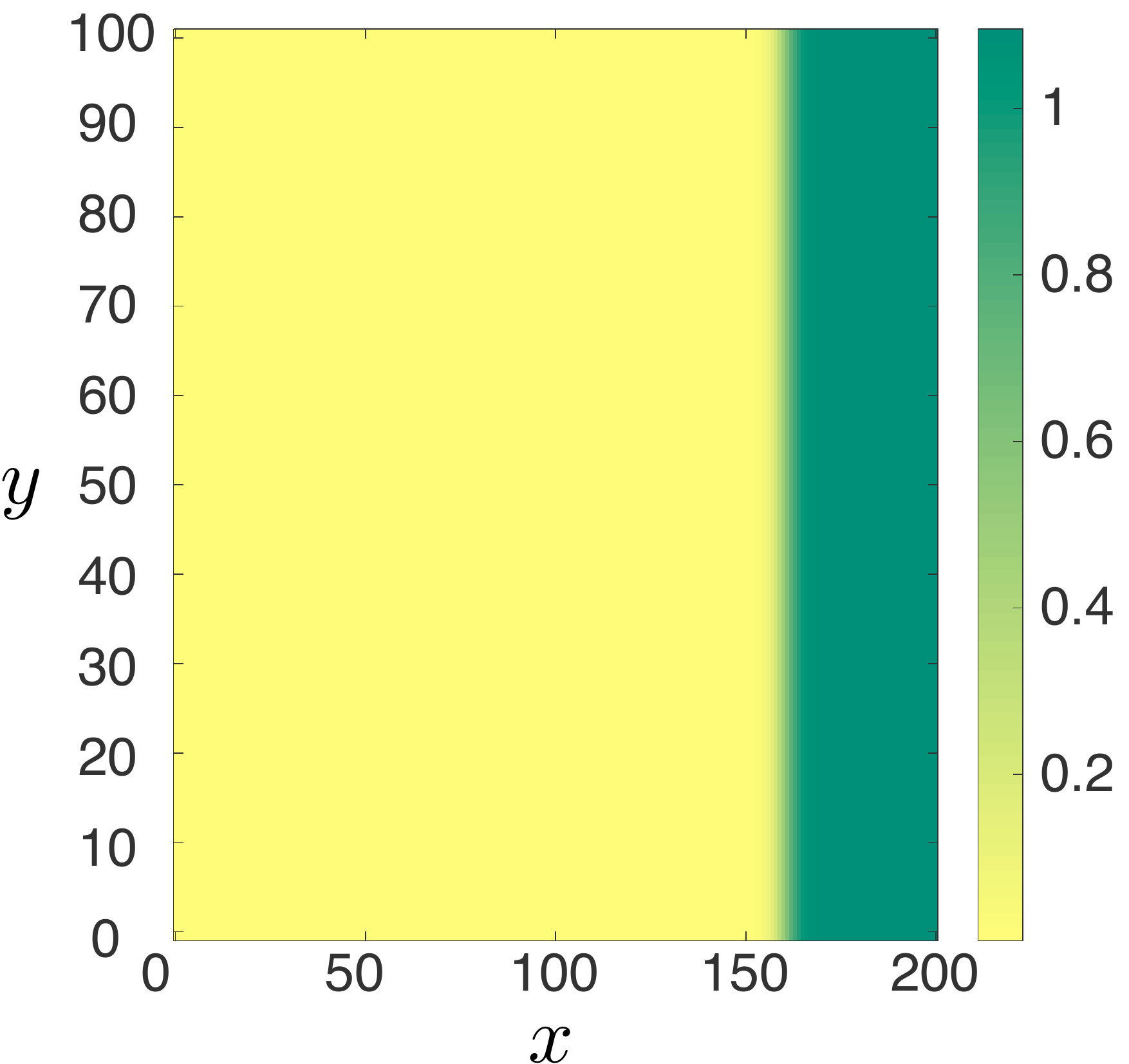}
		\caption{Desert front}
	\end{subfigure}
	\begin{subfigure}[t]{0.18\textwidth}
		\centering
			\includegraphics[width=\textwidth]{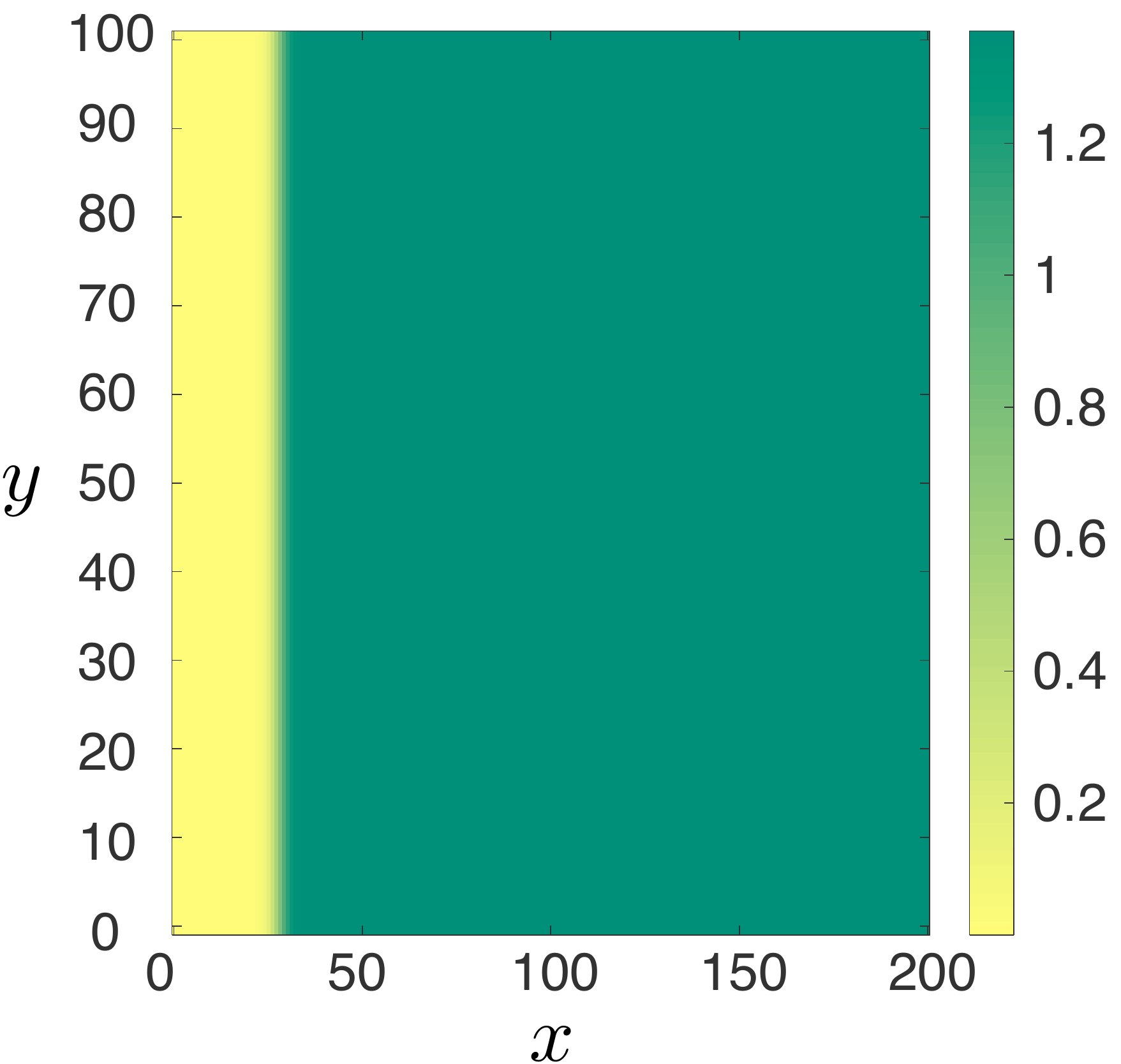}
		\caption{Desert front}
	\end{subfigure}
\caption{Results of direct numerical simulation of the PDE~\eqref{eq:modKlausmeier} for $b = 0.5$, $m = 0.45$, $\varepsilon = 0.01$ and $a = 1.2$ (a,f), $a = 2.0$ (b--d,g--i) or $a = 3.0$ (e,j). Figures a--e show the evolution of a cross section of $v$, i.e. for constant $y$ and figures f--j show the $v(x,y)$ pattern at a specific time. Simulations are run on a finite grid of size $L_x = 200$, $L_y = 100$, accompanied with Neumann boundary conditions for the $y$-direction and either periodic (a--b,f--g) or Neumann (c--e,h--j) boundary conditions in the $x$-direction.}
\label{fig:numericsStraightb0p5}
\end{figure}

\begin{figure}
	\centering
	\begin{subfigure}[t]{0.18\textwidth}
		\centering
			\includegraphics[width=\textwidth]{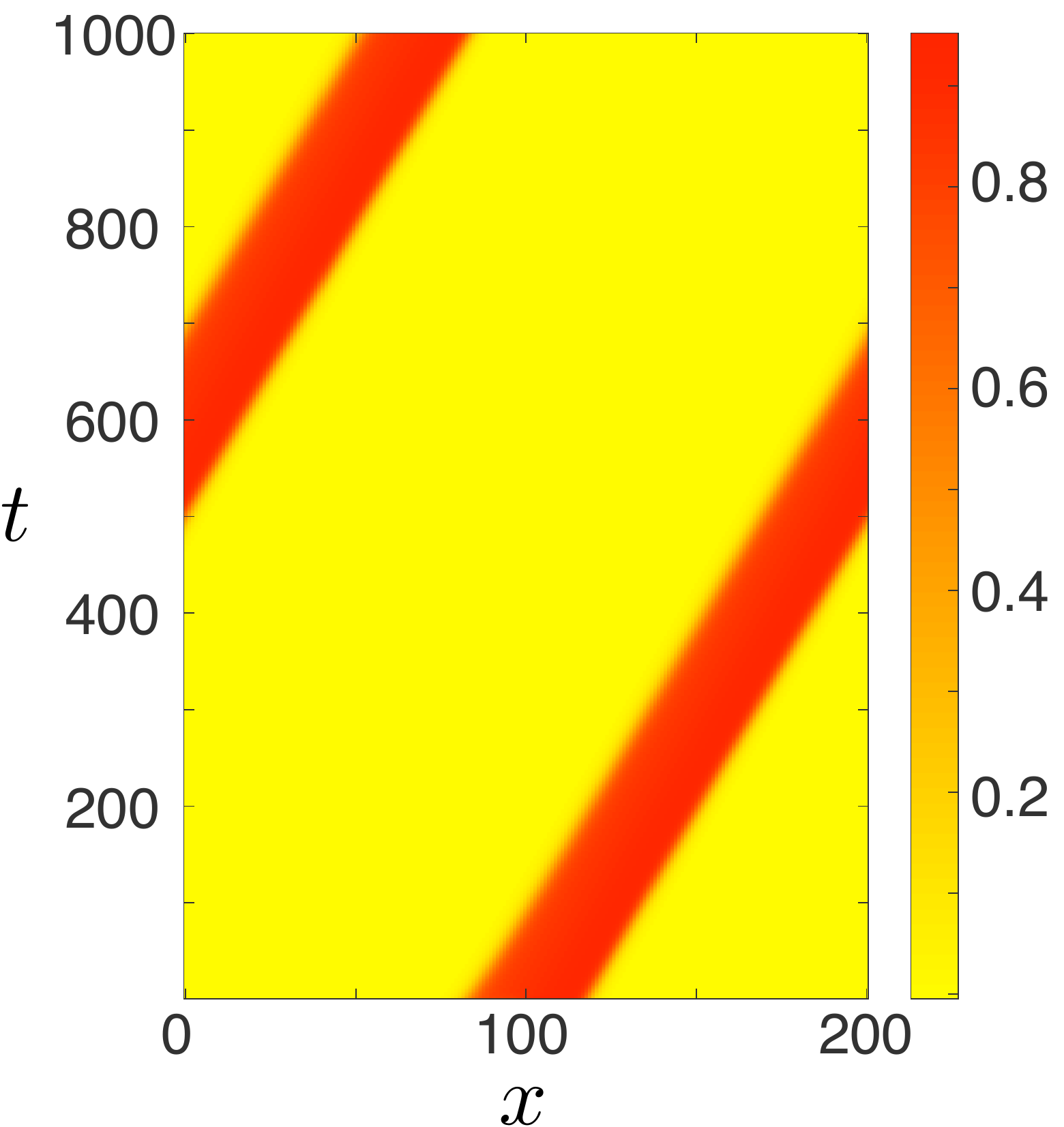}
		\caption{Stripe}
	\end{subfigure}
	\begin{subfigure}[t]{0.18\textwidth}
		\centering
			\includegraphics[width=\textwidth]{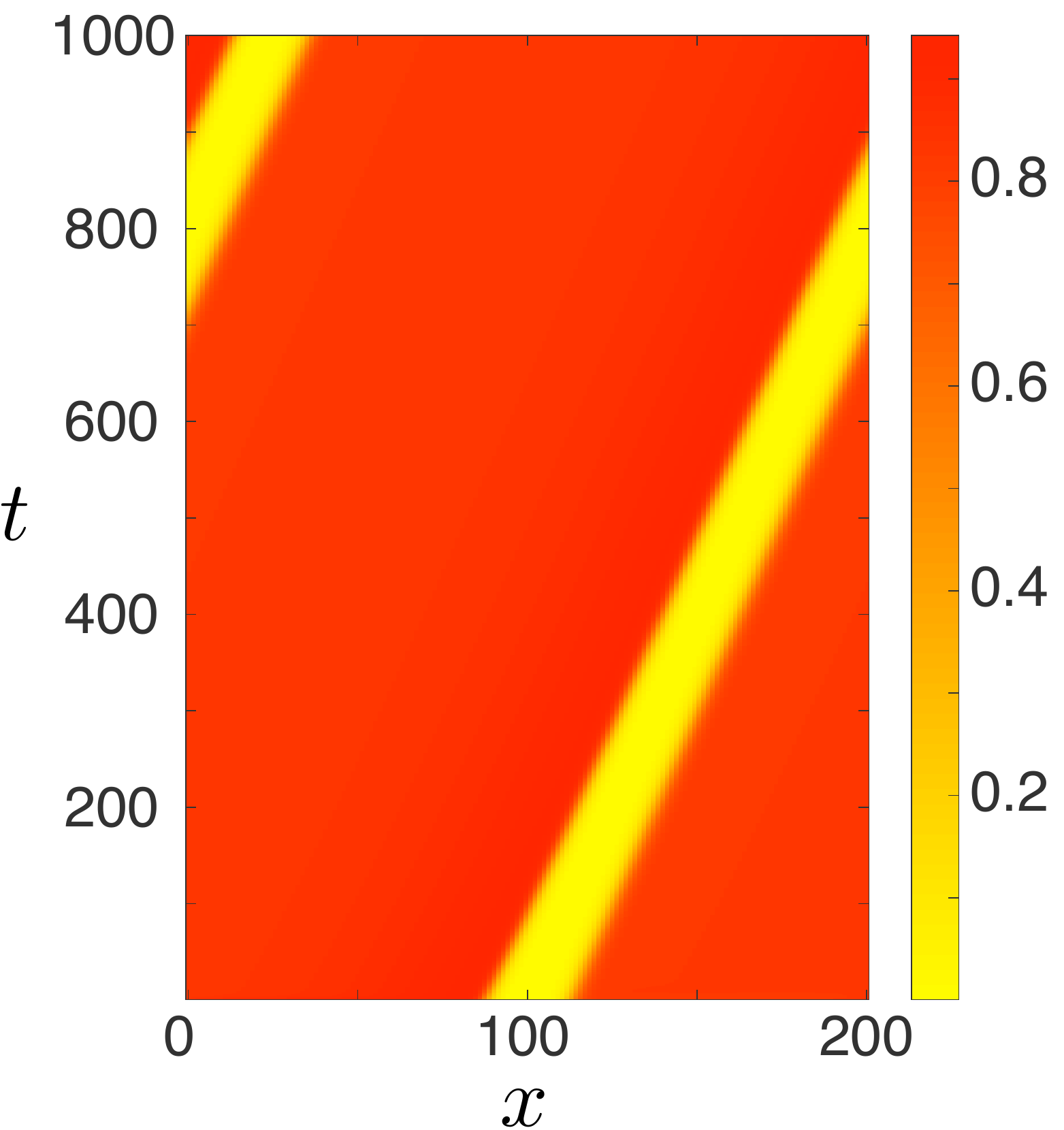}
		\caption{Gap}
	\end{subfigure}
	\begin{subfigure}[t]{0.18\textwidth}
		\centering
			\includegraphics[width=\textwidth]{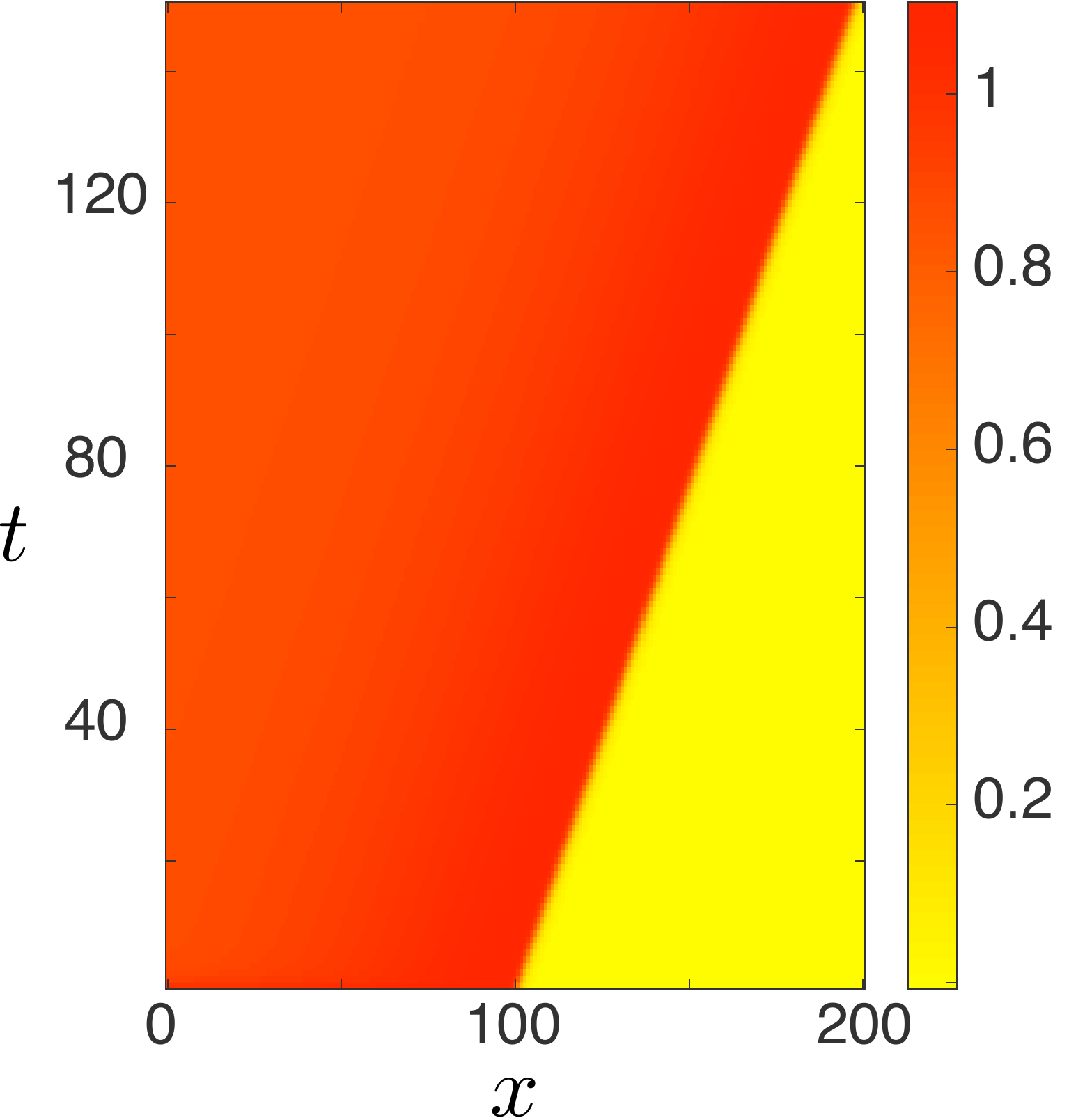}
		\caption{Vegetation front}
	\end{subfigure}
	\begin{subfigure}[t]{0.18\textwidth}
		\centering
			\includegraphics[width=\textwidth]{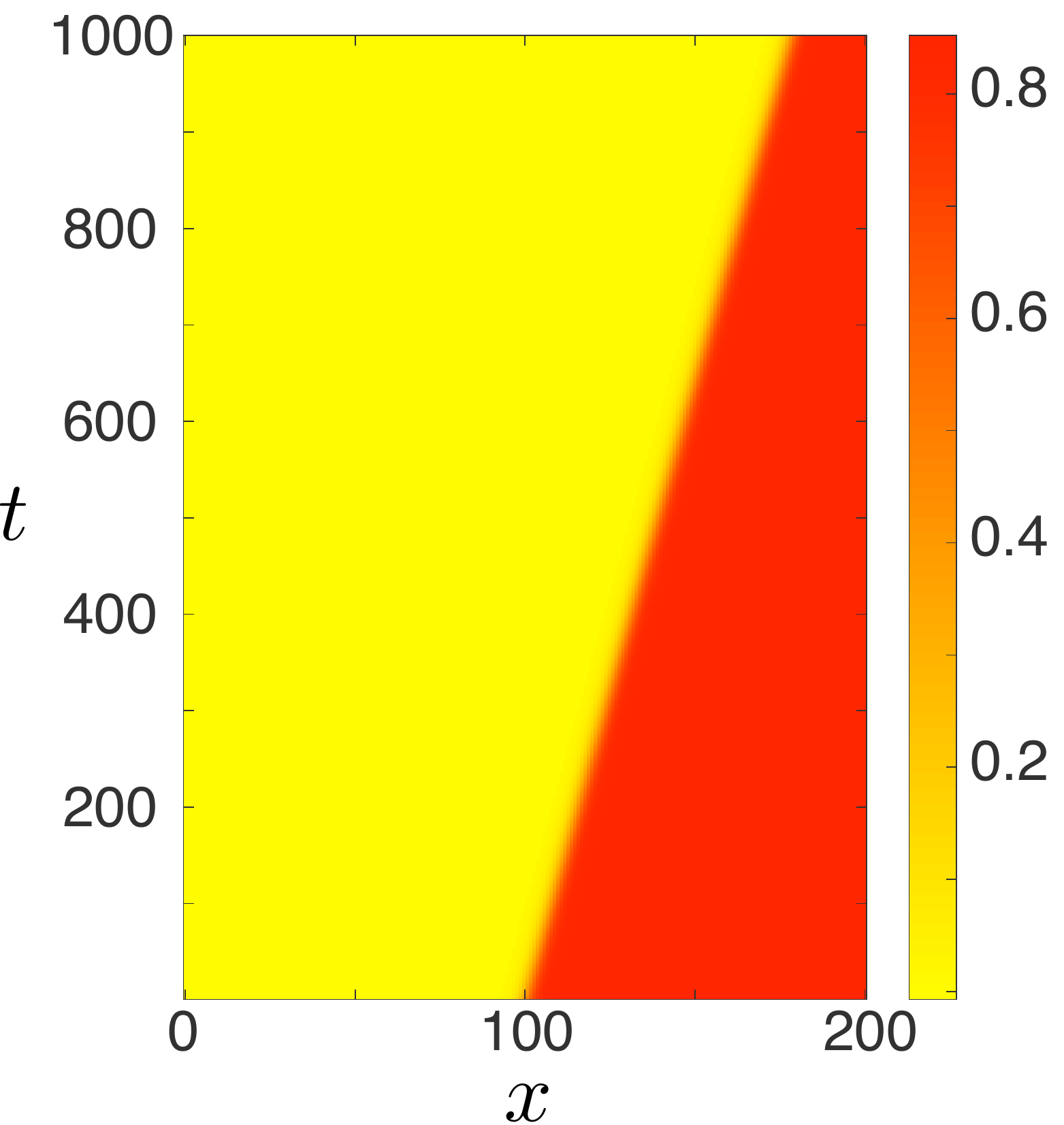}
		\caption{Desert front}
	\end{subfigure}
	\begin{subfigure}[t]{0.18\textwidth}
		\centering
			\includegraphics[width=\textwidth]{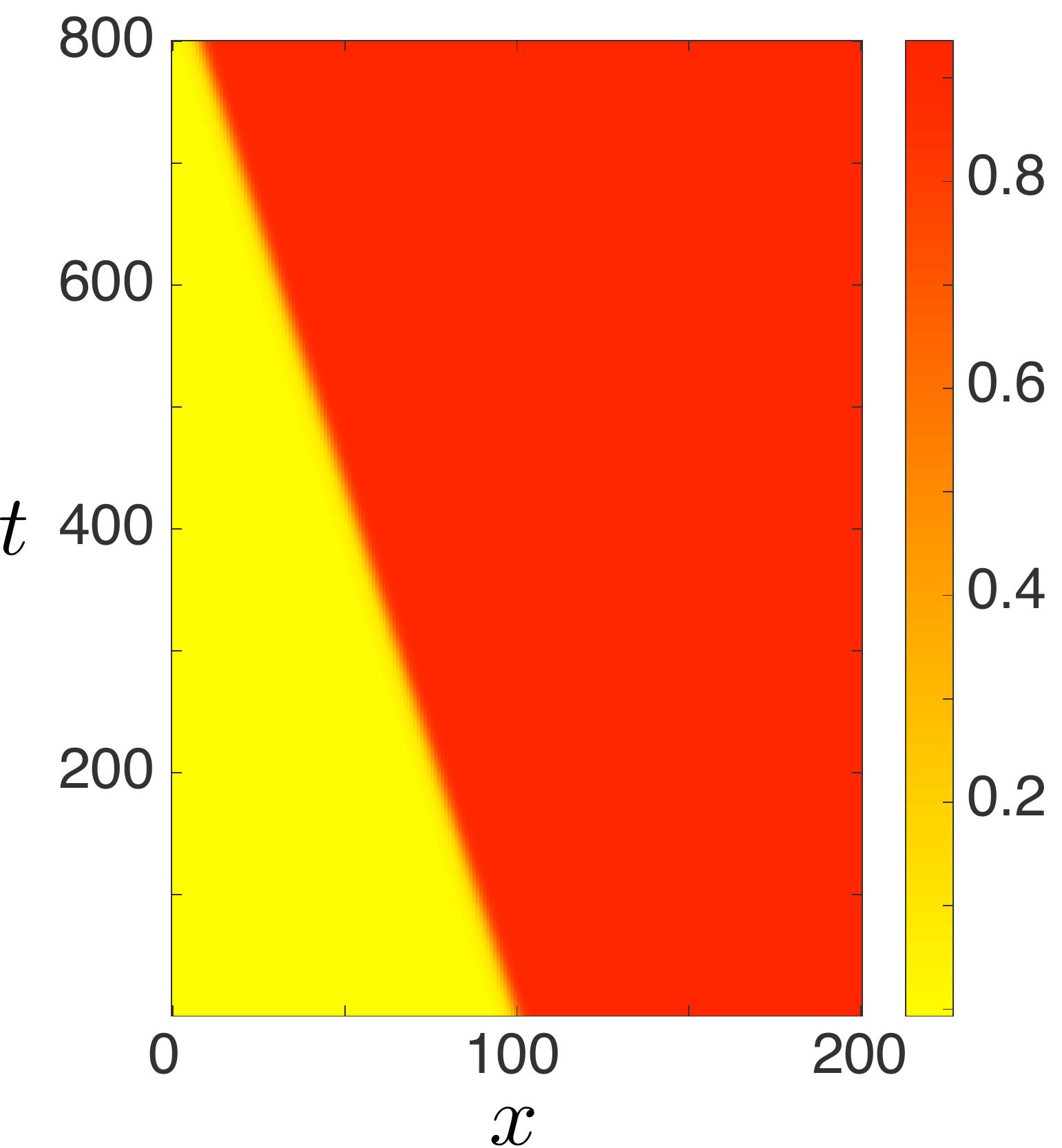}
		\caption{Desert front}
	\end{subfigure}

	\begin{subfigure}[t]{0.18\textwidth}
		\centering
			\includegraphics[width=\textwidth]{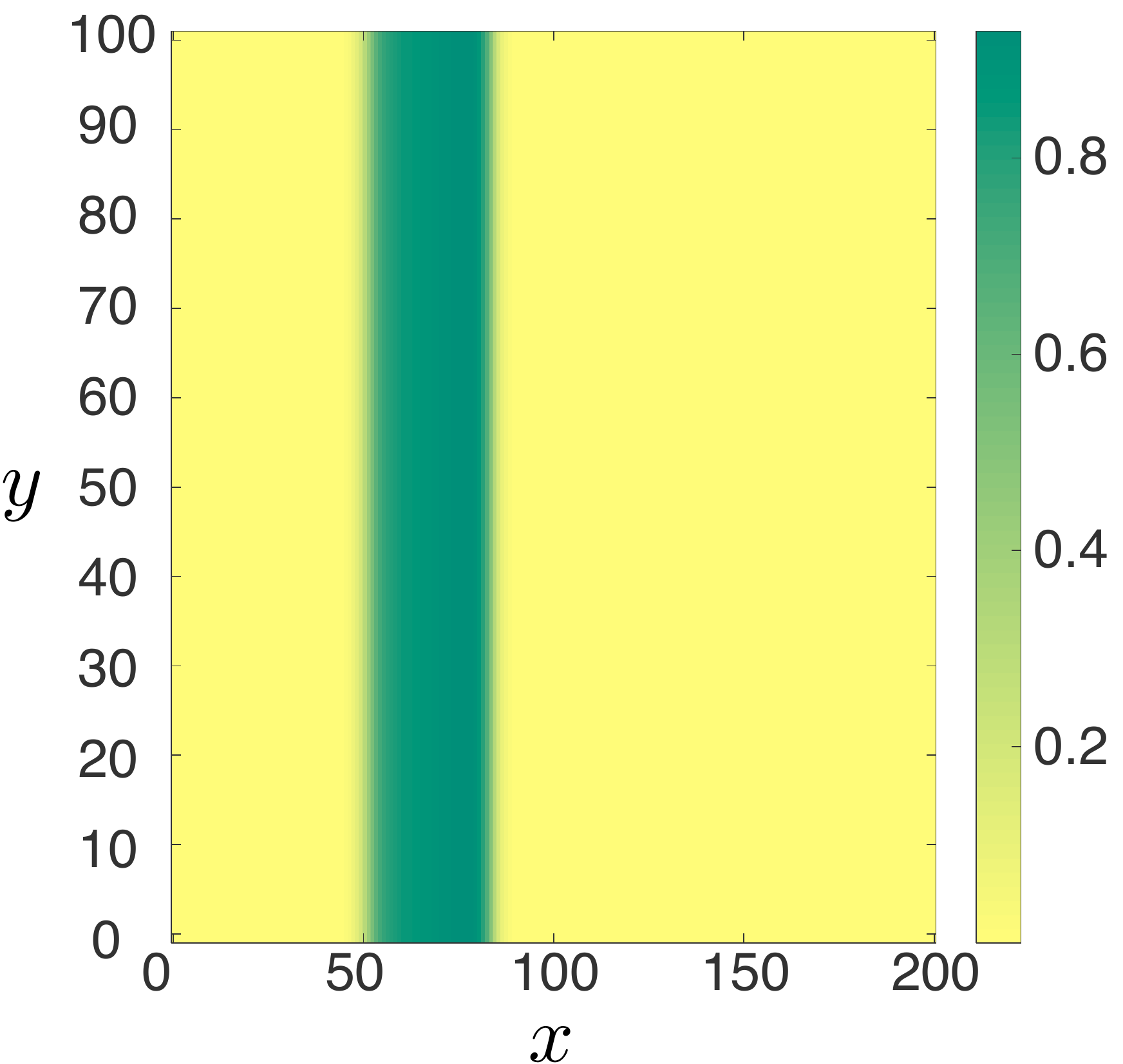}
		\caption{Stripe}
	\end{subfigure}
	\begin{subfigure}[t]{0.18\textwidth}
		\centering
			\includegraphics[width=\textwidth]{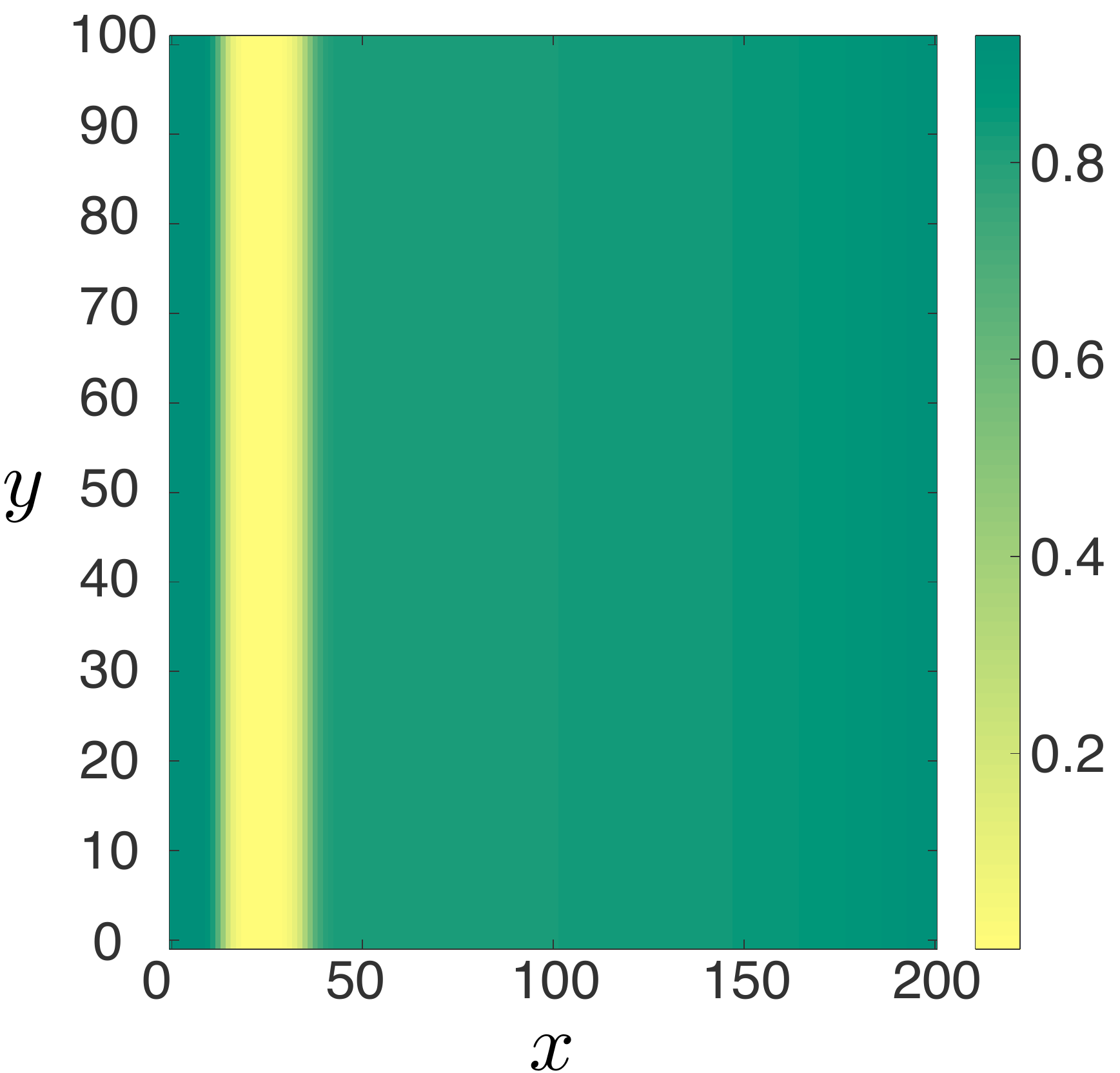}
		\caption{Gap}
	\end{subfigure}
	\begin{subfigure}[t]{0.18\textwidth}
		\centering
			\includegraphics[width=\textwidth]{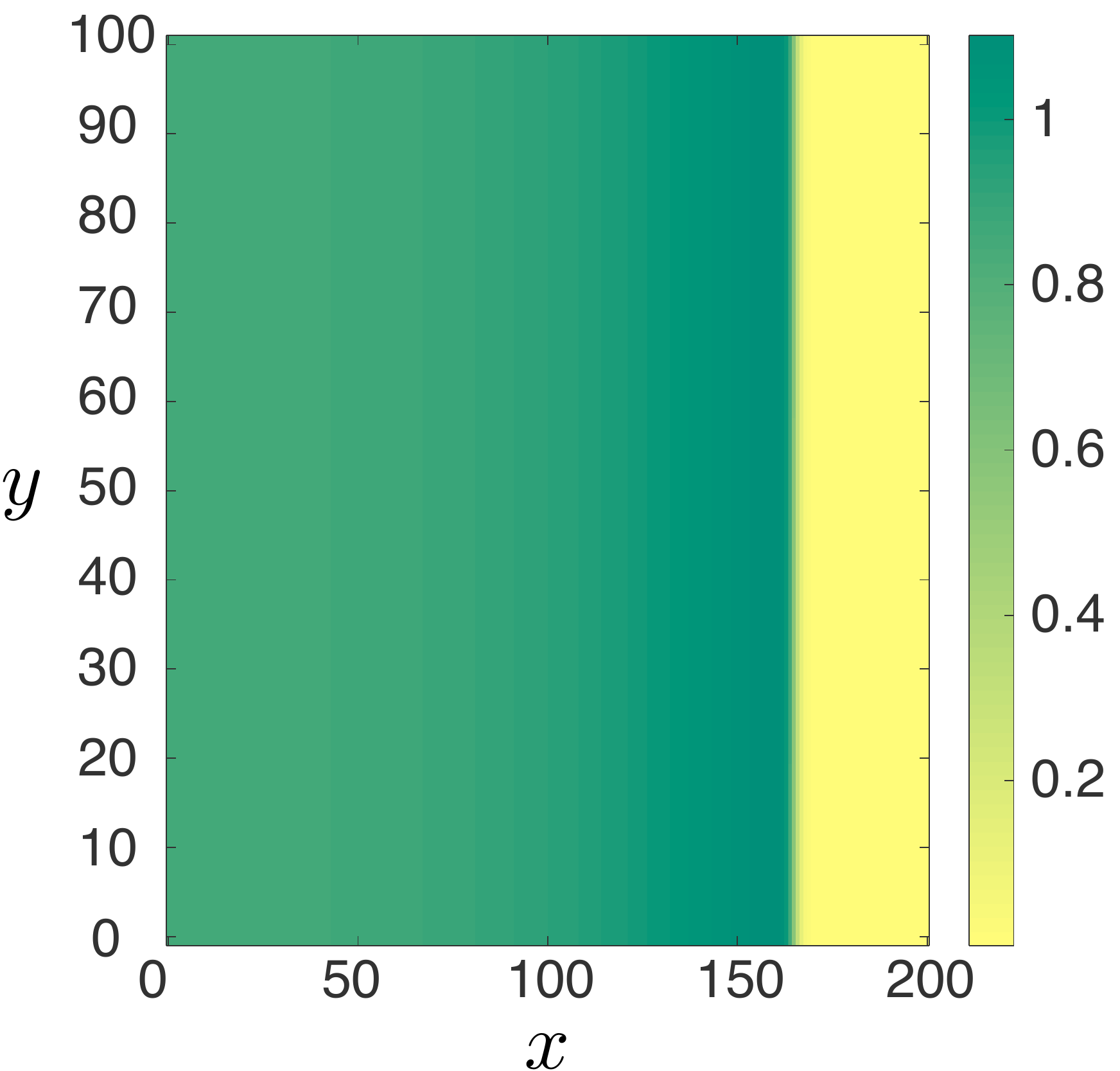}
		\caption{Vegetation front}
	\end{subfigure}
	\begin{subfigure}[t]{0.18\textwidth}
		\centering
			\includegraphics[width=\textwidth]{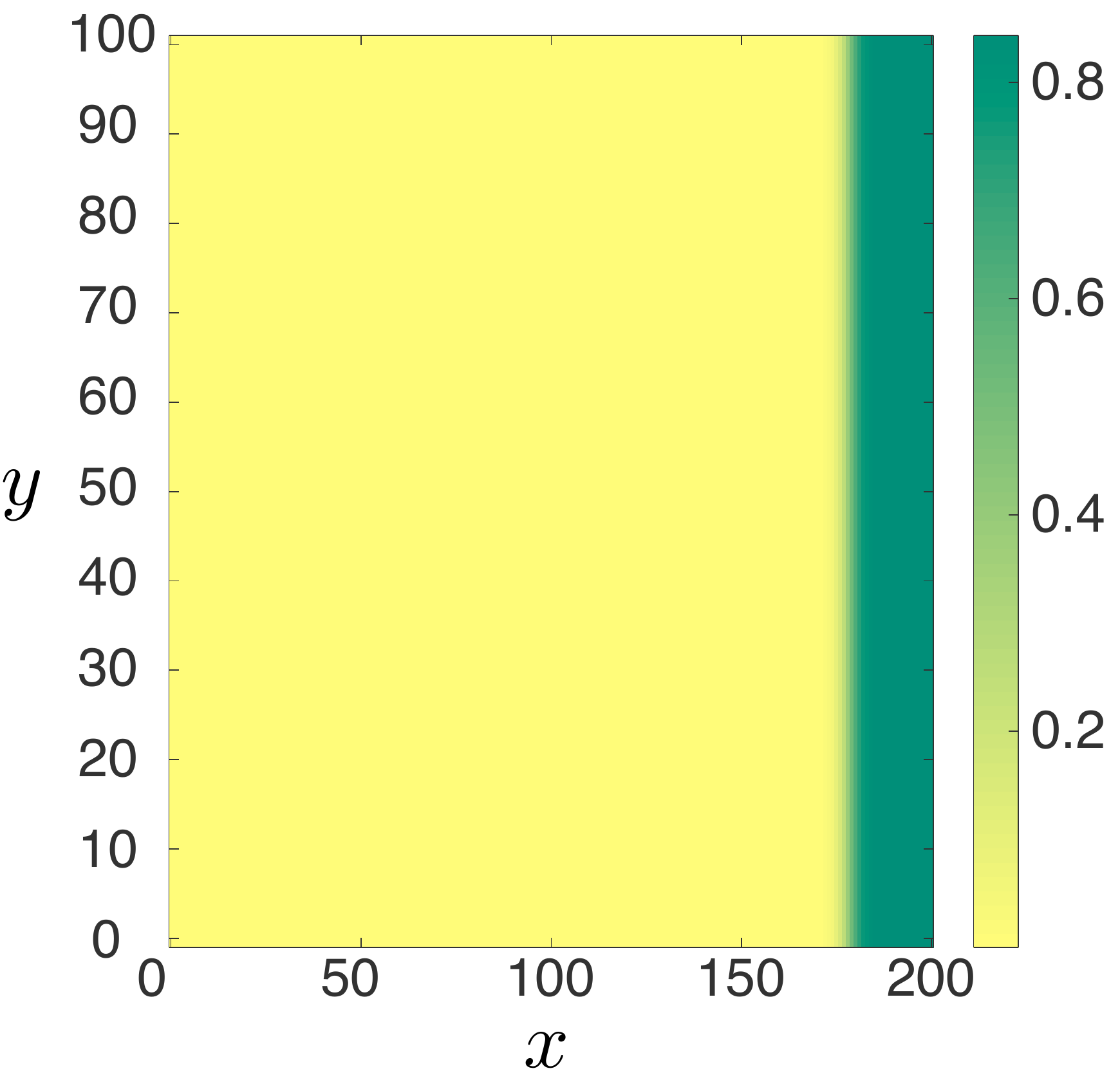}
		\caption{Desert front}
	\end{subfigure}
	\begin{subfigure}[t]{0.18\textwidth}
		\centering
			\includegraphics[width=\textwidth]{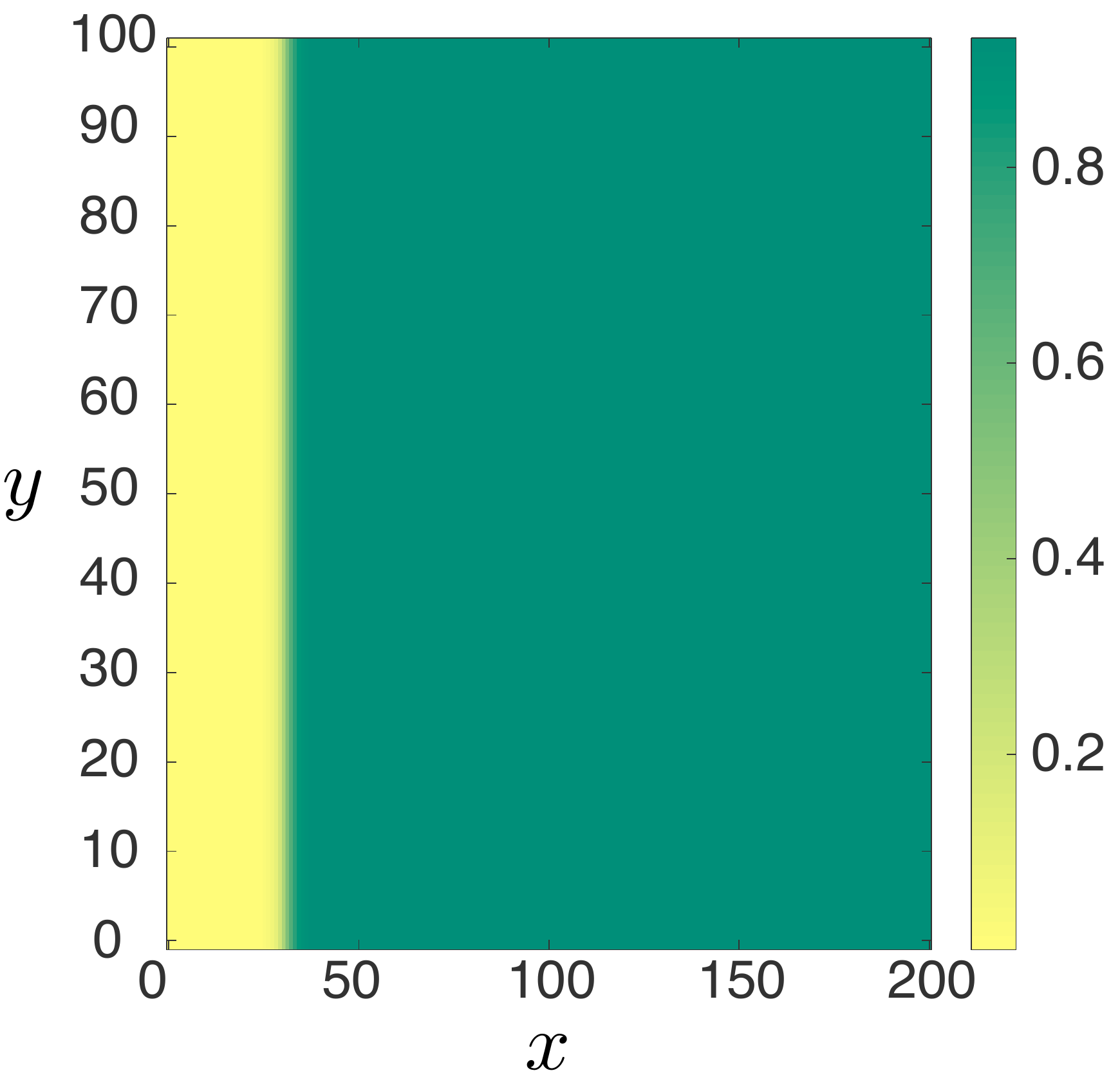}
		\caption{Desert front}
	\end{subfigure}
\caption{Results of direct numerical simulation of the PDE~\eqref{eq:modKlausmeier} for $b = 0.75$, $m = 0.45$, $\varepsilon = 0.01$ and $a = 1.75$ (a,f), $a = 2.4$ (b,g), $a = 2.5$ (c--d,h--i) or $a = 3.0$ (e,j). Figures a--e show the evolution of a cross section of $v$, i.e. for constant $y$ and figures f--j show the $v(x,y)$ pattern at a specific time. Simulations are run on a finite grid of size $L_x = 200$, $L_y = 100$, accompanied with Neumann boundary conditions for the $y$-direction and either periodic (a--b,f--g) or Neumann (c--e,h--j) boundary conditions in the $x$-direction.}
\label{fig:numericsStraightb0p75}
\end{figure}

\begin{figure}

\centering
	\begin{subfigure}[t]{0.18\textwidth}
		\centering
			\includegraphics[width=\textwidth]{Figures/numericsCornerB0p5Stripe}
		\caption{Stripe}
	\end{subfigure}
	\begin{subfigure}[t]{0.18\textwidth}
		\centering
			\includegraphics[width=\textwidth]{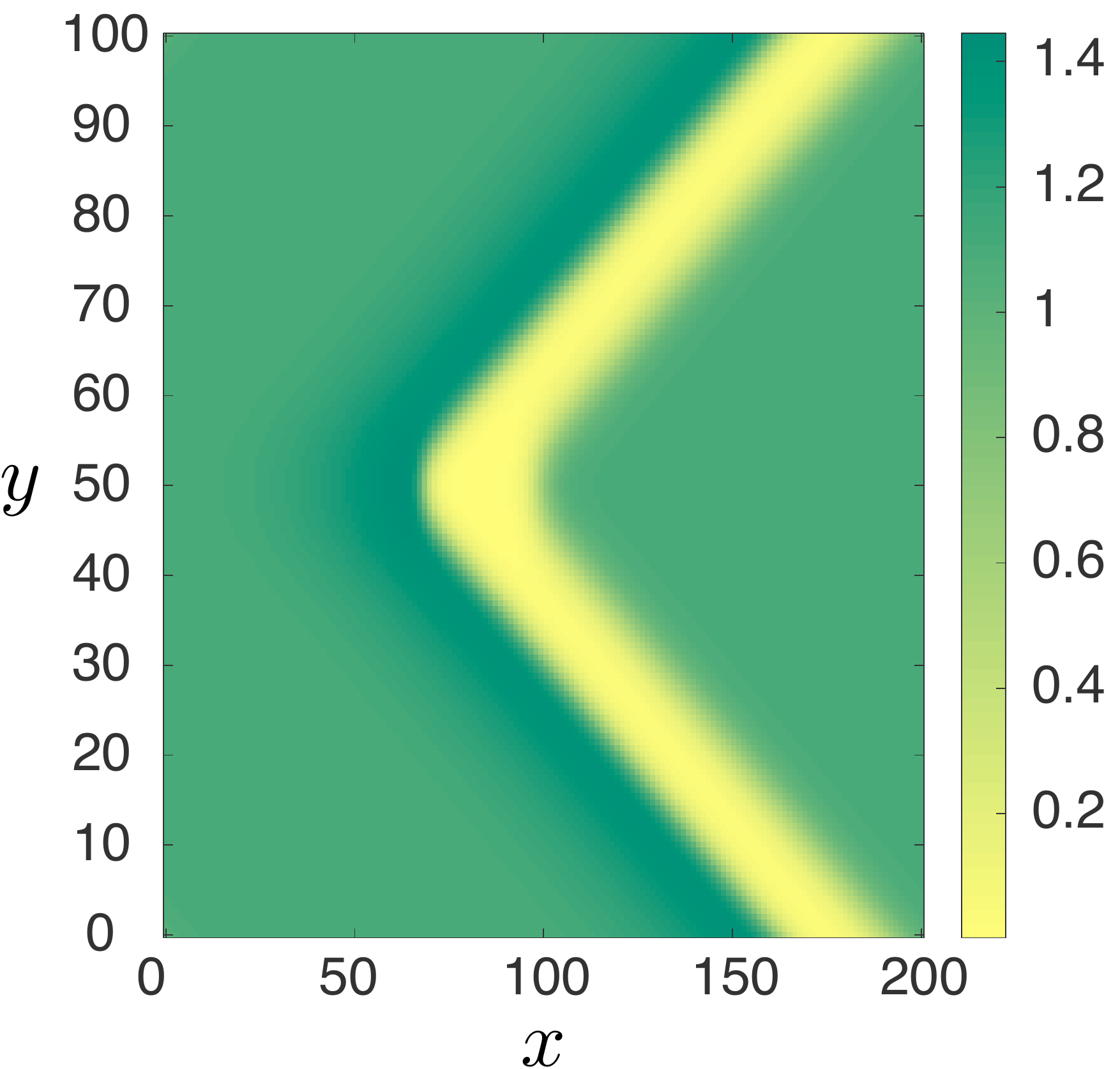}
		\caption{Gap}
	\end{subfigure}
	\begin{subfigure}[t]{0.18\textwidth}
		\centering
			\includegraphics[width=\textwidth]{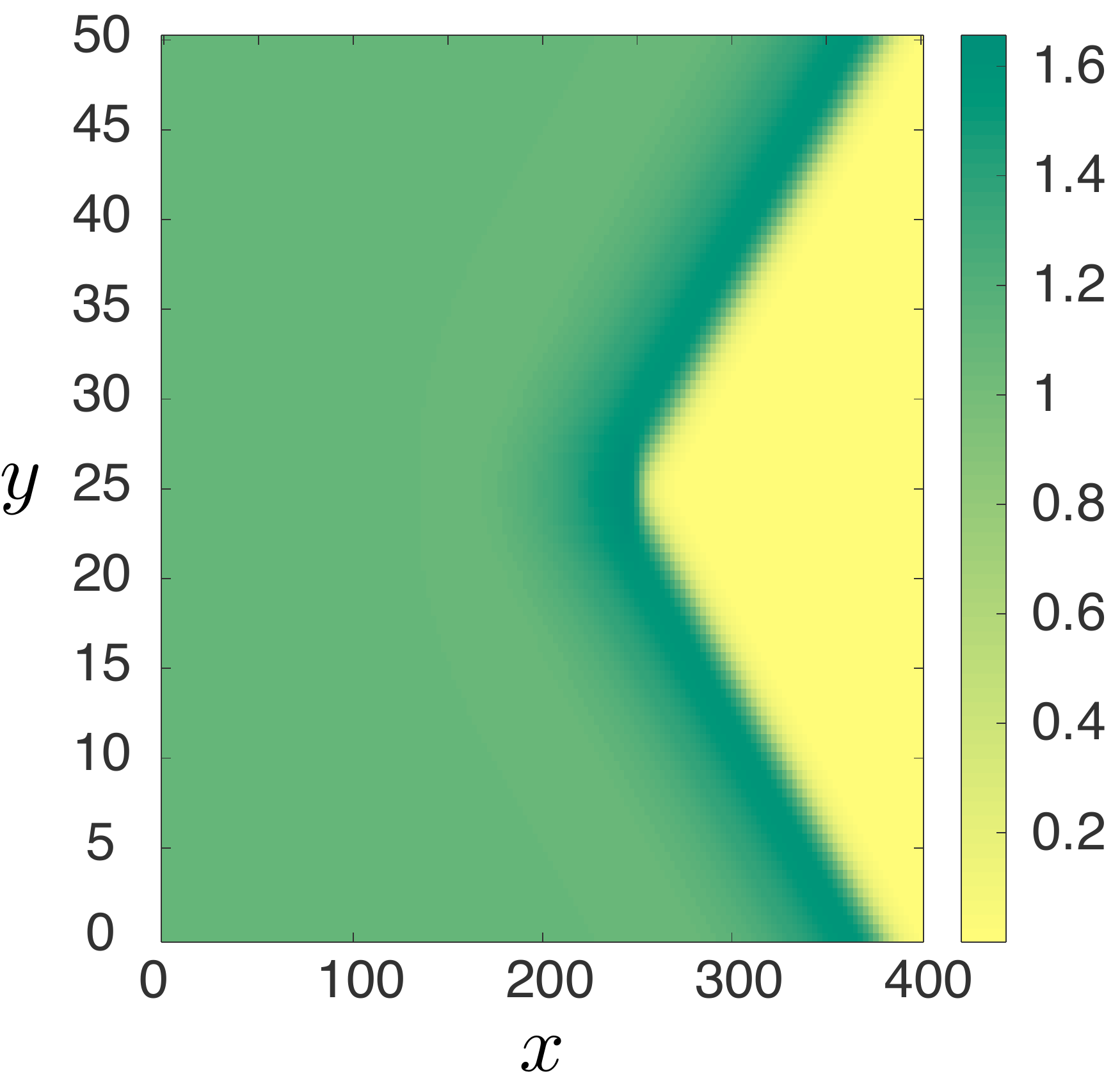}
		\caption{Vegetation front}
	\end{subfigure}
	\begin{subfigure}[t]{0.18\textwidth}
		\centering
			\includegraphics[width=\textwidth]{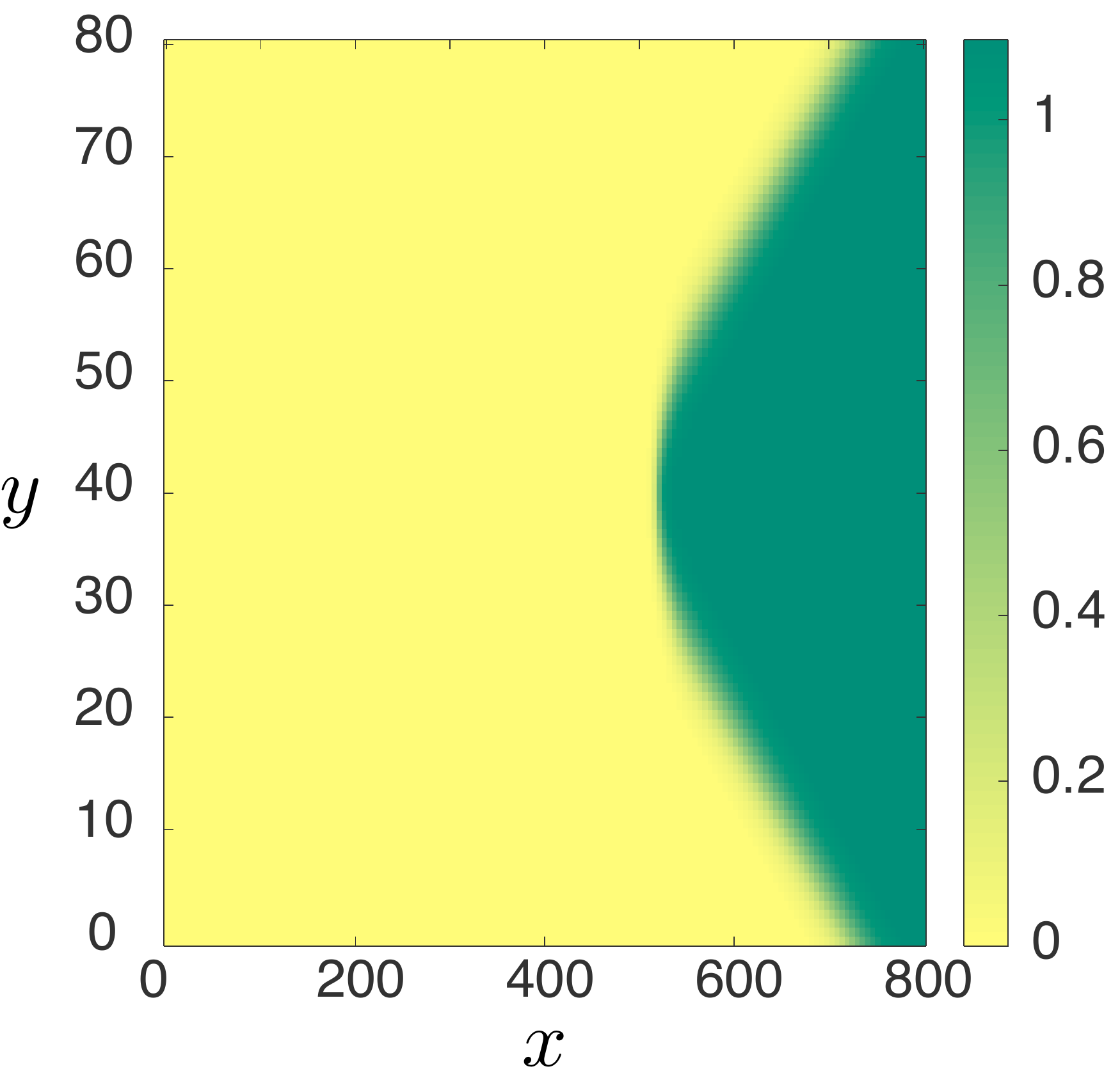}
		\caption{Desert front}
	\end{subfigure}
	\begin{subfigure}[t]{0.18\textwidth}
		\centering
			\includegraphics[width=\textwidth]{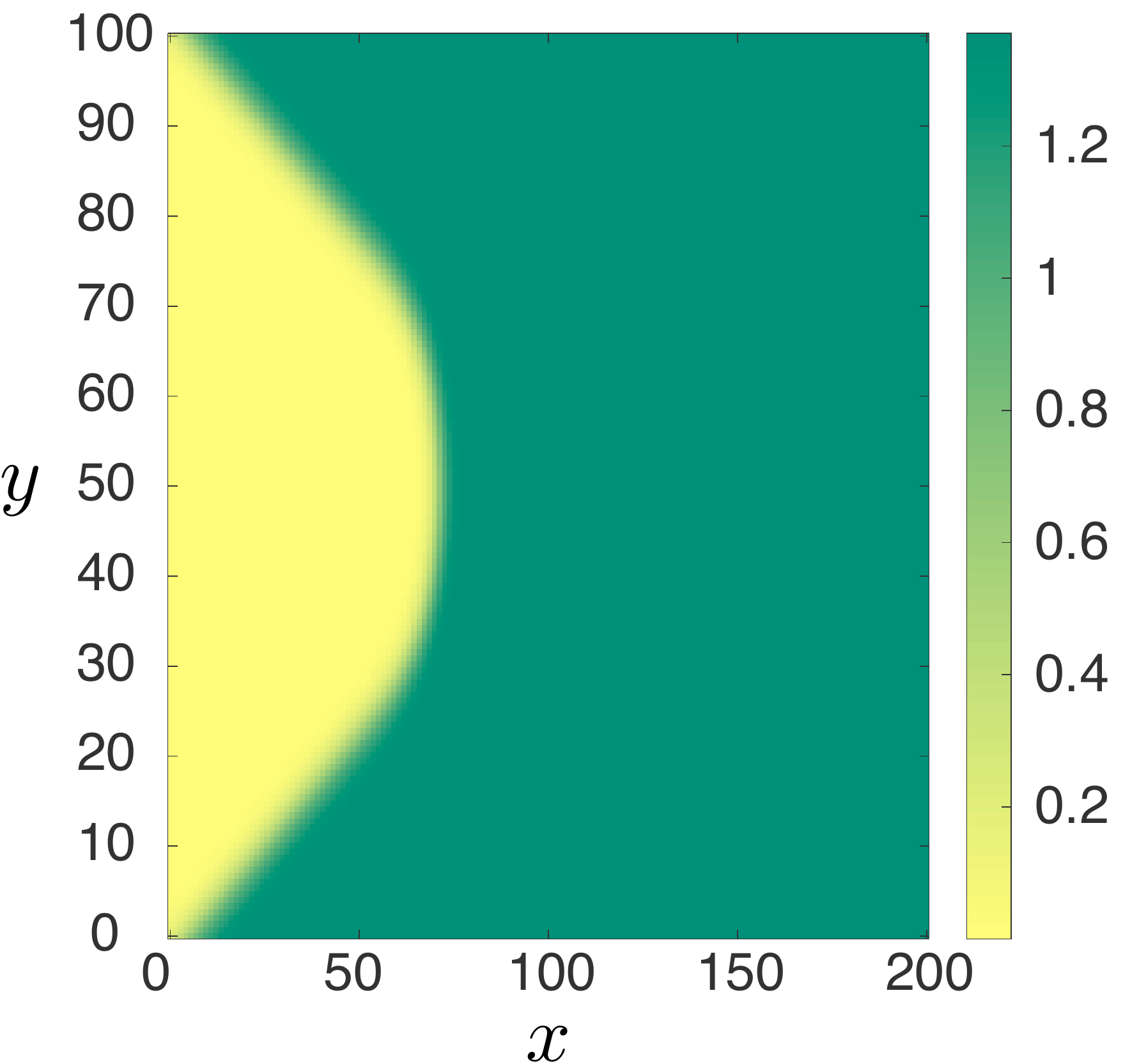}
		\caption{Desert front}
	\end{subfigure}

	\begin{subfigure}[t]{0.18\textwidth}
		\centering
			\includegraphics[width=\textwidth]{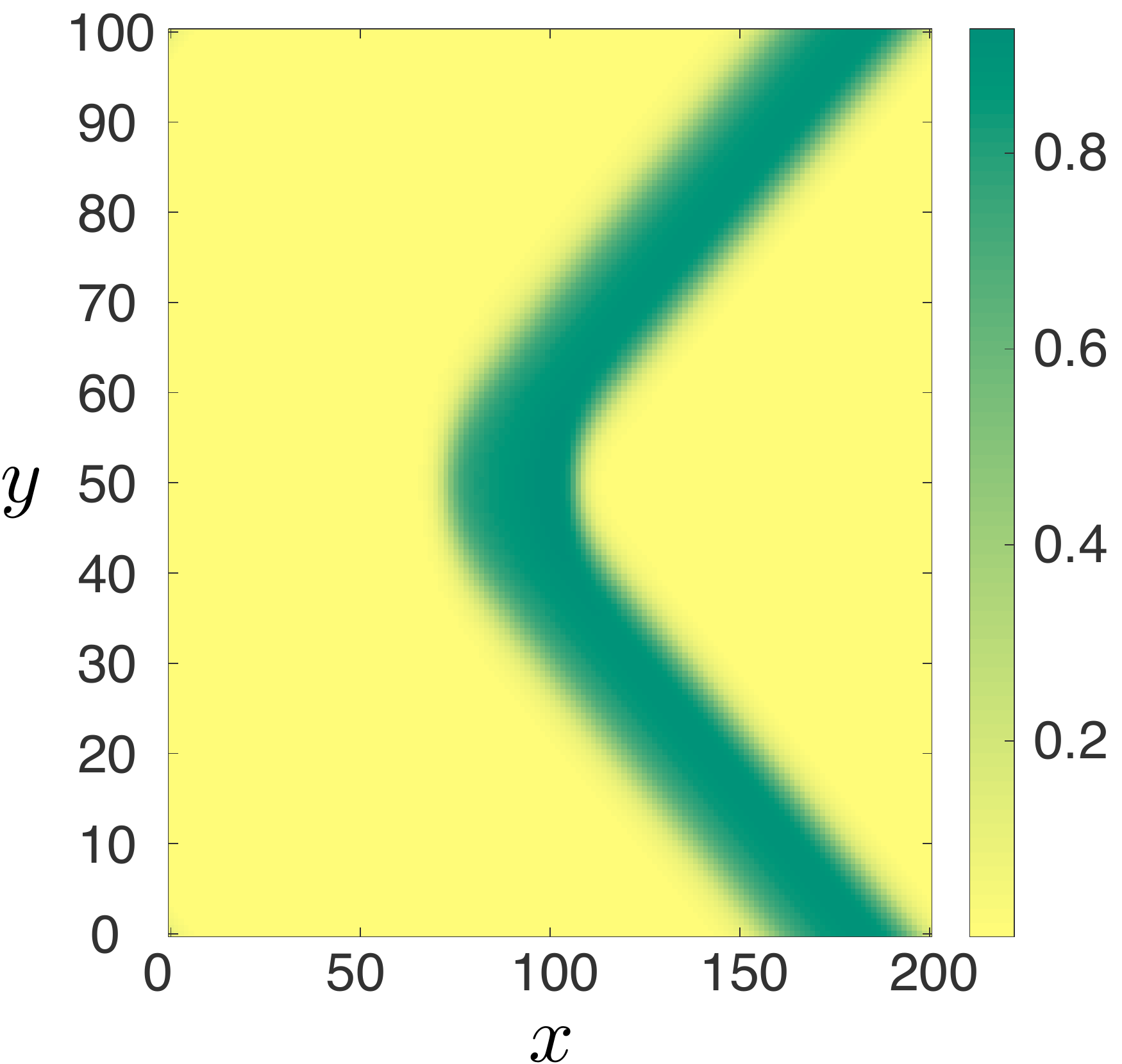}
		\caption{Stripe}
	\end{subfigure}
	\begin{subfigure}[t]{0.18\textwidth}
		\centering
			\includegraphics[width=\textwidth]{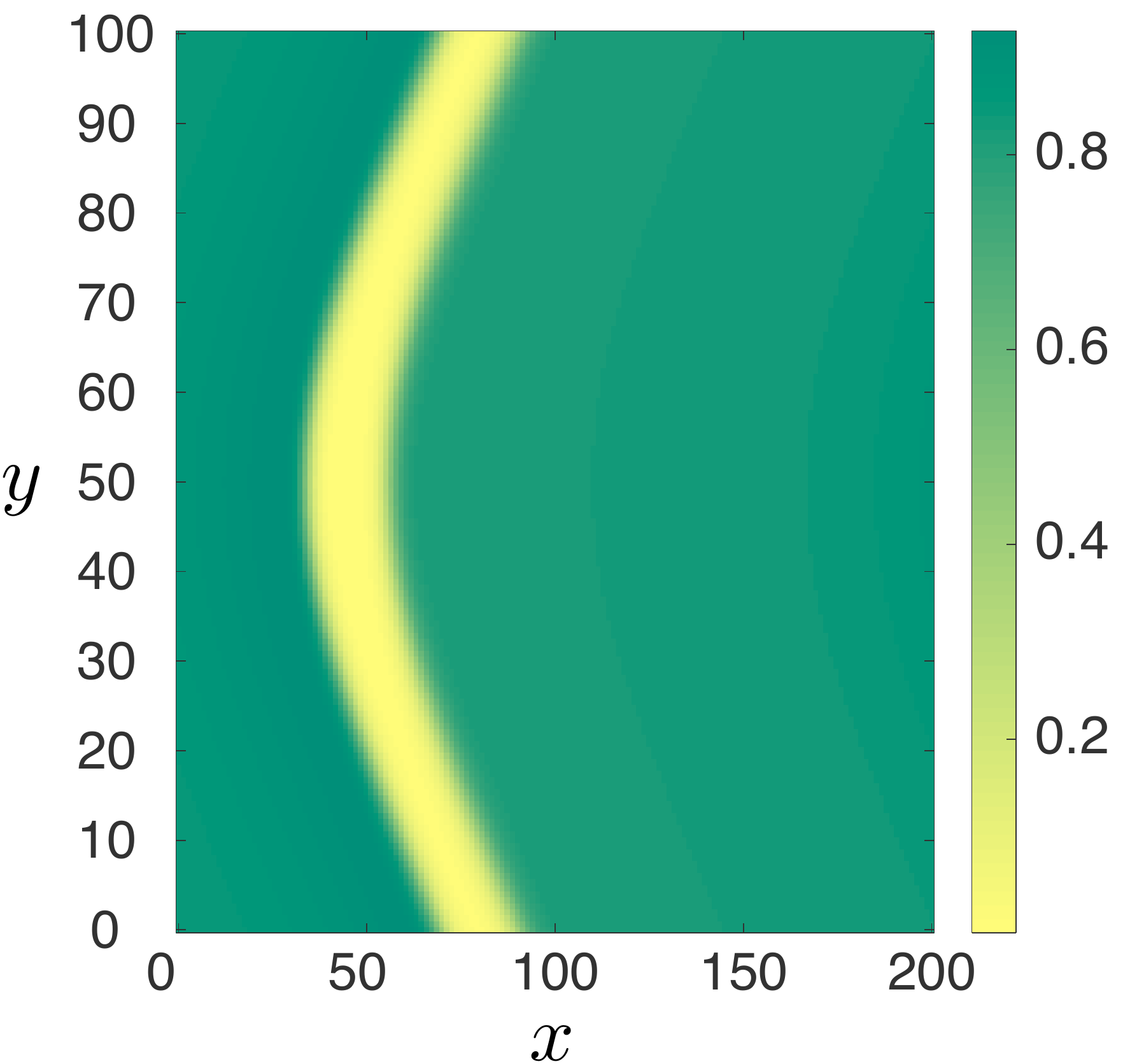}
		\caption{Gap}
	\end{subfigure}
	\begin{subfigure}[t]{0.18\textwidth}
		\centering
			\includegraphics[width=\textwidth]{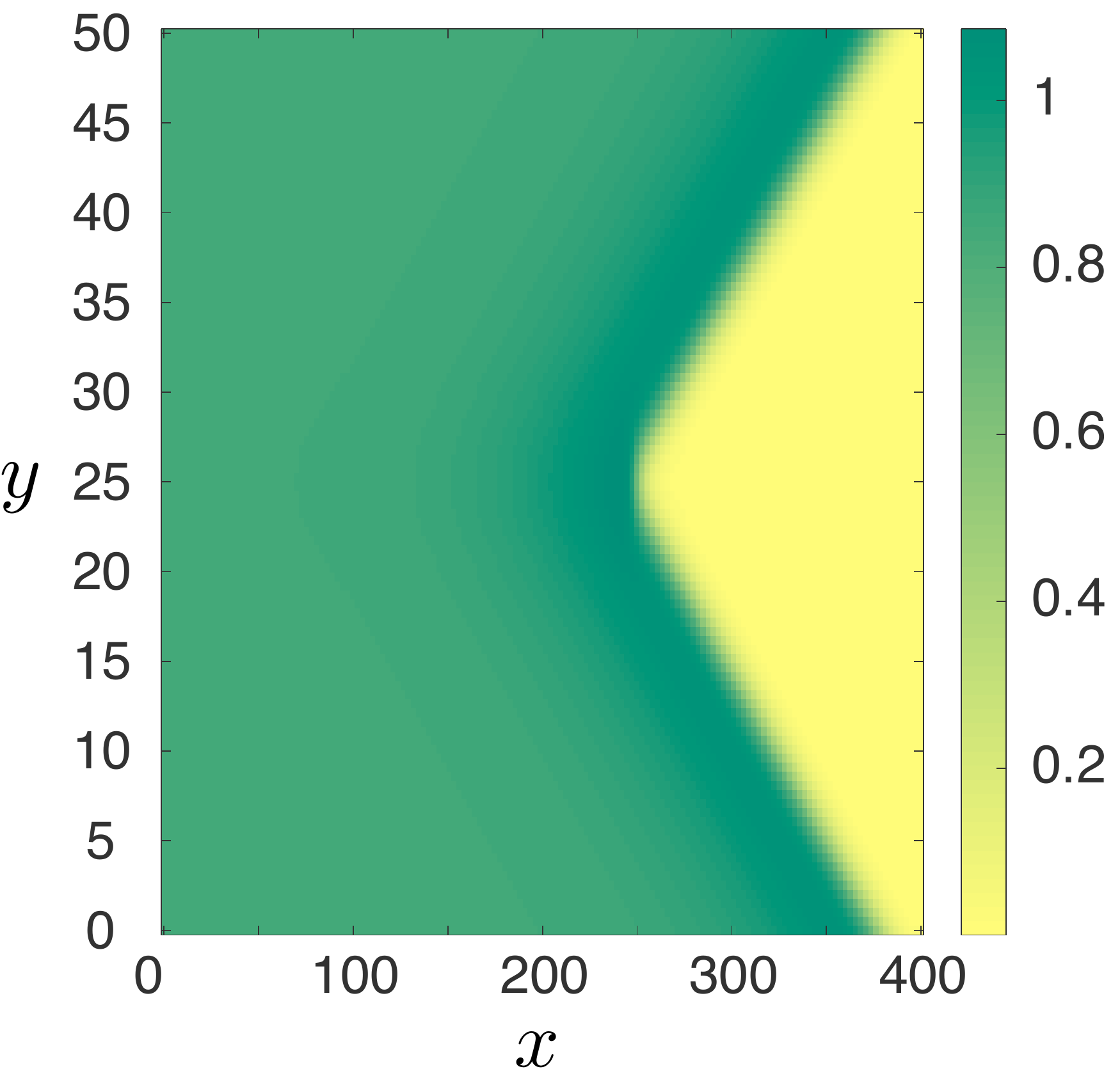}
		\caption{Vegetation front}
	\end{subfigure}
	\begin{subfigure}[t]{0.18\textwidth}
		\centering
			\includegraphics[width=\textwidth]{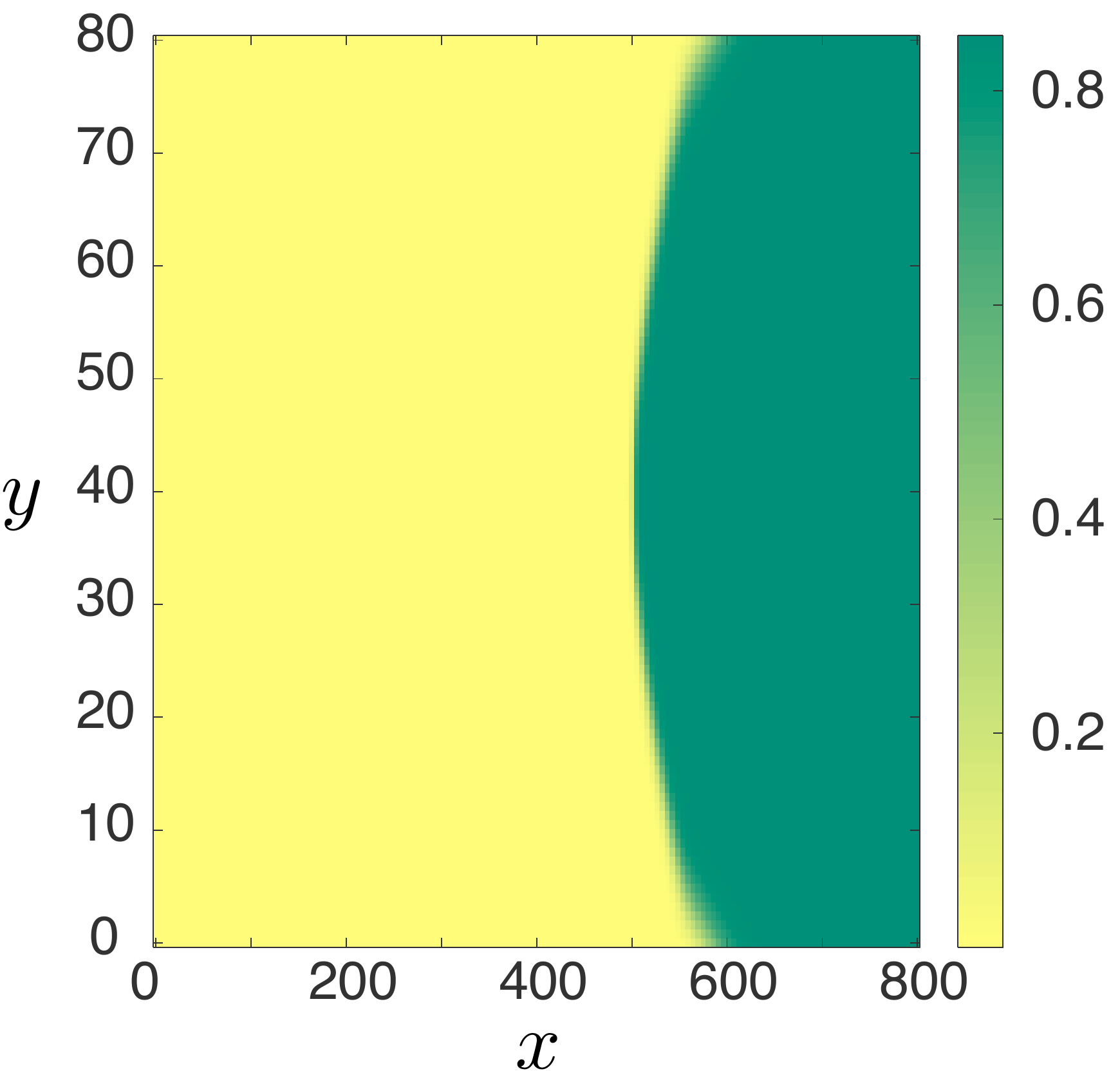}
		\caption{Desert front}
	\end{subfigure}
	\begin{subfigure}[t]{0.18\textwidth}
		\centering
			\includegraphics[width=\textwidth]{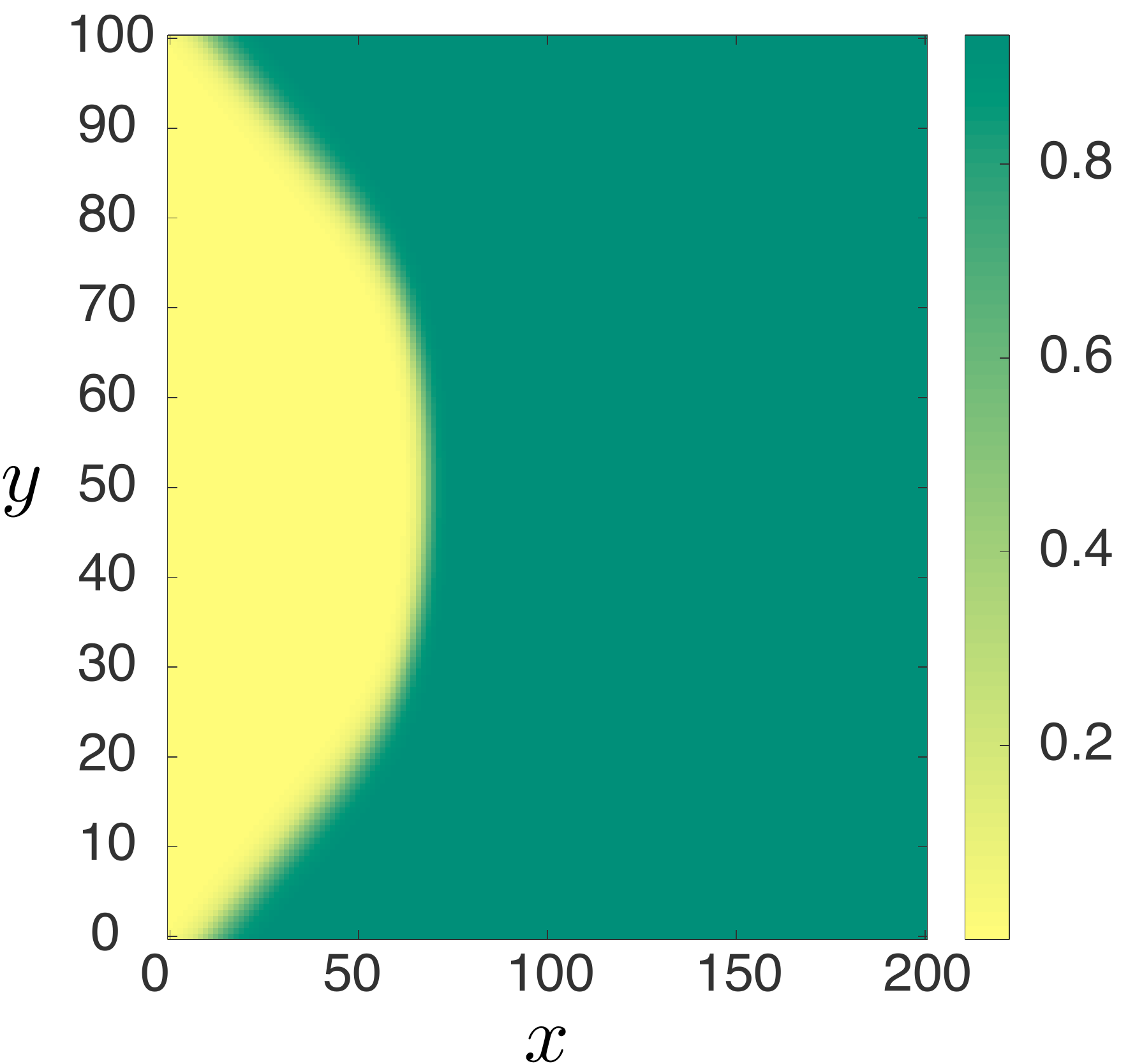}
		\caption{Desert front}
	\end{subfigure}
\caption{$v(x,y)$ configuration of corner solutions in direct numerical simulations of the PDE~\eqref{eq:modKlausmeier} for $m = 0.45$, $\varepsilon = 0.01$, $b = 0.5$ (a--e) or $b = 0.75$ (f--j) and various $a$-values. Simulations are done on a finite grid of various sizes, accompanied with either periodic boundary conditions (a--b,f--g) or Neumann boundary conditions (c--e,h--j) in the $x$-direction and the boundary conditions $v_y(x,L_y) + \alpha v_x(x,L_y) = 0$ and $v_y(x,0) - \alpha u_x(x,0)=0$ in the $y$-direction to accommodate corner solutions, with $\alpha = -1$ (a--d,f,h--i), $\alpha = -0.5$ (g) $\alpha = +1$ (e,j)}
\label{fig:numericsCorners}
\end{figure}

\section{Discussion}\label{sec:discussion}


In this paper we constructed planar traveling stripes, gaps and front-type solutions to the modified Klausmeier model~\eqref{eq:modKlausmeier}. We proved their existence rigorously using geometric singular perturbation methods for a wide range of system parameters $a,b,m$ in the large advection limit $\eps\to0$. We showed that vegetation stripes exist for smaller $a/m$ values, while vegetation gap patterns and front solutions can be found for larger values of $a/m$. For the largest $a/m$ values, stripes and gaps no longer persist, and we find only front-type solutions that correspond to invading vegetation. Contrary to the typical pulse patterns constructed in similar dryland models~\cite{sewaltspatially, BD2018}, the stripes and gaps found in~\eqref{eq:modKlausmeier} are not thin, but have sizable widths -- aligning better with observations of real dryland ecosystems~\cite{von2001diversity, Rietkerk2002, deblauwe2011environmental, gandhi2018influence}.

Furthermore, we showed that all such solutions are $2$D spectrally stable, using exponential dichotomies and Lin's method, based on similar stability analysis of traveling pulse solutions to the FitzHugh--Nagumo equations in~\cite{cdrs}.  We note that, to our knowledge, there are currently no direct results which guarantee nonlinear stability based on spectral stability of traveling wave solutions to~\eqref{eq:modKlausmeier}. Multidimensional nonlinear stability of traveling wave solutions in reaction-diffusion systems, however, has been studied previously~\cite{kapitula1997multidimensional}. By adding a small diffusion term, as in~\eqref{eq:modKlausmeier+D}, we obtain a system which fits into the framework of planar interface propagation studied in~\cite{HS2,HS1}. We expect our results still hold for~\eqref{eq:modKlausmeier+D} using a perturbation argument, provided $D\ll \eps \ll 1$. Further, results relating spectral and nonlinear stability have been found to hold in mixed parabolic-hyperbolic equations such as~\eqref{eq:modKlausmeier} for perturbations in one spatial dimension~\cite{rottmann2012stability}, and we expect that similar results may hold in higher dimensions.  

As far as we are aware, ours is the first construction of $2$D linearly stable traveling stripes in a reaction-diffusion-advection model of vegetation pattern formation. Typically in this class of models, one finds that stripe solutions are stable in $1$D, but destabilize for some range of (small) wavenumbers in $2$D~\cite{siero2015striped, sewaltspatially, doelman2002homoclinic, moyles2016explicitly}. We attribute this phenomenon to the stabilizing effect of the large advection term, as well as the destabilizing effect of water diffusion. By ignoring the diffusion of water and allowing the advection to dominate, the lateral competition for water resources is diminished, and $2$D stability can essentially be reduced to $1$D stability. This is reflected in our stability analysis in which the critical part of the $2$D spectrum is bounded to the left of the $1$D spectrum: In order to compute the $2$D spectrum, a Fourier decomposition in the transverse variable $y$ results in a family of $1$D eigenvalue problems parameterized by the transverse wavenumber $\ell$. These eigenvalue problems can then be solved using the methods of~\cite{cdrs}, and we find that eigenvalues occurring for $\ell\neq0$ can be bounded to the left of those occurring for $\ell=0$, corresponding to the $1$D spectrum. In fact we find that the correspondence is approximately $\lambda \to \lambda-\ell^2$.

An important question is how and why the addition of water diffusion and reduction in the magnitude of the advection term results in instabilities in the resulting patterns. This matches intuition, as water diffusion allows for lateral competition for water resources, which -- if sufficiently large -- could manifest in lateral instabilities. From the mathematical point of view, the onset of these instabilities is not well understood, though we note that one indeed finds lateral instabilities, both numerically and analytically, in similar models where both advection and diffusion are present~\cite{siero2015striped, sewaltspatially, doelman2002homoclinic, moyles2016explicitly}. A natural direction for future research lies in understanding this transition, and in particular the precise relation between the water diffusion and advection which determines the boundary for stability. This is likely to be challenging, given that the singular geometries in the advection-dominant case (as in this paper) versus the diffusion-dominant case are wholly distinct. The traveling wave solutions constructed in this work are all based off of singular fast front-type jumps between one-dimensional slow manifolds, much like one finds in the classic FitzHugh--Nagumo equation. However, typically in the diffusion-dominant regime traveling stripe solutions are constructed as perturbations of fast homoclinic orbits which depart and return to the same two-dimensional slow manifold in a four-dimensional singularly perturbed traveling wave equation~\cite{sewaltspatially, doelman2015explicit}. Hence, even the existence of stripe solutions in an intermediate regime is far from clear, as one must understand how the transition between these two geometries occurs. 

Also novel to our results are the implications for the appearance of curved solutions, even in the absence of terrain curvature. These arise as corner defect solutions~\cite{HS2,HS1}, which resemble two angled planar traveling wave solutions which meet along an interface. We find that the speed of the straight planar traveling wave predicts whether the associated corner solutions are oriented convex upslope or downslope. In particular, since all of the traveling stripe and gap solutions we constructed travel in the uphill direction, the corresponding curved stripes and gaps are oriented convex downslope. The planar front solutions, however, can be oriented either convex downslope or upslope depending on parameters. An interesting direction for future research lies in determining the effect of alternative topographies, in particular topographies which can be viewed as perturbations of constantly sloped terrain, which we expect can be studied using similar methods. A natural question is whether such topographies can destabilize stripe patterns or affect the curvature of these patterns. There are several numerical and observational results in this direction~\cite{gandhi2018influence}, but little is known analytically. A first analytical step towards this can be found in~\cite{BCBD}, in which the impact of non-trivial topographies on 1D stripe patterns is studied.

Finally, we remark on the implications of our results for Klausmeier's original equation~\cite{klausmeier1999regular}, which corresponds to infinite carrying capacity, or setting $b=0$ in~\eqref{eq:modKlausmeier}. As discussed in~\S\ref{sec:criticalmanifold} (see Remark~\ref{rem:b0degenerate}), the limit $b\to0$ is highly singular, and our results no longer hold in this regime. Existence of traveling stripes in this case has been obtained in~\cite{CDklausmeier} using geometric singular perturbation theory and blow-up methods to account for passage near a nonhyperbolic slow manifold. Pulse solutions in that setting consist of portions of two slow manifolds, along with a single fast jump. Stability, however, is not known; this is due to the fact that several rescalings and coordinate transformations are required to recover a slow-fast structure in the corresponding traveling wave equation. The result is that the associated reduced eigenvalue problem across the fast jump can no longer be interpreted in terms of the simpler scalar problem for the corresponding front as in~\S\ref{sec:regionr2}, which precludes the application of Sturm-Liouville theory. However, we expect stability to continue to hold in this regime. In particular, the existence of a single fast jump should result in one matching condition, and hence a single critical eigenvalue $\lambda=0$ due to translation invariance. This intuition supported by the fact that the second critical eigenvalue $\tilde{\lambda}_\mathrm{c}$ of Theorem~\ref{thm:matching} satisfies $\tilde{\lambda}_\mathrm{c}\to-\infty$, when naively taking the limit $b\to0$ for fixed $\eps$. Rigorous verification of the stability of traveling stripes in the Klausmeier equation is the subject of ongoing work.

\paragraph{Acknowledgments.} RB was supported by the Mathematics of Planet Earth program of the Netherlands Organization of Scientific Research (NWO). PC gratefully acknowledges support through NSF grant DMS--1815315. 

\appendix
\section{Stability of steady states}\label{sec:stabilitySteadyStates}

To understand the stability of steady states,~\eqref{eq:desertState} and\eqref{eq:uniformSteadyStates}, against homogeneous perturbations, we linearize~\eqref{eq:modKlausmeier} around the steady states by setting $(U,V)(x,t) = (U^*,V^*) + e^{\lambda t} (\bar{U},\bar{V})$, where $(U^*,V^*)$ is the steady state solution. For the desert-state $(U_0,V_0) = (a,0)$ this gives the linearized system
\begin{equation*}
	\lambda \left( \begin{array}{c} \bar{U} \\ \bar{V} \end{array} \right) = \left( \begin{array}{cc} -1 & 0 \\ 0 & -m \end{array} \right) \left( \begin{array}{c} \bar{U} \\ \bar{V} \end{array} \right).
\end{equation*}
Thus the corresponding eigenvalues are $\lambda = -1 < 0$ and $\lambda = -m < 0$. Both eigenvalues are negative and thus the desert-state $(U_0,V_0) = (a,0)$ is stable against homogeneous perturbations for all parameter values.

Linearization around the other steady states $(U_{1,2},V_{1,2})$ yields the eigenvalue problem
\begin{equation}
	\lambda \left( \begin{array}{c} \bar{U} \\ \bar{V} \end{array} \right) = M \left( \begin{array}{c} \bar{U} \\ \bar{V} \end{array} \right); \qquad M := \left( \begin{array}{cc} -1-V_{1,2}^2 & -2 U_{1,2} V_{1,2} \\ (1 - b V_{1,2}) V_{1,2}^2 & -m+(2-3bV_{1,2}) U_{1,2} V_{1,2} \end{array} \right). \label{eq:ussStabilityProblem}
\end{equation}
The determinant of the matrix on the right-hand side can be computed as
\begin{equation*}
	\det M = \frac{-1 + 2 b V_{1,2} + V_{1,2}^2}{1-b V_{1,2}}\ m.
\end{equation*}
From this, it can be found that the determinant is negative when $V_{1,2} < -b + \sqrt{1+b^2}$ and positive when $V_{1,2} > -b + \sqrt{1+b^2}$. Using~\eqref{eq:uniformSteadyStates}, one can readily obtain that $V_1 < -b + \sqrt{1+b^2}$ and $V_2 > -b + \sqrt{1+b^2}$. Hence the uniform steady state $(U_1,V_1)$ necessarily has a positive eigenvalue and therefore this steady state is unstable. To determine the stability for $(U_2,V_2)$ we need to determine the trace of the matrix $M$. Straightforward computation using the expressions~\eqref{eq:uniformSteadyStates} yields:
\begin{equation*}
	\mbox{Tr}\ \mathcal{M} = -1 - V_2^2 + m\frac{1-2bV_{1,2}} {1-bV_{2}},
\end{equation*}
which we note is always negative if $V_2>\frac{1}{2b}$, corresponding to the condition $\frac{a}{m}>4b+\frac{1}{b}$, and hence the state $(U_2,V_2)$ is stable to homogeneous perturbations in this regime.

\section{Absence of point spectrum in $R_2(\delta,M)$}\label{app:r2proof}

In this section, we complete the proof of Proposition~\ref{propR2}, and show that the region $R_2(\delta,M)$ contains no eigenvalues $\tilde{\lambda}$ of~\eqref{eq:Wevalproblem}.
\begin{proof}[Proof of Proposition~\ref{propR2}]
Following the argument outlined in~\ref{sec:regionr2}, we note that the translated derivative $e^{\eta \xi}\phi_j'(\xi)$ is an exponentially localized solution to~\eqref{eq:Wreducedinvariant2} at $\tilde{\lambda} = 0$, which admits no zeros. Therefore, by Sturm-Liouville theory~\cite[Theorem 2.3.3]{KAP},~\eqref{eq:Wreducedinvariant2} admits no bounded solutions for $\tilde{\lambda} \in R_2(\delta,M)$. Thus, for $\tilde{\lambda} \in R_2(\delta,M)$~\eqref{eq:Wreducedinvariant2} admits an exponential dichotomy on $\R$ with constants $C, \mu > 0$  independent of $\tilde{\lambda}\in R_2(\delta,M)$. Exploiting the lower triangular structure of system~\eqref{eq:Wevalproblemreduced2} the exponential dichotomy of~\eqref{eq:Wreducedinvariant2} can be extended to the  system~\eqref{eq:Wevalproblemreduced2} using variation of constants formulae. We denote the corresponding projections by $Q_{j}^{\mathrm{u},\mathrm{s}}(\xi;\tilde{\lambda})$ for $j = \dagger,\diamond$.

We now consider the eigenvalue problem~\eqref{eq:Wevalproblem} as a perturbation of~\eqref{eq:Wevalproblemreduced}. By Theorem~\ref{thm:existencebounds}, we have that
\begin{align}
\begin{split}
|A_\eta(\xi;\tilde{\lambda},\ell,\eps) - A_{\dagger,\eta}(\xi;\tilde{\lambda})| &=\mathcal{O}(\eps |\!\log \eps |), \quad \xi \in [-L_\eps,L_\eps],\\
|A_\eta(Z_\eps+\xi;\tilde{\lambda},\ell,\eps) - A_{\diamond,\eta}(\xi,\tilde{\lambda})| &=\mathcal{O}(\eps |\!\log \eps |), \quad \xi \in [-L_\eps,\infty).
\end{split}\label{eq:AboundR2}\end{align}
 Denote by $P_j^{\mathrm{u},\mathrm{s}}(\tilde{\lambda})$ the spectral projection onto the (un)stable eigenspace of the asymptotic matrices $A_{j,\eta}^{\pm\infty}(\tilde{\lambda})=\lim_{\xi\to\pm \infty}A_{j,\eta}(\xi;\tilde{\lambda})$ of~\eqref{eq:Wevalproblemreduced}. We note that $ A_{j,\eta}(\xi;\tilde{\lambda})$ converges at an exponential rate to the asymptotic matrix $A_{j,\eta}^{\infty}(\tilde{\lambda})$ as $\xi \to \infty$. Hence, the projections $Q_{j}^{\mathrm{u},\mathrm{s}}(\pm \xi,\tilde{\lambda})$ satisfy
\begin{align} \|Q_{j}^{\mathrm{u},\mathrm{s}}(\pm \xi,\tilde{\lambda}) - P_{j,\pm}^{\mathrm{u},\mathrm{s}}(\tilde{\lambda})\| \leq Ce^{-\tilde{\mu}\xi}, \quad j = \dagger,\diamond, \label{eq:spectralN}\end{align}
for $\xi \geq 0$ for some $\tilde{\mu}>0$ (see for instance~\cite[Lemma 3.4]{PAL}). Using~\eqref{eq:AboundR2} and roughness~\cite[Theorem 2]{COP2}, we obtain exponential dichotomies for~\eqref{eq:Wevalproblem} on $I_\dagger$ and $I_\diamond$ with constants $C,\tfrac{\mu}{2} > 0$ independent of $\tilde{\lambda}\in R_2(\delta,M)$ and projections $\mathcal{Q}_{j}^{\mathrm{u},\mathrm{s}}(\xi;\tilde{\lambda},\eps)$, which satisfy
\begin{align}
\begin{split}
\|\mathcal{Q}_{j}^{\mathrm{u},\mathrm{s}}(\xi;\tilde{\lambda},\eps) - Q_\dagger^{\mathrm{u},\mathrm{s}}(\xi,\tilde{\lambda})\| &\leq C\eps|\!\log \eps |, \\
\|\mathcal{Q}_{j}^{\mathrm{u},\mathrm{s}}(Z_\eps+\xi;\tilde{\lambda},\eps) - Q_\diamond^{\mathrm{u},\mathrm{s}}(\xi,\tilde{\lambda})\| &\leq C\eps|\!\log \eps |,
\end{split}
\label{eq:dichbound}\end{align}
for $|\xi| \leq L_\eps$.

By Proposition~\ref{prop:slowexpdich} system~\eqref{eq:Wevalproblem} admits exponential dichotomies on $I_\ell = (-\infty,L_\eps]$ and $I_r = [L_\eps,Z_{a,\eps} - L_\eps]$  with projections $\mathcal{Q}_{r,\ell}^{\mathrm{u},\mathrm{s}}(\xi;\tilde{\lambda},\eps)$, which satisfy
\begin{align}
\begin{split}
\left\|[\mathcal{Q}_{\ell}^\mathrm{s} - \mathcal{P}](-L_\eps;\tilde{\lambda},\eps)\right\|, \left\|[\mathcal{Q}_{r}^\mathrm{s} - \mathcal{P}](Z_{a,\eps} - L_\eps;\tilde{\lambda},\eps)\right\|, \left\|[\mathcal{Q}_{r}^\mathrm{s} - \mathcal{P}](Z_{a,\eps} + L_\eps;\tilde{\lambda},\eps)\right\| &\leq C\eps|\!\log \eps |,
\end{split} \label{eq:slowprojbound}
\end{align}
where $\mathcal{P}(\xi;\tilde{\lambda},\eps)$ denotes the spectral projection onto the stable eigenspace of $A_\eta(\xi;\tilde{\lambda},\ell,\eps)$.

We now compare the exponential dichotomies for~\eqref{eq:Wevalproblem} constructed on each of the intervals $I_\ell, I_\dagger,I_r,I_\diamond$ at the endpoints of the intervals. Recall that $A_{j,\eta}(\xi;\tilde{\lambda})$ converges at an exponential rate to the asymptotic matrix $A_{j,\eta}^{\pm\infty}(\tilde{\lambda})$ as $\xi \to \pm \infty$ for $j = \dagger,\diamond$. Recalling~\eqref{eq:AboundR2}, we have that
\begin{align*}
|A_\eta(\pm L_\eps;\tilde{\lambda},\ell,\eps) - A_{\dagger,\eta}^{\pm \infty}(\tilde{\lambda})|,  |A_\eta(Z_\eps+L_\eps;\tilde{\lambda},\ell, \eps) - A_{\diamond,\eta}^{-\infty}(\tilde{\lambda})| &\leq C\eps|\!\log \eps |.\end{align*}
By continuity the same bound holds for the spectral projections associated with these matrices. Combining this with~\eqref{eq:spectralN}--\eqref{eq:slowprojbound} we obtain
\begin{align}
\begin{split}
\left\|[\mathcal{Q}_{\ell}^{\mathrm{u},\mathrm{s}} - \mathcal{Q}_\dagger^{\mathrm{u},\mathrm{s}}](-L_\eps;\tilde{\lambda},\eps)\right\|, \left\|[\mathcal{Q}_{r}^{\mathrm{u},\mathrm{s}} - \mathcal{Q}_\dagger^{\mathrm{u},\mathrm{s}}](L_\eps;\tilde{\lambda},\eps)\right\|, \left\|[\mathcal{Q}_{r}^{\mathrm{u},\mathrm{s}} - \mathcal{Q}_\diamond^{\mathrm{u},\mathrm{s}}](Z_\eps-L_\eps;\tilde{\lambda},\eps) \right\| &\leq C\eps|\!\log \eps |.
\end{split} \label{eq:slowboundR2}
\end{align}

Let $\psi(\xi)$ be an exponentially localized solution to~\eqref{eq:Wevalproblem} at some $\tilde{\lambda} \in R_2(\delta,M)$. This implies $\mathcal{Q}_\ell^\mathrm{s}(-L_\eps;\tilde{\lambda},\eps)\psi(-L_\eps) = 0$. By for instance~\cite[Lemma 6.10]{HOLZ} or~\cite[Lemma 6.19]{cdrs}, we have that
\begin{align} |\mathcal{Q}_{r}^\mathrm{s}(L_\eps;\tilde{\lambda},\eps)\psi(L_\eps)| \leq C\eps|\!\log \eps ||\mathcal{Q}_{r}^\mathrm{u}(L_\eps;\tilde{\lambda},\eps)\psi(L_\eps)|, \label{elephant2}\end{align}
using~\eqref{eq:slowboundR2}. Again using~\cite[Lemma 6.19]{cdrs} and~\eqref{eq:slowboundR2} to obtain a similar inequality at the endpoint $Z_{a,\eps} - L_\eps$, we obtain
\begin{align*}  |\mathcal{Q}_{\ell}^\mathrm{s}(Z_{a,\eps} - L_\eps;\tilde{\lambda},\eps)\psi(Z_{a,\eps} - L_\eps)|| \leq C\eps|\!\log \eps ||\mathcal{Q}_{\ell}^\mathrm{u}(Z_{a,\eps} - L_\eps;\tilde{\lambda},\eps)\psi(Z_{a,\eps} - L_\eps)| = 0,\end{align*}
since we assumed $\psi(\xi)$ is exponentially localized. Hence, any exponentially localized solution $\psi(\xi)$ to~\eqref{eq:Wevalproblem} is the trivial solution.\end{proof}

\bibliographystyle{abbrv}
\bibliography{bibfile}

\end{document}